\documentclass[11pt,a4paper]{article}
\usepackage{enumitem}
\usepackage[utf8]{inputenc}
\usepackage{amsmath,amsfonts,amssymb}
\usepackage{amsthm,upref}
\usepackage[a4paper]{geometry}
\usepackage{color} 
\usepackage[dvipsnames]{xcolor}
\usepackage{graphicx}
\usepackage{subcaption}
\usepackage{tabularx,float}
\usepackage[hidelinks]{hyperref}
\usepackage{comment} 
\usepackage{tabularx} 

\newcommand\rsp[1]{{\color{black}{#1}}}
\newcommand\rspp[1]{{\color{black}{#1}}}
\newtheorem{thm}{Theorem}[section]

\newtheorem{definition}[thm]{Definition}

\newtheorem{proposition}[thm]{Proposition}

\newtheorem{lemma}[thm]{Lemma}
\newtheorem{remark}[thm]{Remark}

\newtheorem{conj}[thm]{\rsp{Conjecture}}

\newcommand{\e}{{\mathrm e}}

\usepackage{todonotes}

\newcommand{\secref}[1]{Section~\ref{sec:#1}}

\newcommand{\seclab}[1]{\label{sec:#1}}
\newcommand{\eqlab}[1]{\label{eq:#1}}
\renewcommand{\eqref}[1]{(\ref{eq:#1})}

\newcommand{\figref}[1]{Fig.~\ref{fig:#1}}
\newcommand{\figlab}[1]{\label{fig:#1}}
\newcommand{\propref}[1]{Proposition~\ref{proposition:#1}}
\newcommand{\proplab}[1]{\label{proposition:#1}}
\newcommand{\conjref}[1]{\rsp{Conjecture~\ref{conj:#1}}}
\newcommand{\conjlab}[1]{\label{conj:#1}}

\newcommand{\lemmaref}[1]{Lemma~\ref{lemma:#1}}
\newcommand{\lemmalab}[1]{\label{lemma:#1}}
\newcommand{\remref}[1]{Remark~\ref{remark:#1}}
\newcommand{\remlab}[1]{\label{remark:#1}}
\newcommand{\thmref}[1]{Theorem~\ref{theorem:#1}}
\newcommand{\thmlab}[1]{\label{theorem:#1}}

\newcommand{\defnlab}[1]{\label{defn:#1}}
\newcommand{\defnref}[1]{Definition~\ref{defn:#1}}

\newcommand{\argmax}{\operatorname{argmax}}
\newcommand{\md}{{\mathbf{d}}}
\newcommand{\degree}{\operatorname{deg}}
\newcommand{\conv}{\mathbf{conv}}
\newcommand{\Conv}{\mathbf{trop}}
\newcommand{\trop}{\operatorname{nc}}
\newcommand{\sign}{\operatorname{sign}}
\newcommand{\Fil}{\operatorname{Fp}}

\newcommand{\TDS}{\mathcal T\mathcal D\mathcal S(N,\{\md_k\}_{k\in \mathcal I})}
\newcommand{\TDSN}{\mathcal T\mathcal D\mathcal S}
\usepackage{tikz}

\title{On a tropicalization of planar polynomial ODEs with a 
finite number of structurally stable phase portraits} 
\author {Kristiansen, K. U. and Sarantaris, A. H.} 
\date{%
Department of Applied Mathematics and Computer Science, \\
Technical University of Denmark, \\
2800 Kgs. Lyngby, \\
Denmark,\\
krkri@dtu.dk\\
\vspace{0.5cm}
\today
}
\begin{document}
 
\maketitle
%
%
%
 \begin{abstract}
Recently, concepts from the emerging field of tropical geometry have been used to identify different scaling regimes in chemical reaction networks where dimension reduction may take place. In this paper, we try to formalize these ideas further in the context of planar polynomial ODEs. In particular, we develop a theory of a tropical dynamical system, based upon a differential inclusion, that has a set of discontinuities on a subset of the associated tropical curve. We define a phase portrait, a notion of equivalence and characterize structural stability for a tropical dynamical system. Our main result is that there are finitely many equivalence classes of structurally stable phase portraits.  We believe that the property of finitely many structurally stable phase portraits underlines the potential of the tropical approach, also in higher dimension, as a method to (systematically) obtain and identify skeleton models in chemical reaction networks in extreme parameter regimes.
 \end{abstract} 
\tableofcontents
\section{Introduction}
Polynomial systems of ordinary differential equations (ODEs) occur in many different areas of science. Perhaps most notably they occur as models of chemical reaction networks \cite{feinberg2019a} and in compartmental models for pharmacology, epidemiology and ecology \cite{brauer2012a}. Moreover,
(the second part of) Hilbert’s 16th problem  \cite{ilyashenko2002a,li2003a} – which remains unresolved to this day – is on the existence of an upper bound $H_N$ of the number of limit cycles for planar polynomial
ODEs of a fixed degree $N\in \mathbb N$, $N\ge 2$:
\begin{equation}\eqlab{xyPQ}
\begin{aligned}
 x' &= P(x,y),\\
 y'&=Q(x,y).
\end{aligned}
\end{equation}
These problems are nonlocal (or simply global). For example, in general chemical reaction network, we are interested in the long term dynamics (stationary points, periodic solutions, or even chaos in dimensions $\ge 3$), and although bifurcation theory and numerical computations work well locally and for given parameters, global problems in generality of e.g. \eqref{xyPQ}, are (currently) out of reach.

Tropical geometry is a variant of algebraic geometry based on the semiring where the
usual definitions of addition and multiplication of real numbers are replaced by the following operations 
\begin{align}
 u\oplus v := \max\,(u,v),\quad  u\star v: = u+v,\quad u,v\in \mathbb R\cup\{-\infty\},\eqlab{tropop}
\end{align}
see \cite{maclagan2015introduction}. Tropical monomials $F_{i,j}(u,v)=\alpha_{i,j} \star u^{\star i} \star v^{\star j}=\alpha_{i,j}+iu+jv$ in two variables $u$ and $v$ are therefore affine functions. On the other hand, tropical polynomials are piecewise linear $F_{\max}(u,v)=\max_{i,j} F_{i,j}(u,v)$, much like what one observes when putting the graph
of a classical polynomial on a logarithmic paper. This connection between the tropical semi-ring and logarithmic paper can be formalized through:
\begin{align}\eqlab{logpaper}
\lim_{\epsilon\to 0} \epsilon \log \left(e^{u \epsilon^{-1}}+e^{v\epsilon^{-1}}\right)=u\oplus v,
\end{align}
see \cite{viro2001a}, and relates to the Maslov's dequantization of the real numbers in physics, see \cite{litvinov2007a,litvinov1996a}. The tropical curve of a tropical polynomial $F_{\max}$ is the set of points $(u,v)$ where the maximum is attained by at least two monomials. 

Tropical geometry has in recent years both established
itself as an area of its own right – with tropical solutions of polynomial equations forming a kind
of skeleton for the classical ones based upon standard arithmetic \cite{MR1925796,viro2001a} – but also unveiled deep
connections to numerous branches of pure and applied mathematics \cite{baker2016a,MR2508011,maclagan2015introduction,DBLP:conf/casc/SamalGFR15}. Perhaps most noticeably, O. Viro's \textit{gluing method} in algebraic geometry, which allows algebraic curves to be constructed by a ``cut and paste'' approach, see \cite{MR1036837}, has deep connections to (and even inspired the development of) tropical geometry. 
In \cite{itenberg2000a}, the authors developed a modification of Viro's gluing method to construct  polynomial systems \eqref{xyPQ} of degree $N$ with an order of
\begin{align}\eqlab{lowerboundHN}
\approx \frac{N^2 \log N}{2 \log 2},
\end{align}
as $N\to \infty$,
many limit cycles. To the best of our knowledge, this is the best known asymptotic lower bound of $H_N$, see also \cite{MR4053594}. At the same time, Newton polygons have also been used in dynamical systems theory for the local study of singularities, see \cite{braaksma1992a,dumortier2006a}.
%
%
%

Recently, see e.g. \cite{desoeuvres2022a,kruff2021a,noel2012a,soliman2014a}, tropical geometry has been used to develop computational tools for the analysis and identification of scaling regimes of chemical-reaction networks where dimension reduction may take place. The basis for this approach is related to \eqref{logpaper} and the fact that reduction (in a first approximation) takes place where dominating terms of opposite sign balance out. In tropical terms, the balancing occurs along (a subset of) the tropical curve. This reduction approach can be applied successively, eliminating one dimension at a time and can (potentially) be interpreted (and justified) as slow manifolds in multiple time scale systems, see \cite{desoeuvres2022a,kruff2021a}. 
%
Consequently, it offers an attractive way to deal with the challenges posed by complex, high-dimensional bio-chemical models that involve numerous parameters varying over different scales, making general dynamical systems analysis (including numerical analysis) difficult. Having manageable skeleton models is highly desirable in such cases, see \cite{noel2012a}.
While further theoretical development is necessary, we are taking an initial step in this paper \rsp{by considering} \textit{tropicalized} planar polynomial systems.  
We believe that this serves as a natural starting point for further development of the theory.

Our approach takes as its starting point the setting introduced \rspp{to the first author} by P. Szmolyan, which is partially fleshed out in the Diploma thesis \cite{portisch2021novel} by S. Portisch, wherein \eqref{xyPQ} is embedded into a certain $\epsilon$-family and where $(u,v)$ is introduced through \begin{align}
u=\epsilon \log x,&\quad v=\epsilon \log y.\eqlab{uv}
\end{align}
Taking the limit $\epsilon\to 0$, using \eqref{logpaper} and a certain reparametrization of time, we obtain a piecewise smooth (PWS) system in the $(u,v)$-plane, with the tropical curve containing the discontinuities. 
\rsp{We believe that the idea of tropicalizing polynomial ODE systems was first introduced in \cite{noel2012a}. However, their definition is slightly different from ours (as we for example only work with $(u,v)$) and was motivated by the Litvinov-Maslov corresponding principle \cite{litvinov1996a} rather than (as in the present case) through a formal limiting process of \eqref{xyPQ}.}

PWS systems are important in many different application areas, including mechanics (impact, friction, etc), electronics (switches and diodes, etc.), and control engineering (sliding mode control, etc.), see \cite{pwsbook,jeffrey2018a}. At the same time, PWS systems also pose mathematical problems. For one thing, they do not (in general) define a (classical) dynamical system, due to the lack of uniqueness of solutions. During the past decades, there has therefore been an effort to extend the theory of dynamical systems to PWS systems. This includes the work by \cite{broucke2001a,Sotomayor96} to  extend Peixito's program of structural stability, see also \cite{guardia2011a}. 
Parallel to this effort, there has been an attempt to bound the number of limit cycles in planar PWS linear systems, see \cite{esteban2021a,li2021a,llibre2013a}. \rspp{A uniform bound on the number} of (crossing) limit cycles for such systems with a line of discontinuities was recently \rspp{obtained}  in \cite{carmona2023a}, using the novel characterization of the transition maps in \cite{carmona2021a}.

Finally, recent research has explored how phenomena such as folds, grazing, and boundary singularities in PWS systems can be understood in their smooth version. To achieve this, researchers have refined methods from Geometric Singular Perturbation Theory (GSPT) and blowup, to resolve the special singular limit of smooth systems approaching PWS ones,  see \cite{HUZAK202334,jelbart2021c,jelbart2021b,kristiansen2020a,kristiansen2015a,kristiansen2018a}. 
 These methods \rsp{are also relevant} in the present context when connecting the tropical dynamical system back to \eqref{xyPQ} for $\epsilon>0$ small enough. \rsp{In fact, this was the main focus of \cite{portisch2021novel} in the context of the planar Michaelis-Menten model of enzyme reactions. It was found that an attracting one-dimensional manifold exists across a wide range of parameters. In contrast, in the present paper \textit{we will only analyze (and develop) the singular limit ($\epsilon=0$) case}, and leave this connection to $\epsilon>0$ to future work.
} 
 
 \subsection{Summary of the main results}
 \rsp{
In order to obtain a well-defined system for $\epsilon=0$, we first introduce a novel differential inclusion, that extends Filippov's sliding method, see \cite{filippov1988differential}, in a meaningful way to our setting. As indicated by \eqref{uv}, we will only consider the first quadrant $(x,y)\in \mathbb R_+^2,\,\mathbb R_+:=(0,\infty)$. The other quadrants can be handled in a similar way. We then prove existence of solutions of the differential inclusion (see \thmref{thm0}) and introduce a new notion of equivalence for our tropical dynamical system that we use to classify the structurally stable systems. 
In our main result (see \thmref{thm2}), we find that for any degree $N$ there are finitely many equivalence classes of structurally stable phase portraits. We do not enumerate or put a bound on these in the present work, but we demonstrate our approach by counting the number ($15$ in total) of equivalence classes of structurally stable phase portraits for a generalized autocatalator model (see \propref{genauto}). More generally, we believe that our results demonstrate that the global problems offered by \eqref{xyPQ} become more accessible in their tropical version. }

\subsection{Overview}
The paper is organized as follows: In \secref{trop}, we first review some basic concepts of tropical geometry. Then in \secref{tropode}, we introduce the concept of a tropical (polynomial) dynamical system, see \defnref{tropsystem}. It is described as a differential inclusion with a discontinuity set of the underlying piecewise smooth system as a subset of the tropical curve. 
In \secref{existence}, we \rspp{prove existence of piecewise affine solutions}, which in \secref{orbits} lead us to a definition of orbits (as polygonal curves) and a definition of a phase portrait. In \secref{auto}, we illustrate the concepts on the autocatalator model from \cite{Gucwa2009783}. 

Subsequently, in \secref{equivalence}, we define a (new) notion of equivalence of tropical dynamical systems. In the following sections, see \secref{troppoint}--\secref{crossing}, we then classify structural stability (locally) of \rsp{tropical vertices}, \rsp{tropical singularities}, separatrices and crossing cycles, respectively.  Subsequently, in \secref{graph} we associate a graph (the crossing graph) with a tropical dynamical system. Next, in \secref{structurallystable}, \rspp{we put everything together (including the crossing graph in \secref{graph} that will play a crucial role) to} state and prove our main results (see \thmref{thm2}) on the structural stability of tropical dynamical systems. 
\rsp{We demonstrate these results in the context of a generalized autocatalator model in \secref{genauto}}. 
We conclude the paper in \secref{discussion} with a discussion section. Here we also state some perturbation results (delaying proofs to future work) that relate the tropical dynamical system with solutions of the classical version \eqref{xyPQ} for all $0<\epsilon\ll 1$ and discuss directions for future work. 

\subsection{Tropical Phase Plane}\seclab{TPP}
To generate the tropical phase portraits, we have used the \verb#Matlab# toolbox \verb#Tropical# \verb#Phase# \verb#Plane#  developed by the second author in his bachelor thesis, see \cite{andreas}. The toolbox can be downloaded from 

\begin{center}
\url{https://www.mathworks.com/matlabcentral/fileexchange/132058-tropical-dynamics-toolbox}.
\end{center}
\section{Tropical geometry}\seclab{trop}
Tropical geometry \cite{maclagan2015introduction,morrison2020a} is a variant of algebraic geometry where addition and multiplication are replaced by $\max$ and addition, respectively, see \eqref{tropop},
using $-\infty+u=-\infty$ for every $u\in \mathbb R \cup\{-\infty\}$. \eqref{tropop} defines a semi-ring, where $u=-\infty$ is the neutral element of addition and $u=0$ is the neutral element with respect to multiplication. Notice that there is no additive inverse: Given $u$ and $w$ such that $u>w\ge -\infty$, then the equation $u\oplus v = w$ has no solution $v$. Moreover, $\oplus$ is idempotent: $u\oplus u=u$ for all $u\in \mathbb R\cup\{-\infty\}$.

\begin{remark}\remlab{open}
As usual, $o\subset \mathbb R\cup \{-\infty\}$ is open if either (a) $o$ is open in $\mathbb R$ or (b) $o=K^c\cup\{-\infty\}$ with $K^c$ the complement of a compact set $K\subset \mathbb R$. 
\end{remark}

In tropical geometry, classical monomials $$\alpha_{i,j} u^i v^j, \quad \alpha_{i,j}\in \mathbb R,\, i,j\in \mathbb N,$$ become affine functions:
\begin{align}
 F_{i,j}(u,v)=\alpha_{i,j}+iu+jv,\quad (u,v)\in \mathbb R^2.\eqlab{tropicalmonomial}
\end{align}
We will call $\alpha_{i,j}\in \mathbb R\cup \{-\infty\}$ the \textit{tropical coefficient} of $F_{i,j}$ and define the degree $\degree F_{i,j}$ of $F_{i,j}$ as the tuple $(i,j)\in \mathbb N_0^2$. On the other hand, classical polynomials $$\sum_{0\le i+j\le N}  \alpha_{i,j} u^i v^j,$$ of degree $N$ become piecewise affine functions:
\begin{align}
 F_{\max}(u,v)=\max_{0\le i+j\le N} F_{i,j}(u,v),\quad (u,v)\in \mathbb R^2.\eqlab{tropicalpol}
\end{align}

In tropical geometry, a \textit{tropical curve} $\mathcal T$ is then the set of points $(u,v)\in \mathbb R^2$ where the maximum in \eqref{tropicalpol} is attained by at least two tropical monomials, i.e. 
\begin{align*}
 \mathcal T: = \bigg\{(u,v)\in \mathbb R^2\,:\,\#\argmax_{(i,j)}F_{i,j}(u,v)\ge 2\bigg\},
\end{align*}
see \cite{maclagan2015introduction}.
Here we use $\#$ for the cardinality of a finite set. 


Points $(u,v)$ where $$\#\argmax_{(i,j)}F_{i,j}(u,v)= 2,$$ lie on line segments (isomorphic to open intervals $]a,b[$) given by
\begin{equation}
\begin{aligned}
\mathcal E_{(i_1,j_1),(i_2,j_2)}=\bigg\{(u,v)\,:\,
 F_{i_1,j_1}(u,v) = F_{i_2,j_2}(u,v)>F_{i,j}(u,v),\\
  \text{for all}\quad (i,j)\ne (i_k,j_k),\,k=1,2\bigg\}
  \end{aligned}\eqlab{tropmanifold}
\end{equation}
for some pair $(i_1,j_1),(i_2,j_2)$.
We will call such lines \textit{\rsp{tropical edges}}.  
Notice that $\mathcal E_{(i_1,j_1),(i_2,j_2)}$ (if nonempty) is either vertical (if $j_1=j_2$) or it has rational slope $\frac{i_2-i_1}{j_1-j_2}$. 

On the other hand, points $(u,v)$ where $$\#\argmax_{(i,j)}F_{i,j}(u,v)> 2,$$ are called \textit{\rsp{tropical vertices}}. They are in general isolated with $$\#\argmax_{(i,j)}F_{i,j}(u,v)= 3,$$ see \cite{maclagan2015introduction}, and are in this case given by
 \begin{equation}
\begin{aligned}
 \mathcal P_{(i_1,j_1),(i_2,j_2),(i_3,j_3)} = \bigg\{(u,v)\,:\,F_{i_1,j_1}(u,v) = F_{i_2,j_2}(u,v)=F_{i_3,j_3}(u,v)>F_{i,j}(u,v),\\
 \text{for all}\quad (i,j)\ne (i_k,j_k),\,k=1,2,3\bigg\},
\end{aligned}\nonumber
\end{equation}
for some triplet $(i_1,j_1),(i_2,j_2),(i_3,j_3)$.

Finally, points $(u,v)$ where $$\#\argmax_{(i,j)}F_{i,j}(u,v)=1,$$ lie on convex polygonal domains given by
\begin{align}
 \mathcal R_{(i_1,j_1)} = \bigg\{(u,v)\,:\,F_{i_1,j_1}(u,v) >F_{i,j}(u,v),\quad \text{for all}\quad (i,j)\ne (i_1,j_1)\bigg\},
\end{align}
for some $(i_1,j_1)$. 
\rsp{We call these sets \textit{tropical regions}}. We will frequently drop the subscripts on $\mathcal E$, $\mathcal P$ and $\mathcal R$ if the indices (specifying the relevant monomials) are not important. 

\begin{definition}\defnlab{homotopiccurves}
 Two tropical curves $\mathcal T$ and $\mathcal T'$ are said to be \textnormal{homotopic} if there is a continuous deformation $\mathcal T_t$, $t\in [0,1]$, with $\mathcal T_0=\mathcal T$ and $\mathcal T_1=\mathcal T'$, such that (a) the cardinality of the set of \rsp{tropical vertices} of $\mathcal T_t$ is constant and (b) only the lengths of the \rsp{tropical edges} of $\mathcal T_t$ vary, not their slopes. 
\end{definition}

The \textit{Newton polygon} of a tropical polynomial \eqref{tropicalpol} is defined as the convex hull of the \textit{point configuration} $$\bigg\{\degree F_{i,j}\in \mathbb N_0^2\,:\,\alpha_{i,j}\ne -\infty,\,0\le i+j\le N\bigg\}.$$  Another important concept is that of a \textit{polyhedral subdivision} $\mathcal S$ of the Newton polygon. To explain this, consider the points $$(\degree F_{i,j},\alpha_{i,j})\in \mathbb N_0^2\times (\mathbb R\cup \{-\infty\}),$$  obtained by assigning the tropical coefficients $\alpha_{i,j}$ of $F_{i,j}$ as a third dimension (a height). Then we take the upper envelope of the resulting set of points. Finally, the subdivision $\mathcal S$ is obtained by projecting the edges of the resulting convex polyhedron back to $\mathbb R^2$, where $\mathbb N_0^2$ is naturally embedded, see \cite{de2010a,maclagan2015introduction} and \figref{trop}. The different polygons of $\mathcal S$ will be called \textit{faces}.

The importance of $\mathcal S$ is due to the fact that $\mathcal S$ is \textit{dual to the tropical curve} $\mathcal T$ in the following sense: 
\begin{proposition} \proplab{dualS}
 \cite[Proposition 3.1.6]{maclagan2015introduction}
Consider a polyhedral subdivision $\mathcal S$ associated with a tropical polynomial. Then each face of the subdivision corresponds to a \rsp{tropical vertex}  $(u,v)\in \mathcal T$ where 
\begin{align}\eqlab{tropPcond}
\# \argmax_{(i,j)} F_{i,j}(u,v)>2,
\end{align} and the edges of $\mathcal S$ correspond to 
the edges (the \rsp{tropical edges}) of $\mathcal T$ where $$\# \argmax_{(i,j)} F_{i,j}(u,v)=2.$$
In particular, the subdivision $\mathcal S$ fixes the tropical curve $\mathcal T$ up to homotopy and the dual edges of $\mathcal S$ and $\mathcal T$ are perpendicular (see \figref{trop}(b)). 
\end{proposition}
  A subdivision $\mathcal S$ where all of its faces are triangles (as in \figref{trop}) is called a \textit{triangulation}, and in this case exactly three \rsp{tropical edges} come together at each \rsp{tropical vertex}  (so that the left hand side of \eqref{tropPcond} equals $3$).

\begin{figure}[H]
    \centering
    \begin{subfigure}{0.5\textwidth}
    \centering
        \includegraphics[width=0.95\linewidth]{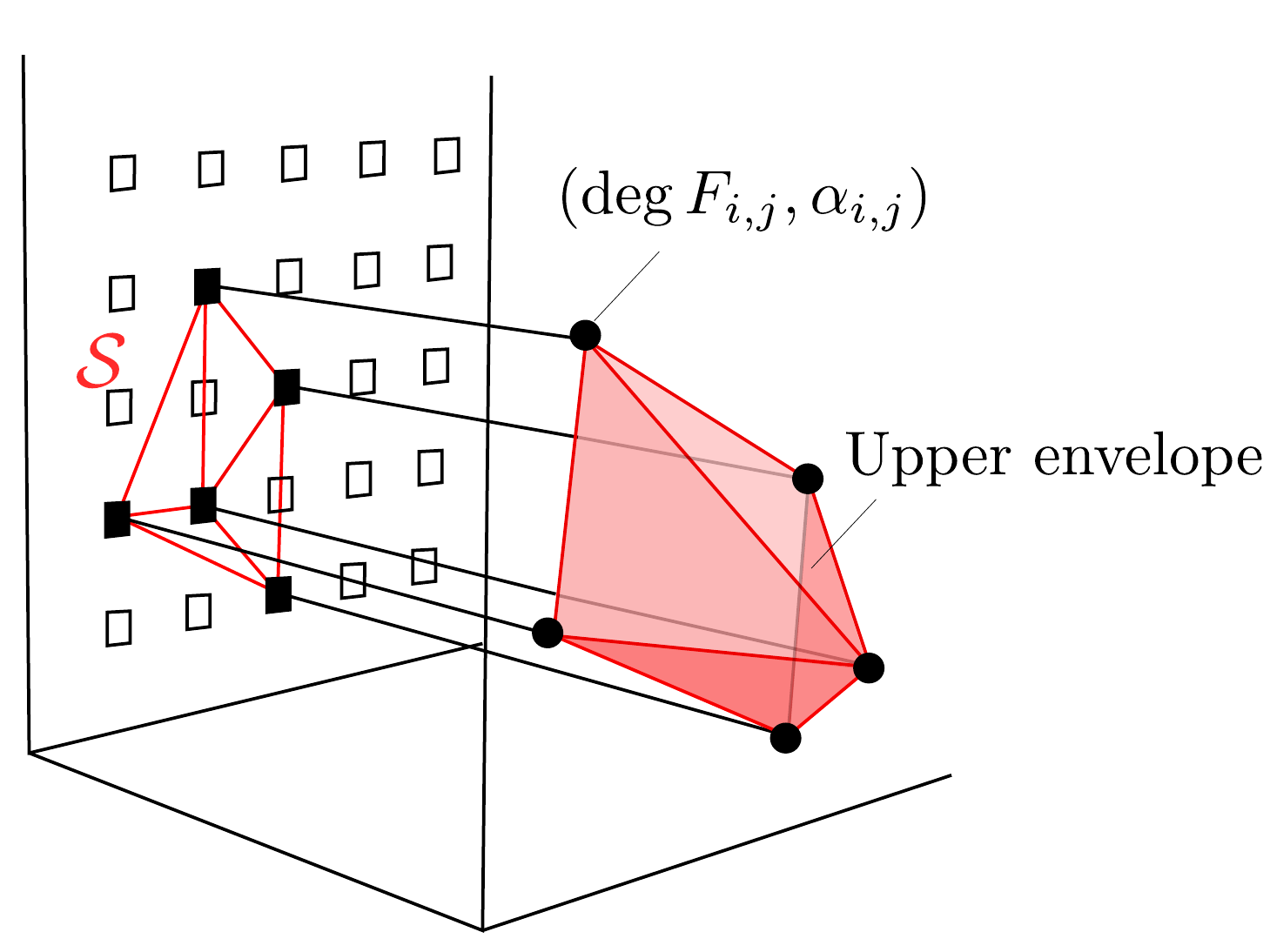}
        \caption{}
    \end{subfigure}%
    \begin{subfigure}{0.5\textwidth}
    \centering
        \includegraphics[width=0.94\linewidth]{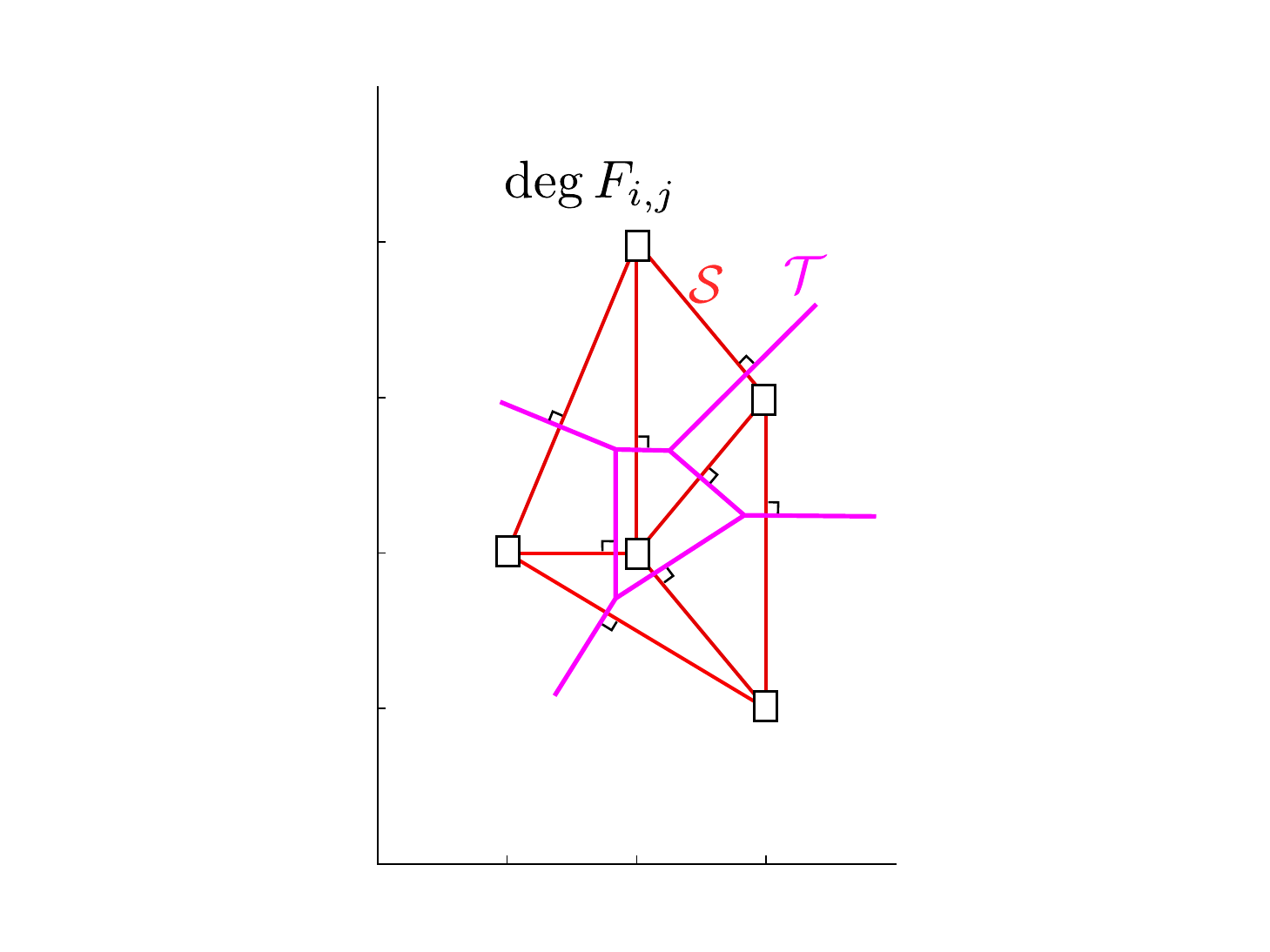}
        \caption{}
    \end{subfigure}
    \caption{In (a): The upper envelope of the set of points $(\degree F_{i,j},\alpha_{i,j})$. The subdivision $\mathcal S$ of the Newton polygon is the projection of the envelope. In (b): The subdivision $\mathcal S$ and the dual tropical curve $\mathcal T$. The subdivision fixes $\mathcal T$ up to homotopy and the \rsp{tropical edges} of $\mathcal T$ are perpendicular to the edges of $\mathcal S$, see \propref{dualS}.}
    \figlab{trop}
    \end{figure}
    
  There are 
\begin{align} M=M(N):=\frac12 (N+2)(N+1)\in \mathbb N,\eqlab{M} \end{align} 
many coefficients $\{\alpha_{i,j}\}_{0\le i+j\le N}$ of a general (tropical) polynomial of degree $N$. 
Let 
\begin{align}\eqlab{alphaenum}
\alpha:=(\alpha_{0,0},\ldots,\alpha_{N,0},\alpha_{0,1},\ldots,\alpha_{N-1,1},\alpha_{0,2},\ldots,\alpha_{N,0}),
\end{align} and put
\begin{align}
 \rspp{\mathbf{A}}(M):=(\mathbb R\cup \{-\infty\})^{M}.\eqlab{AN}
\end{align}
Then $\alpha\in \rspp{\mathbf{A}}(M)$.
%
%
%
%
%
We use the product topology to obtain a topology on $\rspp{\mathbf{A}}(M)$, recall \remref{open}. 
%
%
%
%
%

Now, the subdivisions of the Newton polygon are determined by a set of linear inequalities on the tropical coefficients, see e.g. \cite{de2010a}, and each triangulation defines an open convex polytope in the space of coefficients $\rspp{\mathbf{A}}(M)$. This leads to the following simple fact, see \cite{de2010a}.
\begin{thm}\thmlab{thm0}
Fix a degree $N\in \mathbb N$. Then the subdivision $\mathcal S$ of a tropical polynomial \eqref{tropicalpol} with $\alpha\in \rspp{\mathbf{A}}(M(N))$ is a triangulation on a dense open subset of $\rspp{\mathbf{A}}(M(N))$. This dense subset is the finite union of open convex polytopes.
\end{thm}
In other words, all \rsp{tropical vertices} have $\#\argmax_{(i,j)}F_{i,j}(u,v)= 3$ on a  dense open subset consisting of a finite union of open convex polytopes. Notice that the tropical curves (as well as the subdivisions) are invariant with respect to translations  
\begin{align}\eqlab{translation}a \mapsto \alpha+a:=(\alpha_{0,0}+a,\ldots,\alpha_{N,0}+a), \quad a\in \mathbb R,\end{align} of the tropical coefficients. Each of the convex polytopes in \thmref{thm0} are therefore also invariant with respect to such translations. (One could obviously choose to partition $\rspp{\mathbf{A}}(M)$ into equivalence classes, obtained from the action \eqref{translation}, but this will not be important for our purposes, and we will therefore not do so in the present manuscript.)

The actual number of triangulations $T(n)$ for a point configuration of $n$ points is not known, to the best of our knowledge. There are bounds, however, of the form $2^{c_1 n}\le T(n)\le 2^{c_2 n}$ for some constants $c_2\ge  c_1>0$, see \cite[Chapter 3]{de2010a}. Moreover, not every triangulation is \textit{regular}, in the sense that it is realizable as a projection of an upper envelope, see \cite{de2010a}.

There are many classical results of algebraic geometry that have counterparts in tropical geometry. One even talks of a (\rsp{Litvinov-Maslov}) \textit{correspondance principle} \cite{litvinov1996a,noel2012a}), in the spirit of the {correspondance principle} in Quantum Mechanics, between results in classical algebraic geometry and tropical geometry, see \cite{litvinov2007a}. For example, B\'ezout's theorem also holds in tropical geometry (for details on counting multiplicities, we refer to \cite{maclagan2015introduction}):
\begin{thm}
 Consider two tropical curves $\mathcal T$ and $\mathcal T'$ of degree $K$ and $L$ in $\mathbb R^2$. Then if they intersect transversally, meaning that each intersection is along the \rsp{tropical edges} of $\mathcal T$ and $\mathcal T'$, then they do so $KL$ times (counting multiplicities). 
\end{thm}
\rspp{Moreover, if} $\mathcal T$ and $\mathcal T'$ do not intersect transversally, then 
\begin{align*}
 \lim_{\epsilon\to 0} \mathcal T_\epsilon \cap \mathcal T'_\epsilon,
\end{align*}
is \textit{independent of the choice of perturbations} $\mathcal T_\epsilon$ and $ \mathcal T'_\epsilon$ of $\mathcal T$ and $\mathcal T'$, respectively. This is called the \textit{Stable Intersection Principle}, see \cite[Theorem 1.3.3]{maclagan2015introduction}.





%

\section{\rsp{A tropicalization of planar polynomial ODEs}}\seclab{tropode}
The vector space of real polynomials of degree $N$ in two variables $x,y$ is spanned by 
the monomials $x^ny^m$, $0\le n+m\le N$. The dimension of this space is $M=M(N)$, recall \eqref{M}.
Now, define 
\begin{align}\eqlab{UV}
\mathcal U:=\{1,\ldots,M\},\quad \mathcal V=\mathcal U+M:=\{M+1,\ldots,2M\}.
\end{align}
For our purposes, it will then be convenient to write the polynomials $P$ and $Q$ in \eqref{xyPQ} in the following form
\begin{align}
 P(x,y) = \sum_{k\in \mathcal U} a_k x^{n_k} y^{m_k},\quad Q(x,y)=\sum_{k\in \mathcal V} a_{k} x^{n_{k-M}} y^{m_{k-M}},\eqlab{PQpol}
\end{align}
using an enumeration of $n_k$ and $m_k$: The details are not important, but for concreteness, consider
\begin{align*}
 q_l :=lN -\frac{l(l-3)}{2},\quad \text{for}\quad l=0,\ldots,N,
\end{align*}
and $q_{N+1}:=M+1$. Then we define
\begin{align}\eqlab{nkmk}
 n_k := k-q_l,\,
  m_k := l,\quad \text{for}\quad k=q_l,\cdots,q_{l+1}-1,
 \end{align}
 and each $l=0,\ldots,N$. Notice that $q_{l+1}-q_l=N-l+1$ so that $n_{k} +m_{k} = N$ for $k=q_{l+1}-1$ and all $l=0,\ldots,N$. This enumeration is consistent with \eqref{alphaenum}.

We put 
\begin{align*}
\mathcal I:=\mathcal U\cup \mathcal V=\{1,\ldots,2M\}.
\end{align*}
Then $a_k\in \mathbb R$ for all $k\in \mathcal I$ and $(n_k,m_k)\in \mathbb N_0^2$ and $0\le n_k+m_k\le N$ for all 
$k\in \mathcal I$. 
%


We will now introduce a tropical version of \eqref{xyPQ}. We will only consider the first quadrant $(x,y)\in \mathbb R_+^2,\,\mathbb R_+:=(0,\infty)$; the other quadrants can be handled in a similar way, but we leave the problem of connecting the quadrants to future work (see \secref{discussion} for a further discussion of this aspect). 


For any $\epsilon>0$, consider the transformation: 
\begin{align}
(x,y,\{a_k\}_{k\in \mathcal I})\mapsto (u,v,\{\delta_k\}_{k\in \mathcal I},\{\alpha_k\}_{k\in \mathcal I}),\eqlab{trop}
\end{align}
of $x,y$ and the coefficients $\{a_k\}_{k\in \mathcal I}$, defined by \eqref{uv}
%
and 
\begin{align}
\delta_k=\textnormal{sign}(a_k) :=\begin{cases}
                                1 & \textnormal{for}\quad a_k\ge 0\\
                                -1 & \textnormal{for}\quad a_k<0
                                \end{cases},&\quad \alpha_k = \begin{cases} 
                                \epsilon \log \vert a_k\vert & \textnormal{for}\quad a_k\ne 0\\
                                -\infty & \textnormal{for}\quad a_k= 0
                                \end{cases},
                                                                \eqlab{deltakalphak}
\end{align}
for all $k\in \mathcal I$.
By introducing the tropical monomials
\begin{equation}\eqlab{Fk}
\begin{aligned}
 F_k(u,v):&=\alpha_k+(n_k-1)u+m_k v,\quad k\in \mathcal U,\\
 F_k(u,v):&=\alpha_k+n_{k-M} u+(m_{k-M}-1) v,\quad k\in \mathcal V,
\end{aligned}
\end{equation}
for $(u,v)\in \mathbb R^2$,
an elementary calculation then shows that the application of \eqref{trop} to \eqref{xyPQ}, with $P$ and $Q$ in the form \eqref{PQpol}, leads to the following equations:
\begin{equation}\eqlab{uvEqn0}
\begin{aligned}
 u' & =\epsilon \sum_{k\in \mathcal U} \delta_k \e^{F_k(u,v)\epsilon^{-1}},\\
 v' &=\epsilon \sum_{k\in \mathcal V} \delta_k \e^{F_k(u,v)\epsilon^{-1}}.
\end{aligned}
\end{equation}
Monomial terms of $P$ and $Q$ in the equations for $x$ and $y$ therefore become exponential terms with the tropical monomials $F_k$ as their arguments. 
 Recall that the tuple $(n,m)$ is the degree $\textnormal{deg}\,F$ of $F(u,v)=\alpha+n u + m v$. In contrast to \secref{trop}, we can therefore have $\deg F_k =(-1,*)$ and $(*,-1)$.  This is not a problem; we can define tropical curves, the Newton polygons and the subdivisions completely analogously. Notice, however, from \eqref{Fk}  that 
 \begin{equation}\eqlab{nkmkNeg1}
  \begin{aligned}
 \degree F_k=(-1,*)\Longrightarrow k\in \mathcal U,\\
 \degree F_k=(*,-1)\Longrightarrow k\in \mathcal V.
 \end{aligned}
 \end{equation}

 To be able to study the limit $\epsilon\to 0$ of \eqref{uvEqn0}, we ``tame'' the exponentials (as in \cite{jelbart2021a}) by introducing a new time $t$
defined by 
\begin{align}
 \frac{\mathrm dt}{\mathrm d\tau }=\epsilon \sum_{k\in \mathcal I} \e^{F_k(u,v)\epsilon^{-1}}>0.\eqlab{time}
\end{align}
This produces the equations
\begin{equation}\eqlab{uvEqn}
\begin{aligned}
 \dot u & =\frac{\sum_{k\in \mathcal U} \delta_k \e^{F_k(u,v)\epsilon^{-1}}}{\sum_{k\in \mathcal I} \e^{F_k(u,v)\epsilon^{-1}}},\\
 \dot v &=\frac{\sum_{k\in \mathcal V} \delta_k \e^{F_k(u,v)\epsilon^{-1}}}{\sum_{k\in \mathcal I} \e^{F_k(u,v)\epsilon^{-1}}},
\end{aligned}
\end{equation}
with $\dot{()}=\frac{\mathrm d}{\mathrm dt}$. 
 Let $\mathcal L=\mathcal I,\,\mathcal U$ or $\mathcal V$. We then define 
  \begin{align*}
 F_{\max}^{\mathcal L}(u,v):=\max_{k\in \mathcal L} F_{k}(u,v),\quad (u,v)\in \mathbb R^2,
\end{align*}
 as the tropical polynomial associated with the monomials $F_{k}$, $k\in \mathcal L$, and let $\mathcal T^{\mathcal L}$ denote the associated tropical curve. Similarly, \rsp{tropical edges}, \rsp{vertices} and regions will be denoted by $\mathcal E^{\mathcal L}$, $\mathcal P^{\mathcal L}$ and $\mathcal R^{\mathcal L}$, respectively. Finally, the associated subdivision of the Newton polygon is denoted by $\mathcal S^{\mathcal L}$. 
  However, we frequently drop the superscript if $\mathcal L=\mathcal I$ (since this situation will play the most important role).  Notice that 
\begin{align*}
 F_{\max}^{\mathcal U}(u,v)&:=\max_{k\in \mathcal U} F_k(u,v),\quad (u,v)\in \mathbb R^2,\\
 F_{\max}^{\mathcal V}(u,v)&:=\max_{k\in \mathcal V} F_k(u,v),\quad (u,v)\in \mathbb R^2,
\end{align*}
only involve monomials from the $u$ and $v$-directions, respectively.
\begin{remark}
  The point configurations associated with $F_{\max}^{\mathcal U}$ and $F_{\max}^{\mathcal V}$ each have $M$ points, see \eqref{M}, whereas the point configuration associated with $F_{\max}^{\mathcal I}$ has 
  \begin{align}
   M+N+1 = \frac12 (N+1)(N+4),\eqlab{NIpoints}
  \end{align}
points. This is a simple consequence of \eqref{nkmkNeg1}.
 \end{remark}



  
  Finally, for any $k\in \mathcal I$ we define
 \begin{align}
  \md_k = \begin{cases}
         (\delta_k,0) &\textnormal{if}\quad k\in \mathcal U,\\
         (0,\delta_k) &\textnormal{if} \quad k\in \mathcal V,
        \end{cases}\eqlab{dk}
 \end{align}
  as \textit{the flow vector} associated to the tropical monomial $F_k$ in \eqref{Fk}. 
  \begin{remark}\remlab{degree}
  Notice that in our definition of \eqref{Fk}, we can have $\textnormal{deg}\,F_k = \textnormal{deg}\,F_j$ but only for $k\in \mathcal U$ and $j\in \mathcal V$. In particular, \textnormal{the  tropical pair} $(\md_k,\textnormal{deg}\,F_k)$ is unique. 
  \end{remark}
   



  
\begin{lemma}\lemmalab{limitTrop}
The system
 \eqref{xyPQ} on $(x,y)\in \mathbb R_+^2$ and the system \eqref{uvEqn} on $(u,v)\in \mathbb R^2$ are topologically equivalent for every $\epsilon>0$ whenever \eqref{deltakalphak} holds.

  Moreover, suppose $(u,v)\notin \mathcal T$ such that $(u,v)\in \mathcal R_l$ with $l=\argmax_{k\in \mathcal I} F_k(u,v)$ \rspp{and fix the lists $\{\delta_k\}_{k\in \mathcal I}$ and $\{\alpha_k\}_{k\in \mathcal I}$}. Then 
  the right hand side of \eqref{uvEqn} converges as $\epsilon\to 0$ to the flow vector $\md_{l}$, see \eqref{dk}.

 \end{lemma}
 \begin{proof}
 The first part is by construction, with the diffeomorphism $\mathbb R_+^2 \ni (x,y)\mapsto (u,v)\in \mathbb R^2$ given by \eqref{uv} and the regular transformation of time given by \eqref{time}. The second part follows from a straightforward calculation: Given $(u,v)\notin \mathcal T$, we have $(u,v)\in \mathcal R_l$ with $l=\argmax_{k\in \mathcal I}F_k(u,v)$. Hence
 \begin{align*}
  F_l(u,v)-F_k(u,v)>c>0\quad \textnormal{for all}\quad k\ne l,
 \end{align*}
 for some $c>0$ small enough. Suppose that $l\in \mathcal U$. Then by \eqref{uvEqn}, we obtain
\begin{align*}
 \dot u &= \frac{\delta_l + \mathcal O(\e^{-c/\epsilon})}{1+\mathcal O(\e^{-c/\epsilon})},\\
 \dot v &=\frac{\mathcal O(\e^{-c/\epsilon})}{1+\mathcal O(\e^{-c/\epsilon})},
\end{align*}
and the result follows. The case $l\in \mathcal V$ is identical. 
 \end{proof}
 \begin{remark}
 \rspp{The approach, detailed above, for tropicalizing \eqref{xyPQ} through \eqref{uv}, \eqref{trop} and the singular limit $\epsilon\to 0$, was communicated to the first author by P. Szmolyan in 2019 in a private conversation. It is also partially fleshed out in \cite{portisch2021novel} in the context of the Michaelis-Menten kinetics.}
   \end{remark}

Let 
\begin{align}\eqlab{norm1}
\vert (u,v)\vert_1 := \vert u\vert+\vert v\vert,\end{align} denote the $1$-norm on $\mathbb R^2$, and put $\mathbb N_{0}:=\{0\}\cup \mathbb N$ and $\mathbb N_{-1}:=\{-1\}\cup \mathbb N_0$. 


\begin{definition}\defnlab{Ml}
We define $\mathcal M^u$ ($\mathcal M^v$) as the set of all tropical monomials $F$ (i.e. of the form \eqref{tropicalmonomial}) with degree $\degree F\in \mathbb N_{-1}\times \mathbb N_0$ ($\degree F\in \mathbb N_{0}\times \mathbb N_{-1}$, respectively) and tropical coefficient $\alpha\in \mathbb R\cup \{-\infty\}$.

For each fixed $N\in \mathbb N$, we define $\mathcal M_N^u$ ($\mathcal M_N^v$) as the subset of $\mathcal M^u$ ($\mathcal M^v$, respectively) satisfying
\begin{align*}
  \mathcal M_N^l = \bigg\{F\in \mathcal M^l\,:\,\vert \degree F\vert_1\le N+1\bigg\},\quad l=u,v.
 \end{align*}
 
  \end{definition}

  \subsection{A set-valued vector-field $\Conv$}
   As a consequence of \lemmaref{limitTrop}, \eqref{uvEqn} \rspp{(with $\{\delta_k\}_{k\in \mathcal I}$ and $\{\alpha_k\}_{k\in \mathcal I}$ fixed)} is piecewise smooth (in fact, even piecewise constant) in the limit $\epsilon\to 0$ and the discontinuity set is a subset of $\mathcal T$. In particular, the limit $\epsilon\to 0$ is well-defined for $(u,v)\notin \mathcal T$ and given by 
   \begin{align}
   \begin{pmatrix}
   \dot u\\
   \dot v
 \end{pmatrix} &=
\md_{\argmax_{k\in \mathcal I} F_k(u,v)},\quad (u,v)\notin \mathcal T.\eqlab{uvnotinT}
\end{align}
On the other hand, for $(u,v)\in \mathcal T$, $$\argmax_{k\in \mathcal I} F_k(u,v),$$ contains more than one element by definition.

Let $$\conv (\md_k,\md_l)=\bigg\{q\md_k+(1-q)\md_l\,:\,q\in [0,1]\bigg\},$$ denote the convex hull of $\md_k$ and $\md_l$. If $\md_k\ne \md_l$ then $\conv (\md_k,\md_l)$ is simply the line segment connecting $\md_k$ and $\md_l$. Otherwise, it is a point.  

\begin{remark}
For simplicity, we write $\{(-1,0),(1,0)\}$ as $\{(\pm 1,0)\}$. $\{(0,\pm 1)\}$ is understood in the same way. \rspp{Moreover, for simplicity of notation, we will throughout assume that it is clear from the context whether tuples should be understood as row or column vectors. As usual, $b\cdot c$ will denote the dot product in $\mathbb R^n$ between vectors $b$ and $c$.}
\end{remark}

We now define a set-valued vector-field $\Conv$, \rspp{which will form the basis for our definition of a tropical dynamical system}. For this purpose, we use $2^{\mathbb R^2}$ to denote the power set of $\mathbb R^2$, i.e. the set of all subsets of $\mathbb R^2$. 
\begin{definition}\defnlab{convtrop}
Consider any $(u,v)\in \mathbb R^2$ and write
\begin{align*}
\mathcal L^*= \mathcal L^*(u,v) :=\argmax_{k\in \mathcal L} F_k(u,v)\quad \mbox{and}\quad 
\md_{\mathcal L^*} := \{\md_{k}\,:\, k\in \mathcal L^*\},
\end{align*}
for $\mathcal L= \mathcal U,\mathcal V$ and $\mathcal I$. 
We then define the set-valued vector-field $$(u,v)\mapsto \Conv(u,v) \in 2^{\mathbb R^2},$$ 
as follows:
\begin{enumerate}
\item \label{1} \rspp{If $\md_{\mathcal I^*}$ is a singleton}, so that $\md_{\mathcal I^*}=\{\md\}$ for some flow vector $\md$, then $$\Conv(u,v):=\{\md\}.$$
\item \label{2} Otherwise, $\md\in \Conv(u,v)$ 
if and only if either of the following conditions \rspp{holds true}
\begin{enumerate}
 \item \label{a} \rspp{There exist $\mathbf t\in \md_{\mathcal U^*}$ and $\mathbf s\in \md_{\mathcal V^*}$} such that $\md \in \conv(\mathbf t,\mathbf s)$.
 \item \label{b} $\md = \textnormal{\textbf{0}}$, $\md_{\mathcal U^*}=\{(\pm 1,0)\}$ \textnormal{and} $\md_{\mathcal V^*}=\{(0,\pm 1)\}$.
\end{enumerate}

\end{enumerate}
\end{definition}

\rsp{Notice that $\Conv(u,v)=\{\md_k\}$ for any $(u,v)\in \mathcal R_k$, $k\in \mathcal I$, (as desired, cf. \eqref{uvnotinT}).} 

\rspp{We emphasize that the zero vector $\textnormal{\textbf{0}}$ belongs to $\Conv(u,v)$ if and only if $\md_{\mathcal U^*}=\{(\pm 1,0)\}$ {and} $\md_{\mathcal V^*}=\{(0,\pm 1)\}$, cf.  \ref{b}. The motivation for this choice should be clear enough, see also \secref{auto} below. If $\textnormal{\textbf{0}}\in \Conv(u,v)$, then $(u,v)$ is called a  \textit{tropical singularity}, see \secref{tropeq}. } 

The following property of $\Conv$ will be important later on. 
\begin{lemma}\lemmalab{upper}
$\Conv:\mathbb R^2\to 2^{\mathbb R^2}$ is \textnormal{upper semi-continuous}: For every $(u',v')\in \mathbb R^2$ there exists a neighborhood $X$ of $(u',v')$ such that
\begin{align*}
\Conv(u,v)\subset \Conv(u',v'),
\end{align*}
for all $(u,v)\in X$.
\end{lemma}
\begin{proof}
The tropical monomials are continuous functions and consequently there is a neighborhood $X$ of $(u',v')$ such that 
\begin{align*}
F_l(u,v)<F_i(u,v),
\end{align*}
for all $i\in \mathcal L^*(u',v')$, all $l\notin \mathcal L^*(u',v')$, all $(u,v)\in X$ and $\mathcal L=\mathcal U,\mathcal V,\mathcal I$. Consequently, $\md_{\mathcal L^*(u,v)}\subset \md_{\mathcal L^*(u',v')}$ for all $(u,v)\in X$ and every $\mathcal L=\mathcal U,\mathcal V,\mathcal I$, and the result then follows from the definition of $\Conv$, see \defnref{convtrop}.
\end{proof}

 \rsp{In the following \secref{sliding_crossing}, we define the important concepts of tropical sliding and crossing. We then analyze these cases further in \secref{convsli} and \secref{crossingflow}, where we also compute $\Conv(u,v)$.}

\subsection{Tropical sliding and crossing}\seclab{sliding_crossing}
\rsp{Suppose} that $(u,v)$ belongs to a \rsp{tropical edge} $\mathcal E_{i,j}$, $i,j\in \mathcal I$, $i\ne j$, i.e. a line segment defined by the set of points $(u,v)\,:\,F_i(u,v)=F_j(u,v)>F_k(u,v)$ for all $k\in \mathcal I,k\ne i,j$. This gives
\begin{align}\eqlab{Eijeqn}
 (n_j-n_i) u+(m_j-m_i)v+\alpha_j-\alpha_i=0.
\end{align}
\rspp{Therefore $\mathcal E_{i,j}$ is horizontal if $n_i=n_j$ and vertical if $m_i=m_j$. Otherwise, it has rational slope $\frac{n_j-n_i}{m_i-m_j}$. Here we suppose that $\degree F_i\ne \degree F_j$, recall \remref{degree}.}
\begin{remark}\remlab{degree2}
Notice that if $\degree F_i= \degree F_j$ then the set defined by $F_i=F_j$ is either the empty set or $\mathbb R^2$ ($\Leftrightarrow \alpha_i=\alpha_j$).
\end{remark}
Given \eqref{Eijeqn},
\begin{align}\eqlab{nij}
\mathbf n_{i,j}=\degree F_j -\degree F_i =(n_j-n_i,m_j-m_i),
\end{align} is a normal  vector to $\mathcal E_{i,j}$.

In the language of PW\rspp{S} dynamical systems theory \cite{Bernardo08,filippov1988differential}, $\mathcal E_{i,j}$ is a switching manifold if $\md_i\ne \md_j$.
\begin{definition}\defnlab{crossingsliding}
Suppose that $\mathcal E_{i,j}$ is a switching manifold ($\md_i\ne \md_j$). Then $\mathcal E_{i,j}$ is of the following type:
\begin{enumerate}
\item \textnormal{Crossing} if $\md_i$ and $\md_j$ point in the same direction relative to $\mathcal E_{i,j}$:
 \begin{align}
  (\md_i \cdot \mathbf n_{i,j}) (\md_j \cdot \mathbf n_{i,j})>0.\eqlab{crossingdef}
 \end{align}
 \item \textnormal{Sliding} otherwise: 
 \begin{align}
  (\md_i \cdot \mathbf n_{i,j}) (\md_j \cdot \mathbf n_{i,j})\le 0.\eqlab{sliding}
 \end{align}
 More specifically, when \eqref{sliding} holds, then $\mathcal E_{i,j}$ is of the following \textnormal{sliding-type}:
 \begin{enumerate}
  \item \textnormal{Transversal (tangential) Filippov} if $\md_i\cdot \md_j=0$ \textnormal{and} the left hand side of \eqref{sliding} is nonzero (zero, respectively). 
  \item \textnormal{Transversal (tangential) Nullcline} if $\md_i\cdot \md_j=-1$ \textnormal{and} the left hand side of \eqref{sliding} is nonzero (zero, respectively).
 \end{enumerate}
 \end{enumerate}
\end{definition}
 We will occasionally use a similar terminology of nullcline sliding for the \textit{individual} \rsp{tropical edges} $\mathcal E^\mathcal L$ of $\mathcal T^{\mathcal L}$, $\mathcal L=\mathcal U,\mathcal V$, see \secref{tropeq}. Notice that for crossing, then $\md_i\cdot \md_j=0$ with $i\in \mathcal U$ and $j\in \mathcal V$ (or vice versa).

 \rspp{In contrast} to general piecewise smooth systems with nonconstant vector-fields \cite{filippov1988differential,jeffrey_geometry_2011}, we define sliding to hold even when \eqref{sliding} holds with equality. The reason for doing so, is that such tangencies are \textit{robust} with respect to perturbations of $\alpha_k$ in the present case. Indeed, the slope of $\mathcal E_{i,j}$ only depends upon the degrees $\degree F_i$ and $\degree F_j$ (which are fixed), not on the tropical coefficients $\alpha_i$ and $\alpha_j$. 
 
 We illustrate the different cases in \figref{sliding} and \figref{crossing}. \figref{sliding}(a) and (b) illustrate Filippov sliding, whereas (c) and (d) illustrate nullcline sliding. 
 
 In these figures, we also introduce our convention: Red (brown, see \figref{tropauto}) indicates a vertical flow vector, pointing upwards (downwards, respectively). Blue (light blue), on the other hand, corresponds to a horizontal flow vector, pointing to the right (left, respectively). The little arrow along the switching manifolds in (c) indicate the subdominant direction ($\md_l$ in \figref{sliding}). It is colored red in the present case because $\md_l$ is pointing upwards. Moreover, full lines indicate switching manifolds of sliding type, whereas dashed lines indicate switching manifolds of crossing type (or later \rsp{tropical edges} $\mathcal E_{i,j}$ which are not switching manifolds, i.e. $\md_j=\md_i$).

\begin{figure}[H]
    \centering
    \begin{subfigure}{0.49\textwidth}
    \centering
        \includegraphics[width=0.96\linewidth]{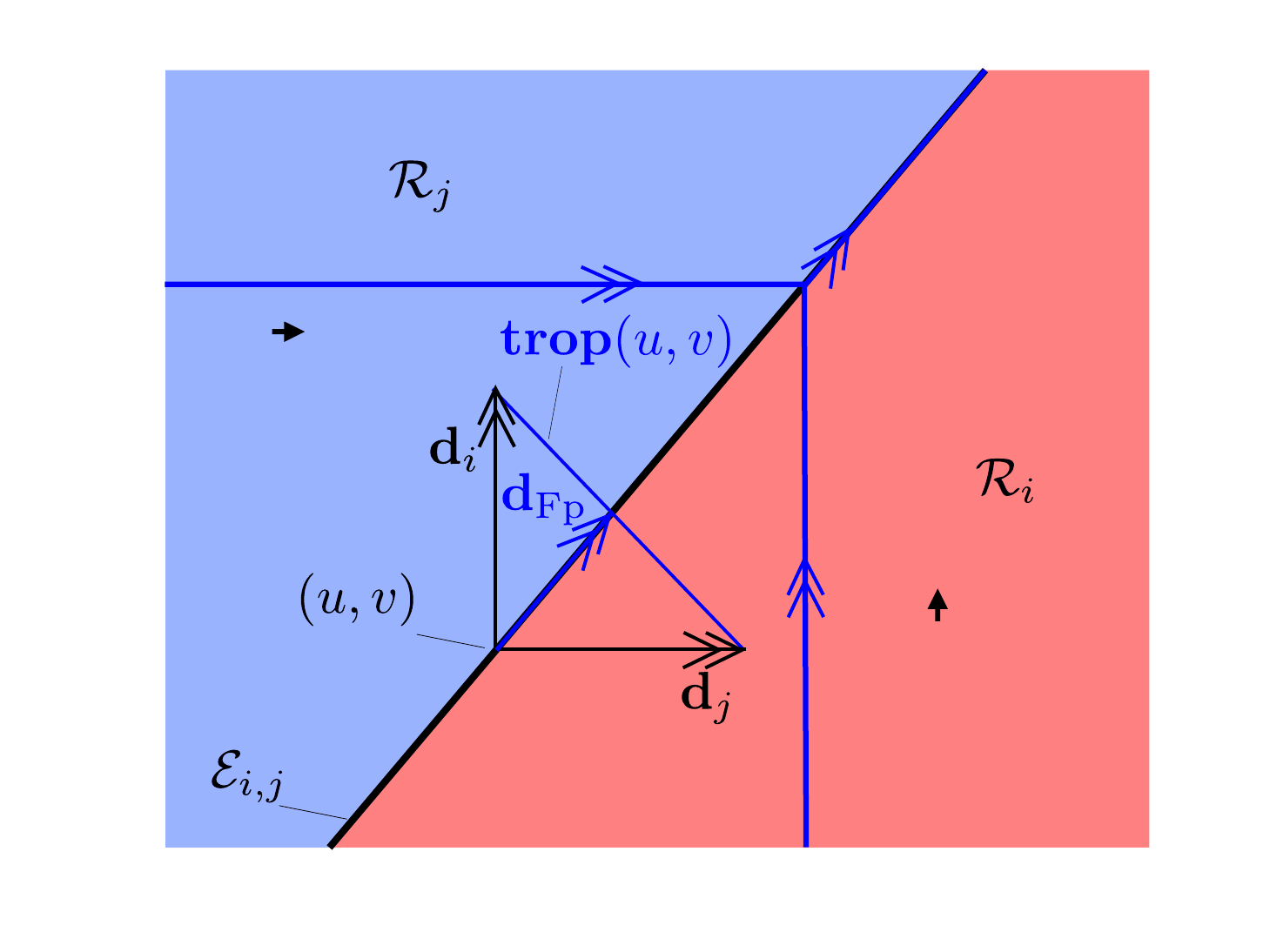}
        \caption{Transversal Filippov sliding}
    \end{subfigure}
    \begin{subfigure}{0.49\textwidth}
    \centering
        \includegraphics[width=0.96\linewidth]{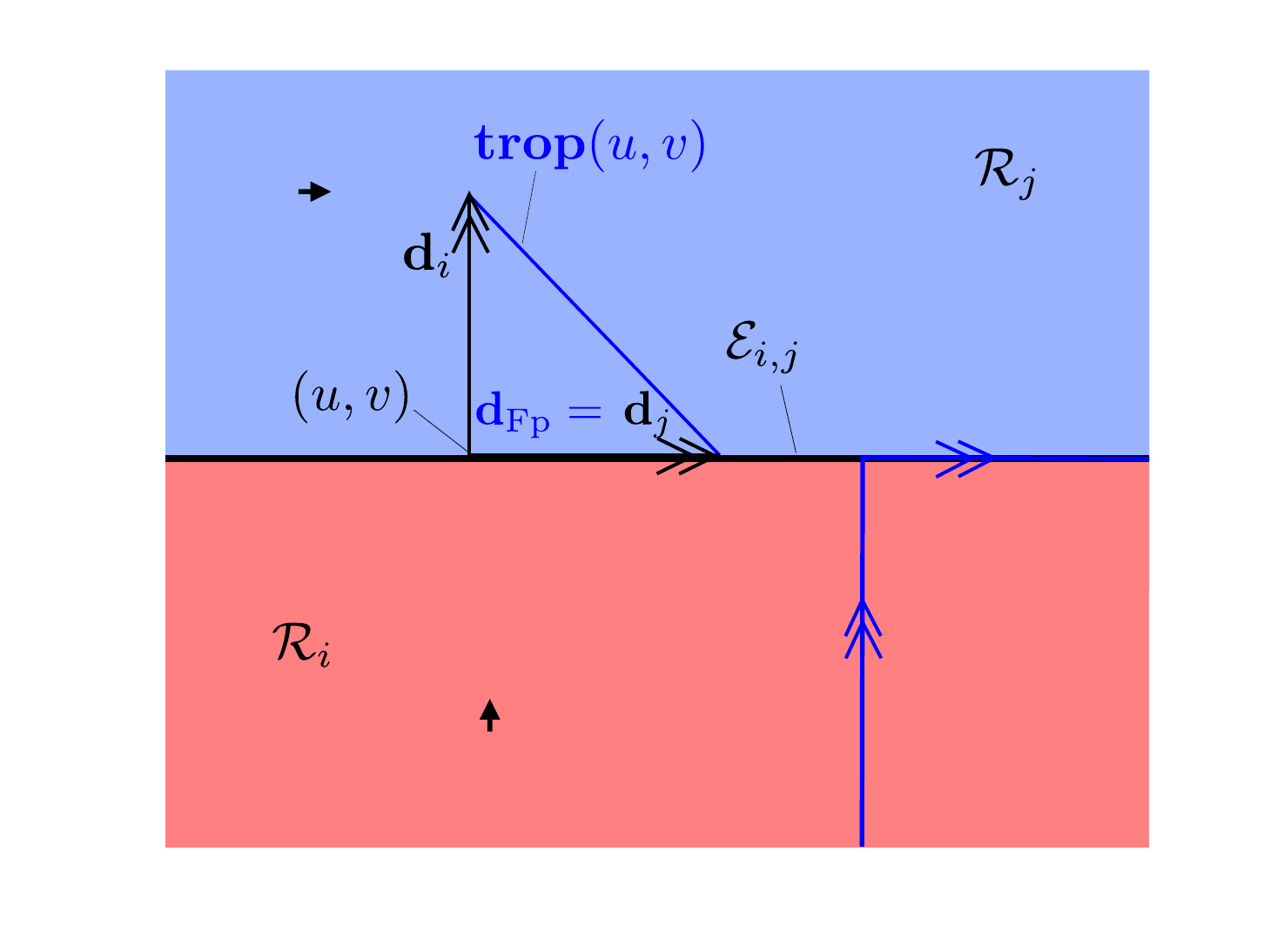}
        \caption{Tangential Filippov sliding}
    \end{subfigure}
     \begin{subfigure}{0.49\textwidth}
    \centering
        \includegraphics[width=0.96\linewidth]{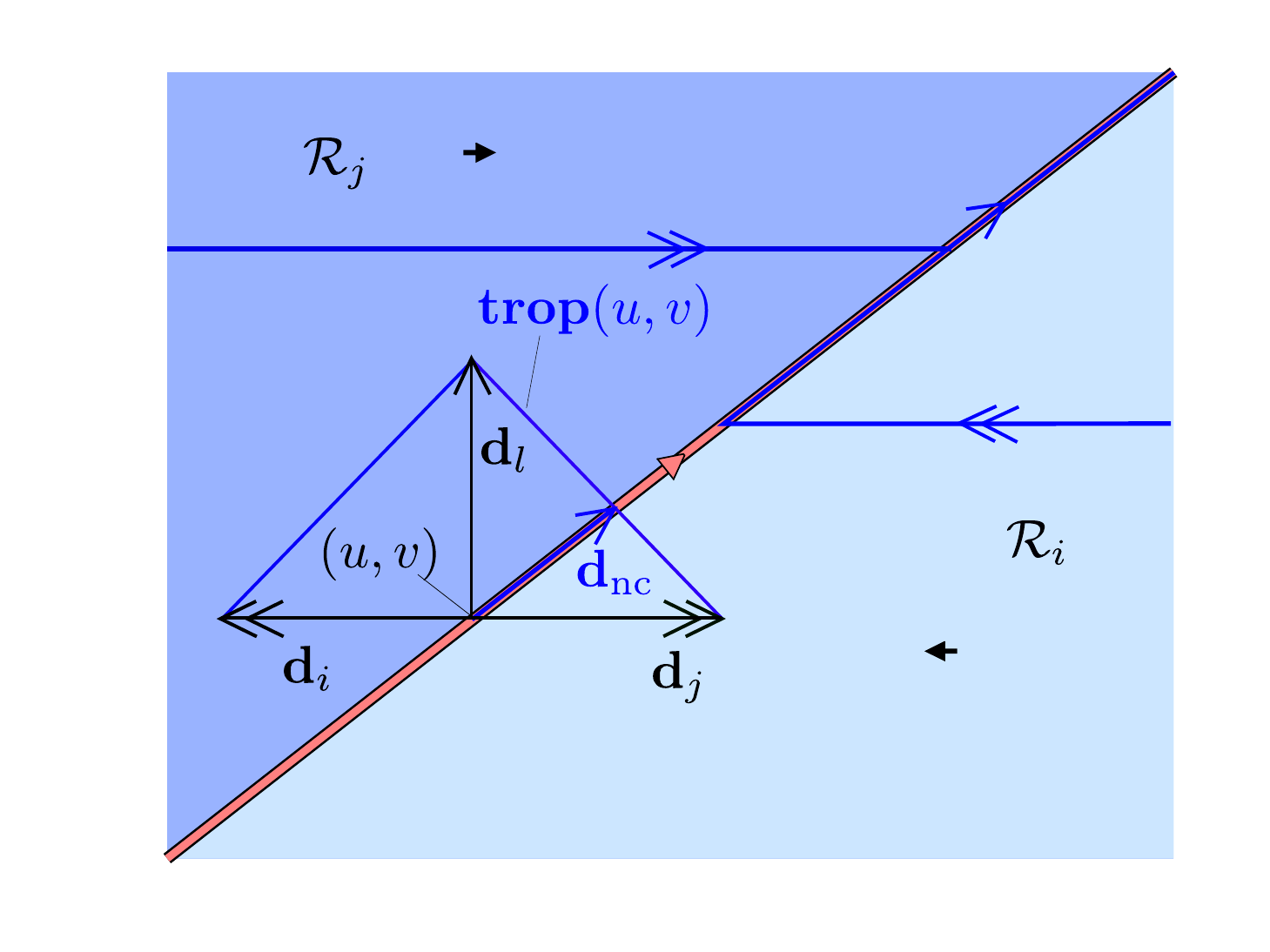}
        \caption{Transversal nullcline sliding}
    \end{subfigure}
    \begin{subfigure}{0.49\textwidth}
    \centering
        \includegraphics[width=0.96\linewidth]{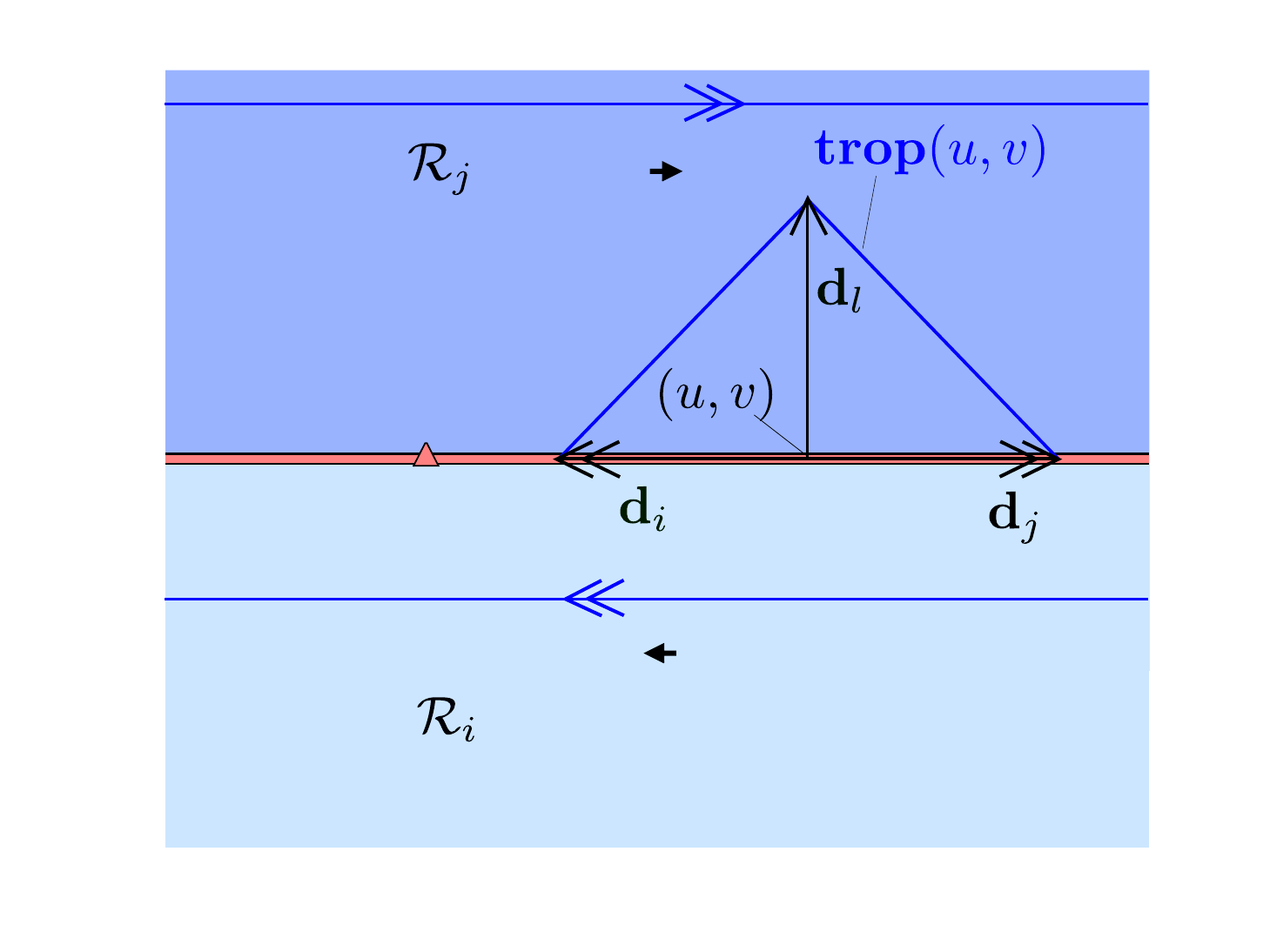}
        \caption{Tangential nullcline sliding}
    \end{subfigure}
    \caption{Illustrations of sliding and the set-valued vector-field $\Conv(u,v)$, $(u,v)\in \mathcal E_{i,j}$. \rsp{The blue curves show orbits of the differential inclusion $(\dot u,\dot v)\in \Conv(u,v)$. In particular,} for Filippov sliding ((a) and (b)), there is a unique vector $\md_{\mathrm{Fp}}$, the Filippov sliding-vector, tangent to the switching manifold. For transversal (not tangential) nullcline sliding (c), in \rspp{a point which is not a} \rsp{tropical singularity}, there is also a unique vector $\md_{\mathrm{nc}}$ that is tangent to the switching manifold. The direction is also indicated by a little (red) triangle, a convention we will also use in the subsequent figures. In (d) (tangential nullcline sliding), only $\md_i$ and $\md_j$ in $\Conv(u,v)$ are tangent to $\mathcal E_{i,j}$ at $(u,v)\in \mathcal E_{i,j}$.} 
    \figlab{sliding}
        \end{figure}

Filippov sliding (transversal or tangential) as well as transversal nullcline sliding can be stable and unstable, depending on whether vectors point towards or away from $\mathcal E_{i,j}$. In \figref{sliding} we only illustrate the stable versions. 

\subsection{\rsp{The set-valued vector-field $\Conv$ for tropical sliding}}\seclab{convsli}

\rsp{In this section, we compute $\Conv(u,v)$ in the different cases of sliding under the assumption that $\textnormal{\textbf{0}}\notin \Conv(u,v)$. The results  will important in our construction of global solutions to $(\dot u,\dot v)\in \Conv(u,v)$ in  \thmref{existence}.} 
\begin{lemma}\lemmalab{ConvFilippov}
 \rsp{Suppose that $\mathcal E_{i,j}$, $i\in \mathcal U$, $j\in \mathcal V$, is a switching manifold of Filippov sliding type (transversal or tangential, see \figref{sliding}(a) and (b)). Then 
 \begin{align}
  \Conv(u,v)=\conv(\md_i,\md_j),\quad \text{for all}\quad (u,v)\in \mathcal E_{i,j},\eqlab{convFilippov}
 \end{align}
and there is a unique vector
\begin{align}
 \md_{\Fil} = q \md_i+(1-q)\md_j,\quad q=\frac{\mathbf n_{i,j} \cdot \md_j}{\mathbf n_{i,j} \cdot \md_j-\mathbf n_{i,j} \cdot \md_i},\eqlab{mdFil}
\end{align}
contained in $\Conv(u,v)$, which is tangent to $\mathcal E_{i,j}$ at $(u,v)\in \mathcal E_{i,j}$.}
\end{lemma}
\begin{proof}
 \rsp{By \defnref{crossingsliding}, $\md_{\mathcal I^*} = \{\md_i,\md_j\}$ and $\md_i\cdot \md_j=0$ with $i\in \mathcal U$ and $j\in \mathcal V$. We are therefore in situation \ref{a} of \defnref{convtrop} and \eqref{convFilippov} follows. Finally, the statement regarding \eqref{mdFil} follows from a elementary computation.}
\end{proof}


The vector \eqref{mdFil} is the \textit{Filippov sliding vector}, known from the Filippov convention of piecewise smooth systems, see \cite{filippov1988differential}. \rspp{In our setting,} it is constant along a sliding switching manifold $\mathcal E_{i,j}$ of Filippov type. \rsp{Notice that in the tangential case then $q=0$ ($\md_{\Fil} =\md_j$) or $q=1$ ($\md_{\Fil} =\md_i$), whereas in the transversal case $q\in (0,1)$.} 

\rsp{Next, we consider nullcline sliding:}
\begin{lemma}\lemmalab{ConvNullcline}
 \rsp{Suppose that $\mathcal E_{i,j}$ is of nullcline sliding type and that $i,j\in \mathcal U$, $i\ne j$, (so that $u$ is the dominant direction, see \figref{sliding}(c) and (d)). Moreover, suppose that $(u,v)\in \mathcal E_{i,j}\cap \mathcal R^{\mathcal V}_l$ for some $l\in \mathcal V$. 
 Then 
 \begin{align}
 \Conv(u,v)=\bigg\{(u',v')\in \mathbb R^2\,:\,\vert (u',v')\vert_1=1\,\, \mbox{and}\,\, (u',v')\cdot \md_l\ge 0\bigg\},\eqlab{ConvNullcline}
\end{align}
recall \eqref{norm1}. Moreover, the following holds:
\begin{enumerate}
 \item \label{transversal} If $\mathcal E_{i,j}$ is of transversal nullcline sliding type, then there
 exists a unique vector 
 \begin{align}
 \md_{\trop} = \frac{\md_l\cdot \mathbf e_{i,j}}{\vert \md_l \cdot \mathbf e_{i,j}\vert} \mathbf e_{i,j},\eqlab{mdtrop}
\end{align}
contained in $\Conv(u,v)$,
  which is tangent to $\mathcal E_{i,j}$ at $(u,v)\in \mathcal E_{i,j}\cap \mathcal R_l^{\mathcal V}$.  Here $\mathbf e_{i,j}$ is a $\vert \cdot \vert_1$-unit vector tangent to $\mathcal E_{i,j}$.
  \item \label{tangential} If $\mathcal E_{i,j}$ is of tangential nullcline sliding type, then $\md_i$ and $\md_j$ are the only vectors contained in $\Conv(u,v)$ that are tangent to $\mathcal E_{i,j}$.
  \end{enumerate}
 }
\end{lemma}
\begin{proof}
 \rsp{By \defnref{crossingsliding}, $\md_{\mathcal I^*} = \{\md_i,\md_j\}$ with $i,j\in \mathcal U$ and $\md_i\ne \md_j$. We are therefore in case \ref{a} of \defnref{convtrop} (case \ref{b} is not relevant because $\md_{\mathcal V^*}=\{\md_l\}$ is a singleton). Hence $\md\in \Conv(u,v)$ if only if $\md \in \conv(\md_i,\md_l)$ or $\md \in \conv(\md_j,\md_l)$ and \eqref{ConvNullcline} therefore follows. The remaining statements in \ref{transversal} and \ref{tangential} are trivial, and further details are therefore left out.}
\end{proof}
\rsp{The case of \lemmaref{ConvNullcline} where $i,j\in \mathcal V$ and $l\in \mathcal U$ (so that $v$ is the dominating direction) is similar and the corresponding details are therefore left out for simplicity. 

The vector \eqref{mdtrop} will be called the \textit{nullcline sliding vector}. It is constant on the open subset $\mathcal E_{i,j}\cap \mathcal R_l^{\mathcal V}$ of $\mathcal E_{i,j}$ (of transversal nullcline sliding type).} 


\begin{remark}
Notice that transversal nullcline sliding is associated with a time scale separation of \eqref{uvEqn} for $\epsilon\to 0$. This is due to the fact that the vector $\md_{l}$ in \figref{sliding}(c) and (d) corresponds to the flow vector of \textit{sub-dominant monomial $F_l$}: $F_l(u,v)<F_i(u,v)=F_j(u,v)$ for $(u,v)\in \mathcal E_{i,j}\cap \mathcal R_l^{\mathcal V}$, $\{i,j\}=\argmax_{k\in \mathcal I} F_k(u,v)$, and $l=\argmax_{k\in \mathcal V} F_k(u,v)$. Consequently, $v$ is slow in comparison with $u$ in (c) and (d) for \eqref{uvEqn} for $0<\epsilon\ll 1$. Our figures reflect this fact insofar that we use two types of arrows: Double-headed arrows for fast directions (including $\md_{\mathrm{Fp}}$) and single headed arrows for slow directions (along $\md_{\mathrm{nc}}$, see (c)) (a convention frequently used in slow-fast systems).
\end{remark}
 \begin{figure}[H]
    \centering
        \includegraphics[width=0.5\linewidth]{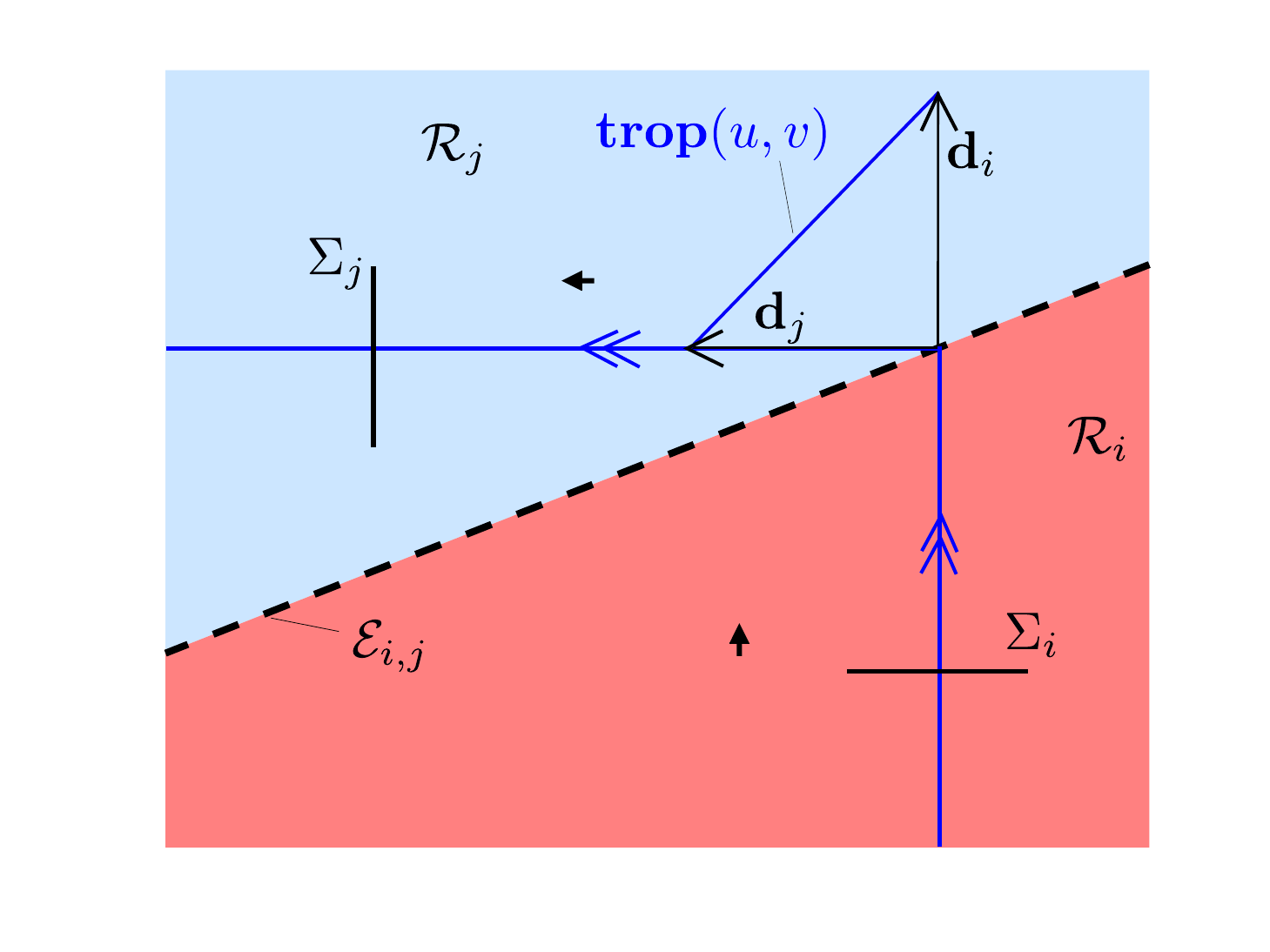}
        \caption{Illustration of crossing. }
\figlab{crossing}
\end{figure}
 
 \subsection{\rsp{The crossing flow}}\seclab{crossingflow}
\rsp{Finally, regarding crossing (see \figref{crossing}) we notice
that a crossing \rsp{edge} $\mathcal E_{i,j}\subset \mathcal T$ with normal vector $ \mathbf n_{i,j}$, see \eqref{nij}, is neither horizontal or vertical (i.e. $n_i\ne n_j$ and $m_i\ne m_j$).  
It follows that the flow of \eqref{uvnotinT} near crossing switching manifolds is well-defined. In fact, as indicated in \figref{crossing}, it coincides with the usual notion of crossing in PWS systems where orbits are uniquely concatenated across the switching manifold. This process leads to local transition maps that we now describe:}

Given that $\md_i\cdot \md_j=0$, we may suppose that $\md_i=(0,*)$  (otherwise, \rspp{by the definition of flow vectors in \eqref{dk}}, we can shift the role of $i$ and $j$). We will also suppose that $\md_i$ points towards $\mathcal E_{i,j}$; if not, we can reverse the direction of time. We are in the situation illustrated in \figref{crossing} (where $\md_i=(0,1)$), and consider segments $\Sigma_l\subset \mathcal R_l$, $l=i,j$, transverse to the flow, defined by fixed $v=\textnormal{const.}$, and $u\in I_i$ for $l=i$, and fixed $u=\textnormal{const.}$, and $v\in J_j$ for $l=j$. Here $I_i$ and $J_j$ are open and nonempty intervals, such that the transition map $P_{i,j}:\Sigma_i\to \Sigma_j$, induced by the forward flow of \eqref{uvnotinT}, is well-defined. Clearly, $P_{i,j}(u)$ satisfies
\begin{align*}
 (n_j-n_i)u + (m_j-m_i) P_{i,j}(u) +\alpha_j-\alpha_i=0,\quad \text{for all}\quad u\in I_i,
\end{align*}
\rsp{recall \eqref{nij}}.
Consequently, we have:
\begin{lemma}\lemmalab{mapPij}
 \begin{align}
  P_{i,j} (u) = \frac{n_j-n_i}{m_i-m_j} u +\frac{\alpha_j-\alpha_i}{m_i-m_j},\quad \textnormal{for all}\quad u\in I_i.\eqlab{mapPij}
 \end{align}
\end{lemma}
\begin{remark}\remlab{crossing}
More generally, we can concatenate orbits across consecutive crossing \rsp{edges} of $\mathcal T$ and we can describe the dynamics by composing maps of the form \eqref{mapPij}; this leads to transition maps of the form $u\mapsto c u+b(\alpha)$, with $\alpha\mapsto b(\alpha)$ affine, and where $c\in \mathbb Q_+$ only depends on the degrees $\degree F_k$, $k\in \mathcal I$. In this sense, \eqref{uvnotinT} (as usual) induces an (incomplete) flow:
\begin{lemma}\lemmalab{crossingflow}
Let $\mathcal T_s\subset \mathcal T$ denote the union of all switching manifolds of sliding type. Then the system \eqref{uvnotinT} defines a usual flow (incomplete, in general)  on $\mathbb R^2\backslash \mathcal T_s$.
\end{lemma} 
We will often need to refer to this flow and we will do so by using the following reference:
\begin{center}
\textnormal{The crossing flow (induced by \eqref{uvnotinT})}.
\end{center}
%
\end{remark}

\rsp{For completeness, we assert that $\Conv(u,v) = \conv(\md_i,\md_j)$ on a crossing switching manifold $\mathcal E_{i,j}$. Indeed, we are 
in case \ref{a} of \defnref{convtrop} with $\md_{\mathcal I^*} = \{\md_i,\md_j\}$ and $i\in \mathcal U$, $j\in \mathcal V$. Most importantly, $\textbf{0}\notin \Conv(u,v)$ in this case. Therefore the flow of the differential inclusion $(\dot u,\dot v)\in \Conv(u,v)$ is well-defined near crossing type switching manifolds and coincides with the crossing flow.}



\subsection{\rsp{Definition of a} tropical dynamical system}

We are now in a position to define our tropical dynamical system. (Recall \defnref{Ml}.)
\begin{definition}\defnlab{tropsystem}
Given (a): two nonempty, \rspp{finite} and disjoint index sets ${\mathcal U}\subset \mathbb N$ and ${\mathcal V}\subset \mathbb N$, (b): ${\mathcal I}:={\mathcal U} \cup {\mathcal V}$ and (c): two sets of tropical pairs:
\begin{enumerate}
\item $(\md_k,F_k)\in \{(\pm 1,0)\} \times \mathcal M^u$ for all $k\in {\mathcal U}$, with all $\deg F_k$, $k\in {\mathcal U}$ distinct,
\item $(\md_k,F_k)\in \{(0,\pm 1)\} \times \mathcal M^v$ for all $k\in {\mathcal V}$, with all $\deg F_k$, $k\in {\mathcal V}$ distinct,
\end{enumerate} 
we then define a \textnormal{tropical dynamical system} ($TDS$ for short) as the differential inclusion:
\begin{align}\eqlab{uvtrop}
 \rspp{(\dot u,
   \dot v)(t)\in \Conv(u,v)(t),}
\end{align}
\rspp{with $(u,v)(t):=(u(t),v(t))\in \mathbb R^2,\, t\in \mathbb R$.}
\rspp{Finally, we define the} \textnormal{degree of the tropical dynamical system} \rspp{as} 
\begin{align*}
 \max_{k\in{ \mathcal I}} \vert \degree F_k\vert_1+1.
\end{align*}
\end{definition}
In this definition, $\mathcal U$ and $\mathcal V$ are general index sets, not necessarily given by \eqref{UV}. This is convenient in examples (as in the autocatalator below, see \secref{auto}) where only a few monomials are present.

However, in our analysis of general degree $N$ tropical dynamical systems, we will stick to $\mathcal U$ and $\mathcal V$ and the enumeration of $(n_k,m_k)$ given by \eqref{UV} and \eqref{nkmk}, respectively. Specifically, we then define:
\begin{definition}
$\TDS$ ($\TDSN$ for short) is the set of tropical dynamical systems $TDS$ of a fixed degree $N\in \mathbb N$ and a fixed list of flow vectors $\{\md_k\}_{k\in \mathcal I}$.
\end{definition}
Consequently, we have that $TDS \in \TDSN$ if and only if $F_k\in \mathcal M_N^u$ ($F_k\in \mathcal M_N^v$) for all $k\in \mathcal U$ ($k\in \mathcal V$, respectively), recall (again) \defnref{Ml}. 

Let $$\alpha:=(\alpha_1,\ldots,\alpha_{2M})\in \rspp{\mathbf{A}}(2M),$$
recall \eqref{AN}.
We then topologize $\TDSN$ in the obvious way:
\begin{definition}\defnlab{openTDS}
An open neighborhood $\mathcal O$ of a tropical dynamical system $$TDS'\in \TDS,$$ with tropical coefficients $\alpha'\in \rspp{\mathbf{A}}(2M)$, is the set of tropical dynamical systems $TDS$ with $\alpha\in O$, $O\subset \rspp{\mathbf{A}}(2M)$ being an open neighborhood of $\alpha'$ in $\rspp{\mathbf{A}}(2M)$.
\end{definition}
In examples, with only a few monomials and a few parameters $(\alpha_1,\ldots,\alpha_L)\in (\mathbb R\cup \{-\infty\})^L$, as in the autocatalator model below (with $L=1$, $\alpha_1=\alpha$), see \secref{auto}, we define a neighborhood completely analogously using open sets $O\subset (\mathbb R\cup \{-\infty\})^L$. 

\begin{remark}\remlab{mathcalOO}
\rspp{In the following, we will use $\mathcal Y\subset \TDSN$ ($\mathcal Y=\mathcal O,\mathcal D$) to denote open sets in $\TDSN$ and denote by $Y\subset \rspp{\mathbf{A}}(2M)$ the associated open subsets of $\rspp{\mathbf{A}}(2M)$. $X$ (as in \lemmaref{upper}) will denote open sets in $\mathbb R^2$. }
\end{remark}



\section{Existence of solutions}\seclab{existence}
Our main aim will be to show that there are finitely many distinct structurally stable phase portraits in $\TDSN$. For this purpose, we need to define what we mean by a phase portrait of a $TDS\in \TDSN$ and how structural stability can be understood. We start by studying solutions of \eqref{uvtrop}. 

As is usual \cite{aubin1984a,filippov1988differential}, we define solutions of the differential inclusion \eqref{uvtrop} through $(u_0,v_0)$ to be absolutely continuous functions:
$$(u,v)
: \mathbb R \to \mathbb R^2,$$ with $(u,v)(0)=(u_0,v_0)$ and for which \eqref{uvtrop} holds almost everywhere. Given the nature of $\Conv(u,v)$, it is natural to restrict attention to piecewise affine solutions, i.e. for every $t_0\in \mathbb R$ there is an $\rspp{\theta}(t_0)>0$ such that $(u,v)\vert_{I}:I\to \mathbb R^2$,  
$I:=(t_0-\rspp{\theta},t_0+\rspp{\theta})$, 
satisfies: $(u,v)\vert_{I}'(t)=\mathrm{const.}$ for all $t\in I$, $t\ne t_0$. If the $\mathrm{const.}$ depends upon the sign of $t-t_0$, then $t_0$ is said to be a \textit{singular point} of the solution $u$. 

\begin{thm}\thmlab{existence}
Consider a tropical dynamical system $TDS$. Then there exists a piecewise affine solution $(u,v):\mathbb R\to \mathbb R^2$ of \eqref{uvtrop} through every point, whose singular points $\{t_n\}_n$ are enumerable and do not accumulate (if $\{t_n\}_n$ is infinite).
\end{thm}
\begin{proof}
We first consider the problem of a local piecewise affine solution $(u,v):I\to \mathbb R^2$, $I=[-\rspp{\theta},\rspp{\theta}]$, $\rspp{\theta}>0$, through a point $(u_0,v_0)$. 

If $\textnormal{\textbf{0}}\in \Conv(u_0,v_0)$ then we are done. 
\rspp{From the results in \secref{convsli} and \secref{crossingflow}, see also \figref{sliding} and \figref{crossing}, we can also construct local piecewise affine solutions near $(u_0,v_0)$} if
\begin{align}\eqlab{case1}
1\le \#\argmax_{k\in \mathcal I}F_k(u_0,v_0)\le 2.
\end{align}
\rspp{Such solutions} can either (a) be extended globally or (b) be extended to points $(u_0,v_0)$ where $$L:=\#\argmax_{k\in \mathcal I}F_k(u_0,v_0)\ge 3.$$ We therefore restrict attention to such points $(u_0,v_0)$.
We may here suppose that $\degree F_i$, $i\in \argmax_{k\in \mathcal I}F_k(u_0,v_0)$, are mutually distinct, because otherwise we can reduce it to \eqref{case1}. The same is true if $\degree F_i$ were co-linear, since then $\mathcal T$ consists of a single tropical \rsp{edge} near $(u_0,v_0)$.

Consequently, we are left with the case where the point $(u_0,v_0)$ is associated with a $L$-sided polygon face of the polyhedral subdivision $\mathcal S$, such that $L$ \rsp{tropical edges} of $\mathcal T$ come together at $(u_0,v_0)$. Moreover, $\textnormal{\textbf{0}}\notin \Conv(u_0,v_0)$. 

There are two cases to consider:
\begin{enumerate}
 \item \label{case1} $\md_{\mathcal I^*}=\md_{\mathcal L^*}$, with either $\mathcal L=\mathcal U$ or $\mathcal V$,
 \item \label{case2} or, $\md_{\mathcal I^*}=\md_{\mathcal U^*}\cup \md_{\mathcal V^*}$.
\end{enumerate}
In case \ref{case1}, we take $\mathcal L=\mathcal V$ for simplicity ($\mathcal L=\mathcal U$ is identical). If $\md_{\mathcal I^*}$ only contains one element, then we are done as we can just follow this direction. We therefore suppose that $\md_{\mathcal I^*}=\{(0,\pm 1)\}$. We can always follow the positive or negative $v$-direction in one direction of time. This is consequence of \lemmaref{upper}; there will be vectors in $\Conv(u_0,v)$, $v\in I\backslash\{v_0\}$, with $I$ a neighborhood of $v_0$, that are either pointing in the same direction or are in opposition; notice that this also holds if $u=u_0$ contains a \rsp{tropical edge} because vectors within the tropical regions are extended to $\mathcal T$ by $\Conv$. Now, if these vectors point in the same direction along $u=u_0$, then we are done. We therefore suppose that these vectors are in opposition and suppose without loss of generality that solutions can be extended backward in time. We will now construct a solution forward in time. For this purpose, we first use that $\textnormal{\textbf{0}}\notin \Conv(u_0,v_0)$ to deduce that $\md_{\mathcal U^*}$ only contains one element. We suppose it is $(1,0)$ ($(-1,0)$ is identical). Then we take an arc going clockwise from the positive $v$-direction to the negative $v$-direction of sufficiently small radius. Along this $180^\circ$
arc, we have to intersect at least one \rsp{tropical edge} (simply due to the fact that the face  of $\mathcal S$ is an $L$-sided polygon).  By \lemmaref{upper}, at least one of these \rsp{tropical edges} has to be of nullcline sliding-type. Indeed, if not, then going along the arc each \rsp{tropical edge} would be of crossing type and we arrive at the contradiction that the vectors along $u=u_0$ point in the same direction. 

Along the \rsp{tropical edge} of nullcline sliding type, we have a tangent vector, cf. \lemmaref{ConvNullcline}, which by \lemmaref{upper} belongs to $ \Conv(u_0,v_0)$. This gives the desired forward solution.

Next, in case \ref{case2},
only one of $\md_{\mathcal U^*}$ and $ \md_{\mathcal V^*}$ can contain two elements. Otherwise $\textnormal{\textbf{0}}\in \Conv(u_0,v_0)$. We suppose $\md_{\mathcal U^*} = \{(1,0)\}$ and that $(0,1) \in\md_{\mathcal V^*}$ without loss of generality. We now construct a forward solution. The case of a backward solution is similar and therefore left out. 

Either we have:
\begin{enumerate} 
 \item[(\textnormal{a})] $(0,1)\notin \Conv(u_0,v)$ and $(1,0)\notin \Conv(u,v_0)$, for all $v>v_0$, $v\sim v_0$ and $u>u_0$, $u\sim u_0$, respectively,
  \item[(\textnormal{b})] or we can continue a solution along the positive $v$-direction or along the positive $u$-direction.
\end{enumerate}
We are therefore left with the first case. Then by \lemmaref{upper}, we have 
 $(1,0)\in \Conv(u_0,v)$ and $(0,1)\in \Conv(u,v_0)$, for all $v>v_0$, $v\sim v_0$ and $u>u_0$, $u\sim u_0$. We can then proceed as in case \ref{case1} above and take a $90^\circ$-arc going clockwise from the positive $v$-direction to the positive $u$-direction, which has to intersect at least one \rsp{tropical edge} of Filippov sliding-type. Along this set we have a unique tangent vector $\md_{\Fil}\in \Conv(u_0,v_0)$, the Filippov sliding vector, recall \lemmaref{ConvFilippov}, which points away from $(u_0,v_0)$. This completes the proof of existence of a local piecewise affine solution.

 Now, we turn to the problem of globalizing the solution. We focus on extending the solution forward in time, since backward in time is identical. For this purpose, we simply apply the same procedure, using $(u_0,v_0)=(u(\rspp{\theta}),v(\rspp{\theta}))$ as initial condition in the first step. This procedure either leads to a global forward solution $(u,v):[0,\infty)\to \mathbb R^2$, as desired, or there is some $T>0$ such that the forward solution $(u,v):[0,T)\to \mathbb R^2$ cannot be continued beyond $T>0$. From our construction, and the fact that there are finitely many \rsp{tropical vertices} $\mathcal P\subset \mathcal T$ and finitely many \rsp{tropical edges} $\mathcal E\subset \mathcal T$, we are left with the only possibility that $(u,v)(t)\to \infty$ as $t\to T^-$. But this is impossible in finite time, since all vectors in $\Conv$ are bounded in Euclidean norm by $1$. Consequently, $T=\infty$. It also follows from these arguments that the singular points $\{t_n\}_n$ of our solution $(u,v)$ do not accumulate, $t_n\not \to t_*$. 

 \end{proof}
 \begin{remark}
 Not all solutions are necessarily piecewise affine. Consider for example the case where 
 \begin{align}\eqlab{degenerate}
 F_i(u',v')=F_j(u',v')>F_k(u',v')\quad \text{for all} \quad k\in \mathcal I,k\ne i,j,
 \end{align} 
 where $\degree F_i=\degree F_j$, $i\in \mathcal U$, $j\in \mathcal V$, for some $(u',v')$. In this case, \eqref{degenerate} holds on a neighborhood $X$ of $(u',v')$. Within $X$, we have $\Conv(u,v)=\conv(\md_i,\md_j)$, $\md_i\cdot \md_j=0$. Consequently, any differentiable curve $s\mapsto (u,v)(s)\in X$, with $\operatorname{sign}(u'(s)\delta_i)=\operatorname{sign}(v'(s)\delta_j)$ can be re-parametrized as a solution $t\mapsto (u,v)(s(t))$ of \eqref{uvtrop}.  
  Having said that, \eqref{degenerate} is not stable to perturbations, in the sense that if $\alpha_i\ne \alpha_j$, then \eqref{degenerate} holds nowhere, 
  recall \remref{degree2}.
  
Moreover, even if \eqref{degenerate} does not hold for any pair $\degree F_i=\degree F_j$, then there may still exist solutions where the singular points $\{t_n\}_n$ of a solution of \eqref{uvtrop} accumulate, such that $t_n\to t_*$. These are so called \textnormal{zeno trajectories}, and they are covered for classical Filippov systems in \cite[Theorem 8.1]{broucke2001a}. They occur, for example, at so-called \textit{elliptic sectors} in our case; an example of this case is a neighborhood of the origin in \figref{tropauto2}(b) (compare with \cite[Fig. 8]{broucke2001a}). It is reminiscent of a canard point \cite{krupa2001a} and here solutions can be continued through the \rsp{tropical vertex}  $\mathcal P_{456}$ and along the repelling transversal nullcline sliding manifold $\mathcal E_{46}$ and then jump back down towards the stable transversal nullcline sliding manifold $\mathcal E_{45}$, arbitrarily close to $\mathcal P_{456}$. In our setting, zeno trajectories also occur along tangential nullcline sliding, see \figref{sliding}(d) and \remref{hybrid}.

Having said that, \thmref{thm0} guarantees that there is \textnormal{always} a piecewise affine solution of \eqref{uvtrop} through every point, whose singular points $\{t_n\}_n$ do not accumulate. 
 \end{remark}
 \section{Orbits and phase portraits}\seclab{orbits}


Solutions of a tropical dynamical system are nonunique (as usual for piecewise smooth systems, see \cite{filippov1988differential} and \figref{sliding} for an example (due to sliding)), and the tropical system does therefore not define a usual flow. Nevertheless, 
we can still use the existence of piecewise affine solutions (see \thmref{existence}) to define an orbit (nonunique) through any given point, including a forward orbit and a backward orbit, see also \cite{broucke2001a,filippov1988differential}. This also leads to $\alpha$ and $\omega$-limit sets. Moreover, since the solutions of \thmref{existence} are piecewise affine, we obtain orbits as oriented polygonal curves, i.e. oriented and connected line segments (finitely many or countably infinitely many), and we will restrict attention to such polygonal orbits. 

We can write a polygonal orbit $\gamma$ as the union of a set of edges $E=\{l_i\}_i$ and a set of vertices $V=\{q_i\}_i$ in such a way that (a) $l_i$ is the oriented line segment from $q_i$ to $q_{i+1}$ and (b) $l_i$ and $l_{i+1}$ have different orientations in the sense that the vertices $V$ correspond to the singular points $\{t_i\}_i$ (that do not accumulate) of a solution having $E\cup V$ as its image. Here $i$ runs over:
\begin{enumerate} 
 \item A finite index set, when the orbit is isomorphic to a closed interval or a closed curve.
 \item $(-\mathbb N)=\{\ldots,-2,-1\}$ or $\mathbb N$, when the orbit is isomorphic to $(-\infty,0]$ (infinitely many line segments in backward time) or $[0,\infty)$ (infinitely line many segments in forward time), respectively.
 \item $\mathbb Z$, when the orbit is isomorphic to $\mathbb R$ (infinitely many line segments in forward and backward time).
\end{enumerate}
If $\gamma$ is a point (corresponding to a \rsp{tropical singularity} ) then $E$ is the empty set and $V=\{\gamma\}$, whereas if $\gamma$ is a single straight line then $E=\{\gamma\}$ and $V$ is the empty set. 

We then define the \textit{phase portrait of a tropical dynamical system} as the set of polygonal orbits. 

For an oriented line segment $l_i$, we will use ``flow orientation'' to refer to the orientation of time, whereas we use ``line orientation'' to refer to the geometric orientation (i.e. the slope) of $l_i$ in $\mathbb R^2$. 
A periodic orbit $\gamma$ is then a closed polygonal orbit of a piecewise affine periodic solution of \eqref{uvtrop}  and it is a limit cycle if it is the $\alpha$ or $\omega$-limit set of points $p\notin \gamma$. 
Due to nonuniqueness, we can (without further assumptions) have two limit cycles $\gamma_1$ and $\gamma_2$ with nonempty intersection (see \figref{foldfold} below for an example). Moreover, polygonal orbits in the neighborhood $X$ of a limit cycle $\gamma$ may (due to sliding) intersect $\gamma$ in finite forward or backward time. As is usual in piecewise smooth systems, we define:
\begin{definition}\defnlab{crossing} Suppose that $\gamma$ is a periodic orbit. Then it is said to be a \textnormal{crossing cycle} if it is a periodic orbit of the crossing flow (induced by \eqref{uvnotinT}, recall \remref{crossing}), with all intersections of the discontinuity subset of $\mathcal T$ (where $\md_i\ne \md_j$) belonging to switching manifolds of crossing-type. Otherwise, $\gamma$ is called a \textnormal{sliding cycle}. 
\end{definition}
Notice here that a crossing cycle does not pass through \rsp{tropical vertices} of $\mathcal T$. This is convenient for our purposes, see \secref{crossing2}.
\begin{remark}
The paper \cite{broucke2001a} establishes generic structural stability (in the usual sense) for Filippov systems on surfaces. Here orbits are not \textnormal{amalgamated} (in contrast to the approach by Filippov \cite{filippov1988differential}) along the sliding switching manifolds. Applying this viewpoint to \figref{sliding}(a), we only have one orbit through any point $\mathcal E_{i,j}$  (the set $\mathcal E_{i,j}$ itself). In the present paper, when we refer to the crossing flow, then we are only amalgamating orbits across the switching manifolds of crossing type (crossing orbits), ``stopping'' when these reach sliding segments. However, for our definition of equivalence, see \secref{equivalence}, it will be important to \textnormal{amalgamate orbits along the switching manifolds of sliding type} as in \cite{filippov1988differential}. Consequently, with our convention, we will have infinitely many orbits through any point on $\mathcal E_{i,j}$ in \figref{sliding}(a)). 
\end{remark}

\section{\rsp{Example: A} tropicalized autocatalator}\seclab{auto}
In this section, we illustrate the concepts further by working on a tropicalized version of the \textit{autocatalator model}:
\begin{equation}\eqlab{auto}
\begin{aligned}
 \dot x &=\rspp{\theta} (\mu-x-xy^2),\\
 \dot y&=-y+x+xy^2,
\end{aligned}
\end{equation}
with $0<\rspp{\theta}\ll 1$ and $\mu>0$. \eqref{auto} is referred to as the autocatalator, since it models an autocatalytic chemical process with $(x,y)\in \mathbb R_+^2$ being (scaled) concentrations; the autocatalytic nature of the process is modeled by the nonlinear term $xy^2$, see \cite{petrov1992a}. The system was analyzed in \cite{Gucwa2009783} using GSPT and has since then become the prime example of a slow-fast polynomial system supporting nonstandard relaxation oscillations (for $\mu>1$ fixed and all $0<\rspp{\theta}\ll 1$, see \cite[Theorem 2.3]{Gucwa2009783}).\footnote{\rspp{Note that $\theta$ is called $\epsilon$ in \cite{Gucwa2009783}. We adopt this change in notation to avoid confusion with $\epsilon$ in \eqref{uv} and \eqref{tropautoalpha} below}.} 

In further details, \cite{Gucwa2009783} shows that there is a Hopf bifurcation at $\mu\approx 1$, which is of canard type \cite{krupa2001a}, in the sense that the Hopf cycles grow (through canard cycles) in amplitude by an $\mathcal O(1)$-amount in a parameter regime of a width of order $\mathcal O(\e^{-c/\rspp{\theta}})$. But in contrast to e.g. the van der Pol system \cite{krupa2001a}, the canard cycles grow unboundedly for \eqref{auto} as $\rspp{\theta}\to 0$. 

We illustrate the phase portrait of \eqref{auto} for $\mu=2$ and $\rspp{\theta}=0.001$ in \figref{auto} (computed in \verb#Matlab# using \verb#ODE45# with low tolerances ($10^{-12}$)). Due to the unboundedness of the limit cycle $\gamma_\rspp{\theta}$ (in red) as $\rspp{\theta}\to 0$ in the $(x,y)$-coordinates, see \figref{auto}(a), we use the zoomed-out coordinates $(x,\rspp{\theta} y)$ in \figref{auto}(b). The blue dashed line in \figref{auto}(a) is the $y$-nullcline for \eqref{auto}. It defines a critical manifold for $\rspp{\theta}=0$, being attracting for $y<1$ and repelling for $y>1$. The point $(x,y)=(\frac12,1)$ is a fold point, in particular a canard point for $\mu=1$ when the $x$-nullcline (in purple and dashed in \figref{auto}) intersects the $y$-nullcline at the fold of the critical manifold. 

For $\mu<1$ and $0<\rspp{\theta}\ll 1$, there is a globally attracting stable node; in \figref{auto}(a) for $\mu=2$ (on the other side of the Hopf) it is an unstable node (black dot).

\begin{figure}[H]
    \centering
    \begin{subfigure}{0.5\textwidth}
    \centering
        \includegraphics[width=0.95\linewidth]{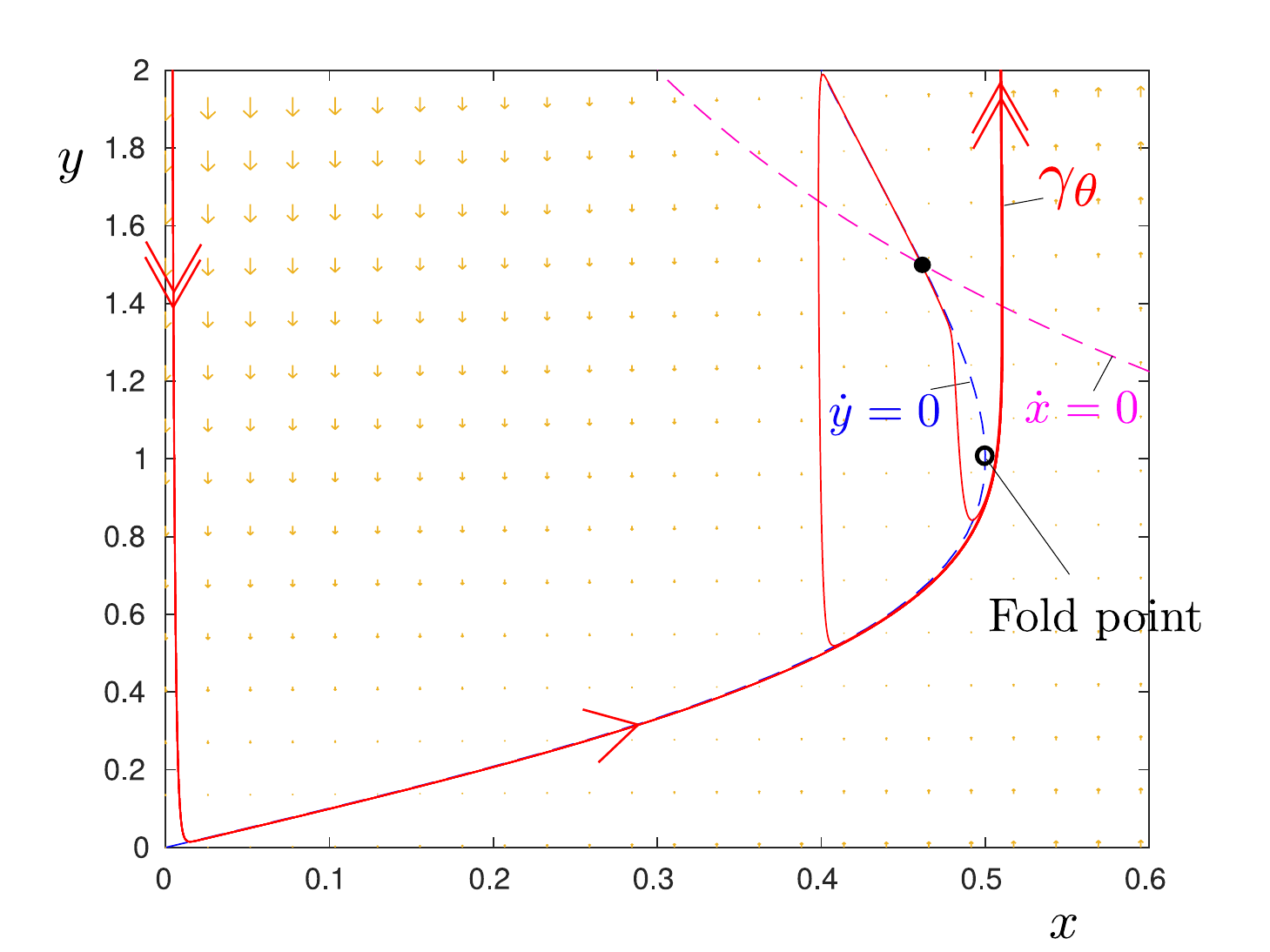}
        \caption{}
    \end{subfigure}%
    \begin{subfigure}{0.5\textwidth}
    \centering
        \includegraphics[width=0.95\linewidth]{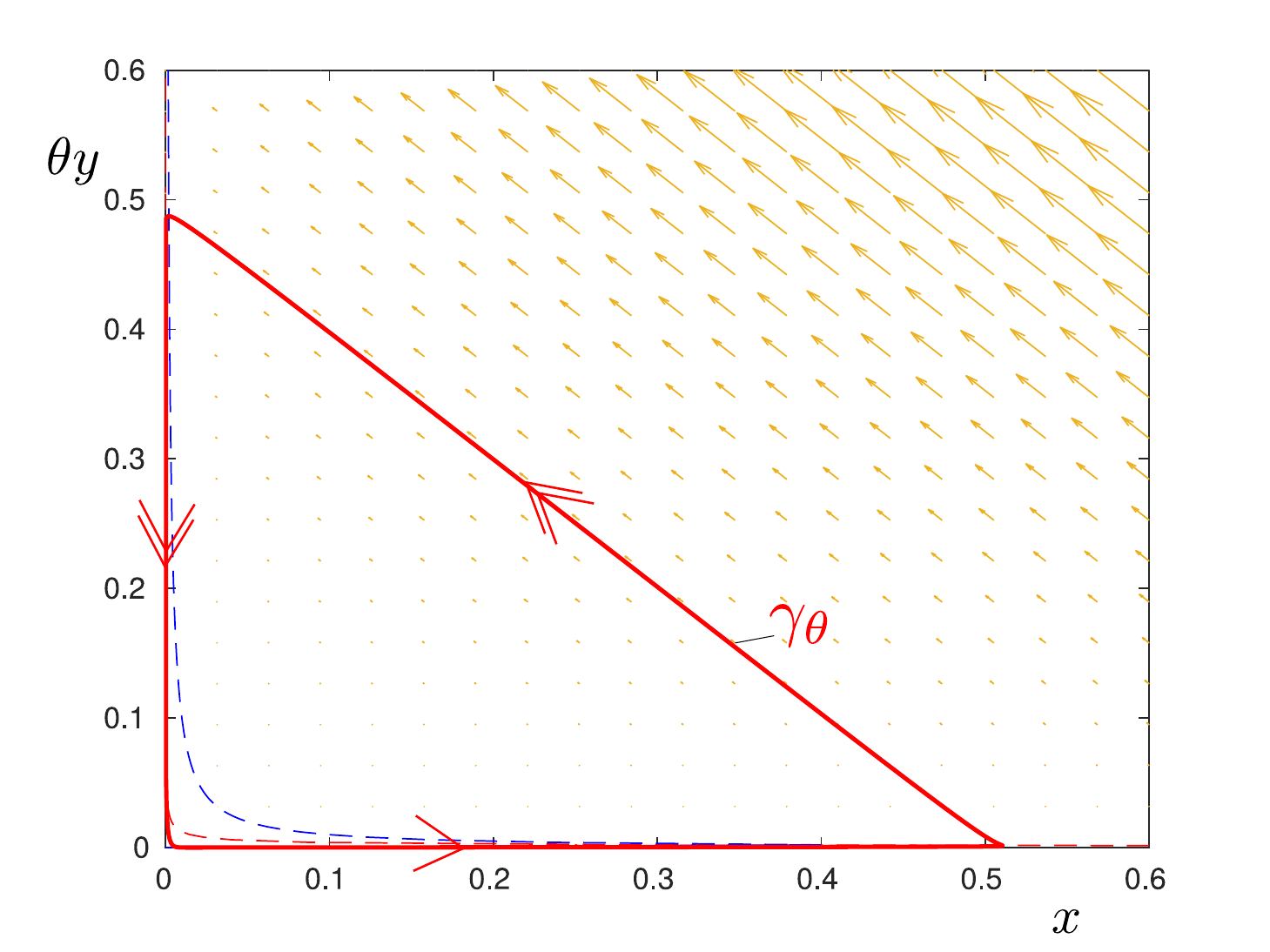}
        \caption{}
    \end{subfigure}
    \caption{Phase portrait of \eqref{auto} for $\mu=2$ and $\rspp{\theta}=0.001$, using two different scalings: $(x,y)$ in (a) and $(x,\rspp{\theta} y)$ in (b). The limit cycle $\gamma_\rspp{\theta}$ \rspp{(in red)} is unbounded in the $(x,y)$-coordinates as $\rspp{\theta}\to 0$ but bounded in $(x,\rspp{\theta} y)$. }
    \figlab{auto}
    \end{figure}

    We now consider a tropical version of \eqref{auto} by defining $u$ and $v$ as in \eqref{uv} and setting
    \begin{align}
     \rspp{\theta} = \e^{-1/\epsilon},\quad \mu = \e^{\alpha/\epsilon}.\eqlab{tropautoalpha}
    \end{align}
    Taking the limit $\epsilon\to 0$ \rspp{(with $\alpha$ fixed)}, we are then led to consider the tropical dynamical system (\rspp{of degree $3$}) defined by the following tropical pairs:
\begin{equation}\eqlab{tropauto}
\begin{alignedat}{3}
 F_1(u,v)&:=\alpha-1-u,\quad &\md_1&:=(1,0),\\
 F_2(u,v)&:=-1,\quad  &\md_2&:=(-1,0),\\
 F_3(u,v)&:=-1+2v, \quad &\md_3&:=(-1,0),\\
 F_4(u,v)&:=0, \quad &\md_4&:=(0,-1),\\
 F_5(u,v)&:=u-v,\quad &\md_5&:=(0,1),\\
 F_6(u,v)&:=u+v,\quad &\md_6&:=(0,1),
\end{alignedat}
\end{equation}
and $\mathcal U=\{1,2,3\}$, $\mathcal V=\{4,5,6\}$. 
\rspp{Notice that $\alpha$ is the single parameter and that it only enters through the tropical coefficient of $F_1$}.
In \figref{tropauto}(a), we show the graph of the associated tropical polynomial $$F_{\max}(u,v)=\max_{k\in \{1,\ldots,6\}}F_k(u,v),$$ for $\alpha=\frac14$. In \figref{tropauto}(b), we illustrate the corresponding subdivision of the Newton polygon. Notice that we label the subdivision according to the associated flow vector in the tropical pair, see \eqref{tropauto}. The purple edges then indicate where we have crossing (see \eqref{crossingdef}); \rspp{this colouring relates to the concept of a crossing graph introduced in \secref{graph} below}.  In \figref{tropauto}(c), we show the resulting tropical phase portrait for the same value of $\alpha$. The numbers $1,3,4,5,6$ in \figref{tropauto} correspond to the number of the dominating monomial, see \eqref{tropauto}. Notice that $2$ is missing. This is due to the fact that the set $\mathcal R_2$ is the empty set in the present case: $F_2(u,v)=-1<F_4(u,v)=0\le F_{\max}(u,v)$ for all $(u,v)\in \mathbb R^2$ and all $\alpha\in \mathbb R\cup \{-\infty\}$. The corresponding flow vectors are indicated as black arrows in \figref{tropauto}(c).

We notice the following in \figref{tropauto}(c). Firstly, there is a point $\mathcal Q_{1346}=(\frac14,\frac14)$ where $\textnormal{\textbf{0}}\in \Conv(u,v)$ (white disc with black boundary). This is  due to the intersection of the transversal nullcline sliding switching manifolds $\mathcal E_{13}^{\mathcal U}$ and $\mathcal E_{46}^{\mathcal V}$. They are both of unstable type and $\mathcal Q_{1346}$, which we call a \rsp{tropical singularity}, acts like an unstable node (or source); see further details in \secref{tropeq}. Secondly, there is also an attracting sliding limit cycle (in red). Notice that it slides along $\mathcal E_{14}$ (tangential Filippov) and along $\mathcal E_{13}^{\mathcal U}$ and $\mathcal E_{45}^{\mathcal V}$ (both stable transversal nullcline). The \rsp{tropical vertex}  $\mathcal P_{456}$ is reminiscent of a classical fold point (see \figref{auto}(a)) of a slow-fast system, insofar that $\mathcal P_{456}$ also separates an attracting nullcline from a repelling one. 

\begin{figure}[H]
    \centering
    \begin{subfigure}{0.5\textwidth}
    \centering
        \includegraphics[width=0.96\linewidth]{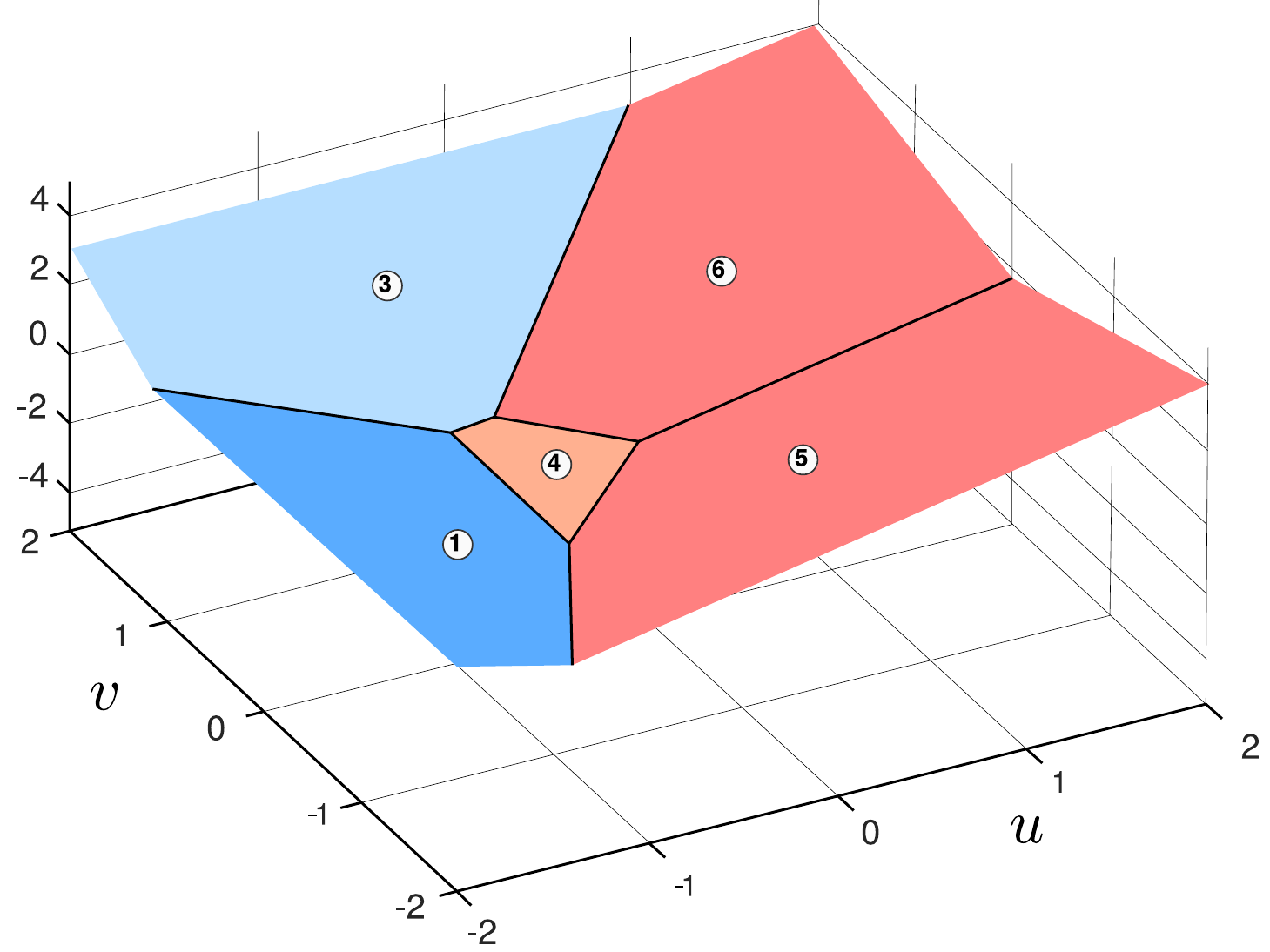}
        \caption{}
    \end{subfigure}%
    \begin{subfigure}{0.5\textwidth}
    \centering
        \includegraphics[width=0.96\linewidth]{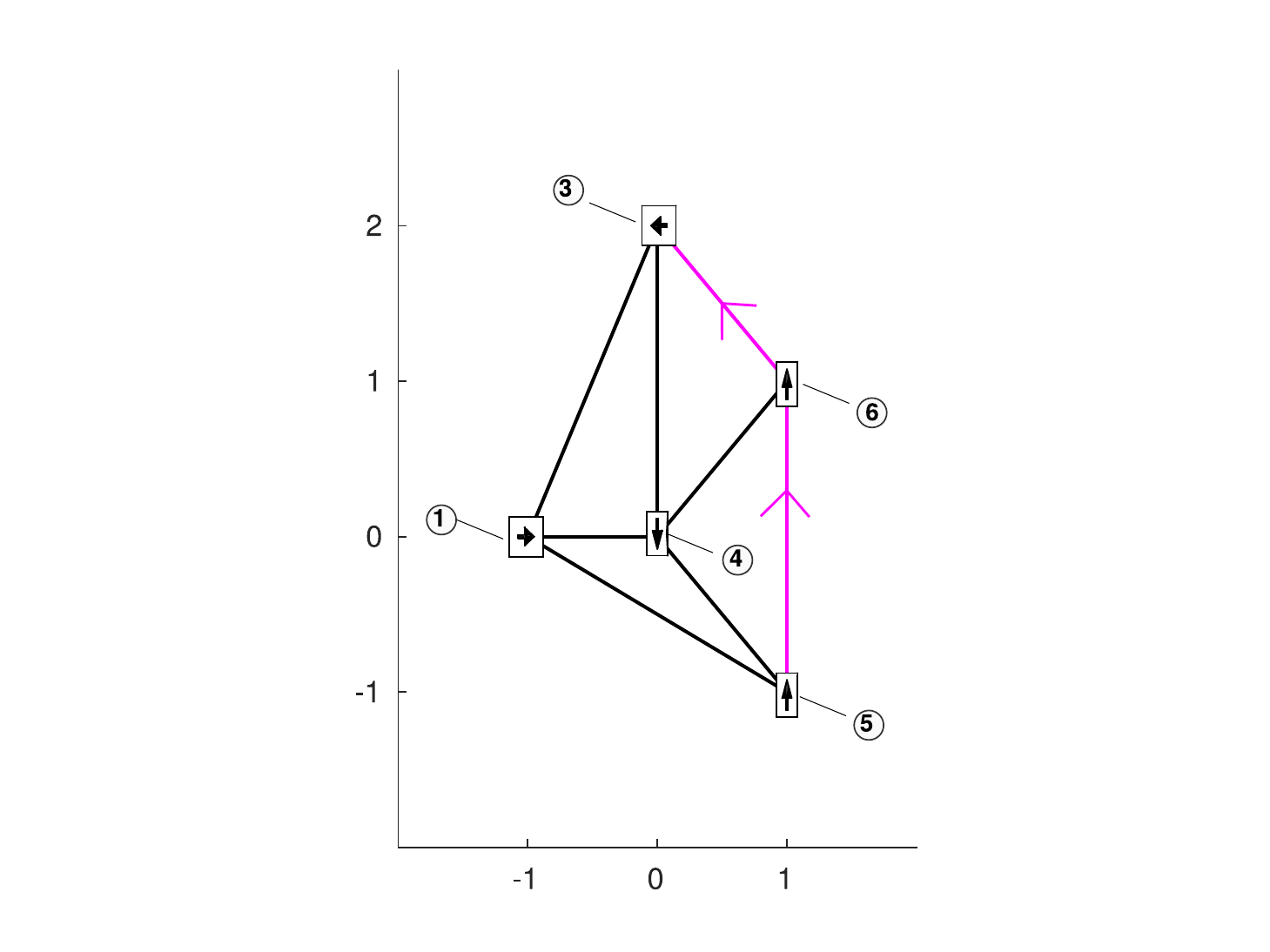}
        \caption{}
    \end{subfigure}
    \begin{subfigure}{0.5\textwidth}
    \centering
        \includegraphics[width=0.96\linewidth]{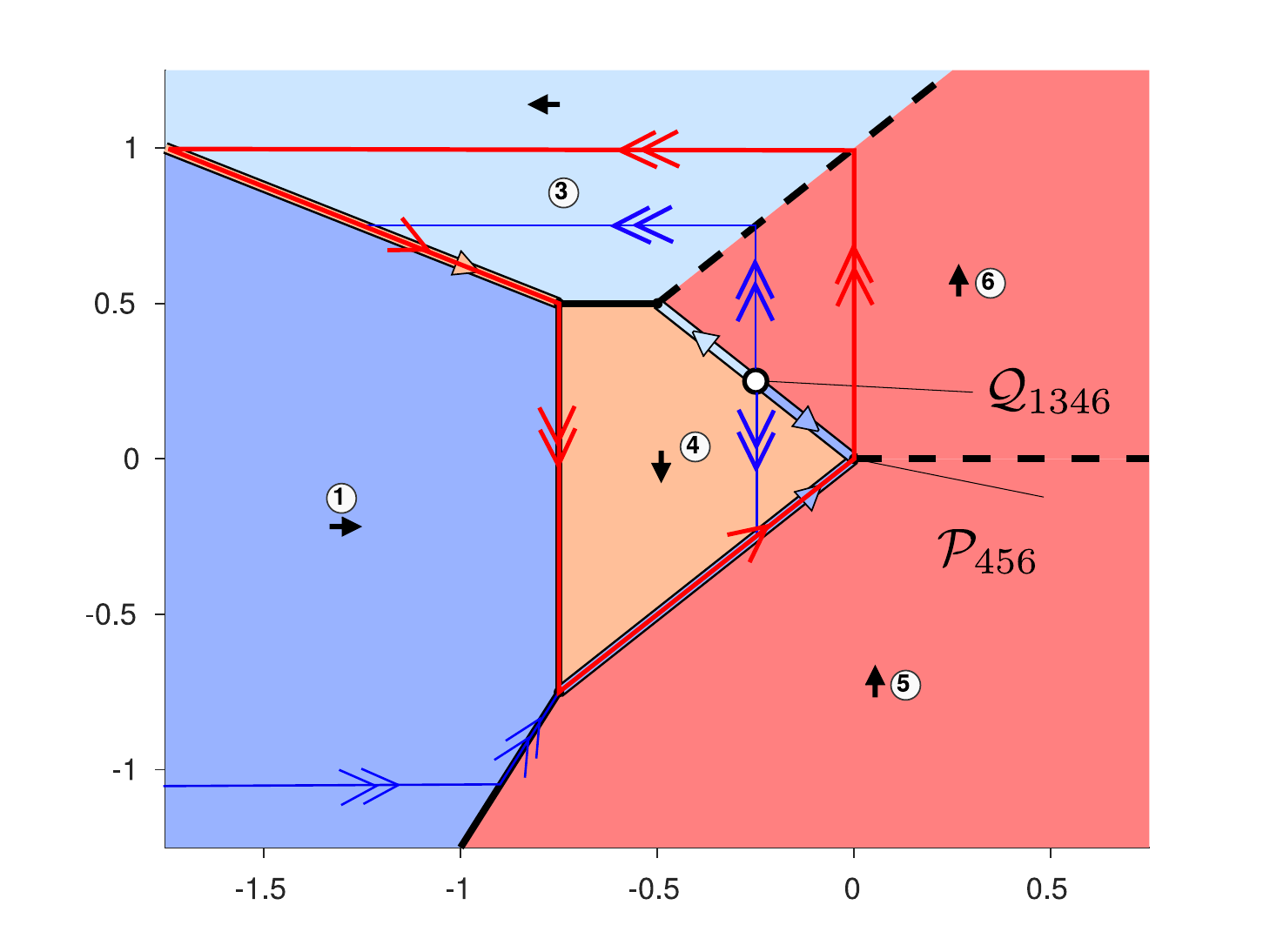}
        \caption{}
    \end{subfigure}     
    \caption{In (a): The graph of the tropical polynomial $F_{\max}$ associated with the tropical monomials in \eqref{tropauto}. In (b): The (labelled) subdivision associated with monomial in \eqref{tropauto}. In all figures, $\alpha=\frac14$. Finally, in (c): The tropical phase portrait with a source $\mathcal Q_{1346}$ and a sliding limit cycle (in red).  }
    \figlab{tropauto}
    \end{figure}

In \figref{auto}, we used $\mu=2$. This corresponds to $\alpha=\epsilon \log 2\to 0$ as $\epsilon\to 0$ (or equivalently $\rspp{\theta} \to 0$) upon using \eqref{tropautoalpha}. We therefore illustrate the case $\alpha=0$ in \figref{tropauto2}, see (b), together with three other values of $\alpha$: $\alpha=-\frac14$ in (a), $\alpha=\frac12$ in (c) and $\alpha=\frac34$ in (d). We notice the following: For $\alpha=0$, the \rsp{tropical singularity}  $\mathcal Q_{1346}$ has moved onto the \rsp{tropical vertex}  $\mathcal P_{456}$. Just like in the canard situation of \eqref{auto}, this value of $\alpha$ also marks the onset of limit cycles; compare \figref{tropauto}(c) for $\alpha=\frac14$ with \figref{tropauto2}(a) for $\alpha=-\frac14$. In the former case, we have a stable limit cycle, whereas in the latter case, we have a sink $\mathcal Q_{1345}$ on $\mathcal E_{45}^{\mathcal V}$, which attracts all points. Essentially, the phase portraits for any $\alpha<0$ and any $\alpha\in (0,\frac12)$ are qualitatively similar to those in \figref{tropauto2}(a) and \figref{tropauto}(c), respectively. (We will make this more precise below in \secref{equivalence}). In the case $\alpha=0$, we see (as in the classical canard explosion) cycles (in red) of different amplitudes co-existing. 
The compelling thing about the tropical phase portrait is that all scales (on the logarithmic scale induced by \eqref{uv}) are visible all at once. The blowup method offers a similar ``compactification'', see \cite[Fig. 7]{Gucwa2009783}. In comparison, we had to use two separate scales in \figref{auto} to visualize the dynamics.   

The limit cycle we see in \figref{tropauto}(c) for $\alpha=\frac14$ is not covered by the analysis of \cite{Gucwa2009783} (in a singular limit). Indeed, in \cite{Gucwa2009783} the authors keep $\mu=\mathcal O(1)$ as $\rspp{\theta}\to 0$. Based on the findings of the tropical dynamical system, see also \conjref{per1} below, we find that the limit cycles of \cite{Gucwa2009783} can be continued for $\mu$-values growing unboundedly (at a rate $\mu=\e^{\alpha/\epsilon}\to \infty$ as $\epsilon\to 0$ with $\alpha\in (0,\frac12)$). (Formally this requires a separate GSPT analysis, which we will leave to future work.) Within this context, it is interesting to \rspp{note} that there is another bifurcation at $\alpha=\frac12$ (see \figref{tropauto2}(c)), where the source $\mathcal Q_{1346}$ has moved on top of a different \rsp{tropical vertex}, $\mathcal P_{1346}$, so that $\textnormal{\textbf{0}}\in \Conv(\mathcal P_{1346})$.  This \rsp{tropical vertex}  (as indicated by the subscripts $1346$) is degenerate in the sense that the four monomials $F_1,F_3,F_4$ and $F_6$ attain the same value at this point. In \figref{tropauto3}, we show the subdivision associated with this case. In agreement with \figref{tropauto2}(c), the subdivision is not a triangulation (in contrast to \figref{tropauto}(b).

The value $\alpha=\frac12$ marks the termination of oscillations. In particular, as indicated in \figref{tropauto2}(d), we have a sink $\mathcal Q_{1346}$ on $\mathcal E_{13}^{\mathcal U}$ for any $\alpha \in(\frac12,1)$. In future work, it would be interesting to study the bifurcation at $\alpha=\frac12$ further as a singular bifurcation for $\rspp{\theta}\to 0$.  We expect that the techniques, developed in \cite{jelbart2021b,jelbart2021c,uldall2021a} for the analysis of bifurcations of smooth systems approaching nonsmooth ones, as well as those in \cite{kristiansen2017a} for dealing with exponential nonlinearities, will all be useful in this regard.  

Notice that although the situation near $\alpha=\frac12$ again bears some similarities with the canard explosion, it is clearly different, as it involves a mixture of fast directions. Indeed, along $\mathcal E_{46}$ the $v$-direction is fast/dominating, whereas along $\mathcal E_{13}$ the $u$-direction is fast/dominating. We notice that $\alpha=\frac12$ corresponds to 
\begin{align*}
 \mu =\rspp{\theta}^{-\frac12}, 
\end{align*}
by \eqref{tropautoalpha}. The following result demonstrates the effectiveness of the tropical approach in identifying significant scalings. 
\begin{lemma}
 The system \eqref{auto} has two Hopf bifurcations at
 \begin{align*}
  \mu &=\rspp{\theta}^{-\frac12}\sqrt{\frac{1-2\rspp{\theta}+\sqrt{1-8\rspp{\theta}}}{2}} = \rspp{\theta}^{-\frac12}\left(1 -\frac{3}{2}\rspp{\theta}+\mathcal O(\rspp{\theta}^2)\right),\\
  \mu &=\rspp{\theta}^{-\frac12}\sqrt{\frac{1-2\rspp{\theta}-\sqrt{1-8\rspp{\theta}}}{2}}=1+2\rspp{\theta}+\mathcal O(\rspp{\theta}^2),
 \end{align*}
for all $0<\rspp{\theta}\ll 1$. 
 \end{lemma}
 \begin{proof}
  Follows from a direct calculation: The unique \rspp{singularity} is at 
  \begin{align*}
   (x,y)=\left(\frac{\mu}{\mu^2+1},\mu\right),
  \end{align*}
and the determinant and trace of the linearization are given by
\begin{align*}
 \rspp{\theta}(1+\mu^2),\quad \frac{\mu^2-1}{\mu^2+1}-\rspp{\theta} (\mu^2+1),
\end{align*}
respectively. With the determinant being positive, we obtain the result, specifically the expressions for $\mu$ by setting the trace equal zero. 
 \end{proof}

Finally, regarding \figref{tropauto2}(d) for the value $\alpha=\frac34$, we notice that the purple orbit (reminiscent of a strong stable manifold of $\mathcal Q_{1346}$) goes through the \rsp{tropical vertex}  $\mathcal P_{146}$. At first sight, this might seem special. One would perhaps expect that if we vary $\alpha$ then this connection would break up. To see that this is not the case, we can either note that the line $\mathcal E_{36}$ is the bisector of the line $(\mathcal P_{146},\mathcal Q_{1346})$ or proceed by the direct calculation:

$\mathcal Q_{1346}$ has coordinates $(\alpha,\alpha)$ and it lies on $\mathcal E_{13}^{\mathcal U}$ for $\alpha>\frac12$. The purple orbit, obtained by following the flow vector $\md_3$ backwards from $\mathcal Q_{1346}$ then intersects the crossing switching manifold $\mathcal E_{36}$, given by $v=u+1$, at a point 
\begin{align}\eqlab{point1}
(-1+\alpha,\alpha).
\end{align}
At the same time, the \rsp{tropical vertex}  $\mathcal P_{146}$, obtained by setting $F_1(u,v)=F_4(u,v)=F_6(u,v)$, has coordinates 
\begin{align}\eqlab{point2}
(-1+\alpha,1-\alpha). 
\end{align}
The $u$-components of \eqref{point1} and \eqref{point2} coincide and the situation illustrated in \figref{tropauto2}(d) is therefore not special, but \textit{persistent}.

    \begin{figure}[H]
    \centering
    \begin{subfigure}{0.49\textwidth}
    \centering
        \includegraphics[width=0.96\linewidth]{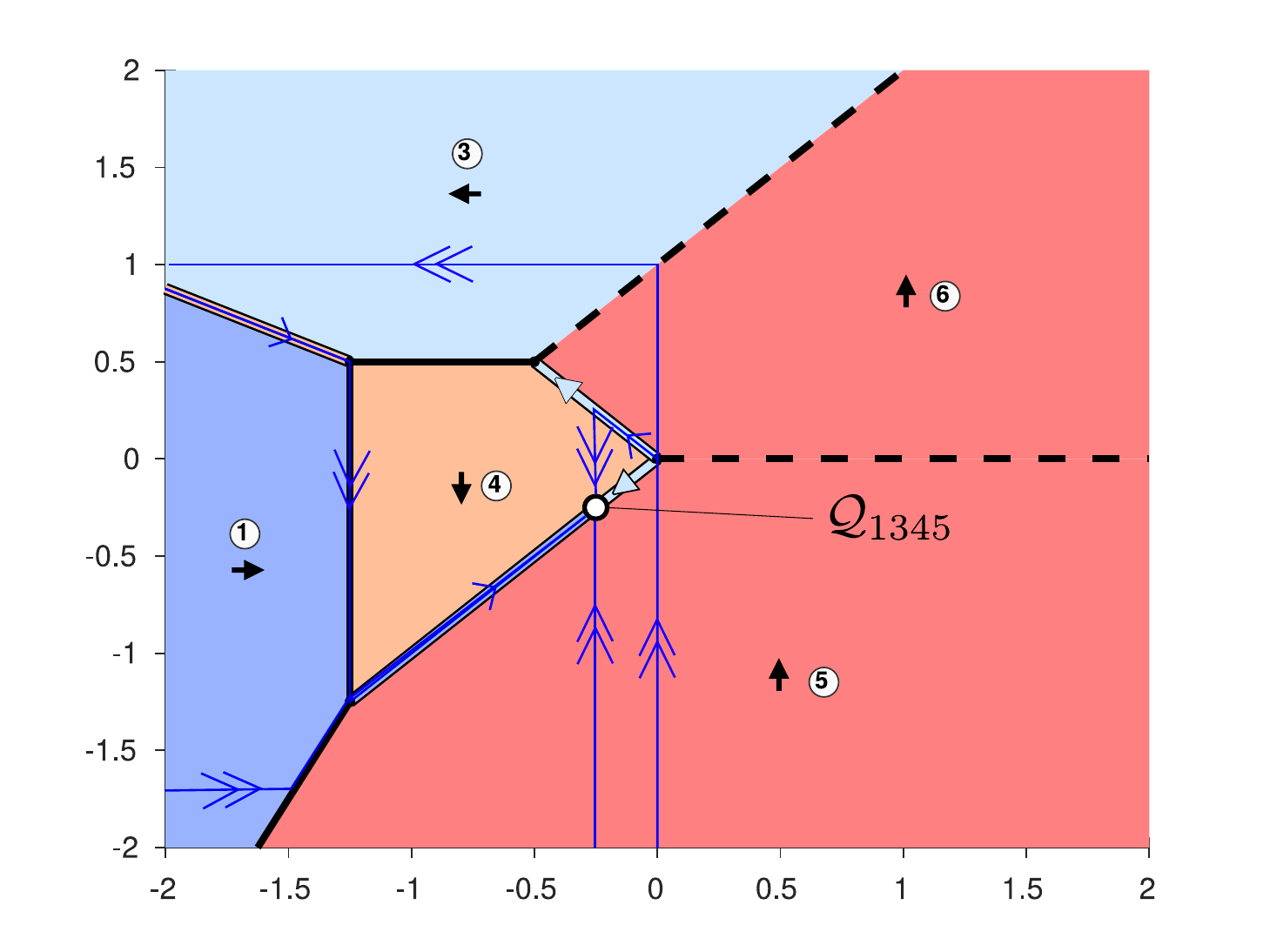}
        \caption{$\alpha=-\frac14$}
    \end{subfigure}%
    \begin{subfigure}{0.49\textwidth}
    \centering
        \includegraphics[width=0.96\linewidth]{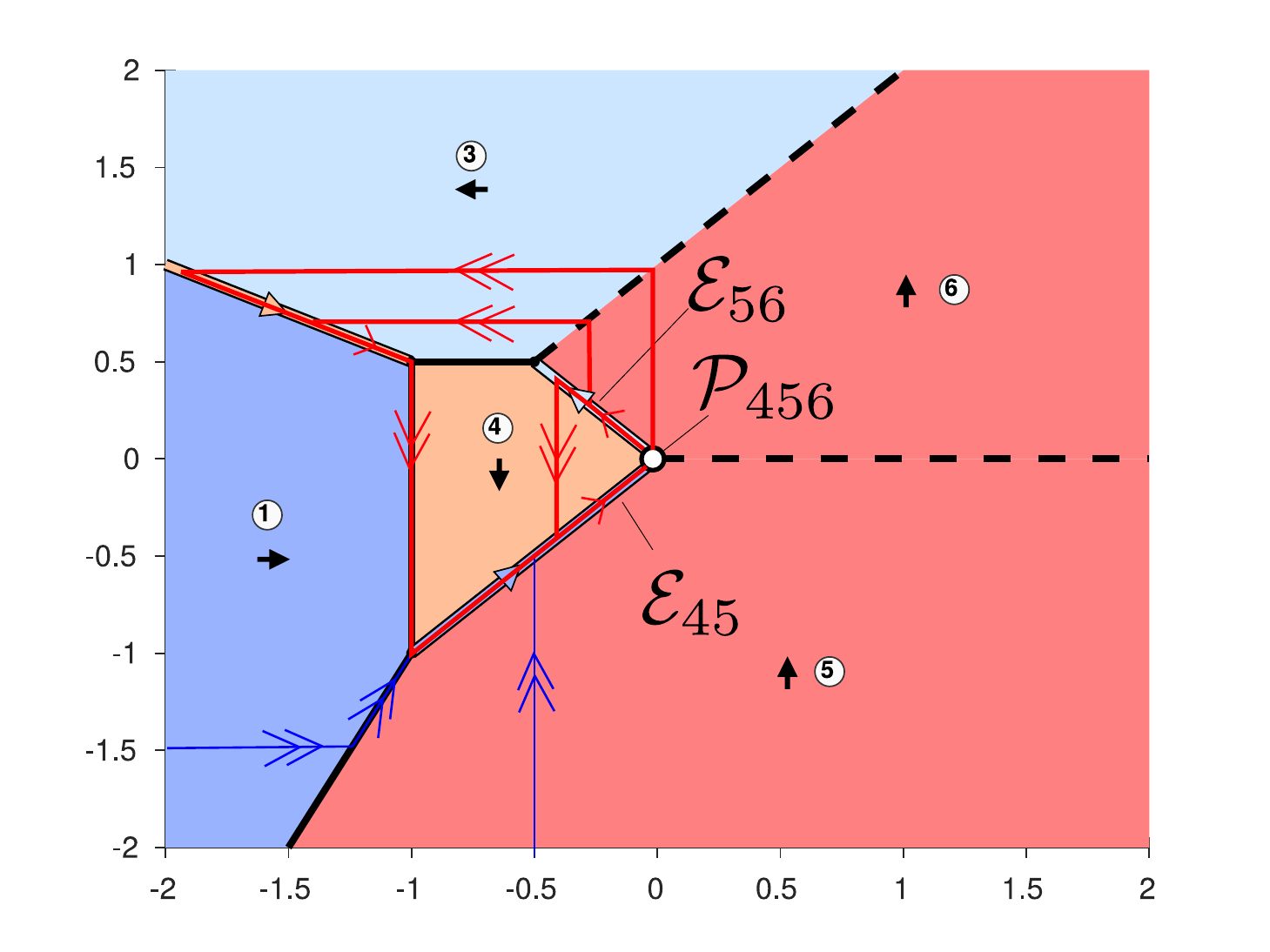}
        \caption{$\alpha=0$}
    \end{subfigure}
    \begin{subfigure}{0.49\textwidth}
    \centering
        \includegraphics[width=0.96\linewidth]{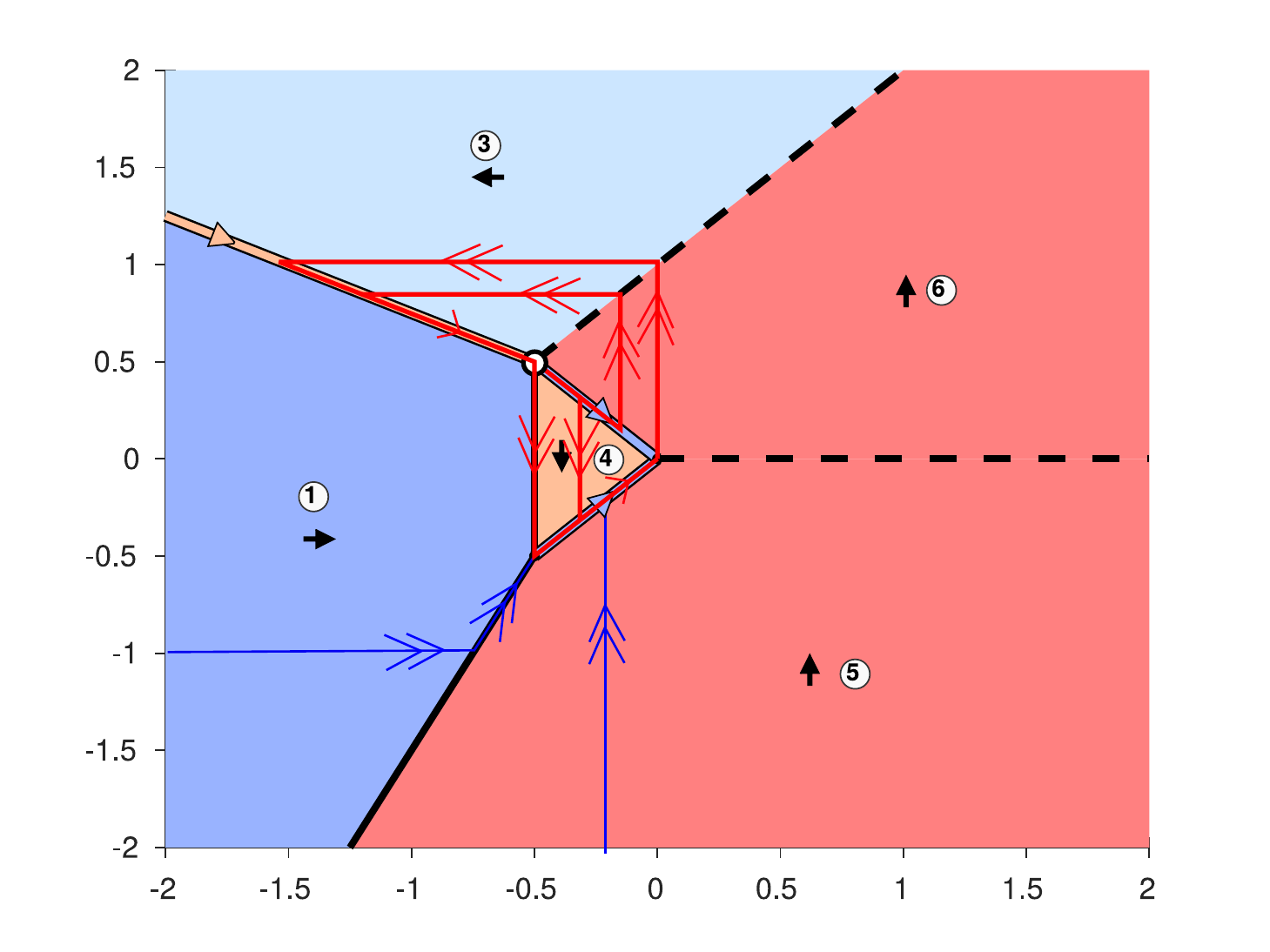}
        \caption{$\alpha=\frac12$}
    \end{subfigure}    
     \begin{subfigure}{0.49\textwidth}
    \centering
        \includegraphics[width=0.96\linewidth]{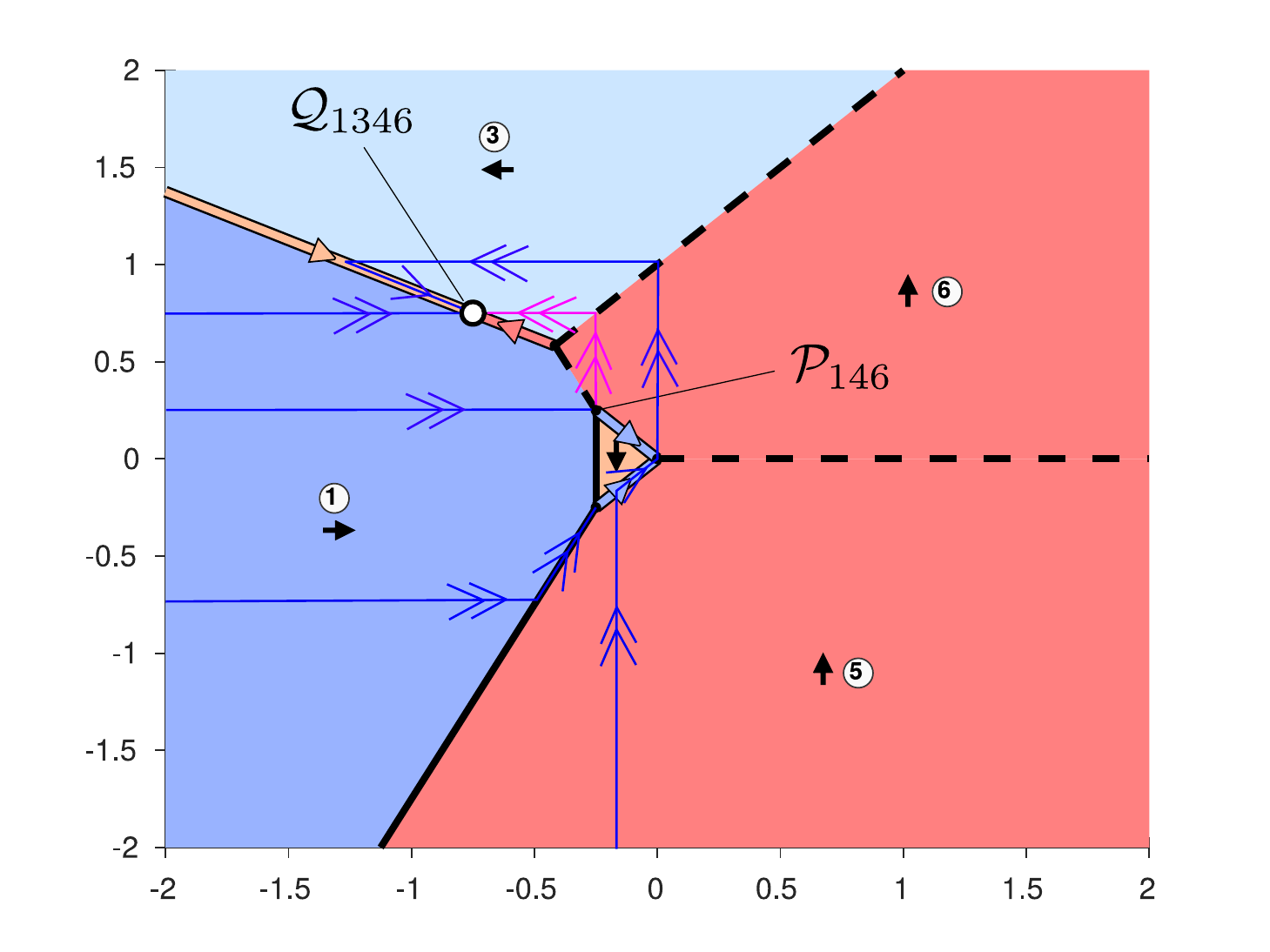}
        \caption{$\alpha=\frac34$}
    \end{subfigure}  
       \caption{The phase portraits of the tropicalized autocatalator given by the tropical pairs \eqref{tropauto} for four different values of $\alpha$ (see sub-captions). $\alpha=0$ (b) and $\alpha=\frac12$ (c) are bifurcation points, reminiscent of canard points in slow-fast systems, where the stability of the \rsp{tropical singularity}  (white circle) changes from a source to a sink (and vice versa) and where limit cycles are created/destroyed, respectively, in  a dramatic (or explosive) fashion. See also \figref{tropauto}(c).} 
    \figlab{tropauto2}
    \end{figure}
    \begin{figure}
    \centering
        \includegraphics[width=0.66\linewidth]{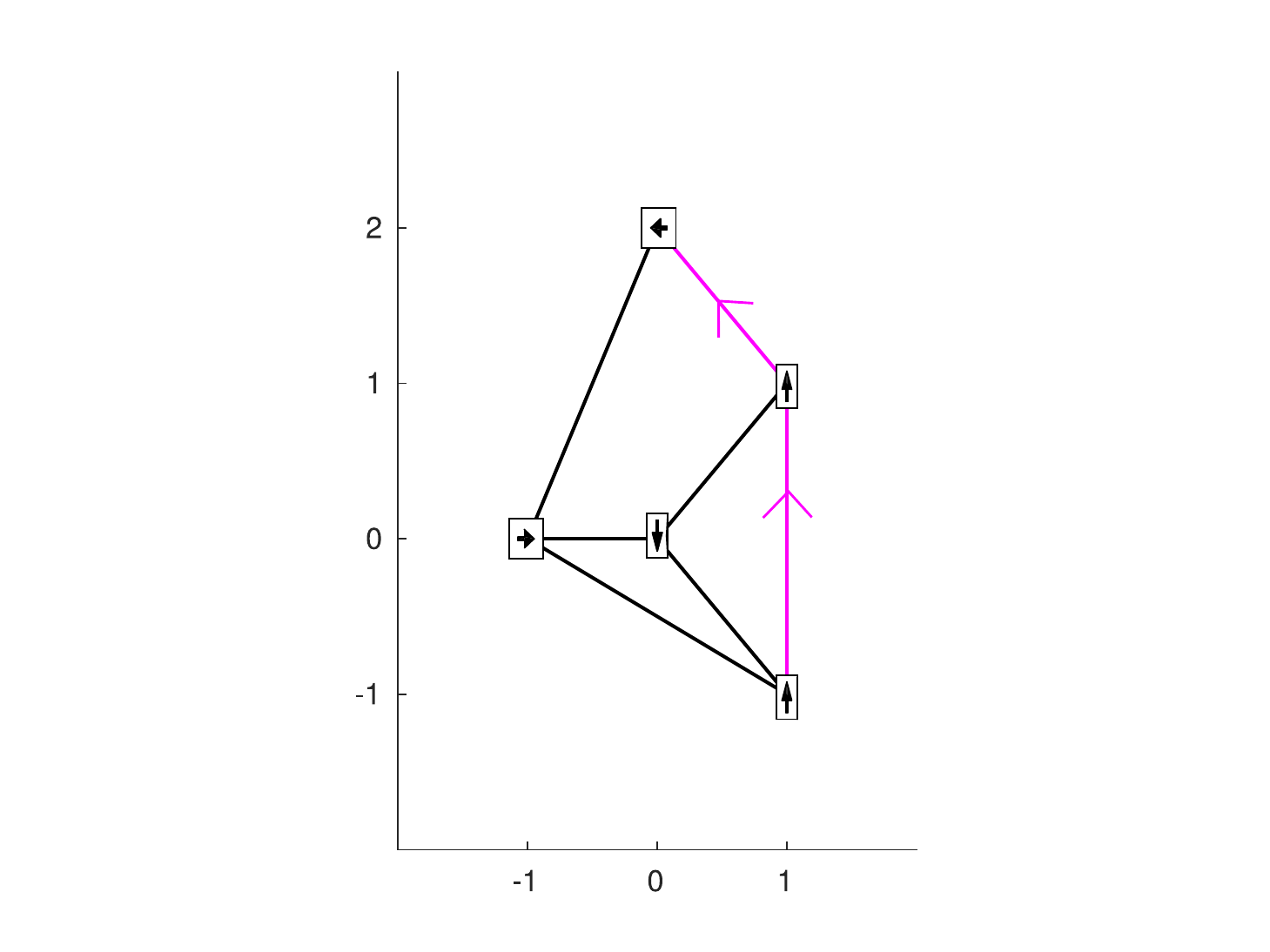}
        \caption{The (labelled) subdivision associated with the tropical monomials \eqref{tropauto} for $\alpha=\frac12$, see \figref{tropauto2}(c). In contrast to the subdivision in \figref{tropauto}(b), the subdivision is not a triangulation.}
\figlab{tropauto3}
\end{figure}
\section{Equivalence of tropical dynamical systems}\seclab{equivalence}
In smooth dynamical systems theory, we say that two planar systems are topologically equivalent if there is a(n) (orientation preserving) homeomorphism $h: \mathbb R^2\to \mathbb R^2$  that maps orbits to orbits (often on invariant compact submanifolds such as the (Poincar\'e) sphere, i.e. $h:S^2\to S^2$, leaving the equator invariant in the case of the Poincar\'e sphere) and preserve the direction of time, see \cite{perko1991a}. For equivalence in PWS systems, (a) the discontinuity set is either assumed to be fixed (and smooth) or (b) $h$ is also assumed to be a homeomorphism on the discontinuity sets, see \cite{broucke2001a,filippov1988differential}. This is called $\Sigma$-equivalence in \cite{guardia2011a}. \cite{guardia2011a} also defines topological equivalence as equivalence in the usual sense, without the requirement of neither (a) not (b). However, \cite[Proposition 2.17]{guardia2011a} and \cite[Remark 2.19]{guardia2011a} show that only diffeomorphisms preserve Filippov in general and therefore demonstrate the subtleties with these notions of equivalence.

In any case, we believe that these definitions of equivalence based on homeomorphisms are too restrictive, if not just simply too difficult to work with, for tropical dynamical systems. This belief is based on the fact that $h$ has to preserve lines, including horizontal and vertical ones; after all, every orbit of \eqref{uvnotinT} is a polygonal curve with horizontal and vertical segments in the tropical regions and (possibly) inclined segments along \rsp{tropical edges} of sliding type.  Instead, we suggest the following definition of equivalence (which is also more in tune with \defnref{homotopiccurves} and tropical geometry in general):

Consider a polygonal orbit $\gamma$ and let $E=\{l_i\}_i$ denote the edges of $\gamma$ and $V=\{q_i\}_i$ its vertices. As before, (a) $i$ runs over a finite index set, $\pm \mathbb N$ or $\mathbb Z$, (b) $l_i$ is the line segment connecting $q_i$ with $q_{i+1}$, and (c) $l_i$ and $l_{i+1}$ have different orientations (either flow or line orientations). Finally, if $\gamma$ is a point, corresponding to a \rsp{tropical singularity}, then $E$ is the empty set and $V=\{\gamma\}$, whereas if $\gamma$ is a single straight line then $E=\{\gamma\}$ and $V$ is the empty set.

\begin{definition}\defnlab{orbitshomotopy}
An oriented polygonal curve $\gamma'$, having $E'=\{l_i' \}_i$ as its edges and $V'=\{q_i'\}_i$ as its vertices, is said to be \textnormal{homotopic} to an oriented polygonal curve $\gamma$, having $E=\{l_i\}_i$ as its edges and $V=\{q_i\}_i$ as its vertices, if there are continuous functions $f_i:[0,1]\to \mathbb R^2$, such that $V=\{f_i(0)\}_i$, $V'=\{f_i(1)\}_i$ and $V(t)=\{f_i(t)\}_i$ for each $t\in [0,1]$ are vertices of an oriented polygonal curve, where only the lengths and positions of the associated edges $E(t)$ change, not their flow and line orientations, and such that $f_i(t)\ne f_j(t)$ for all $i\ne j$ and all $t\in [0,1]$. 

If $\gamma'$ is a single straight line, so that $V'=\emptyset$, then $\gamma'$ is homotopic to $\gamma$ if $\gamma$ is also a straight line having the same flow and line orientations (i.e. $\gamma$ and $\gamma'$ are translations of eachother with the same flow orientation). 
\end{definition}

Since continuous mappings can be composed, we obtain \textit{equivalence classes of orbits that are oriented polygonal curves}. Consider a tropical dynamical system $TDS$ and all of its orbits that are oriented polygonal curves. We then group homotopic (polygonal) orbits into equivalence classes and define equivalence of tropical dynamical systems as follows. 
\begin{definition}\defnlab{equivalence}
Two tropical dynamical systems $TDS$ and $TDS'$ are equivalent, if they have the same polygonal orbit equivalence classes. 
\end{definition}


We will define local equivalence analogously, saying that $TDS$ and $TDS'$ are locally equivalent if they have the same \textit{local} polygonal orbit equivalence classes on an open set $X$ of $\mathbb R^2$. (Here a local polygonal orbit is just the intersection of a polygonal orbit with an open set $X$.)
\begin{remark}\remlab{slope}
Notice that the line segments of two homotopic polygonal curves have the same slopes. Consequently, in our definition of equivalence, two equivalent tropical dynamical systems have sliding switching manifolds with the same slope (line orientation). One could argue that this leads to \textnormal{too many} structurally stable tropical dynamical systems. But we believe the definition is a natural one, since the line segments of the polygonal orbits only have rational slope. In particular, horizontal and vertical lines are intrinsic for the tropical dynamical systems. If we were to allow the slopes to vary, then we obtain lines with irrational slopes under homotopy. Finally, the definition is in line with \defnref{homotopiccurves}.
\end{remark}
\begin{definition}\defnlab{structuralstability}
 Fix $N\in \mathbb N$ and a list of flow vectors $\{\md_k\}_{k\in \mathcal I}$. A tropical dynamical system $$TDS'\in \TDS,$$ is \textnormal{structurally stable} if there is a neighborhood $\mathcal O$ of $TDS'$ in $\TDSN$ such that $TDS$ and $TDS'$ are equivalent for all $TDS\in \mathcal O$. Otherwise, $TDS'$ is said to be \textnormal{structurally unstable}.
\end{definition}
We will define local structurally stability completely analogously when it holds on an open neighborhood $X$ of $(u,v)$ in $\mathbb R^2$.

\begin{lemma}\lemmalab{triangulation}
 Suppose that a tropical dynamical system $TDS'\in \TDSN$ is structurally stable and let $\mathcal O$ be any neighborhood of $TDS'$ in $\TDSN$. Then there is a $TDS\in \mathcal O$ for which the following holds:
 \begin{enumerate}
 \item \label{gp1} The polyhedral subdivisions $\mathcal S^{\mathcal L}$ of the Newton polygons associated to the tropical polynomials $F_{\max}^{\mathcal L}$ are triangulations for $\mathcal L=\mathcal U,\mathcal V$ and $\mathcal I$.
 \item \label{gp2} $\degree F_i = \degree F_j,\,i\ne j \Longrightarrow \alpha_i\ne \alpha_j$. 
 \item \label{gp3} $\mathcal T^{\mathcal U}$ and $\mathcal T^{\mathcal V}$ intersect transversally (i.e. all intersections of $\mathcal T^{\mathcal U}$ and $\mathcal T^{\mathcal V}$  occur along their \rsp{tropical edges}).
 \end{enumerate}
\end{lemma}
\begin{proof}
 This should be clear enough, see also \cite{de2010a}.
\end{proof}
This means that for structurally stable systems, we may assume that (or more accurately, consider a representative of the structural stable system for which)  the tropical curves $\mathcal T^{\mathcal U}$, $\mathcal T^{\mathcal V}$ and $\mathcal T^{\mathcal I}$ consist of \rsp{tropical edges} and \rsp{tropical vertices} where precisely three \rsp{tropical edges} come together. Also, for structurally stable systems, we may assume (in the same sense as before) that \rsp{tropical singularities} $(u',v')$, where $\textnormal{\textbf{0}}\in \Conv(u',v')$, lie on the \rsp{tropical edges} of $\mathcal T$. This is a consequence of item \ref{gp3} of \lemmaref{triangulation}, see also \lemmaref{troppoint} below.

\begin{definition}\defnlab{generalposition}
We say that \rspp{$(\mathcal T,\mathcal T^{\mathcal U},\mathcal T^{\mathcal V})$} (of a tropical dynamical system $TDS$) is \textnormal{in general position}, if items \ref{gp1}--\ref{gp3} of \lemmaref{triangulation} hold.
\end{definition}

\begin{lemma}\lemmalab{usual} 
  Consider a tropical dynamical system $TDS'\in \TDS$ and suppose that $TDS'$ is structurally stable in the following (strong) sense of homeomorphisms (by $\Sigma$-equivalence, see e.g. \cite[definition 2.15]{guardia2011a}): \rspp{For every $\rspp{\theta}>0$}, there is a sufficiently small neighborhood $\mathcal O$ of $TDS'$ such that for any $TDS\in \mathcal O$ there is a homeomorphism $h:\mathbb R^2\to \mathbb R^2$, satisfying:
  \begin{enumerate}
  \item \label{hO} $h$ maps orbits of $TDS'$ to orbits of $TDS$, preserving the direction of time.
   \item \label{hT} $h$ maps  $\mathcal T'$ homeomorphically onto $\mathcal T$, including \rsp{tropical vertices} of $\mathcal T'$ to \rsp{tropical vertices} of $\mathcal T$,
  \item \label{hI} \rspp{$h$ is near-identity: $\vert h(u,v)-(u,v)\vert<\rspp{\theta}$ for all $(u,v)\in \mathbb R^2$.}   
  \end{enumerate}
Then $TDS'$ is also structural stable in the sense of  \defnref{structuralstability}.
 \end{lemma}
 \begin{proof}
   Let $\gamma'$ be any polygonal orbit of $TDS'$. Then  $\gamma=h(\gamma')$ is a polygonal orbit of $TDS\in \mathcal O$ by property \ref{hO}, \rspp{and as a consequence of property \ref{hT}, $h$ maps the vertices and edges of $\gamma'$ homeomorphically onto the edges and vertices of $\gamma$. Morever, by taking $\rspp{\theta}>0$ small enough in property \ref{hI}, we conclude that the corresponding edges of $\gamma'$ and $\gamma$ have to have the same slope; recall here that the slopes are rational and that there are only (with $N\in \mathbb N$ fixed) finitely many possibilities in $\mathcal T\mathcal D\mathcal S$. Hence $\gamma'$ and $\gamma$ are homotopic. In turn, we conclude that $TDS'$ is  structural stable in the sense of  \defnref{structuralstability}. }
 \end{proof}

\rspp{We leave the details of going the other way in \lemmaref{usual} to future work}. In the following, we study local phenomena (\rsp{tropical vertices} and \rsp{tropical singularities}) and characterize the structurally stable situations (locally). For these local phenomena, we will construct the local equivalence in terms of a translation, mapping orbits to orbits locally.

\begin{definition}
We will say that a transformation $T_\alpha:(u,v)\mapsto (u,v)+b( \alpha)$, with $\alpha\mapsto b(\alpha)\in \mathbb R^2$ affine, is \textnormal{an affine translation with respect to $\alpha$}.
\end{definition}

\section{Tropical \rsp{vertices}}\seclab{troppoint}
We define \rsp{tropical vertices} as points  $\mathcal P:(u,v)$ where $\#\argmax_{k\in \mathcal I}F_k(u,v)\ge 3$. 

\begin{lemma}\lemmalab{troppoint}
Consider a tropical dynamical system $TDS'$ with \rspp{$(\mathcal T',\mathcal T'^{\mathcal U},\mathcal T'^{\mathcal V})$} in general position and suppose that $\mathcal P':(u',v')$ is a \rsp{tropical vertex}. Then
\begin{enumerate}
 \item \label{item1} $\argmax_{k\in \mathcal I}F_k(u',v')=\{i,j,l\}$ for some distinct $i,j,l\in \mathcal I$.
 \item \label{item2} $\mathcal P'$ is isolated, i.e. there is a neighborhood $X$ of $(u',v')$ where $\mathcal P'$ is the only \rsp{tropical vertex}.
 \item \label{item3} The degrees $\degree F_k$, $k=i,j,l$ are distinct.
 \item \label{item4} $\textnormal{\textbf{0}}\notin \Conv(u',v')$. 
\end{enumerate}
\end{lemma}
\begin{proof}
  The statements \ref{item1}--\ref{item3} directly follow from the definition of $\rspp{(\mathcal T',\mathcal T'^{\mathcal U},\mathcal T'^{\mathcal V})}$ being in general position, see \defnref{generalposition}. Now regarding item \ref{item4}, we use item \ref{item1} and the definition of $\Conv$ to conclude that $\textnormal{\textbf{0}}\in \Conv(u',v')\Rightarrow \mathcal P'=\mathcal P^{'\mathcal L}$ for either $\mathcal L=\mathcal U$ or $\mathcal L=\mathcal V$. But then $\mathcal P'$ is the intersection point of $\mathcal T^{'\mathcal U}$ and $\mathcal T^{'\mathcal V}$, and this contradicts $\rspp{(\mathcal T',\mathcal T'^{\mathcal U},\mathcal T'^{\mathcal V})}$ being in general position, see item \ref{gp3} of \defnref{generalposition}. 
\end{proof}

\begin{lemma}\lemmalab{troppoint1}
Suppose that the assumptions of \lemmaref{troppoint} hold and that $$TDS'\in \TDS.$$ Then there is a neighborhood $\mathcal O$ of $TDS'$ and a neighborhood $X$ of $(u',v')$, such that there is a unique \rsp{tropical vertex}  $\mathcal P(\alpha)\in X$ of $TDS\in \mathcal O$, having coordinates $(u(\alpha),v(\alpha))$ with $(u(\alpha'),v(\alpha'))=(u',v')$, of the following form:
\begin{align}\eqlab{walpha}
 w(\alpha) = w'+b_w\cdot (\alpha_i-\alpha_i',\alpha_j-\alpha_j',\alpha_l-\alpha'_l),\quad w=u,v,
\end{align}
with $b_w:=(b_{w,i},b_{w,j},b_{w,l})\ne  (0,0,0)$, $b_{w,i}+b_{w,j}+b_{w,l}=0$ and $\argmax_{k\in \mathcal I}F_k(u',v')=\{i,j,l\}$.
\end{lemma}
\begin{proof}
Given item \ref{item1} of \lemmaref{troppoint}, we obtain two equations for the coordinates of $(u,v)$ (see 
 e.g. \eqref{Eijeqn}):
 \begin{align*}
  \begin{pmatrix}
   \degree F_i - \degree F_j  \\
   \degree F_i - \degree F_l
  \end{pmatrix}\begin{pmatrix}
  u\\
  v
\end{pmatrix} +\begin{pmatrix}
\alpha_i-\alpha_j\\
\alpha_i-\alpha_l\end{pmatrix}=0.
 \end{align*}
 But then by item \ref{item2} of \lemmaref{troppoint}  and the fact that $\mathcal P'$ is isolated, we conclude that the coefficient matrix is regular. The fact, that the sum of the components of $b_w$ is zero, is the consequence of the invariance of the tropical system to translations of the tropical coefficients, see \eqref{translation}.
\end{proof}

The \rsp{tropical edges} emanating (locally) from $\mathcal P(\alpha)$ are clearly affine translations of those of $\mathcal P'$. In particular, the slopes of the \rsp{tropical edges} do not depend on $\alpha$ (but only on the degrees $\degree F_k$, $k=i,j,l$.)


\begin{lemma}\lemmalab{troppoint2}
 Consider a tropical dynamical system $$TDS'\in \TDS,$$ with \rspp{$(\mathcal T',\mathcal T'^{\mathcal U},\mathcal T'^{\mathcal V})$ }in general position and suppose that $\mathcal P'$ with coordinates $(u',v')$ is a \rsp{tropical vertex}  of $\mathcal T'$. Then the $TDS'$ is locally structurally stable in a neighborhood $X$ of $(u',v')$.
\end{lemma}
\begin{proof}
 We first use \lemmaref{troppoint1} and define the affine translation $T_\alpha$ with respect to $\alpha$ by 
 \begin{align*}
  T_\alpha(u,v)=(u,v)+(u(\alpha)-u',v(\alpha)-v'),
 \end{align*}
 with $w(\alpha)$, $w=u,v$, given by \eqref{walpha}. Clearly, $T_\alpha$ maps $\mathcal T'\cap X$ to $\mathcal T\cap X$ and since $\textnormal{\textbf{0}}\notin \Conv(u',v')$ it follows that $\Conv(\mathcal P')=\Conv(\mathcal P(\alpha))$ and that $T_\alpha$ locally map orbits of $TDS'$ to orbits of $TDS\in \mathcal O$, with $\mathcal O$ being a neighborhood of $TDS'$. Consequently, $TDS'$ is locally structurally stable.
\end{proof}
\begin{definition}\defnlab{pointgen}
We say that a \rsp{tropical vertex}  $\mathcal P':(u',v')$ of a tropical dynamical system $TDS'$ is \textnormal{in general position} if items \ref{item1}--\ref{item4} of \lemmaref{troppoint} all hold true.
\end{definition}


For a given degree $N$, there are finitely many equivalence classes of locally structurally stable \rsp{tropical vertices} $\mathcal P$.  We will not worry about the exact number, but these different systems are given by different slopes (line orientations) of the \rsp{tropical edges} meeting at $\mathcal P$ and the different relevant flow vectors $\md_k$, $k=i,j,l$. This follows from \lemmaref{troppoint} and \lemmaref{troppoint2}.

In \figref{tropicalpoints}, we sketch different local phase portraits in a table. 
The rows of \figref{tropicalpoints} correspond to different conditions on how orbits approach the \rsp{tropical vertices}, whereas the columns correspond to different conditions on how orbits leave. The examples are not unique, the list is not exhaustive (in fact, we have left some out on purpose) and some cases appear more than once (e.g. upon applying certain symmetries). Moreover, each example do come in different equivalence classes since different slopes of the \rsp{tropical edges} lead to distinct phase portraits, according to our definition \defnref{equivalence}, see also \remref{slope}. The main point of \figref{tropicalpoints} is to demonstrate the variety of possibilities. 
Examples for each of the cases can easily be constructed using the subdivision of the Newton polygon. 
Empty entries reflect cases that are (probably!) not possible. We \rspp{leave a more} thorough classification of the \rsp{tropical vertices} \rspp{to} future work. 

\begin{figure}
        \includegraphics[width=1.0\linewidth]{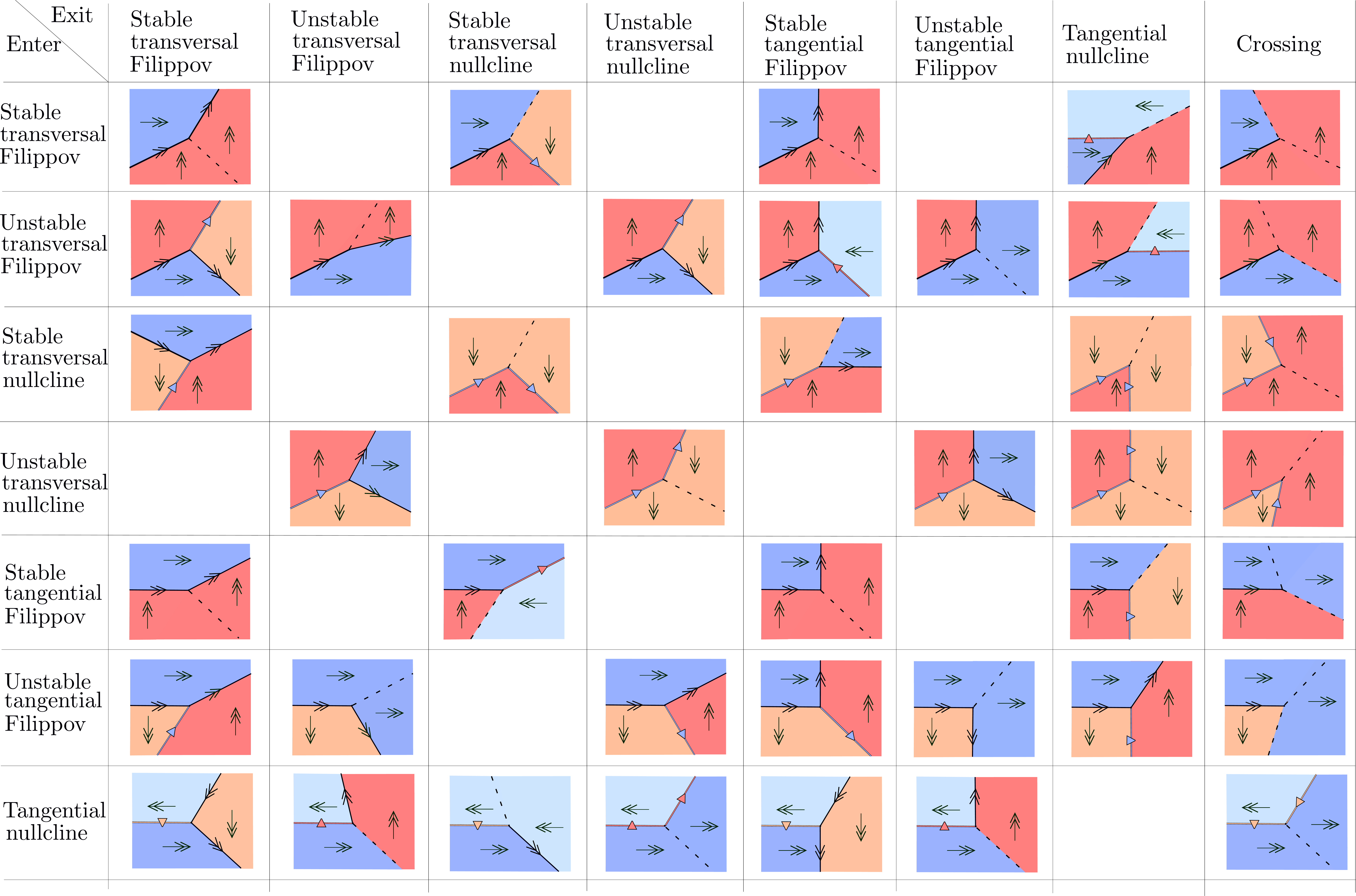}
        \caption{A table with different local dynamics near a \rsp{tropical vertex}  $\mathcal P$ with $\rspp{(\mathcal T,\mathcal T^{\mathcal U},\mathcal T^{\mathcal V})}$ in general position (e.g. $\textnormal{\textbf{0}}\notin \Conv(\mathcal P)$). The table is organized around different types of entrance (the rows) and different types of exits (the columns). Empty entries reflect cases that are (probably!) not possible. The list is not exhaustive and the cases shown are examples; they are not unique. In fact, according to our notion of topological equivalence, different slopes lead to different equivalence classes, see \remref{slope}. }
\figlab{tropicalpoints}
\end{figure}

\section{Tropical \rsp{singularities}}\seclab{tropeq}
A point $(u,v)$ is a \rsp{tropical singularity}  (or just \rsp{singularity} for short) of a tropical dynamical system $TDS$ if $\textnormal{\textbf{0}}\in \Conv(u,v)$. 
We now classify the structurally stable ones, and therefore suppose that $\rspp{(\mathcal T,\mathcal T^{\mathcal U},\mathcal T^{\mathcal V})}$ is in general position, recall \defnref{generalposition}.  By \lemmaref{troppoint}, $(u,v)$ is not a \rsp{tropical vertex}. Consequently, by the definition of $\Conv$, see \defnref{convtrop}, a \rsp{tropical singularity}  $(u,v)$ is the intersection of $\mathcal E_{i,j}^{\mathcal U}\subset \mathcal T^{\mathcal U}$ and $\mathcal E_{l,p}^{\mathcal V}\subset \mathcal T^{\mathcal V}$ with both sets \rspp{being} switching manifolds of nullcline sliding-type (either transversal or tangential). With $\rspp{(\mathcal T,\mathcal T^{\mathcal U},\mathcal T^{\mathcal V})}$ being in general position, the intersection is transverse at $(u,v)$, see item \ref{gp3} of \defnref{generalposition}, such that
\begin{align}\eqlab{detcond}
 \operatorname{det} \begin{pmatrix}
   \degree F_i - \degree F_j  \\
   \degree F_l - \degree F_p
  \end{pmatrix}\ne 0.
\end{align}
Moreover, $\mathcal E_{i,j}^{\mathcal U}$ or $\mathcal E_{l,p}^{\mathcal V}$ is a \rsp{tropical edge} of $\mathcal T$ but not both (since then $(u,v)$ would be a \rsp{tropical vertex}  of $\mathcal T^{\mathcal I}$ with $\#\argmax_{k\in \mathcal I} F_k(u,v)=4$, in contradiction with $\rspp{(\mathcal T,\mathcal T^{\mathcal U},\mathcal T^{\mathcal V})}$ being in general position). Suppose the former. (The latter is identical). 
Then in the case of transversal nullcline sliding, \rspp{the nullcline sliding vector} $\md_{\trop}$, recall \lemmaref{ConvNullcline}, is well-defined on either side of $(u,v)$ on $\mathcal E_{i,j}^{\mathcal U}$ and it is discontinuous at $(u,v)$. 

\begin{definition}\defnlab{sinks}
Under the assumptions stated above, we say that $(u,v)$ is a 
 \begin{enumerate}
  \item \textnormal{A sink} if $\mathcal E_{i,j}^{\mathcal U}$ is a stable transversal nullcline sliding switching manifold, and $\md_{\trop}$ on either side of $(u,v)$ points towards $(u,v)$. 
  \item \textnormal{A source} if $\mathcal E_{i,j}^{\mathcal U}$ is an unstable transversal nullcline sliding switching manifold, and $\md_{\trop}$ on either side of $(u,v)$ points away from $(u,v)$. 
  \item \textnormal{A strong-stable saddle} if $\mathcal E_{i,j}^{\mathcal U}$ is a stable transversal nullcline sliding switching manifold, and $\md_{\trop}$ on either side of $(u,v)$ points away from $(u,v)$. 
  \item \textnormal{A strong-unstable saddle} if $\mathcal E_{i,j}^{\mathcal U}$ is an unstable transversal nullcline sliding switching manifold,  and $\md_{\trop}$ on either side of $(u,v)$ points towards $(u,v)$.
  \item \textnormal{A hybrid point} if $\mathcal E_{i,j}^{\mathcal U}$ is a tangential nullcline sliding switching manifold.
 \end{enumerate}
In the situations where reference to whether the dominating directions are stable or unstable is not important, we just write either a strong-stable saddle or a strong-unstable saddle as a saddle. 
\end{definition}

%

%

We illustrate the different \rsp{tropical singularities} in \figref{tropeq}. In particular, \figref{tropeq}(a) is a sink whereas \figref{tropeq}(b) is a strong-stable saddle. The source and the strong-unstable saddle are identical upon time reversal. Finally, \figref{tropeq}(c)--(d) are hybrid points;  notice that the hybrid points come in two types: One where the dominant and sub-dominate vectors have a well-defined direction of rotation, see (c), (like a center) and one where there is no direction of rotation (like a saddle), see (d). 
Due to our definition of equivalence, see \remref{slope}, for two \rsp{tropical singularities} of sink, source, or saddle-type to be equivalent, the transversal sliding switching manifolds have to have the same slopes. However, for a fixed degree $N$, there are finitely many possible slopes (the actual number is not important for us, but it can be expressed in terms of Euler's totient number) and therefore finitely many equivalence classes of locally structurally stable \rsp{tropical singularities}.

\begin{figure}[H]
    \centering
    \begin{subfigure}{0.49\textwidth}
    \centering
        \includegraphics[width=0.96\linewidth]{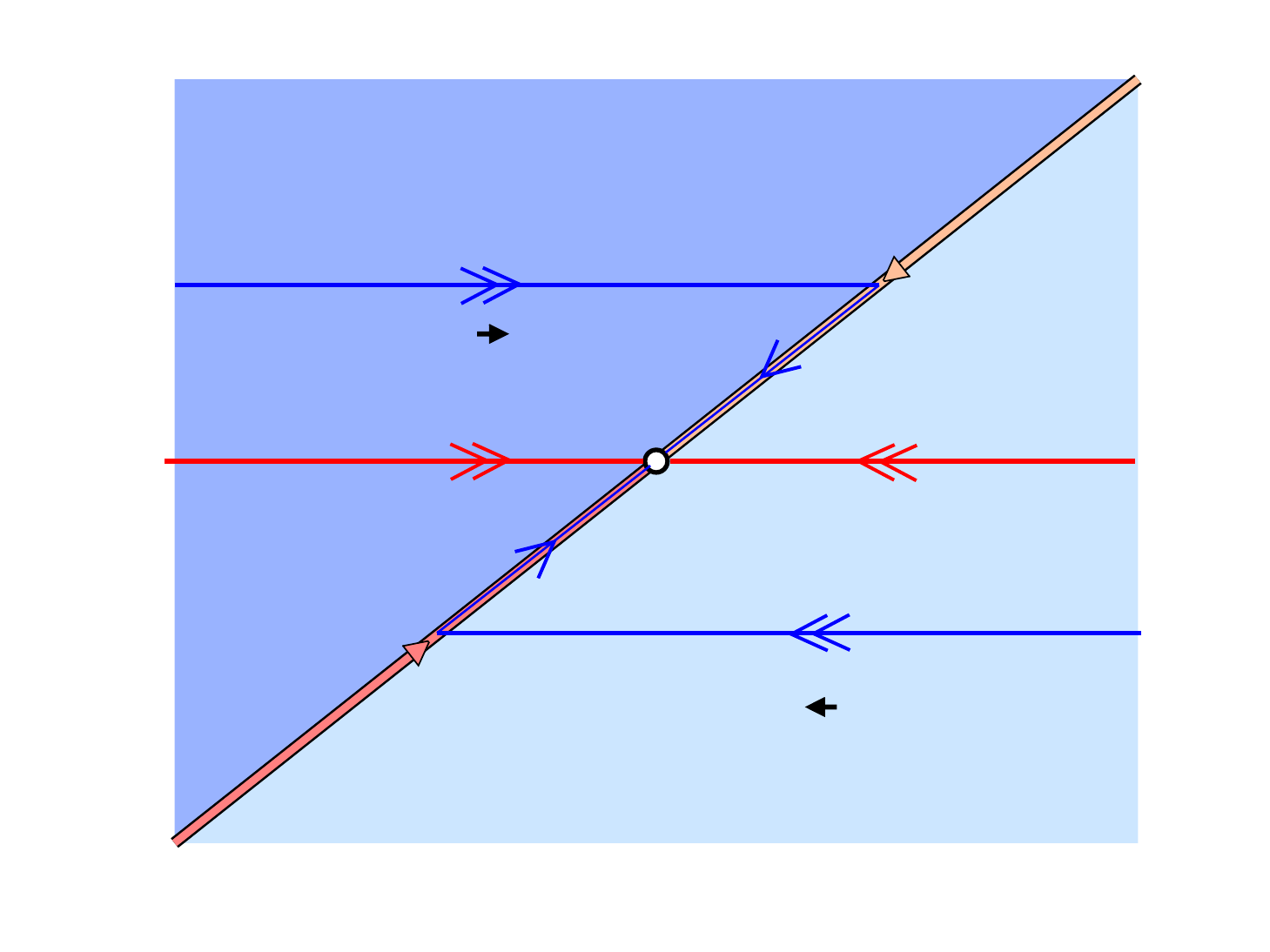}
        \caption{Sink}
    \end{subfigure}
    \begin{subfigure}{0.49\textwidth}
    \centering
        \includegraphics[width=0.96\linewidth]{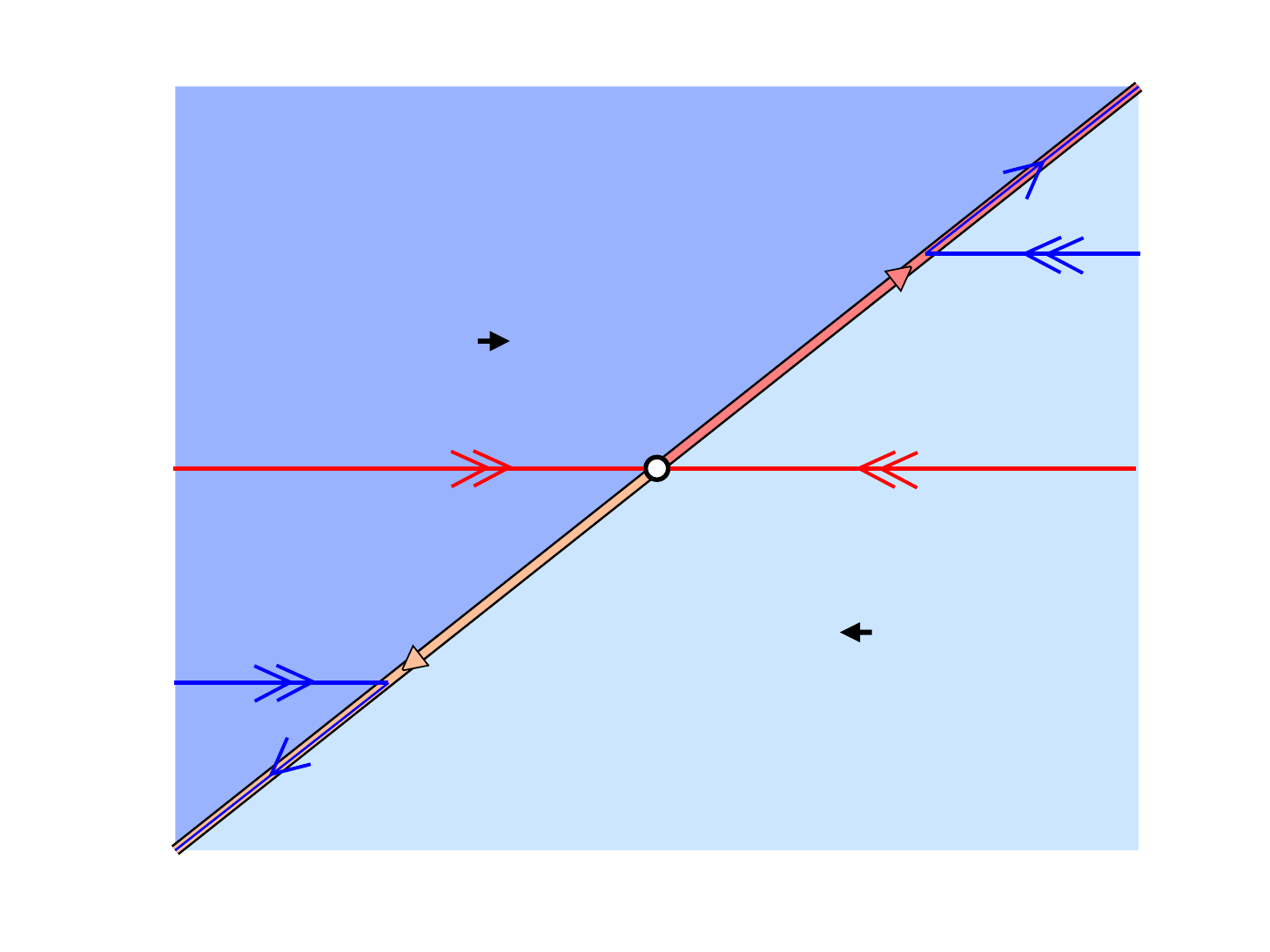}
        \caption{Saddle}
    \end{subfigure}
     \begin{subfigure}{0.49\textwidth}
    \centering
        \includegraphics[width=0.96\linewidth]{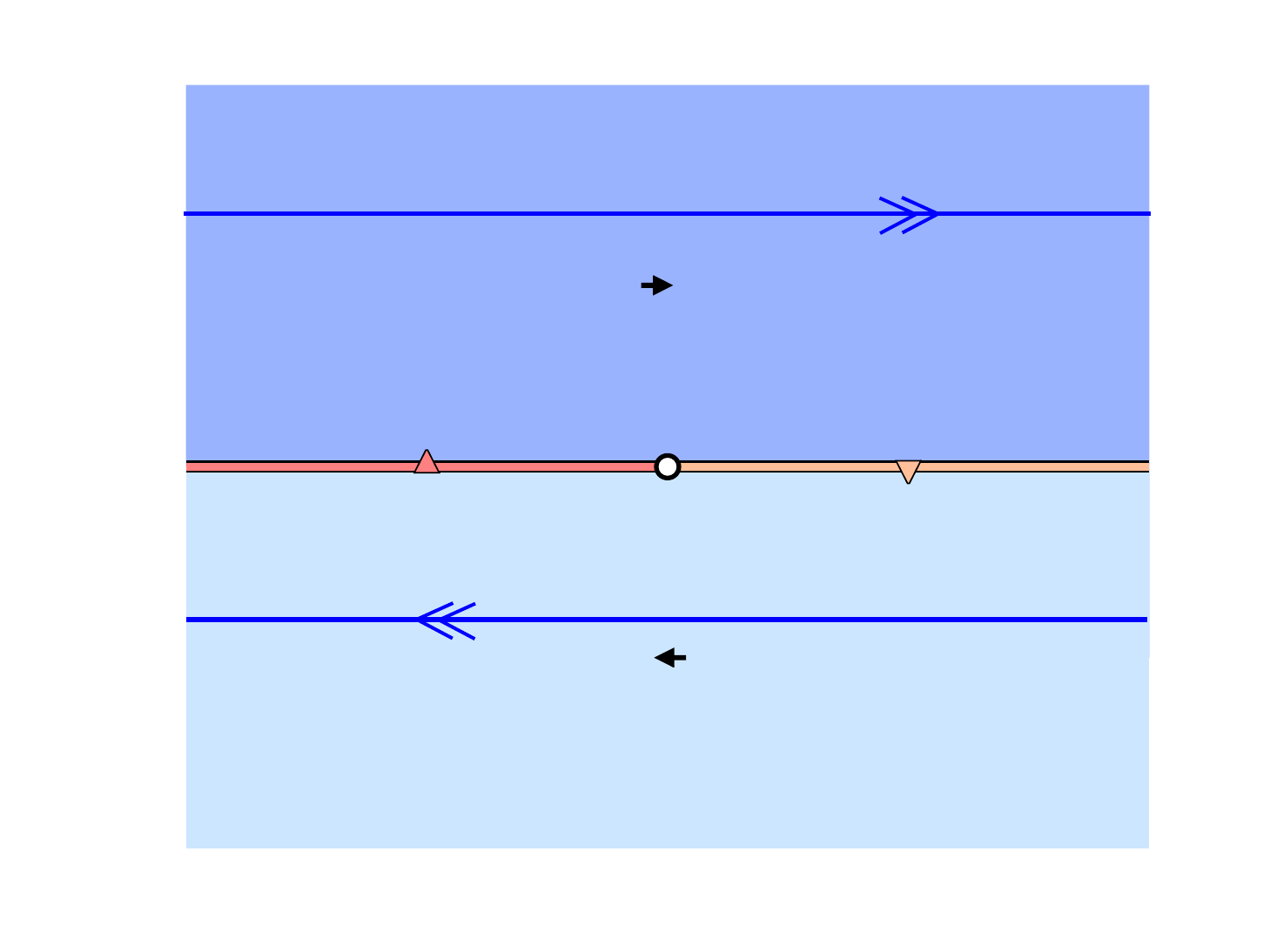}
        \caption{Hybrid point (center-type)}
    \end{subfigure}
    \begin{subfigure}{0.49\textwidth}
    \centering
        \includegraphics[width=0.96\linewidth]{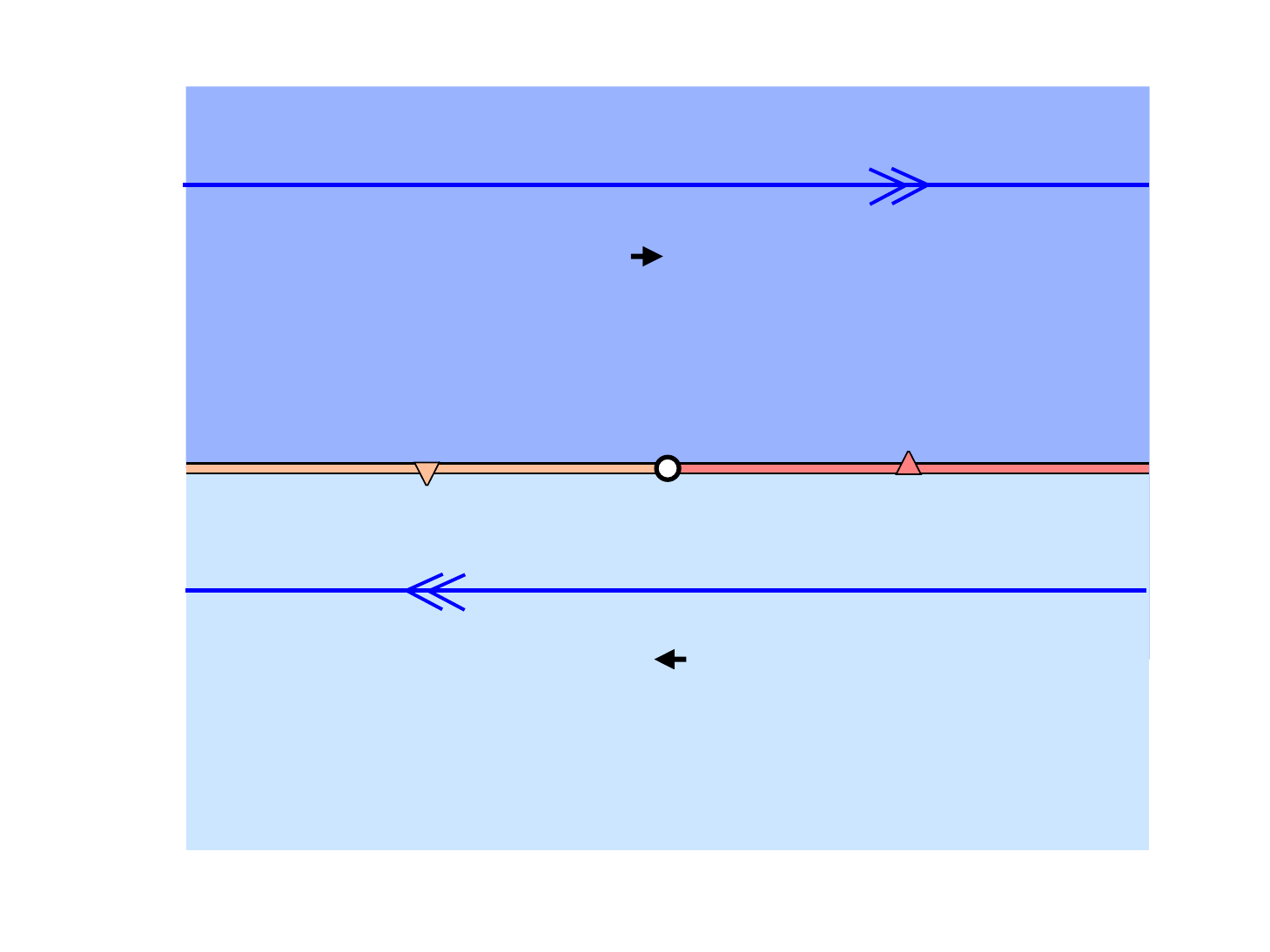}
        \caption{Hybrid point (saddle-type)}
    \end{subfigure}
    \caption{Illustration of the different \rsp{tropical singularities}, see sub-caption. In the case of the sink and saddle we indicate the arrival separatrices in red, see \secref{separatrixconnection}. The triangles along the sliding manifolds indicate the direction of $\md_{\trop}$ in (a) and (b) and the direction of the flow vector of the subdominant direction ($v$ in this case) in (c) and (d).}
    \figlab{tropeq}
        \end{figure}

\begin{lemma}\lemmalab{tropeq}
Consider a tropical dynamical system $$TDS'\in \TDS,$$ with \rspp{$(\mathcal T',\mathcal T'^{\mathcal U},\mathcal T'^{\mathcal V})$} in general position and suppose that $\mathcal P'$ is a \rsp{tropical singularity}. Then $\mathcal P'$ is a sink, source, saddle or a hybrid point and it is locally structurally stable on a neighborhood $X$ of $\mathcal P'$.

Moreover,  for any $TDS\in \mathcal O$, with $\mathcal O$ a sufficiently small neighborhood of $TDS'$, there is a unique \rsp{tropical singularity}  $\mathcal P(\alpha)$, of the same type (sink/source/saddle/hybrid point) as $\mathcal P'$, and whose coordinates $(u(\alpha),v(\alpha))$ are affine functions of $\alpha\in O$ (and invariant with respect to \eqref{translation}).
\end{lemma}
\begin{proof}
We first obtain $\mathcal P(\alpha)$ with the coordinates $(u(\alpha),v(\alpha))$ by proceeding as in \lemmaref{troppoint1}. Then we define an affine translation $T_\alpha$ with respect to $\alpha\in O$ by
 \begin{align*}
  T_\alpha(u,v)=(u,v)+(u(\alpha)-u',v(\alpha)-v').
 \end{align*}
Clearly, $T_\alpha$ sends $\mathcal T'$ to $\mathcal T(\alpha)$ and locally orbits of $TDS'$ to orbits of $TDS\in \mathcal O$, since $\Conv(\mathcal P')=\Conv(\mathcal P(\alpha))$ for all $\alpha\in O$. 
\end{proof}

\begin{definition}\defnlab{eqgen}
 We say that a \rsp{tropical singularity}  is \textnormal{in general position} if it is a sink, source, saddle or a hybrid point.
\end{definition}
\begin{remark}\remlab{hybrid}
 Notice that a sink, a source and saddle each have nine local orbits up to equivalence (one point, two separatrices, two along the switching manifolds and then four with fast and slow segments; only five (including the \rsp{singularity} itself) are shown in \figref{tropeq}(a)  and (b) in the cases of a sink and saddle). Notice also that solutions are nonunique at a \rsp{tropical singularity}.  
 
 On the other hand, with our definition of an orbit, a hybrid point has infinitely many polygonal orbits, as we can slide back and forth any finite number of times before either leaving the neighborhood or ending at the \rsp{singularity}. In the present manuscript, it will not be important, but in future work one might consider restricting the set of orbits further and rule out pathological examples.   
\end{remark}

\section{Separatrices and separatrix connections}\seclab{separatrixconnection}
In \secref{troppoint} and \secref{tropeq}, we focused on local phenomena. We now turn our attention to global ones. For this purpose, \rsp{tropical vertices} and \rsp{tropical singularities} play an important role. Indeed if there is an orbit connecting two such points by crossing only, then this situation \textit{may be} structurally unstable, see \rsp{the example in} \figref{separatrixconnection} and the figure caption for further details.

\begin{figure}[H]
    \centering
    \begin{subfigure}{0.5\textwidth}
    \centering
        \includegraphics[width=0.96\linewidth]{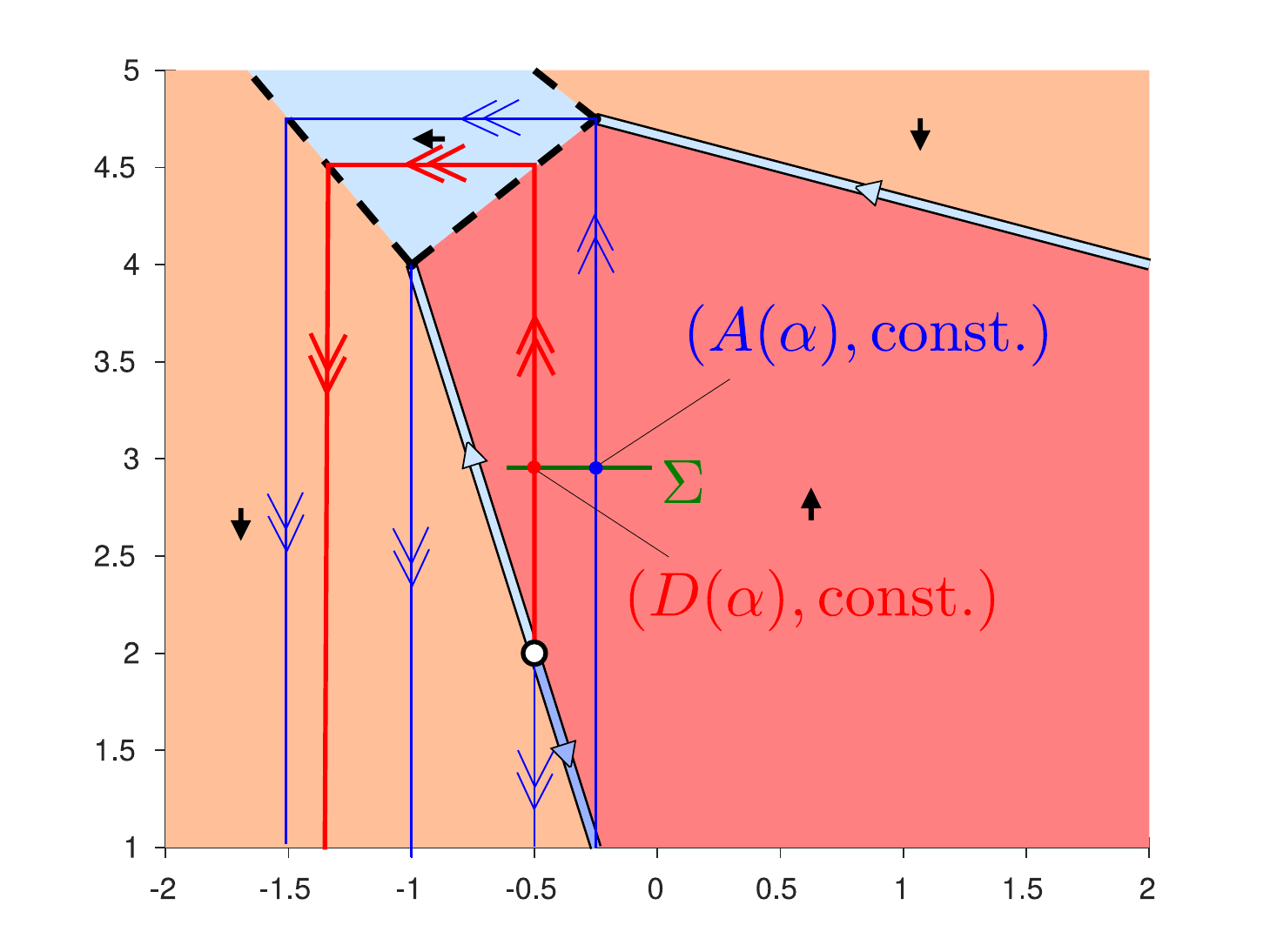}
        \caption{$\alpha=-14$}
    \end{subfigure}%
    \begin{subfigure}{0.5\textwidth}
    \centering
        \includegraphics[width=0.96\linewidth]{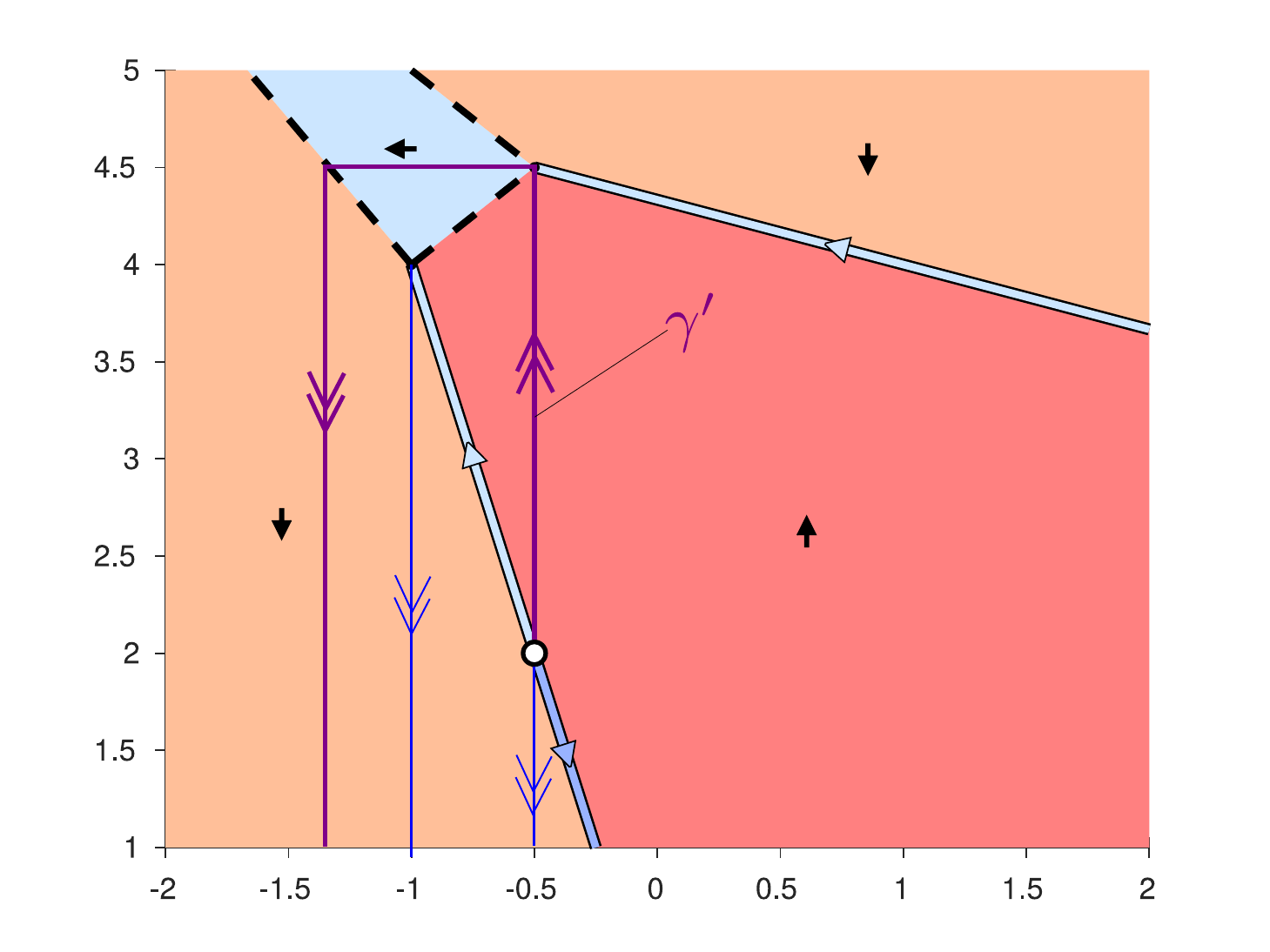}
        \caption{$\alpha=-13$}
    \end{subfigure}
    \begin{subfigure}{0.5\textwidth}
    \centering
        \includegraphics[width=0.96\linewidth]{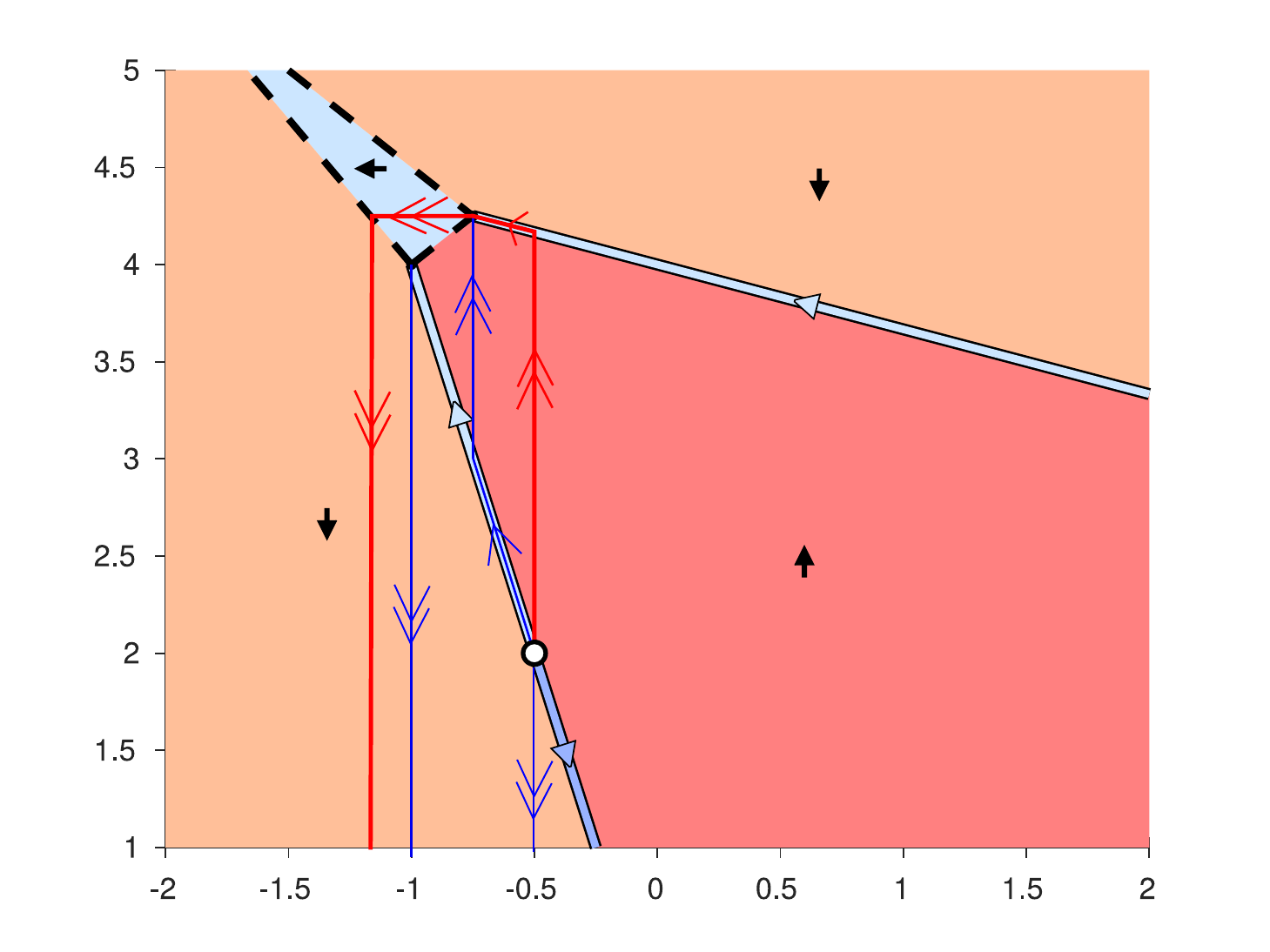}
        \caption{$\alpha=-12$}
    \end{subfigure}
    \caption{\rsp{Example of an} unfolding of a separatrix connection. The system is given by \eqref{crossing1} below. For $\alpha=-13$, see (b), there is a separatrix connection (in purple) between the tropical source at $(-0.5,2)$ and the \rsp{tropical vertex}  at $(-0.5,4.5)$. This gives rise to two non-equivalent phase portraits upon perturbation, see  (a) and (c) for $\alpha=-14$ and $\alpha=-12$, respectively. In particular, the red orbits in (a) and (c) are not homotopic to any orbits in (c) and (a), respectively, cf. \defnref{orbitshomotopy}.}
    \figlab{separatrixconnection}
    \end{figure}

    We now introduce the notion of arrival and departure separatrices (\cite{broucke2001a} uses similar concepts).
\begin{definition}
 Consider a point $\mathcal P$, which is either a \rsp{tropical vertex}  or a \rsp{tropical singularity}  of a tropical dynamical system. Then \textnormal{a departure (an arrival) separatrix} is a forward (backward, respectively) orbit of the crossing flow with initial condition at $\mathcal P$.
\end{definition} 

A departure separatrix of $\mathcal P$ either corresponds to a global forward solution $(u,v):[0,\infty)\to \mathbb R^2$, $(u,v)(0)=\mathcal P$, of the crossing flow (being either unbounded as $t\to \infty$ or having a crossing limit cycle (see \secref{crossing}) as the $\omega$-limit set) or it intersects a sliding switching manifold $\mathcal E\subset \mathcal T$ (transversal nullcline type) or a \rsp{tropical vertex}  $\mathcal Q$ in a (smallest) finite time $T>0$: $(u,v)(T)\in \mathcal E$, $(u,v)(T)=\mathcal Q$, respectively. The same holds for an arrival separatrix in backward time. 

Sinks and strong-stable saddles have two arrival separatrices, see \figref{tropeq}(a) and (b) (red curves), whereas sources and strong-unstable saddles each have two departure separatrices.

Notice that a hybrid point does \textit{not} have a separatrix (since this would necessarily be along the \rsp{tropical edge} and it is therefore not given by crossing, see \figref{tropeq}(c) and (d)). \rspp{This is our definition and it is motivated} by the fact that connections along sliding segments are not (in themselves) associated with bifurcations. 

By \lemmaref{troppoint} and \lemmaref{tropeq}, we \rspp{directly obtain the} following. 
\begin{lemma}\lemmalab{this}
Consider a tropical dynamical system $$TDS'\in \TDS,$$ and a point $\mathcal P':(u',v')$, which is either (a) a \rsp{tropical vertex}  in general position or (b) a tropical sink, source or saddle, recall \defnref{pointgen} and \defnref{sinks}, respectively. (We will refer to this as the type of $\mathcal P'$.) Then there is a neighborhood $X$ of $\mathcal P'$ such that the separatrices of the point $\mathcal P(\alpha)$ of $TDS\in \mathcal O$, which is of the same type as $\mathcal P'$, are affine translations with respect to $\alpha\in O$ in  $X$.
\end{lemma}

\begin{definition}
 We say that an orbit $\gamma$ of the crossing flow is \textnormal{a separatrix connection} if it is both a departure separatrix and an arrival separatrix of points $\mathcal P$ and $\mathcal Q$, respectively. In the affirmative case, we will say that $\gamma$ \textnormal{departs from $\mathcal P$} and \textnormal{arrives at $\mathcal Q$}.  
 If $\mathcal P=\mathcal Q$ then it is a \textnormal{homoclinic separatrix connection}; otherwise, it is said to be a \textnormal{heteroclinic separatrix connection}. 
 \end{definition}
 
\begin{lemma}\lemmalab{Deltasep}
 Consider a tropical dynamical system   $$TDS'\in \TDS,$$ with \rspp{$(\mathcal T',\mathcal T'^{\mathcal U},\mathcal T'^{\mathcal V})$} in general position and tropical coefficients $\alpha'$. Suppose that $\gamma'$ is a separatrix connection, departing from $\mathcal P'$ and arriving at $\mathcal Q'$. Here each $\mathcal P'$ and $\mathcal Q'$ is either a \rsp{tropical vertex}  in general position or a tropical sink, source or saddle.

  Then there is a neighborhood $\mathcal O$ of $TDS'$ such that for each $TDS\in \mathcal O$ there are points $\mathcal P(\alpha)$ and $\mathcal Q(\alpha)$ of the same type as $\mathcal P'$ and $\mathcal Q'$, respectively.
Moreover, there is a section $\Sigma$ transverse to $\gamma'$, see  \figref{separatrixconnection}(a) and (b), defined by either (i) $v=\textnormal{const.}$, $u\in I$,  or (ii) $u=\textnormal{const.}$, $v\in I$, where $I$ is an open interval, such that the departure separatrix of $\mathcal P(\alpha)$ and the arrival separatrix of $\mathcal Q(\alpha)$ intersect $\Sigma$ for all $\alpha\in O$ in the following points:
 \begin{enumerate}
  \item[(i)] $(D(\alpha),\textnormal{const.})$ and $(A(\alpha),\textnormal{const.})$.
  \item[(ii)] $(\textnormal{const.},D(\alpha))$ and $(\textnormal{const.},A(\alpha))$,
 \end{enumerate}
in cases (i) and (ii), respectively.  In particular,
%
 \begin{align*}
  \Delta(\alpha): = D(\alpha)-A(\alpha)=b\cdot (\alpha-\alpha'),\quad \alpha\in O,
 \end{align*}
 for some
 \begin{align}
 \rspp{b=(b_1,\ldots,b_{2M})\in \mathbb Q^{2M}},\eqlab{splittingb}
 \end{align}
 is an affine function of $\alpha$. Moreover, it is invariant with respect to the action \eqref{translation} (i.e. $\sum_{k=1}^{2M} b_k=0$) and roots of $\Delta$ correspond to separatrix connections of $\mathcal P$ and $\mathcal Q$. 
\end{lemma}
\begin{proof}
 This should be clear enough, recall \lemmaref{this}. In particular, the expression for $\Delta$ is a consequence of  \lemmaref{mapPij} and \lemmaref{troppoint}. 
\end{proof}
\begin{remark}$\Delta$ only depends on the $\alpha_i$'s with $i\in \alpha_{\argmax_{k\in \mathcal I} F_k(u,v)}$ for some $(u,v)\in \gamma$ (like the quantities in \lemmaref{mapPij} and \lemmaref{troppoint}), but this will not be important. 
\end{remark}

The \rsp{vector $b$ \eqref{splittingb}} will be called \textit{a splitting constant} associated with $\gamma'$. Clearly, the splitting constant depends upon the section $\Sigma$, but at the same time:
\begin{lemma}\lemmalab{bbp}
If $b$ and $b'$ are two splitting constants then $b=cb'$ for some $c\in \mathbb Q$, $c\ne 0$.
\end{lemma}
\begin{proof}
Let $\Delta(\alpha)=b\cdot (\alpha-\alpha')$ and $\Delta'(\alpha)=b'\cdot (\alpha-\alpha')$ be two different distance functions. They are defined along different sections $\Sigma$ and $\Sigma'$, respectively. Consequently, we can conjugate $\Delta$ and $\Delta'$ through the transition map from $\Sigma$ to $\Sigma'$. This map is affine and invertible, since it is a composition of maps of the form \lemmaref{mapPij}. The result therefore follows.
\end{proof} 
Most importantly, $b=0\Leftrightarrow b'=0$, since $b$ can vanish; for a simple example, recall \figref{tropauto2}(d) (purple orbit). 
\begin{lemma}
 Suppose that the assumptions of \lemmaref{Deltasep} hold true and let $b'$ denote a splitting constant associated to the separatrix connection $\gamma'$. Then the tropical dynamical system $TDS'$ is locally structurally stable on a neighborhood $X$ of $\gamma'$ if $b'=0$. 

  \end{lemma}
\begin{proof}
 If $b'=0$ then $\Delta(\alpha)\equiv 0$ and there is a separatrix connection $\gamma(\alpha)$, with $\gamma(\alpha')=\gamma'$, for all $\alpha\in O$. The system is locally structurally stable, as no equivalence classes are produced or destroyed by perturbations. 
\end{proof}
\begin{remark}\remlab{bnonzero}
On the other hand, if the splitting constant $b'$ is nonzero for a separatrix connection $\gamma'$, then for any neighborhood $O$ of $\alpha'$, there is an $\alpha\in O$ such that $TDS$ does not have separatrix connection near $\gamma$. This may not lead to structural instability; $\mathcal P'$ not being a discontinuity point is an obvious example, but see also \figref{tropauto1H} below for a (more) nontrivial example. In any case, for a structurally stable system $TDS'$, we may assume that (or more accurately, consider a representative of the structural stable system for which) all separatrix connections $\gamma'$ of $TDS'$ have a vanishing splitting constant, i.e. $b'=0$. 
\end{remark}

\begin{definition}\defnlab{sepgen}
We say that a separatrix connection $\gamma$ is \textnormal{in general position} if its splitting constant $b\in \mathbb Q^{2M}$ vanish: $b=0$.
\end{definition}
This definition is independent of the section $\Sigma$, cf. \lemmaref{bbp}.


\section{Crossing cycles}\seclab{crossing}
The dynamical system defined by the crossing flow (induced by \eqref{uvnotinT}, recall \remref{crossing}) can have closed orbits. We call these crossing cycles, when they only intersect the discontinuity set of $\mathcal T$ along switching manifolds of crossing type, recall also \defnref{crossing}. We describe these, as in classical dynamical systems theory, using a return map. 

\begin{lemma}\lemmalab{returnmap}
Consider a tropical dynamical system $TDS'\in \TDSN$ with tropical coefficients $\alpha'$. Suppose that $\rspp{(\mathcal T',\mathcal T'^{\mathcal U},\mathcal T'^{\mathcal V})}$ is in general position and that there is a crossing cycle $\gamma'$ of the crossing flow, i.e. a closed polygonal orbit that does not go through the \rsp{tropical vertices} of $\mathcal T'$, see \defnref{cyclegen}. Then there is a neighborhood $O$ of $\alpha'$ in $\rspp{\mathbf{A}}(2M)$ and a section $\Sigma$ defined by $v=\textnormal{const.}, u\in I$, transverse to $\gamma'$ with $\gamma'\cap \Sigma=(u',\textnormal{const.})$ and $I$ an open interval, such that for all $TDS\in \mathcal O$, we have a well-defined return map $P(\cdot,\alpha):\Sigma\to \Sigma$ of the following form:
\begin{align}\eqlab{Pmap}
 P(u,\alpha) = u'+ c (u-u')+b\cdot (\alpha-\alpha'),\quad u\in \Sigma,\alpha\in O,
\end{align}
satisfying $P(u',\alpha')=u'$,
with $c\in \mathbb Q_+$, and $b=(b_1,\ldots,b_{2M})\in \mathbb Q^{2M}$ with $\sum_{k=1}^{2M} b_k=0$.
\end{lemma}
\begin{proof}
 We simply compose the transition maps in \lemmaref{mapPij}.
\end{proof}

The value $c>0$ is independent of the section $\Sigma$ and it is called \textit{the multiplier} of $\gamma'$. This is standard, see \cite{perko1991a}. The \rspp{vector $b\in \mathbb Q^{2M}$} will be called a \textit{splitting constant} associated to $\gamma'$. 

As usual, if $c=1$ then the fixed point $u=u'$ of $P(\cdot,\alpha')$ is non-hyperbolic, whereas if $c\ne 1$ then it is hyperbolic, being attracting for $c\in (0,1)$ and repelling for $c>1$ . This is standard, see \cite{perko1991a}. Since $P(\cdot,\alpha')$ is affine, $\gamma'$ is only a limit cycle in the hyperbolic case. In particular, $\gamma'$ is not isolated when $c=1$, but surrounded by closed orbits (since $I$ is open). 
\begin{lemma}\lemmalab{crossinggen}
 Suppose that the assumptions of \lemmaref{returnmap} hold true and let $c$ denote the multiplier of $\gamma'$ and let $b$ be an associated splitting constant, see \eqref{Pmap}. Then $TDS'$ is locally structurally stable in a neighborhood $X$ of $\gamma'$ if and only if either of the following statements hold true:
 \begin{enumerate}
  \item[(\textnormal{a})] $\gamma'$ is hyperbolic ($c\ne 1$). 
  \item[(\textnormal{b})] $\gamma'$ is nonhyperbolic ($c=1$) \textnormal{and} $b=0$.
 \end{enumerate}

\end{lemma}
\begin{proof}
 Suppose first that neither (a) nor (b) holds. Then $\gamma'$ is nonhyperbolic with $c=1$ and $b\ne 0$. But then there is an $\alpha\in O$ such that the Poincar\'e map does not have a fixed-point. This shows the only if part. 

 Next, for the if part. Suppose first that (a) holds. Then for all $\alpha\in O$ there is a limit cycle $\gamma(\alpha)$, satisfying $\gamma(\alpha')=\gamma'$, with the same stability as $\gamma'$. In particular, $\gamma(\alpha)$ intersects the section $\Sigma$ in the unique fixed-point of $P(\cdot,\alpha)$. This proves the locally structurally stability in this case. Finally, suppose (b). Then $P(\cdot,\alpha')=P(\cdot,\alpha)=\mathrm{id}$ for all $\alpha\in O$. Since all crossing switching manifolds are translated in an affine way with respect to $\alpha$, the result follows.
\end{proof}

\begin{definition}\defnlab{cyclegen}
 We say that a crossing cycle $\gamma$, with multiplier $c$ and a splitting constant $b$, is \textnormal{in general position} if either (a) or (b) \rspp{of \lemmaref{crossinggen}} hold true. 
\end{definition}

%

 \section{The crossing graph}\seclab{graph}
 Consider a tropical dynamical system $TDS$ with a polyhedral subdivision $\mathcal S$ of the Newton polygon associated with $F_{\max}$. For simplicity, we suppose that $\rspp{(\mathcal T,\mathcal T^{\mathcal U},\mathcal T^{\mathcal V})}$ is in general position. Then for each vertex $\degree F_k$ of the subdivision, we assign the flow vector $\md_k$ as a label to $\degree F_k$, as we did in \secref{auto}. This leads to our \textit{labelled subdivision}. From the subdivision, we obtain $\mathcal T$ (up to homotopy, see \propref{dualS}). Each labelled $\degree F_k$ corresponds to a tropical region $\mathcal R_k$ with flow vector $\md_k$.
 
 \begin{definition}\defnlab{graph}
  The labelled subdivision $\mathcal S$ defines a directed graph $\mathcal G$ (\textnormal{the crossing graph}) in the following way:
   The labelled degrees $\degree F_k$ of $\mathcal S$ are vertices of $\mathcal G$ and there is a directed edge from $\degree F_i$ to $\degree F_j$, $i\ne j$, if the edge $(\degree F_i,\degree F_j)$ is a subset of $\mathcal S$ \textnormal{and}
   \begin{align}\eqlab{graphG}
    (\md_j \cdot (\degree F_j-\degree F_i))(\md_i\cdot (\degree F_j-\degree F_i))>0.
   \end{align}
 \end{definition}
 By construction, the crossing graph $\mathcal G$ is a planar graph. Notice the following:
 \begin{lemma}
  Suppose that there is directed edge of the crossing graph $\mathcal G$ between $\degree F_i$ and $\degree F_j$. Then either (a) $\md_j = \md_i$ or (b) $\md_j\cdot \md_i =0$ \textnormal{and} there is crossing along the associated \rsp{tropical edge} $\mathcal E_{i,j}$. 
 \end{lemma}
\begin{proof}
 The statement follows from the definition of crossing, see \eqref{crossingdef}, and the fact that the \rsp{tropical edges} $\mathcal E_{i,j}$ are perpendicular to edges of the subdivision, see \propref{dualS}.
\end{proof}


 Notice that there is specifically a directed edge of $\mathcal G$ between $\degree F_i$ and $\degree F_j$ if $\mathcal E_{i,j}\subset \mathcal T$ with $\md_i=\md_j$ (so that $\mathcal E_{i,j}$ is not a switching manifold). 
 
 \begin{lemma}\lemmalab{graph}
  Suppose that a tropical dynamical system $TDS$ has a crossing cycle that passes through $n$ tropical regions: $\mathcal R_{k_1},\ldots,\mathcal R_{k_n}$, with $k_1,\ldots,k_n\in \mathcal I$ distinct and $n\in \mathbb N,n\ge 4$. Then the crossing graph $\mathcal G$ has a cycle through the associated vertices $\degree F_{k_1},\ldots, \degree F_{k_n}$. 
 \end{lemma}
 \begin{proof}
  This should be clear enough. 
 \end{proof}
 
 The graph also encodes the multiplier $c$ because this quantity is determined as a product of the slopes, see \lemmaref{returnmap} and \eqref{mapPij}, of the crossing switching manifolds. 
 
We cannot go the other way in \lemmaref{graph}, because the edges of the crossing graph $\mathcal G$ only relate to existence of one single orbit segment going from one tropical region to another. So a cycle of the crossing graph $\mathcal G$ does not imply that there is a closed orbit passing through each of the corresponding tropical regions; this depends upon the degrees and the position of the separatrices of the \rsp{tropical vertices}, see the examples below in \secref{crossing1} and \secref{crossing2}. In other words, the number of cycles in the crossing graph $\mathcal G$ is (only) an upper bound for the number of crossing cycles. We refer to \cite{aldred2008a} for optimal bounds on cycles of planar graphs, see also \secref{discussion} below.

For our purposes, the crossing graph $\mathcal G$ will be useful to understand the possible separatrix connections. To explain this, consider a \rsp{tropical singularity}  $\mathcal Q$. If we are in general position, then it belongs to a \rsp{tropical edge} $\mathcal E_{ij}$ at an interface between $\mathcal R_i$ and $\mathcal R_j$, i.e. $\mathcal E_{i,j}\subset \overline{\mathcal R}_i \cap \overline{\mathcal R}_j$. To study the departure separatrices of $\mathcal Q$, we therefore follow the paths in the directed graph from the corresponding vertices $\degree F_i$ and $\degree F_j$. To study the arrival separatrices, we have to reverse the directions of the edges of the graph.

Similarly, for a \rsp{tropical vertex}  $\mathcal P$ in general position, we have three adjacent tropical regions $\mathcal R_i$, $\mathcal R_j$ and $\mathcal R_l$, $i,j,l\in \mathcal I$ distinct, i.e. $\mathcal P=\overline{\mathcal R}_i\cap \overline{\mathcal R}_j\cap \overline{\mathcal R}_l$, and we can follow the departure separatrices of $\mathcal P$ by following the paths in the directed graph from the corresponding vertices $\degree F_i$, $\degree F_j$ and $\degree F_l$. Arrival separatrices are studied in a similar way by reversing the directions of the edges. 

In other words, each \rsp{tropical singularity}  of sink, source or saddle-type is associated with two vertices of the crossing graph $\mathcal G$, whereas each \rsp{tropical vertex}  in general position is associated with three vertices of $\mathcal G$. We conclude the following.
\begin{lemma}\lemmalab{connectiongraph}
Consider a tropical dynamical system with $\rspp{(\mathcal T,\mathcal T^{\mathcal U},\mathcal T^{\mathcal V})}$ in general position and suppose that there is a separatrix connection, departing from $\mathcal P$ and arriving at $\mathcal Q$. Then there is a path in the crossing graph $\mathcal G$ between the vertices associated with $\mathcal P$ and the vertices associated with $\mathcal Q$.
 \end{lemma}

\begin{remark}
The labelled subdivision $\mathcal S=\mathcal S^{\mathcal I}$ (and the crossing graph $\mathcal G$) describe the flow on all \rsp{tropical edges}, except on \rspp{edges with nullcline sliding}. \rspp{The flow on these} are essentially determined by the intersection of the tropical curves $\mathcal T^{\mathcal U}$ and $\mathcal T^{\mathcal V}$ of the $u$- and $v$-polynomials, $F_{\max}^{\mathcal U}$ and $F_{\max}^{\mathcal V}$, respectively. These intersections can be encoded in the Minkowski sum of $\mathcal S^{\mathcal U}$ and $\mathcal S^{\mathcal V}$, see \cite{de2010a} for details. 
\end{remark}

\subsection{\rsp{Example:} \rspp{A bifurcating crossing cycle}}\seclab{crossing1}
\rspp{In this section,} we consider a simple one-parameter-family of tropical dynamical systems, defined by the following tropical pairs:
\begin{equation}\eqlab{crossing1}
\begin{alignedat}{3}
 F_1(u,v)&:=u,\quad &\md_1&:=(1,0),\\
 F_2(u,v)&:=-5+3u+3v,\quad  &\md_2&:=(-1,0),\\
 F_3(u,v)&:=v, \quad &\md_3&:=(0,-1),\\
 F_4(u,v)&:=4u+2v, \quad &\md_4&:=(0,1),\\
 F_5(u,v)&:=\alpha+5u+5v,\quad &\md_5&:=(0,-1),
\end{alignedat}
\end{equation}
\rspp{with $\mathcal U=\{1,2\},\,\mathcal V=\{3,4,5\}$, that (among other things) supports a bifurcating crossing cycle}.
A simple calculation shows that for all $\alpha<-11$, we obtain the labelled subdivision $\mathcal S$ in \figref{crossing1}(a). We restrict attention to this set of $\alpha$-values. 

We have also illustrated the associated crossing graph $\mathcal G$ in \figref{crossing1}(a) (purple edges). There is a single cycle of the graph given by
\begin{align}
 (1,0)\to (4,2)\to (3,3)\to (0,1)\to (1,0).\eqlab{cycleG}
\end{align}

There are three important \rsp{tropical vertices} associated with the tropical polynomial $F_{\max}(u,v)=\max_{k\in \{1,\ldots,5\}} F_k(u,v)$:
\begin{align*}
 \mathcal P_{134}=(0,0),\quad \mathcal P_{234} = (-1,4),
\end{align*}
and
\begin{align}\eqlab{P245}
 \mathcal P_{245} = \left(-\frac{15}{4}-\frac{\alpha}{4},\frac{5}{4}-\frac{\alpha}{4}\right),
\end{align}
for $\alpha<-11$. There is also a unique \rsp{tropical singularity}  $\mathcal Q_{1234}$ at 
\begin{align*}
 \left(-\frac12,2\right),
\end{align*}
due to the intersection of $\mathcal E^{\mathcal U}_{12}$ and $\mathcal E^{\mathcal V}_{34}$. 

In \figref{crossing1}(b), we illustrate the phase portrait of the tropical system for $\alpha = -25$. Here we see a stable crossing limit cycle (in red) intersecting $\Sigma$ (in green) defined by $v=0$ in the point $u=2$. In agreement with the cycle \eqref{cycleG}, it goes from $\mathcal R_1$ (blue) to $\mathcal R_4$ (red) and then to $\mathcal R_2$ (light blue) and finally to $\mathcal R_\rspp{3}$ (brown), before it repeats itself. It is easy to show that the return map $P(\cdot,\alpha):\Sigma\to \Sigma$ satisfies:
\begin{align*}
 P(u,\alpha)  = \frac{4}{9}u+\frac{10}{9}.
\end{align*}
We see that $u=2$ is a stable fixed point ($P(2,\alpha)=2$, $c:=P'_u(2,\alpha)=\frac49 \in (0,1)$) and that $P$ is independent of $\alpha$, see \lemmaref{returnmap}. However, this expression is only valid if 
\begin{align*}
 2<-\frac{15}{4}-\frac{\alpha}{4},
\end{align*}
where the right hand side is the $u$-component of $\mathcal P_{245}$, see \eqref{P245}. This gives
\begin{align*}
 \alpha <-23.
\end{align*}
At $\alpha=-23$ there is a structurally unstable homoclinic separatrix connection, arriving and departing from $\mathcal P_{245}$, see \figref{crossing1bif}(a). There are also structurally unstable separatrix connections at the following values of $\alpha\in (-11,-23)$.
\begin{enumerate}
\item For $\alpha=-\frac{167}{9}$, there is a separatrix connection, departing from $\mathcal Q_{1234}$ (moving upwards) and arriving at $\mathcal P_{245}$, see \figref{crossing1bif}(b).
 \item For $\alpha=-\frac{53}{3}$, there is a separatrix connection, departing from $\mathcal P_{234}$ and arriving at $\mathcal P_{245}$, see \figref{crossing1bif}(c).
 \item For $\alpha=-\frac{49}{3}$, there is a separatrix connection, departing from $\mathcal Q_{1234}$ (moving downwards) and arriving at $\mathcal P_{245}$, see \figref{crossing1bif}(d).
 \item For $\alpha=-15$, there is a separatrix connection, departing from $\mathcal P_{134}$ and arriving at $\mathcal P_{245}$, see \figref{crossing1bif}(e).
 \item For $\alpha=-13$, there is another separatrix connection, departing from $\mathcal Q_{1234}$ and arriving at $\mathcal P_{245}$. See \figref{separatrixconnection}(b).
\end{enumerate}
For $\alpha=-11$, we have $\mathcal P_{234}=\mathcal P_{245}:=\mathcal P_{2345}$ (not shown). There are no crossing cycles for $\alpha\in [-23,-11)$, cf. \defnref{crossing}, but as is clear from \figref{crossing1bif}, there is a (unique) stable sliding cycle for every $\alpha\in [-23,-11)$. However, the $\alpha$-parameter family defined by \eqref{crossing1} is only structurally stable within $\alpha<-11$ for $\alpha\notin \{-23,-\frac{167}{9},-\frac{53}{3}, -\frac{49}{3},-15,-13\}$.

\begin{figure}[H]
    \centering
    \begin{subfigure}{0.5\textwidth}
    \centering
        \includegraphics[width=0.96\linewidth]{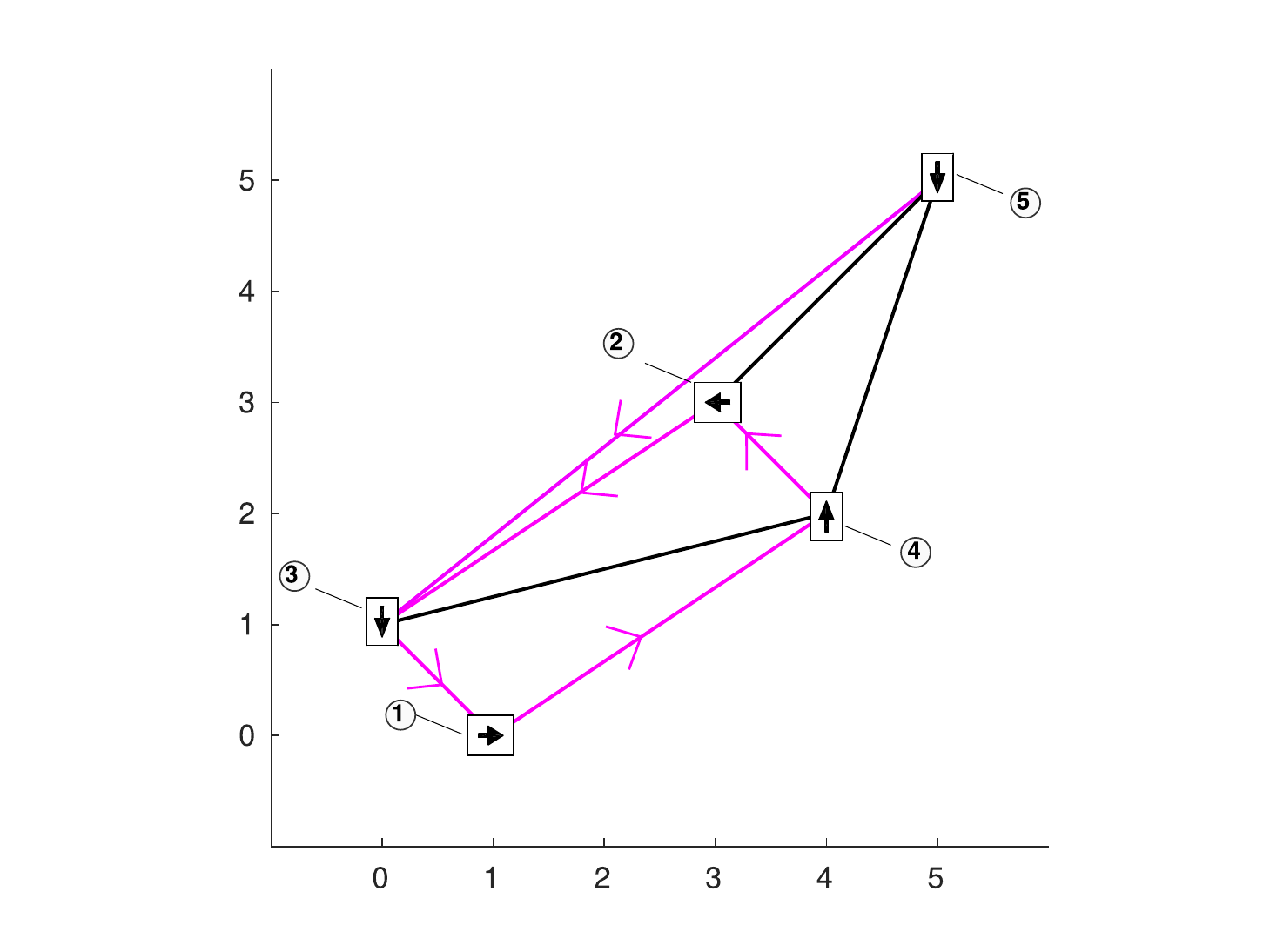}
        \caption{Labelled subdivision and the crossing graph $\mathcal G$ (purple)}
    \end{subfigure}%
    \begin{subfigure}{0.5\textwidth}
    \centering
        \includegraphics[width=0.96\linewidth]{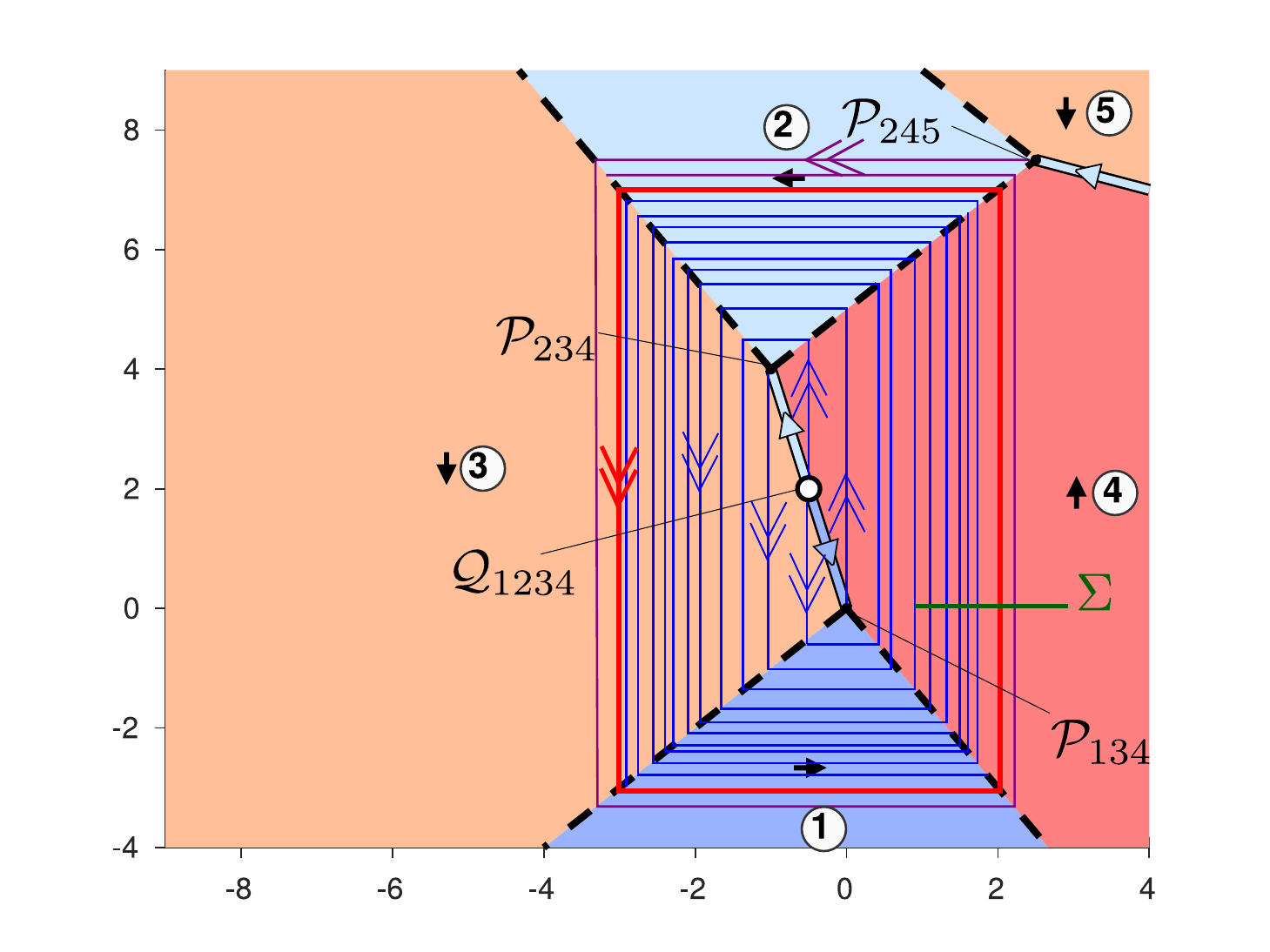}
        \caption{Phase portrait}
    \end{subfigure}
    \caption{In (a): The labelled subdivision and the crossing graph $\mathcal G$ (purple) for the tropical dynamical system defined by \eqref{crossing1} for $\alpha<-11$. There is a single cycle of the graph, see \eqref{cycleG}. In (b): The associated phase portrait for $\alpha=-25$. There is a crossing limit cycle in red. The system is structurally stable (within the one-parameter family defined by \eqref{crossing1}).}
    \figlab{crossing1}
    \end{figure}
    
    \begin{figure}[H]
    \centering
     \begin{subfigure}{0.495\textwidth}
    \centering
        \includegraphics[width=0.96\linewidth]{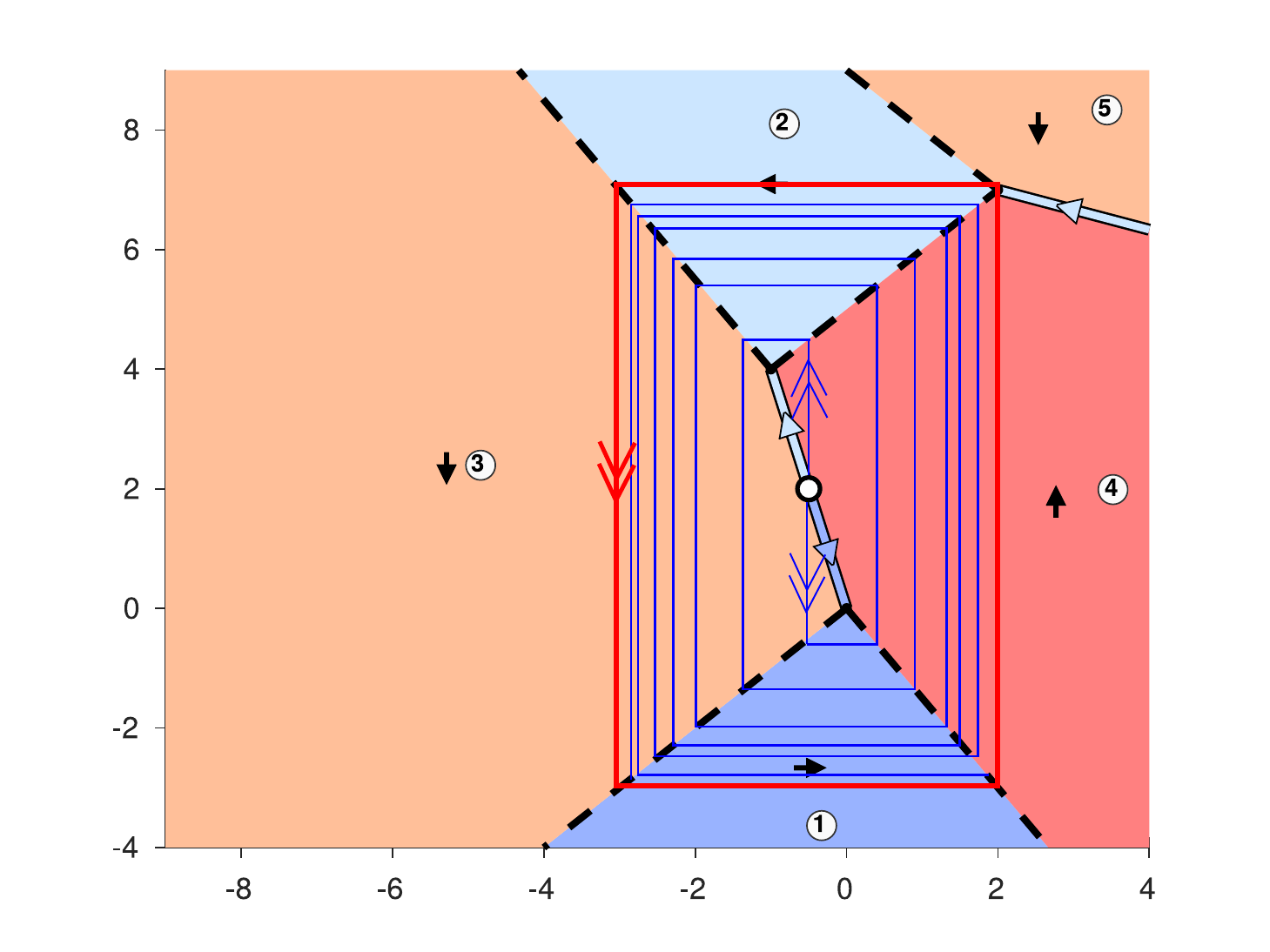}
        \caption{$\alpha=-23$}
    \end{subfigure}
    \begin{subfigure}{0.495\textwidth}
    \centering
        \includegraphics[width=0.96\linewidth]{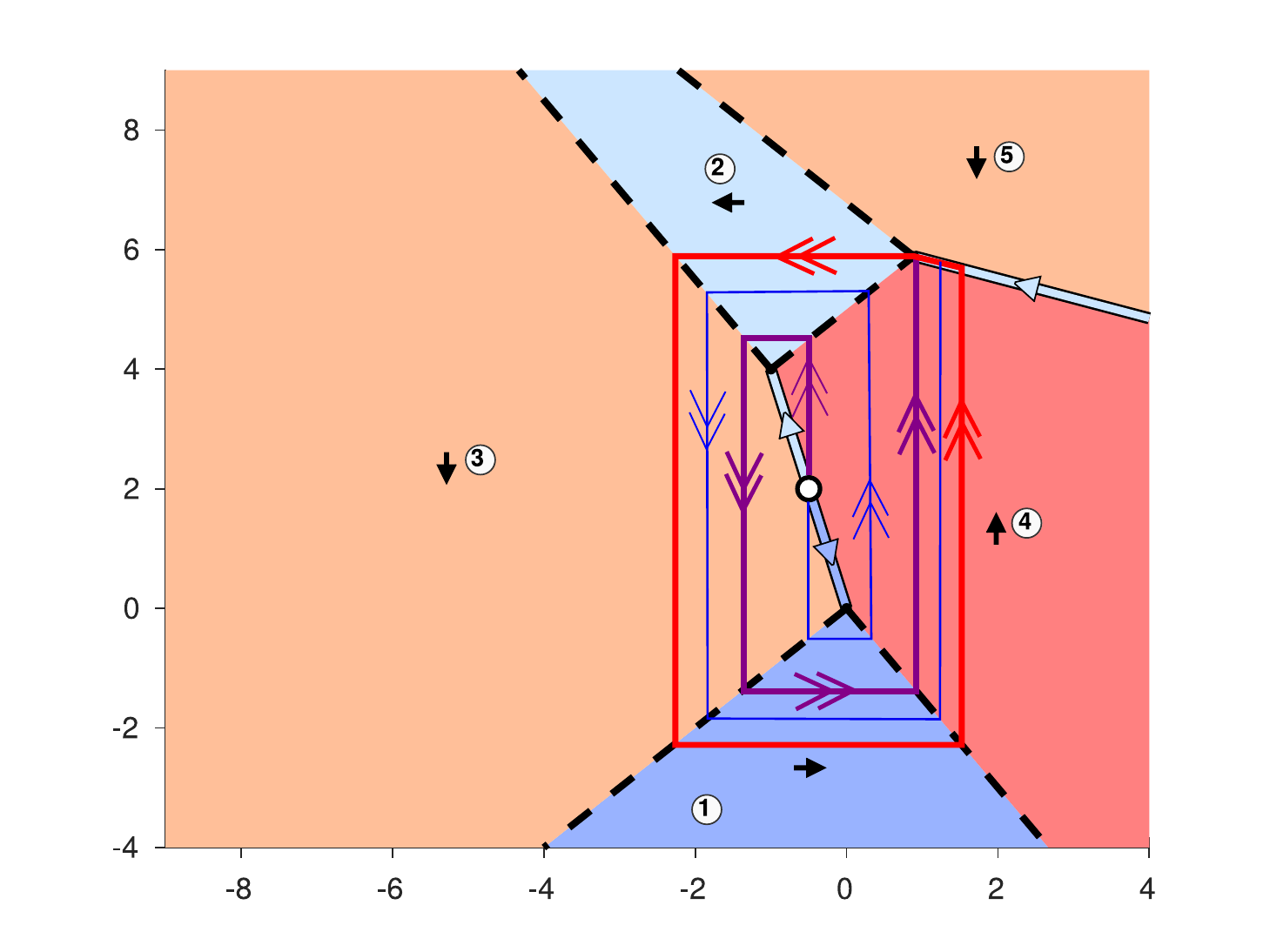}
        \caption{$\alpha=-\frac{167}{9}$}
    \end{subfigure}%
    \begin{subfigure}{0.495\textwidth}
    \centering
        \includegraphics[width=0.96\linewidth]{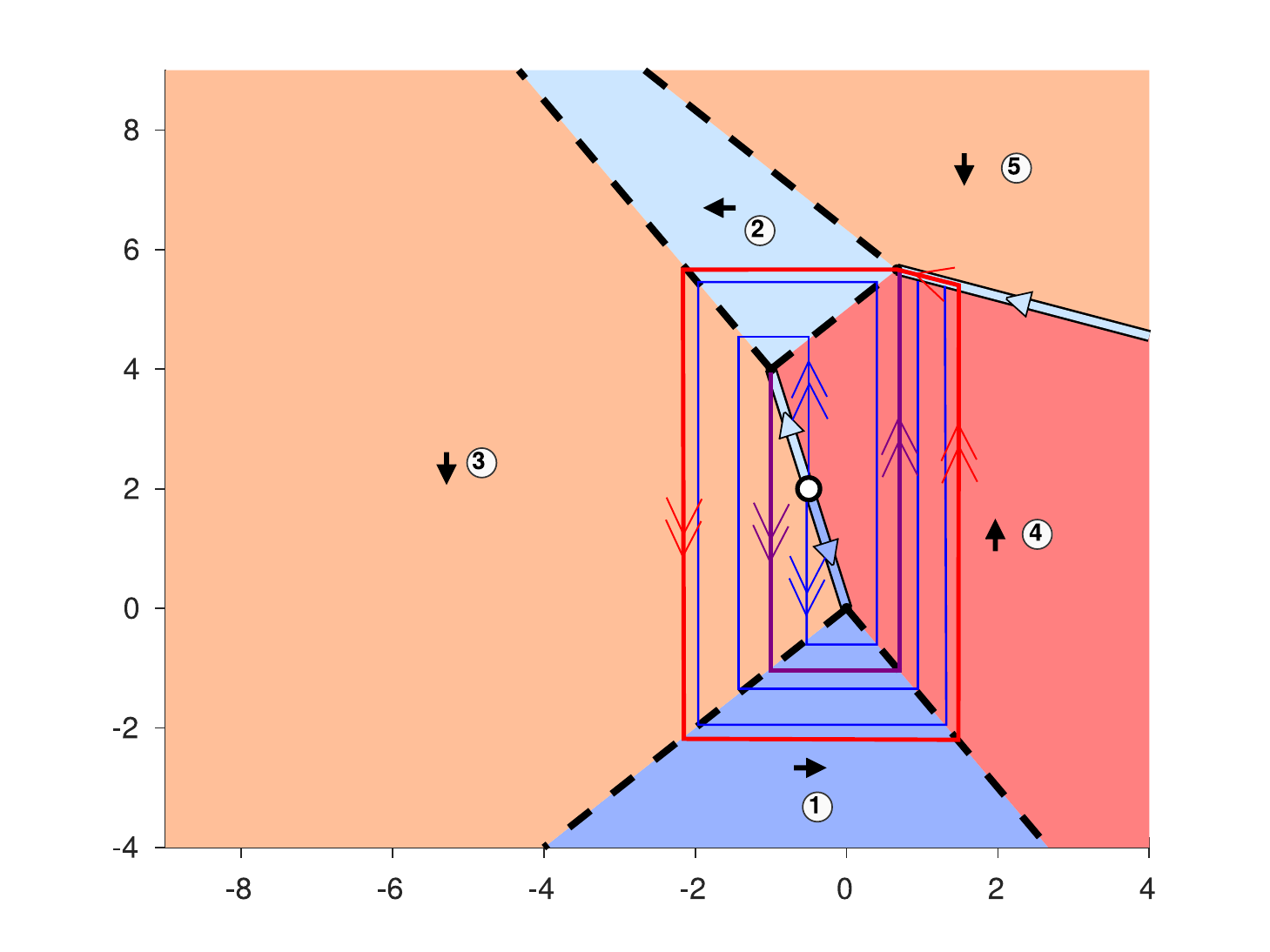}
        \caption{$\alpha=-\frac{53}{3}$}
    \end{subfigure}
     \begin{subfigure}{0.495\textwidth}
    \centering
        \includegraphics[width=0.96\linewidth]{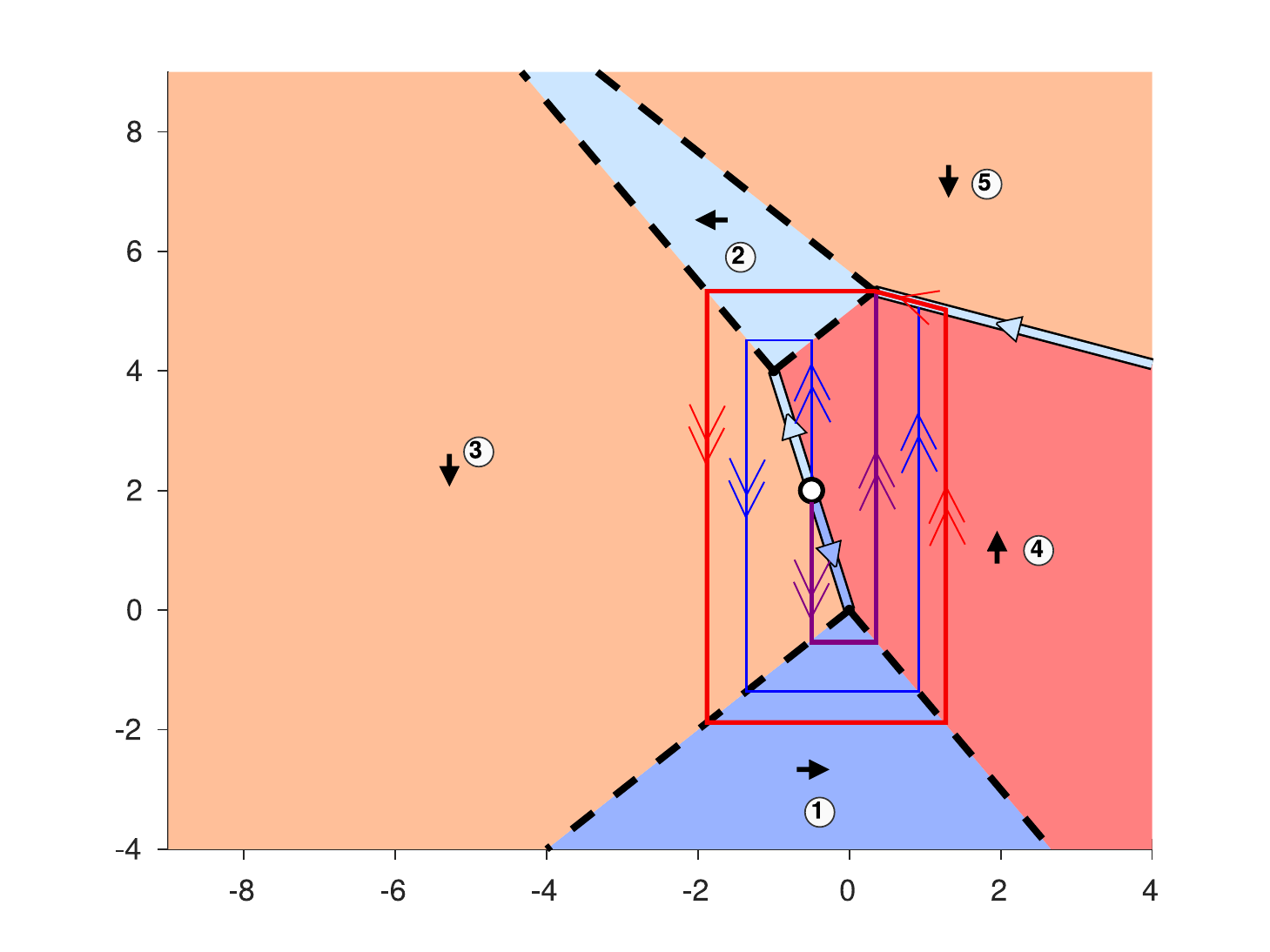}
        \caption{$\alpha=-\frac{49}{3}$}
    \end{subfigure}
     \begin{subfigure}{0.495\textwidth}
    \centering
        \includegraphics[width=0.96\linewidth]{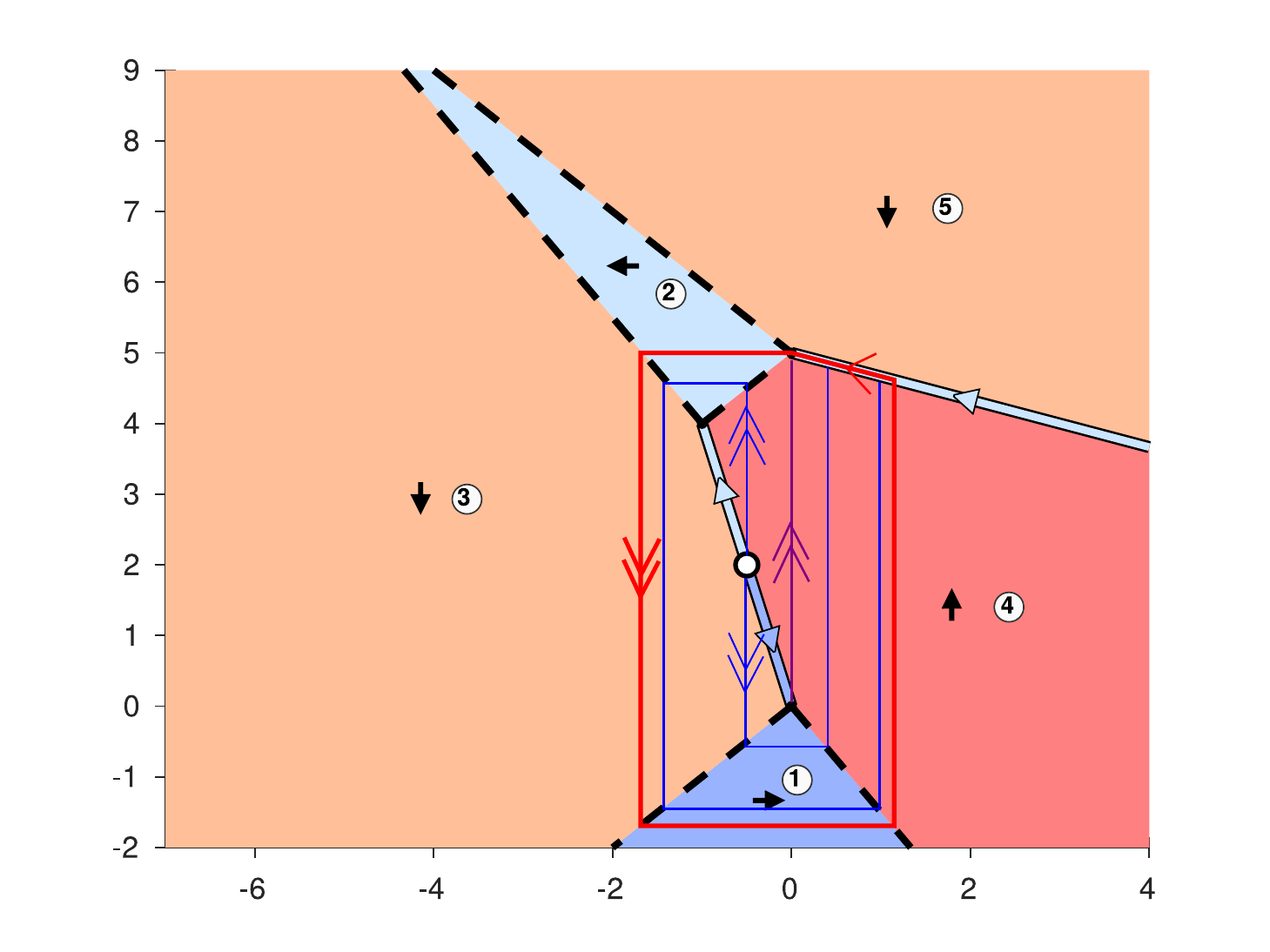}
        \caption{$\alpha=-15$}
    \end{subfigure}
    \caption{Bifurcations of the tropical dynamical system defined by \eqref{crossing1} due to separatrix connections. In (a), where $\alpha=-23$, the crossing limit cycles goes through the point $\mathcal P_{245}$ giving rise to a homoclinic separatrix connection. In (b), where $\alpha=-\frac{167}{9}$, there is a separatrix connection, departing from $\mathcal Q_{1234}$ and arriving at $\mathcal P_{245}$. In (c), where $\alpha=-\frac{53}{3}$, there is a separatrix connection, departing from $\mathcal P_{234}$ and arriving at $\mathcal P_{245}$ (see \figref{crossing1}). In (d), where $\alpha=-\frac{49}{3}$, there is a separatrix connection, departing from the tropical source $\mathcal Q_{1234}$ and arriving at $\mathcal P_{245}$. Finally, in (e), where $\alpha=-15$, there is a separatrix connection departing from $\mathcal P_{134}$ and arriving at $\mathcal P_{245}$.  }
    \figlab{crossing1bif}
    \end{figure}
%
\subsection{\rsp{Example: A} structurally stable separatrix connection}\seclab{crossing2}
Now, consider instead the following one-parameter-family of tropical dynamical systems, defined by the following tropical pairs:
\begin{equation}\eqlab{crossing2}
\begin{alignedat}{3}
 F_1(u,v)&:=2u,\quad &\md_1&:=(1,0),\\
 F_2(u,v)&:=-2+2u+3v,\quad  &\md_2&:=(-1,0),\\
 F_3(u,v)&:=v, \quad &\md_3&:=(0,-1),\\
 F_4(u,v)&:=4u+v, \quad &\md_4&:=(0,1),\\
 F_5(u,v)&:=\alpha+5u+4v,\quad &\md_5&:=(0,-1),
\end{alignedat}
\end{equation}
\rspp{with $\mathcal U=\{1,2\}$, $\mathcal V=\{3,4,5\}$}.
It is easy to show that for all $\alpha<-3$, we obtain the labelled subdivision $\mathcal S$ in \figref{crossing2}(a). We restrict attention to this set of $\alpha$-values. There is a single cycle of the crossing graph $\mathcal G$ given by
\begin{align}\eqlab{cycleG2}
 (2,0)\to (4,1)\to (2,3)\to (0,1)\to (2,0).
\end{align}

There are as before three important \rsp{tropical vertices} of the tropical polynomial $F_{\max}(u,v)=\max_{k\in\{1,\ldots,5\}}F_k(u,v)$:
\begin{align*}
 \mathcal P_{134}=(0,0),\quad \mathcal P_{234} = (0,1),
\end{align*}
and
\begin{align}\nonumber
 \mathcal P_{245} = \left(-\frac{3}{4}-\frac{\alpha}{4},\frac{1}{4}-\frac{\alpha}{4}\right),
\end{align}
for $\alpha<-3$. However, there are now two \rsp{tropical singularities}: $\mathcal Q_{1234}$ at 
\begin{align*}
 \left(0,\frac23\right),
\end{align*}
due to the intersection of $\mathcal E^{\mathcal U}_{12}$ and $\mathcal E^{\mathcal V}_{34}$, and $\mathcal Q_{1245}$ at 
\begin{align*}
 \left(-\alpha-2,\frac23\right),
\end{align*}
due to the intersection of $\mathcal E^{\mathcal U}_{12}$ and $\mathcal E^{\mathcal V}_{45}$. $\mathcal Q_{1234}$ is a hybrid point (center-like), whereas $\mathcal Q_{1245}$ is a tropical saddle.

In \figref{crossing2}(b) we show the phase portrait for $\alpha=-4$. 
The return map to $\Sigma$, defined by $v=0$, $u\in (0,-\frac{3}{4}-\frac{\alpha}{4})$, is the identity in this case (so that $c=1$, $b=0$ in \lemmaref{returnmap}). Consequently, there is a family of (nonhyperbolic) crossing cycles. This also means that there is a structurally stable homoclinic separatrix connection, departing and arriving at the point $\mathcal P_{245}$. A simple calculation, shows that there are no additional bifurcations for $\alpha<-3$; in particular $\mathcal P_{245}$ cannot connect to $\mathcal Q_{1245}$ by crossing only. At $\alpha=-3$, we have $\mathcal P_{245}=\mathcal P_{234}$ and $\rspp{(\mathcal T,\mathcal T^{\mathcal U},\mathcal T^{\mathcal V})}$ is not in general position. Interestingly, the purple cycle in \figref{crossing2}(b) is a limit cycle, as the red orbit coincides with it beyond $\mathcal P_{245}$. (Notice that by \defnref{crossing} the purple curve is a sliding cycle, not a crossing cycle). 

\begin{remark}\remlab{hybridpoint}
%
 The tropical system \eqref{crossing2} is the $\epsilon=0$ singular limit of the following system
 \begin{equation}\eqlab{thisss}
 \begin{aligned}
  x' &= x^3(1-e^{-2\epsilon^{-1}} y^3),\\
  y' &=y^2(-1+x^4  -e^{\alpha \epsilon^{-1}} x^5 y^3),
 \end{aligned}
\end{equation}
written in the $(u,v)$-coordinates, see \eqref{uv}.  A direct calculation shows that \eqref{thisss} has two singularities in the first quadrant for any $\alpha<-2$ and all $0<\epsilon\ll 1$. Indeed, these are given by the equations
\begin{align*}
 x^4-1-e^{(\alpha+2)\epsilon^{-1}} x^5=0,\quad y =e^{\frac23 \epsilon^{-1}},
\end{align*}
and the first equation has two positive solutions $x>0$ in this case. This is in agreement with \figref{crossing2}.

Notice also that for $\alpha<-2$ and $0<\epsilon\ll 1$ then there is a \rspp{singularity} near $x\approx 1$. A direct calculation, shows that the Jacobian at this point has positive determinant and a negative trace given by 
\begin{align}
-3 e^{\frac{3\alpha+8}{3}\epsilon^{-1}}x^5<0. \eqlab{trace}
\end{align}
The \rspp{singularity} is therefore hyperbolic, specifically an attracting focus for \eqref{thisss} for  $\alpha<-2$ and all $0<\epsilon\ll 1$. The corresponding \rsp{tropical singularity}  at $(u,v)=(0,\frac23)$ for $\alpha<-3$ (where the trace \eqref{trace} is exponentially small with respect to $\epsilon\to 0$) is not attracting for the tropical dynamical system, see \figref{crossing2}. Indeed, it is a hybrid point (of center type). 
\end{remark}

\begin{figure}[H]
    \centering
    \begin{subfigure}{0.5\textwidth}
    \centering
        \includegraphics[width=0.96\linewidth]{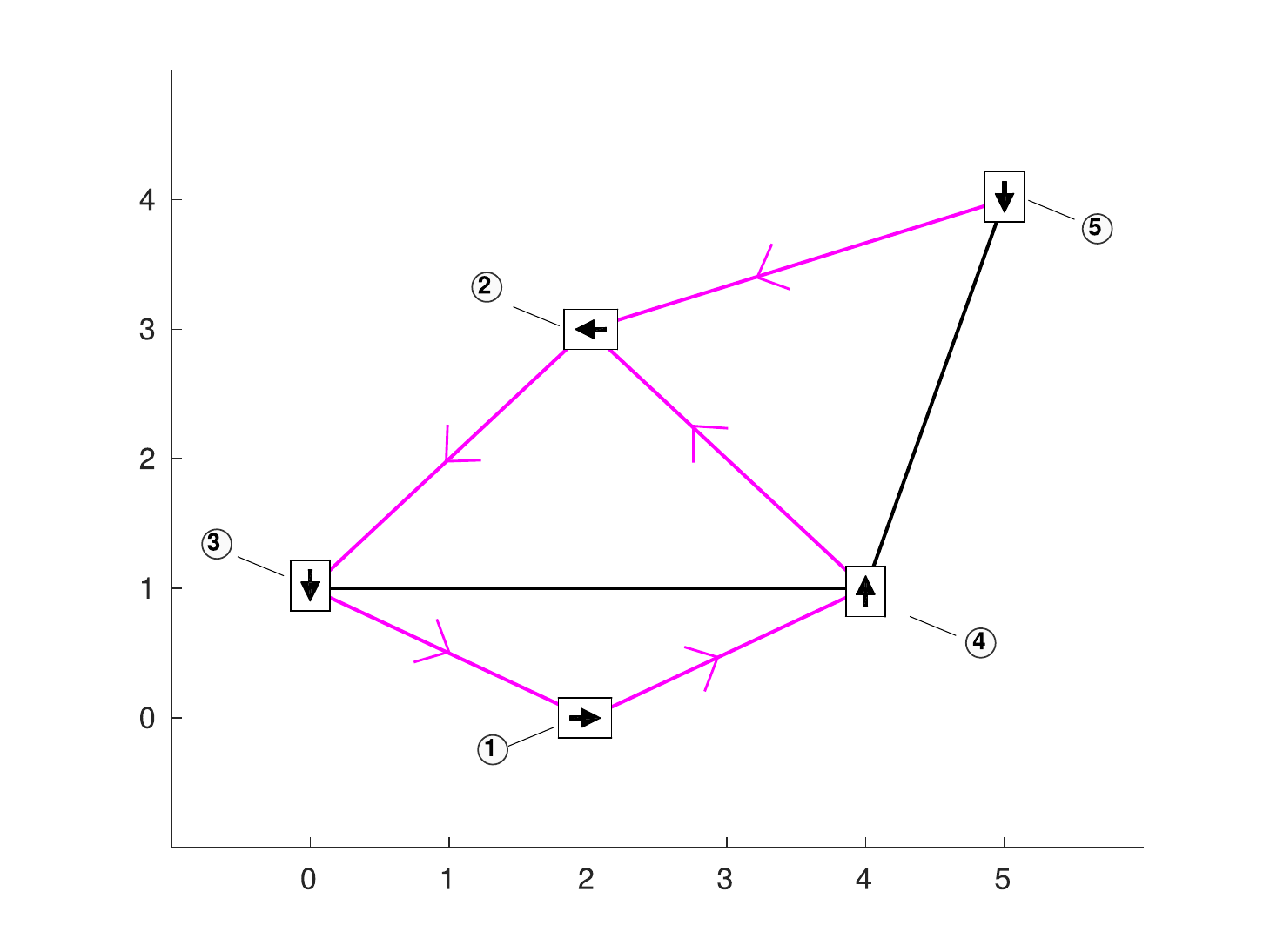}
        \caption{Labelled subdivision and the crossing graph $\mathcal G$ (purple)}
    \end{subfigure}%
    \begin{subfigure}{0.5\textwidth}
    \centering
        \includegraphics[width=0.96\linewidth]{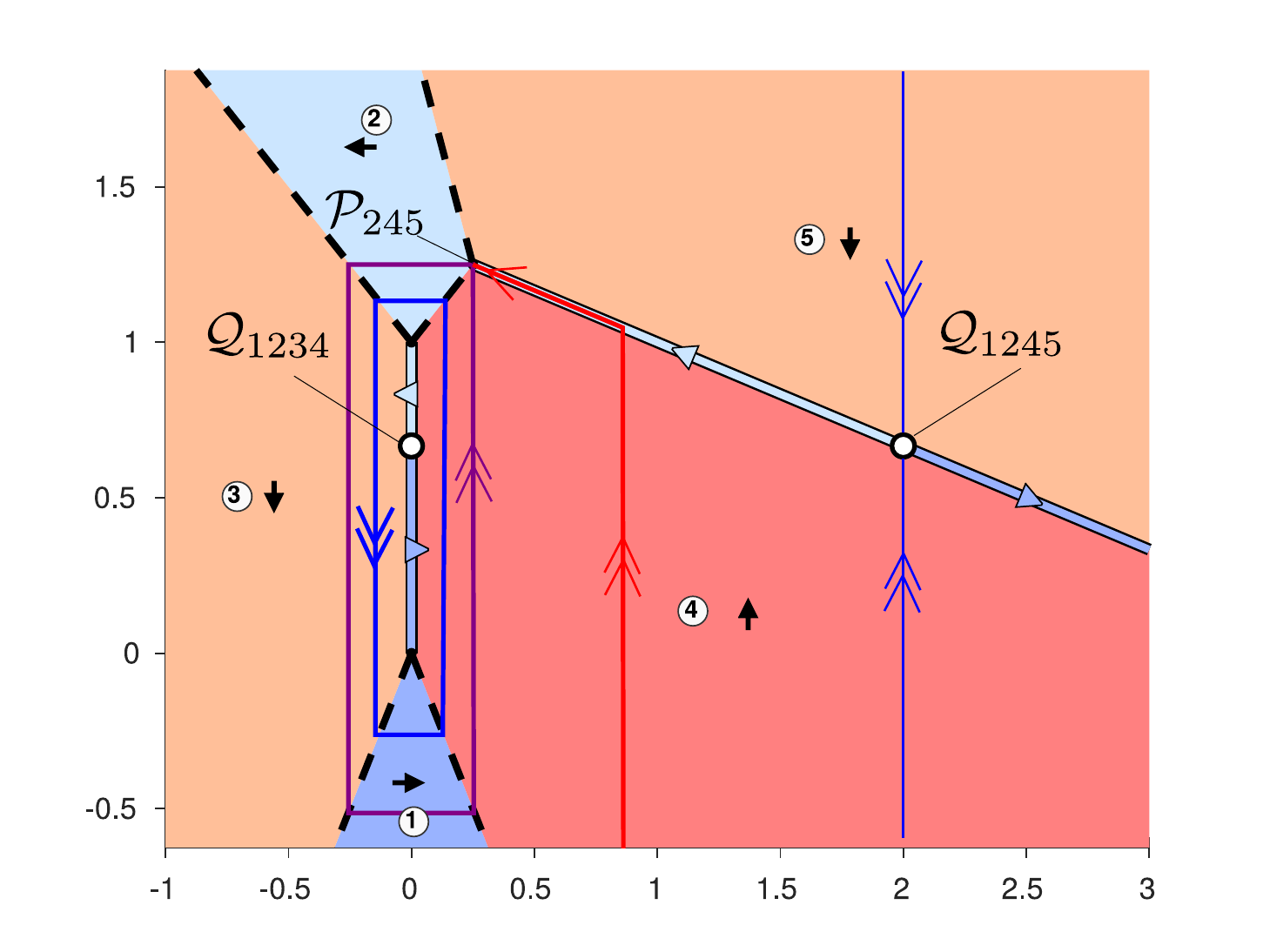}
        \caption{Phase portrait}
    \end{subfigure}
    \caption{In (a): The labelled subdivision and the crossing graph $\mathcal G$ (purple) for the tropical dynamical system defined by \eqref{crossing2} for $\alpha<-3$.  In (b): The associated phase portrait for $\alpha=-4$. There is a single cycle of the graph in (a), see \eqref{cycleG2}, and a family of crossing cycles (all nonhyperbolic with $c=1$, recall \lemmaref{returnmap}) in (b). The purple cycle is a separatrix connection the point $\mathcal P_{245}$. It is structurally stable (due to $c=1$). }
    \figlab{crossing2}
    \end{figure}

\section{Main results: Structurally stable tropical dynamical systems}\seclab{structurallystable}
We are now finally in a position to state our main results. 
  \begin{thm}\thmlab{thm1}
   Consider a tropical system $$TDS'\in \TDS,$$ and assume that 
   \begin{enumerate}
   \item $\rspp{(\mathcal T',\mathcal T'^{\mathcal U},\mathcal T'^{\mathcal V})}$,
   \item all separatrix connections, 
   \item and all crossing cycles,
   \end{enumerate}
   of $TDS'$ are in general position, cf. \defnref{generalposition}, \defnref{sepgen}, and \defnref{cyclegen}, respectively. 
Then $TDS'$ is structurally stable.
%
  \end{thm}
  \begin{proof}
     It follows from the previous sections, that all \rsp{tropical vertices}, singularities, separatrix connections, and all crossing cycles are locally structurally stable. 
     
     
 Consider any departure (or arrival) separatrix of a point $\mathcal P'$, which is not a (structurally stable) separatrix connection. Either it is crossing only (of \eqref{uvnotinT}), or it intersects a sliding manifold (Filippov or transversal nullcline) away from arrival separatrices of \rsp{tropical singularities}.  
  Suppose the former. Then since the crossing flow defines a usual dynamical system, either the separatrix goes unbounded or the $\omega$-limit ($\alpha$-limit) set is a crossing limit cycle, which is hyperbolic by assumption. 
    The same holds for the departure (arrival) orbit of the perturbed $\mathcal P(\alpha)$, $\alpha\in O$, since it is not a separatrix connection. Moreover, all intersections with the crossing switching manifolds of $\mathcal T$ of  $TDS\in \mathcal O$, with $\mathcal O$ a neighborhood of $TDS'$, vary in an affine way with respect to $\alpha\in O$. 
    
    Now, we turn to stable (unstable) sliding. In forward (backward) time, orbits can only leave a sliding segment at a \rsp{tropical vertex}, through a departure (arrival, respectively) orbit. Otherwise, it goes unbounded (only possible if the sliding segment is unbounded) or it goes to a \rsp{tropical singularity}. Since $\rspp{(\mathcal T',\mathcal T'^{\mathcal U},\mathcal T'^{\mathcal V})}$ is in general position, see \lemmaref{troppoint2} and \lemmaref{tropeq}, the same is true for $\mathcal T$ of $TDS\in \mathcal O$. We can therefore conclude that $TDS'$ and $TDS\in \mathcal O$ have the same orbit equivalence classes. Consequently, $TDS'$ is structurally stable.    
    
    
  \end{proof}

\begin{thm}\thmlab{thm2}
There is an open dense subset $\mathcal D$ of $$\TDS,$$ with the associated subset $D\subset \rspp{\mathbf{A}}(2M)$ (\rspp{recall the notation in \remref{mathcalOO}}) consisting of a finite union of convex polytopes, upon which we have the following: 
\begin{enumerate}
 \item \label{item1thm2} Any $TDS\in \mathcal D$ is structurally stable.
\item \label{item2thm2} There is a finite number of equivalence classes of structurally stable systems on $\mathcal D$.
\end{enumerate}
  \end{thm}
  \begin{proof}
To prove the theorem, we focus on item \ref{item2thm2} and show that there is a finite number of equivalence classes of structurally stable systems. This will in turn show the existence of $\mathcal D$ in item \ref{item1thm2}.

  By  \lemmaref{triangulation}, for a structurally stable system we may assume that (or more accurately, consider a representative of the structural stable system for which) $\rspp{(\mathcal T,\mathcal T^{\mathcal U},\mathcal T^{\mathcal V})}$ is in general position. There are finitely many (regular) triangulations of $\mathcal S^{\mathcal U}$ and $\mathcal S^{\mathcal V}$. For each fixed pair of triangulation of $\mathcal S^{\mathcal U}$ and $\mathcal S^{\mathcal V}$, there are finitely many  triangulations of $\mathcal S^{\mathcal I}$. We then fix $\mathcal S^{\mathcal U}$, $\mathcal S^{\mathcal V}$ and $\mathcal S^{\mathcal I}$. This corresponds to fixing a convex polytope $\rspp{O}\subset \rspp{\mathbf{A}}(2M)$, see \thmref{thm0}. Within this set, the corresponding tropical curves  $\mathcal T^{\mathcal U},\mathcal T^{\mathcal U}$ and $\mathcal T^{\mathcal I}$ are fixed up to homotopy, but their relative positions are not, \cite{maclagan2015introduction}. However, there are finitely many ways that $\mathcal T^{\mathcal U}$ and $\mathcal T^{\mathcal V}$ can intersect transversally (up to homotopy) along their \rsp{tropical edges} -- these are described by the Minkowski sum, see \cite{maclagan2015introduction} -- and each of these cases correspond to a convex polytope $\rspp{O}'\subset Y$. Within $\rspp{O}'$, all singularities and \rsp{tropical vertices} are in general position, see \lemmaref{troppoint2} and \lemmaref{tropeq}. Finally, the crossing graph $\mathcal G$, recall \defnref{graph}, is fixed for $\alpha\in \rspp{O}'$.
  
  We now subdivide $\rspp{O}'$ further through separatrix connections (crossing cycles are handled in a similar way using the (finitely many) cycles of $\mathcal G$). Consider a departure separatrix from a point $\mathcal P$ with $\alpha\in \rspp{O}'$. (The case of an arrival separatrix is handled in the same way.) Then in a neighborhood of $\mathcal P$ this set varies in an affine way with respect to $\alpha$. Now, following \lemmaref{connectiongraph}, we use the crossing graph $\mathcal G$ to follow the departure separatrix. There are finitely many possible separatrix connections to consider. If a separatrix connection exists for some $\alpha\in \rspp{O}'$ with a splitting constant $b\ne 0$, then we divide $\rspp{O}'$ into two convex polytopes (corresponding to $\Delta(\alpha)\gtrless 0$). Proceeding in this way, we obtain finitely many convex polytopes -- the  union of which is open and dense -- and where \thmref{thm1} applies. In turn, we have finitely many equivalence classes of structurally stable systems. This completes the proof.

  
  \end{proof}

%

\section{\rsp{Example}: A generalized autocatalator}\seclab{genauto}
In this section, we consider a generalized tropical autocatalator model defined by the following tropical pairs:
\begin{equation}\eqlab{tropgenauto}
\begin{alignedat}{3}
 F_1(u,v)&:=\alpha_1-u,\quad &\md_1&:=(1,0),\\
 F_2(u,v)&:=\alpha_2,\quad  &\md_2&:=(-1,0),\\
 F_3(u,v)&:=\alpha_3+2v, \quad &\md_3&:=(-1,0),\\
 F_4(u,v)&:=\alpha_4, \quad &\md_4&:=(0,-1),\\
 F_5(u,v)&:=\alpha_5+u-v,\quad &\md_5&:=(0,1),\\
 F_6(u,v)&:=\alpha_6+u+v,\quad &\md_6&:=(0,1),
\end{alignedat}
\end{equation}
\rspp{with $\mathcal U=\{1,2,3\}$, $\mathcal V=\{4,5,6\}$}. 
In comparison with \eqref{tropauto}, all tropical coefficients $\alpha_1,\ldots,\alpha_6$ are ``free'', in the sense that $\alpha_1=\alpha-1$, $\alpha_2=-1$, $\alpha_3=-1$, $\alpha_4=0$, $\alpha_5=0$, $\alpha_6=0$ in \eqref{tropgenauto} gives \eqref{tropauto}. In particular, the tropical flow vectors in \eqref{tropgenauto} are the same as in \eqref{tropauto}. \textit{We believe that this mimics the general situation in chemical reaction networks, where the signs of the terms are known, but the (magnitude of the) parameters are uncertain/unknown.} 

\begin{proposition}\proplab{genauto}
 There are $15$ different structurally stable phase portraits of the tropical dynamical system defined by \eqref{tropgenauto}.
\end{proposition}

To prove this claim, we follow the approach of the proof of \thmref{thm2} and first determine the different $\rspp{(\mathcal T,\mathcal T^{\mathcal U},\mathcal T^{\mathcal V})}$ in general position, see \defnref{generalposition}.  For this purpose, we notice that $\mathcal S^{\mathcal U}$ and $\mathcal S^{\mathcal V}$ for $\alpha_i\in \mathbb R$, $i\in \{1,\ldots,6\}$, each consist of a single triangle, see \figref{SUV}, since there are three monomials (with $\degree F_k$ not co-linear) in each direction. Next, for $\mathcal S^{\mathcal I}$, we use a refinement poset \cite{de2010a} of the point configuration associated with \eqref{tropgenauto}, see \figref{subgenauto}. The notion of a refinement gives an ordering of the subdivision into layers according to how far a subdivision is from  being a triangulation, see \cite{de2010a}.  The upper-most subdivision is the trivial subdivision (where all points lie on the same plane in the $(\degree F,\alpha)$-space). The second layer (row) consists of the almost-triangulations which connect exactly two triangulations in the third layer (row) by a \textit{flip} (indicated by the lines). A flip \cite[Section 2.4]{de2010a} is a local change that transforms one triangulation into another. ( In general, the numbers of rows and columns of a refinement poset depend upon the point configuration.) We enumerate the triangulations in the final layer by $1$, $2$, $\ldots$, $5$ as indicated in the figure. Notice that the cases $1$, $4$ and $5$ come in different pairs according to whether $(0,0)$ is labelled as $\md_2=(-1,0)$ ($H=$horizontal) or as $\md_4=(0,-1)$ ($V=$vertical), cf. item \ref{gp2} of \lemmaref{triangulation}, see \defnref{generalposition}. We write $C^H$ and $C^V$, $C=1,4,5,$ to separate these cases further. In the following, we consider the cases separately.

\begin{figure}
    \centering
        \includegraphics[width=0.45\linewidth]{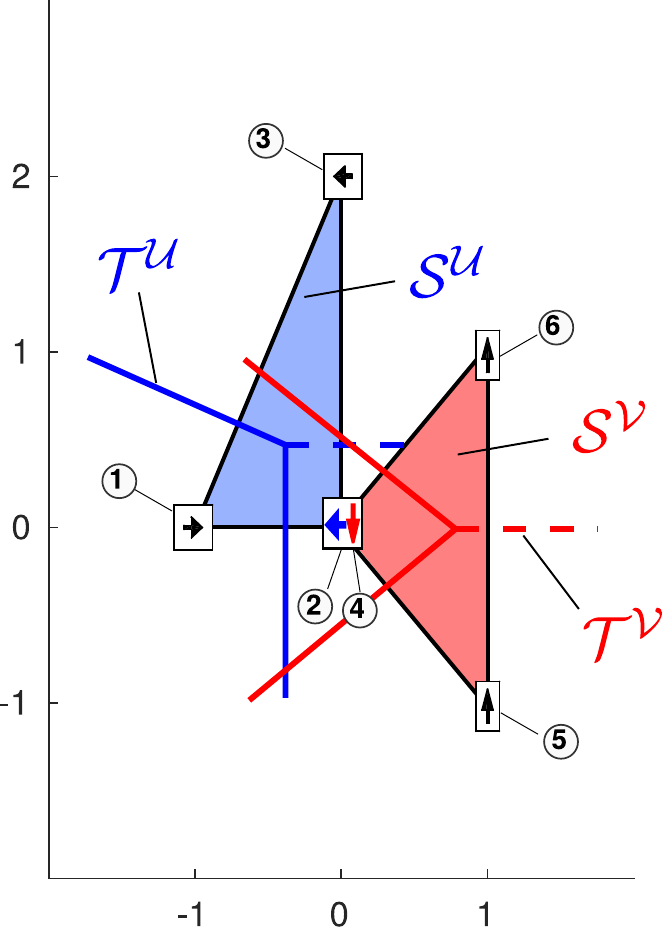}
        \caption{The subdivisions $\mathcal S^{\mathcal U}$ and $\mathcal S^{\mathcal V}$ and their dual tropical curves $\mathcal T^{\mathcal U}$ and $\mathcal T^{\mathcal V}$, respectively. Only the full lines correspond to nullcline sliding. }
\figlab{SUV}
\end{figure}

\begin{figure}
    \centering
        \includegraphics[width=1.0\linewidth]{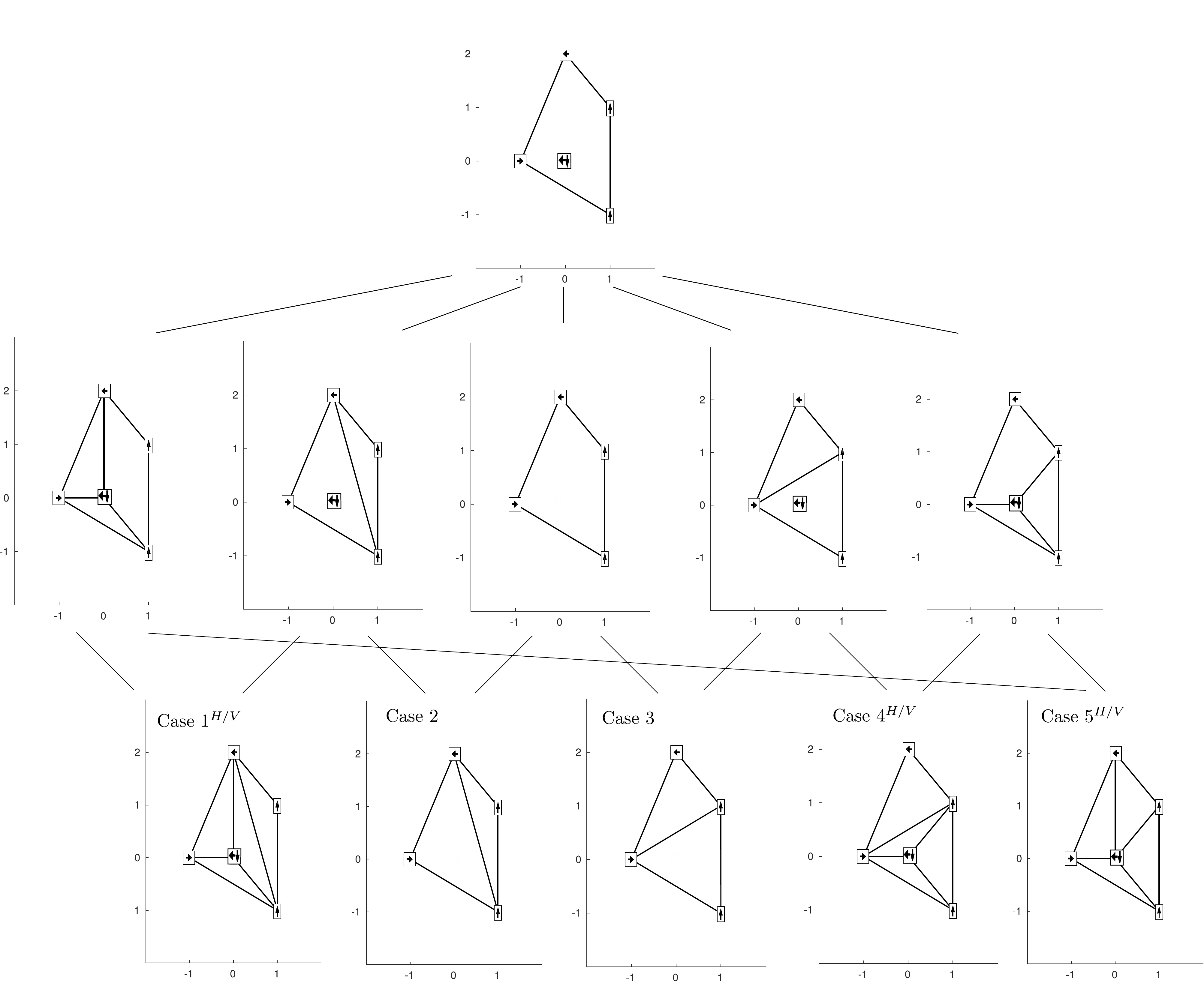}
        \caption{Refinement poset of the point configuration associated with the generalized tropical autocatalator model, see \eqref{tropgenauto}.}
\figlab{subgenauto}
\end{figure}
\subsection{Case $1^H$}
We illustrate the labelled subdivision $\mathcal S^{\mathcal I}$ and the associate tropical curves $\mathcal T^{\mathcal U}$ (blue), $\mathcal T^{\mathcal V}$ (red), and $\mathcal T^{\mathcal I}$ (black) in \figref{tropauto1H}. We also illustrate $\mathcal T^{\mathcal U}$ and $\mathcal T^{\mathcal V}$ and indicate that subsets of each of the \rsp{tropical edges} of $\mathcal T^{\mathcal U}$ belong to $\mathcal T^{\mathcal I}$. On the other hand, only one \rsp{tropical edge} of $\mathcal T^{\mathcal V}$ ($\mathcal E_{56}$) is a subset of $\mathcal T^{\mathcal I}$. This leads to three distinct subcases of case $1^H$ where $\rspp{(\mathcal T,\mathcal T^{\mathcal U},\mathcal T^{\mathcal V})}$ in general position, cf. item \ref{gp3} of \lemmaref{triangulation}, see \defnref{generalposition}. These are indicated in the figure. Here we use that $\mathcal T^{\mathcal V}$ cannot intersect $\mathcal E_{35}$ without changing the subdivision of $\mathcal S^{\mathcal I}$. This should be clear enough, see also \figref{SUV}.  However, $1^H_c\sim 1^H_b$, since the separatrix connection in between does not generate new equivalence classes, and there are therefore only two (up to equivalence) distinct phase portraits. We provide examples of these  in \figref{genauto11H} (parameters are indicated in the caption).  
\begin{figure}
    \centering
        \includegraphics[width=0.75\linewidth]{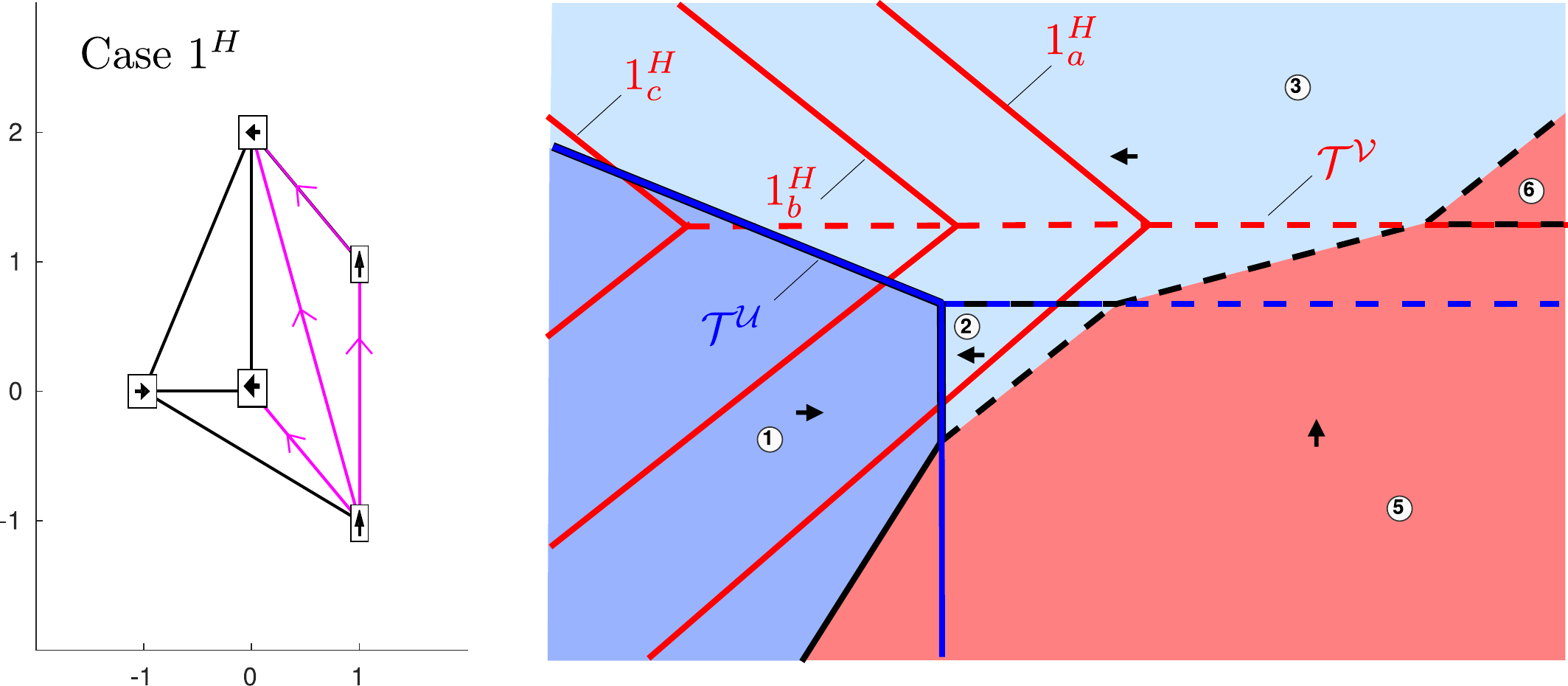}
        \caption{The subdivision associated with case $1^H$ and the associated tropical curves. There are three different cases of $\rspp{(\mathcal T,\mathcal T^{\mathcal U},\mathcal T^{\mathcal V})}$ in general position. Subsets of each of the \rsp{tropical edges} of $\mathcal T^{\mathcal U}$ (blue) belong to $\mathcal T^{\mathcal I}$ (in black) as indicated, whereas only one \rsp{tropical edge} of $\mathcal T^{\mathcal V}$ ($\mathcal E_{56}$ in red) is a subset of $\mathcal T^{\mathcal I}$. The separatrix connection, that occurs between $1^H_b$ and $1^H_c$, does not generate new orbit equivalence classes.}
\figlab{tropauto1H}
\end{figure}

\begin{figure}
\centering
   \begin{subfigure}{0.495\textwidth}
    \centering
        \includegraphics[width=0.96\linewidth]{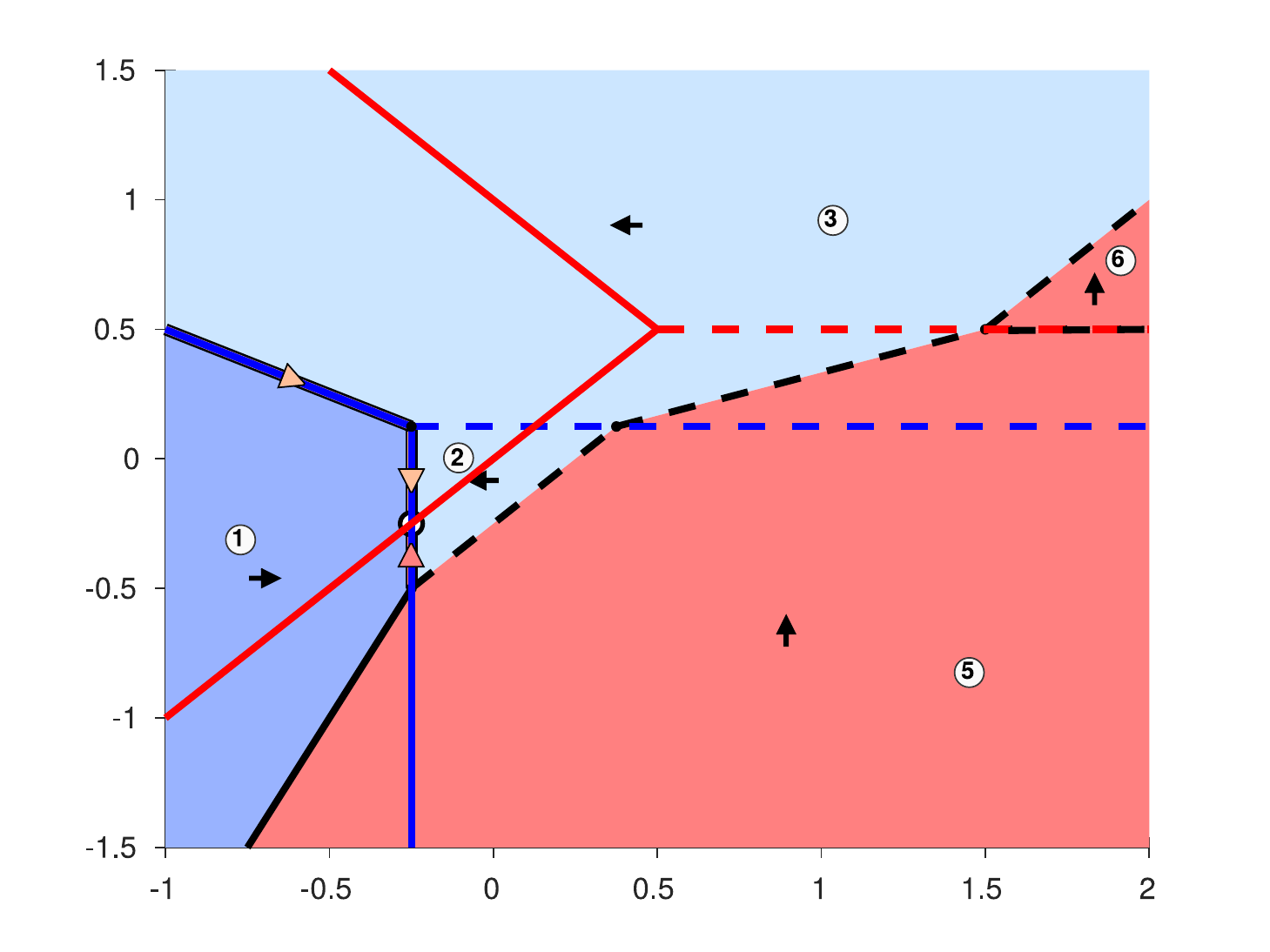}
        \caption{$1^H_a$}
    \end{subfigure}%
    \begin{subfigure}{0.495\textwidth}
    \centering
        \includegraphics[width=0.96\linewidth]{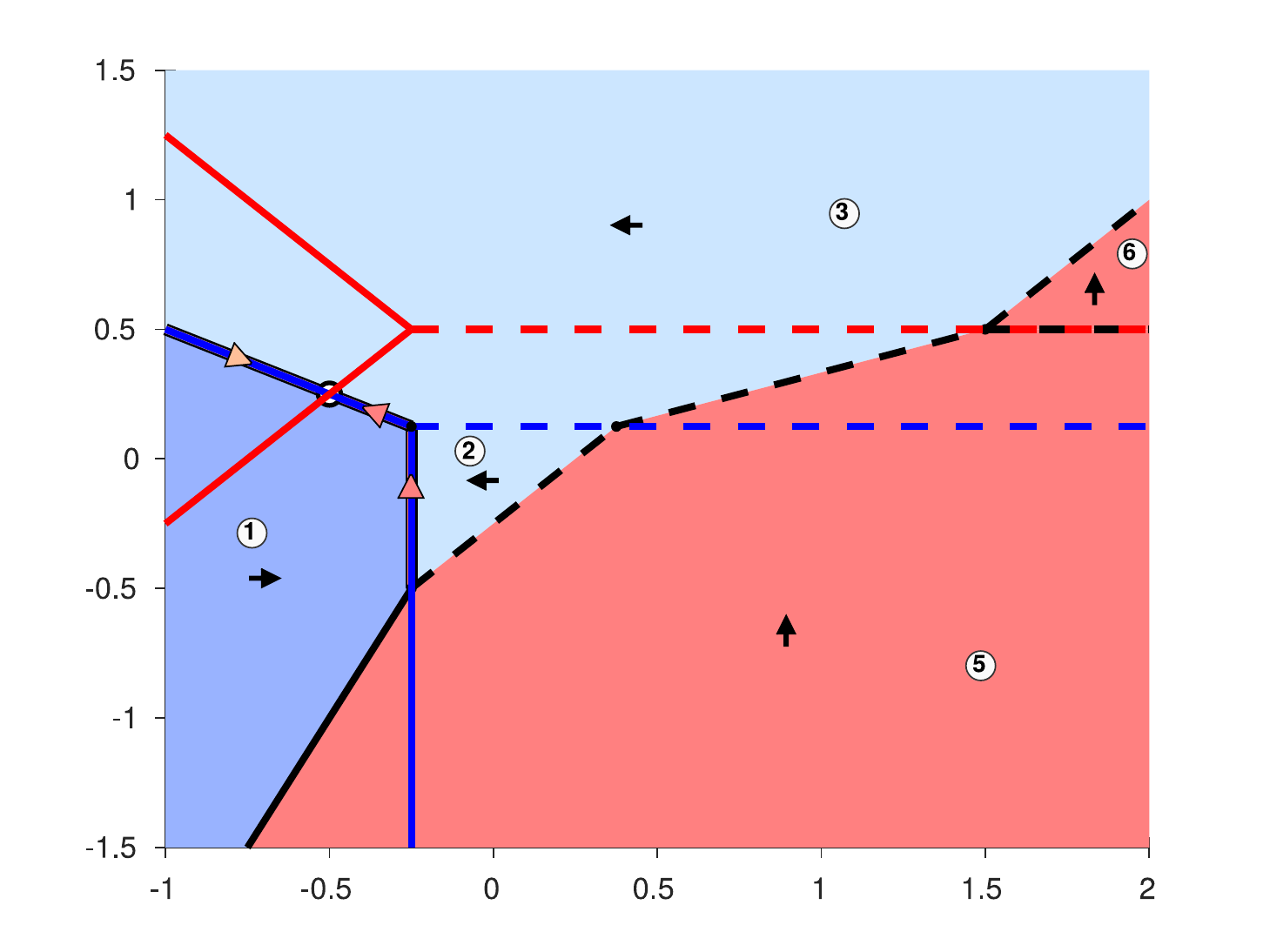}
        \caption{$1^H_b$}
    \end{subfigure}
        \caption{Two nonequivalent phase portraits $1^H_a$ (a) and $1^H_b$ (b) for the case $1^H$, see \figref{tropauto1H}. $\mathcal T^{\mathcal U}$ and $\mathcal T^{\mathcal V}$ are blue and red, whereas $\mathcal T^{\mathcal I}$ is in black. Full lines indicate sliding. The values of the tropical coefficients in the examples are: $\alpha_1 = 0$, $\alpha_2=0.25$, $\alpha_3=0$, $\alpha_5=0$, $\alpha_6=-1$ and $\alpha_4 = 0$ in (a) and $\alpha_4= -0.75$ in (b). }
\figlab{genauto11H}
\end{figure}

\subsection{Case $1^V$}
We illustrate the labelled subdivision $\mathcal S^{\mathcal I}$ and the associate tropical curves $\mathcal T^{\mathcal U}$ (blue), $\mathcal T^{\mathcal V}$ (red), and $\mathcal T^{\mathcal I}$ (black) in \figref{tropauto1V} (similar to \figref{tropauto1H}). There are two different subcases of case $1^V$ with $\rspp{(\mathcal T,\mathcal T^{\mathcal U},\mathcal T^{\mathcal V})}$ in general position. These corresponds to two different intersections of $\mathcal T^{\mathcal U}$ (blue) and $\mathcal T^{\mathcal V}$ (red) along $\mathcal E_{45}$, cf. item \ref{gp3} of \lemmaref{triangulation}, see \defnref{generalposition}. However, the phase portraits are clearly equivalent. We illustrate an example of the single phase portrait (up to equivalence) associated to case $1^V$ in \figref{genauto11V} (see parameters in the figure caption). 

\begin{figure}
    \centering
        \includegraphics[width=0.75\linewidth]{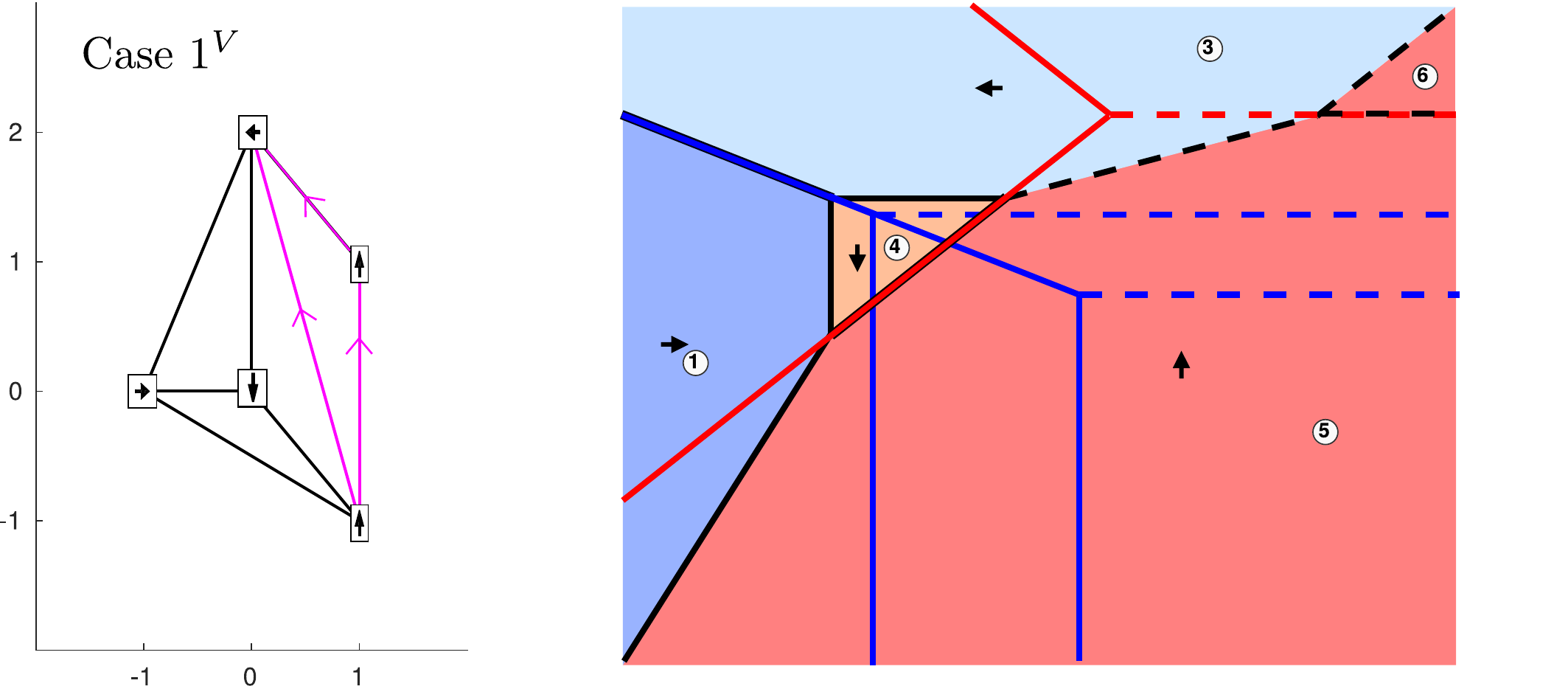}
        \caption{The subdivision associated with case $1^V$ and the associated tropical curves. There are two different cases of $\rspp{(\mathcal T,\mathcal T^{\mathcal U},\mathcal T^{\mathcal V})}$ in general position.  }
\figlab{tropauto1V}
\end{figure}

\begin{figure}
    \centering
        \includegraphics[width=0.5\linewidth]{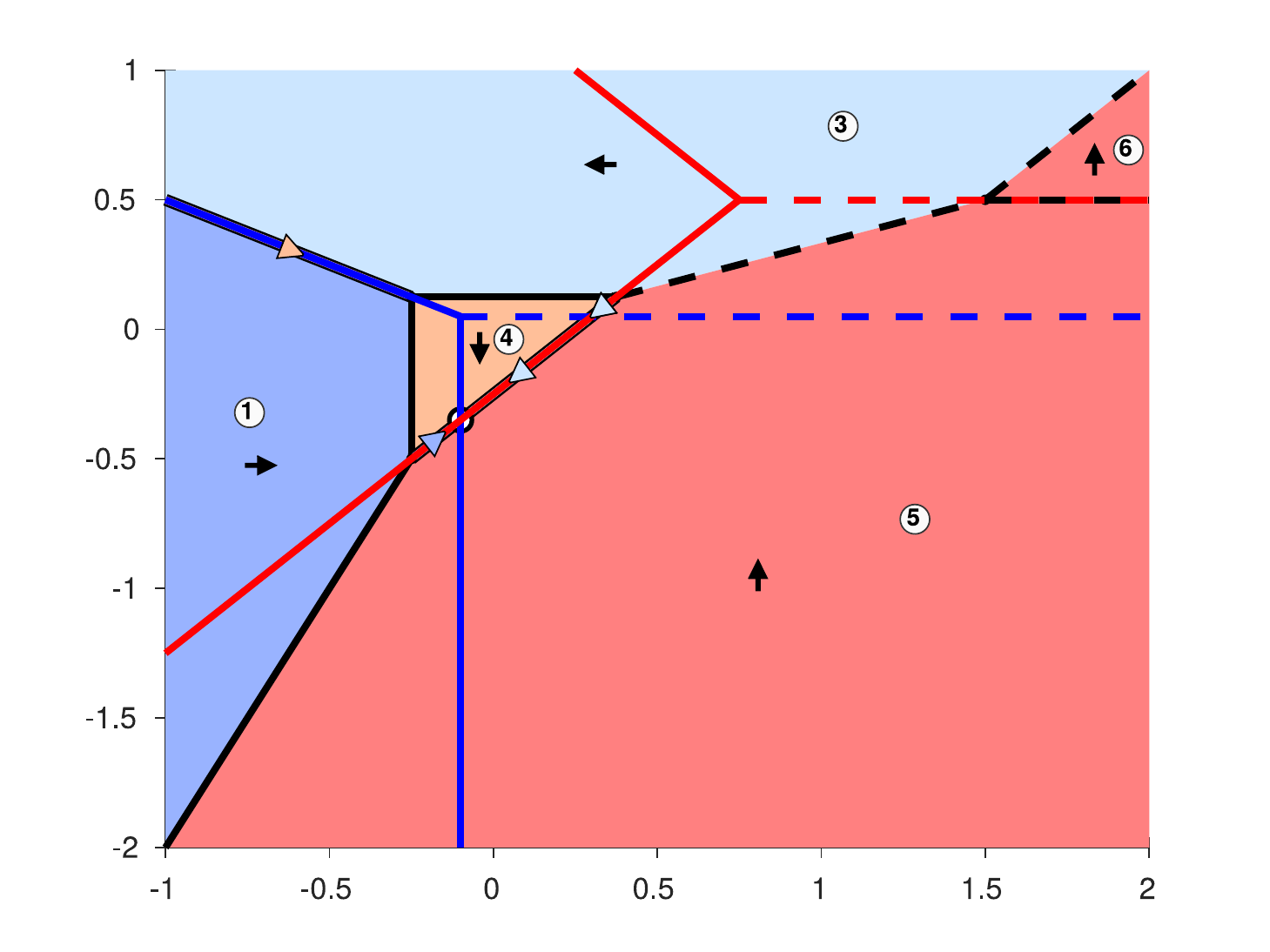}
        \caption{Phase portrait in the case $1^V$, see \figref{tropauto1V}. The values of the tropical coefficients are: $\alpha_1 = 0$, $\alpha_2=0.1$, $\alpha_3=0$, $\alpha_4=0.25$, $\alpha_5=0$ and $\alpha_6 = -1$. }
\figlab{genauto11V}
\end{figure}
\subsection{Case $2$}
We illustrate the labelled subdivision $\mathcal S^{\mathcal I}$ and the associate tropical curves $\mathcal T^{\mathcal U}$ (blue), $\mathcal T^{\mathcal V}$ (red), and $\mathcal T^{\mathcal I}$ (black) in \figref{tropauto22}. There is only one case of $\rspp{(\mathcal T,\mathcal T^{\mathcal U},\mathcal T^{\mathcal V})}$ being in general position. Here we again use that $\mathcal T^{\mathcal V}$ cannot intersect $\mathcal E_{35}$ without changing the subdivision $\mathcal S^{\mathcal I}$. We illustrate an example of the single phase portrait (up to equivalence) associated to case $2$ in \figref{genauto12} (see parameters in the figure caption). 
\begin{figure}
    \centering
        \includegraphics[width=0.75\linewidth]{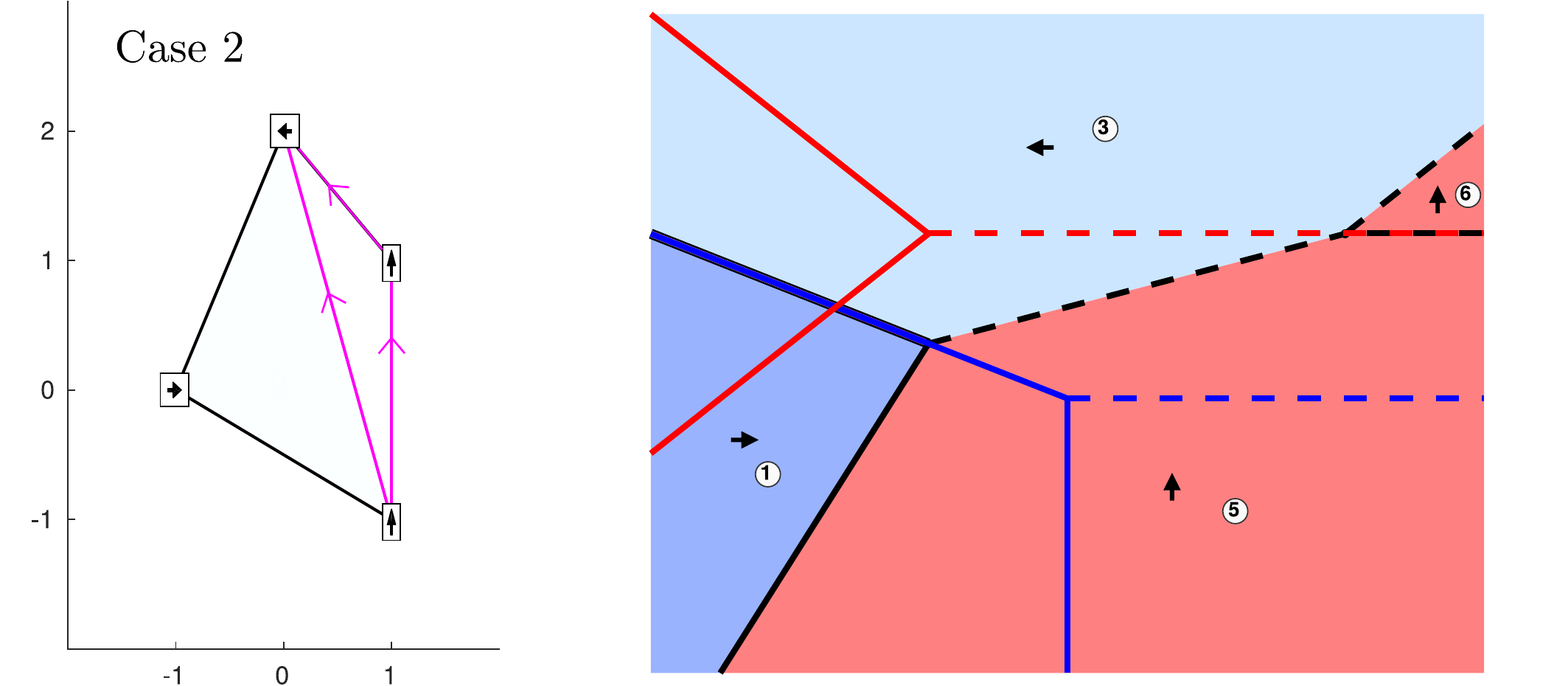}
        \caption{The subdivision associated with case $2$ and the associated tropical curves.  }
\figlab{tropauto22}
\end{figure}

\begin{figure}
    \centering
        \includegraphics[width=0.5\linewidth]{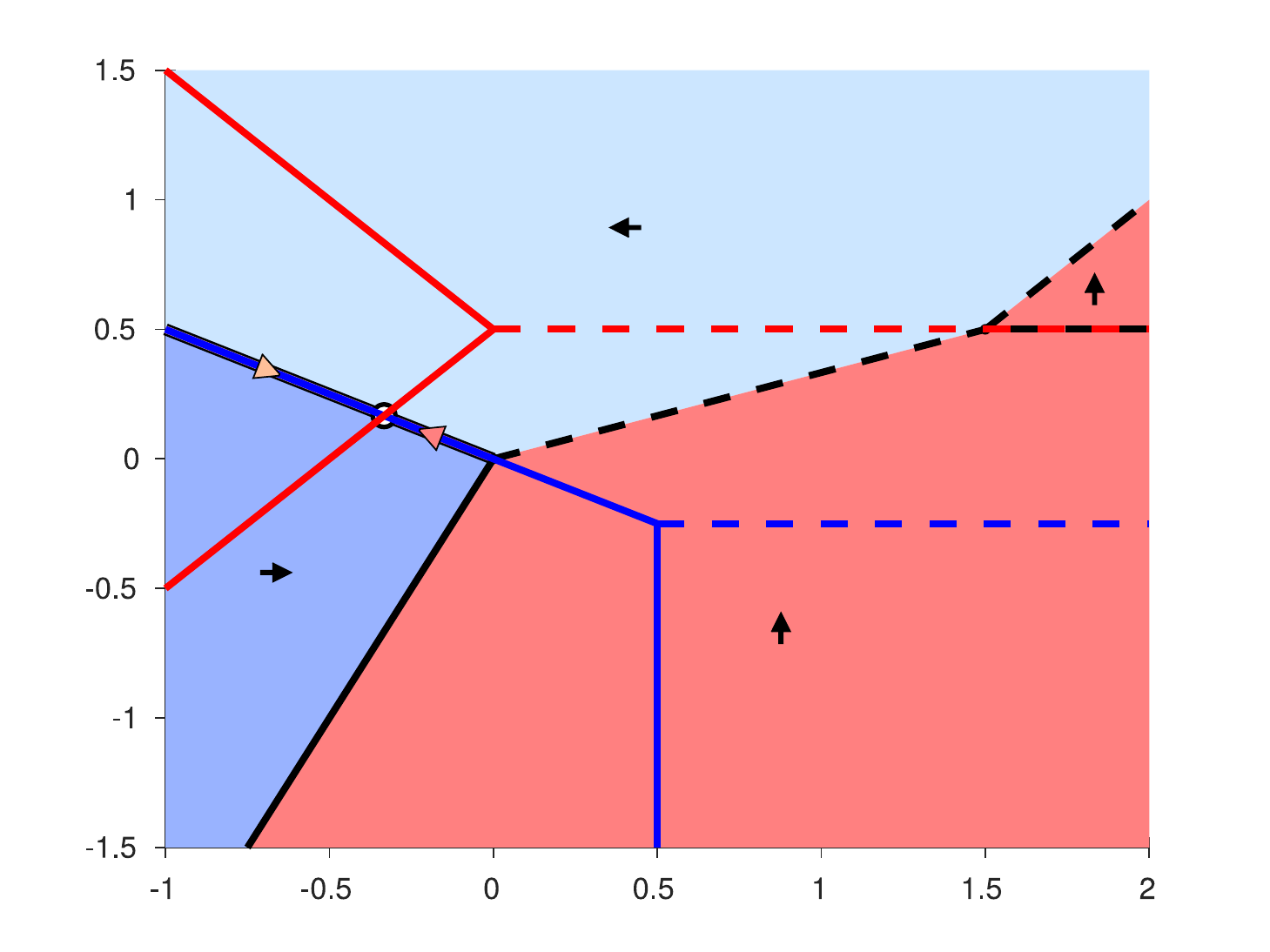}
        \caption{Phase portrait in the case $2$, see \figref{tropauto22}. The values of the tropical coefficients are: $\alpha_1 = 0$, $\alpha_2=-0.5$, $\alpha_3=0$, $\alpha_4=-0.5$, $\alpha_5=0$ and $\alpha_6 = -1$. }
\figlab{genauto12}
\end{figure}

\subsection{Case $3$}
We illustrate the labelled subdivision $\mathcal S^{\mathcal I}$ and the associate tropical curves $\mathcal T^{\mathcal U}$ (blue), $\mathcal T^{\mathcal V}$ (red), and $\mathcal T^{\mathcal I}$ (black) in \figref{tropauto33}. There is only one case of $\rspp{(\mathcal T,\mathcal T^{\mathcal U},\mathcal T^{\mathcal V})}$ being in general position. Here we again use that $\mathcal T^{\mathcal V}$ cannot intersect $\mathcal E_{35}$ and $\mathcal E_{36}$ without changing the subdivision $\mathcal S^{\mathcal I}$. We illustrate an example of the single phase portrait (up to equivalence) associated to case $3$ in \figref{genauto13} (see parameters in the figure caption). 
\begin{figure}
    \centering
        \includegraphics[width=0.75\linewidth]{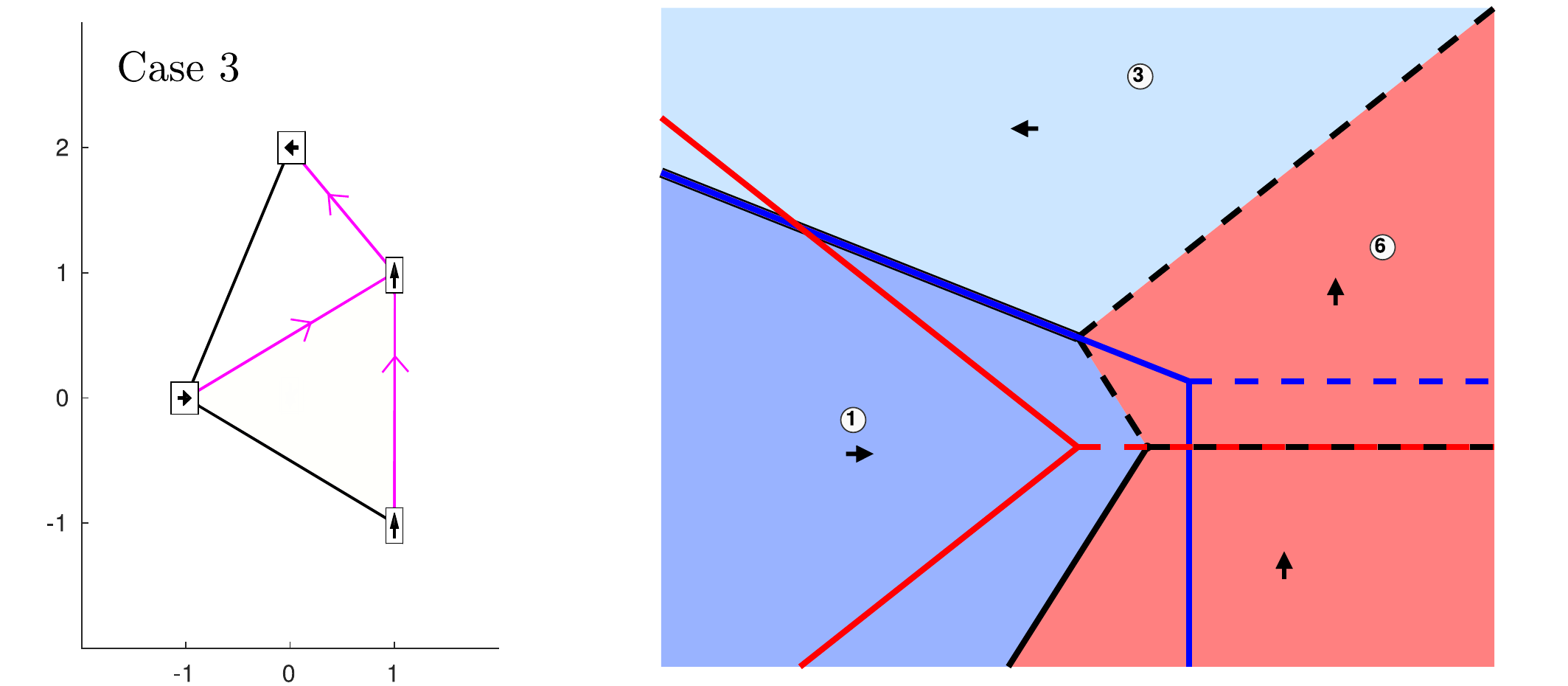}
        \caption{The subdivision associated with case $3$ and the associated tropical curves.  }
\figlab{tropauto33}
\end{figure}

\begin{figure}
    \centering
        \includegraphics[width=0.5\linewidth]{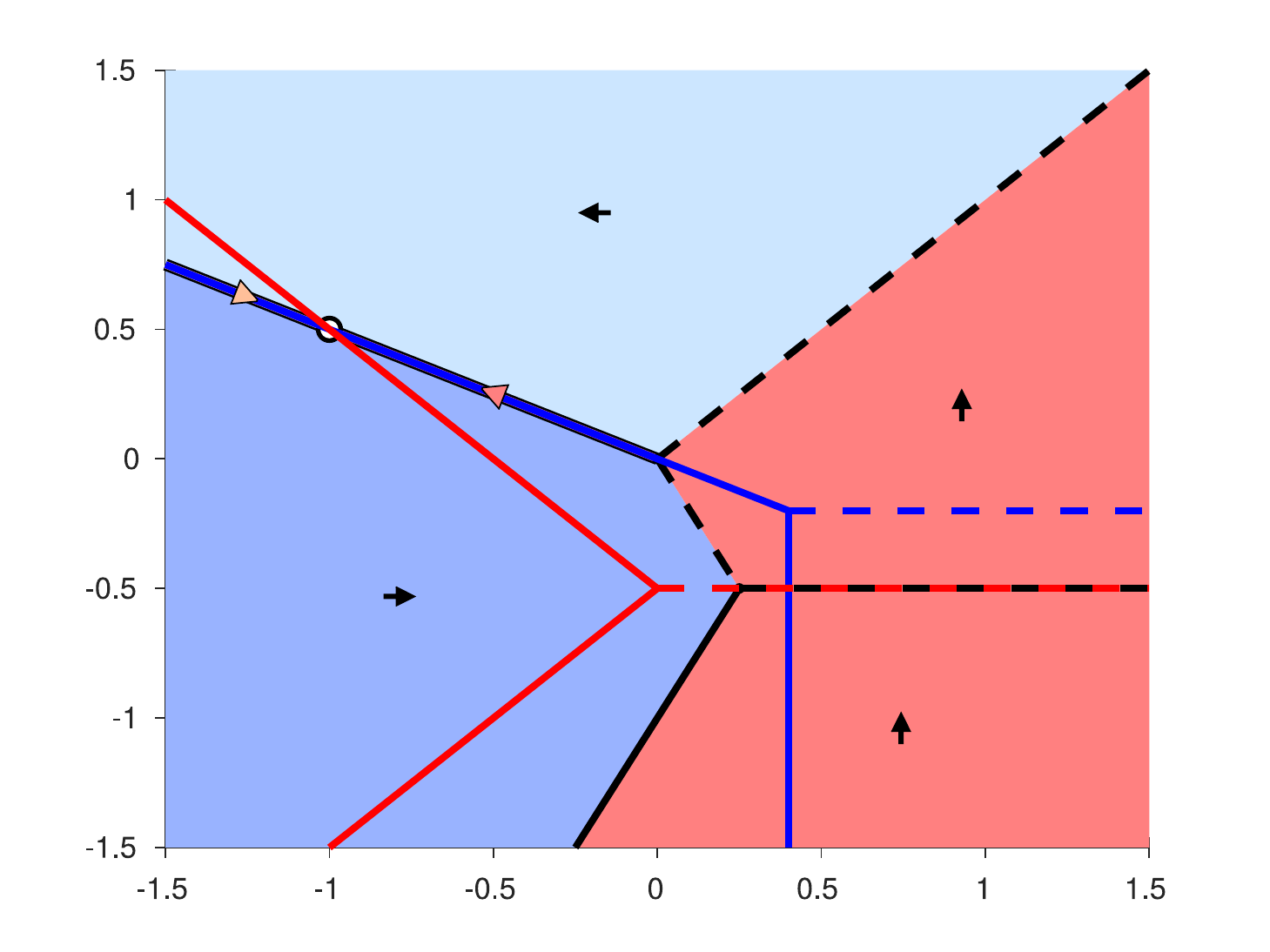}
        \caption{Phase portrait in the case $3$, see \figref{tropauto33}. The values of the tropical coefficients are: $\alpha_1 = 0$, $\alpha_2=-0.4$, $\alpha_3=0$, $\alpha_4=-0.5$, $\alpha_5=-1$ and $\alpha_6 = 0$. }
\figlab{genauto13}
\end{figure}

\subsection{Case $4^H$}
We illustrate the labelled subdivision $\mathcal S^{\mathcal I}$ and the associate tropical curves $\mathcal T^{\mathcal U}$ (blue), $\mathcal T^{\mathcal V}$ (red), and $\mathcal T^{\mathcal I}$ (black) in \figref{tropauto4H}. There are two cases of $\rspp{(\mathcal T,\mathcal T^{\mathcal U},\mathcal T^{\mathcal V})}$ being in general position. We denote these by $4^H_a$ and $4^H_b$. There is a bifurcation in between, where two singularities along $\mathcal E_{12}$ coalesce. Interestingly, this bifurcation is also related to another global bifurcation of a separatrix connection (indicated in black). We illustrate  examples of the two phase portraits (up to equivalence) associated to case $4^H$ in \figref{genauto14H} (see parameters in the figure caption). 
\begin{figure}
    \centering
        \includegraphics[width=0.75\linewidth]{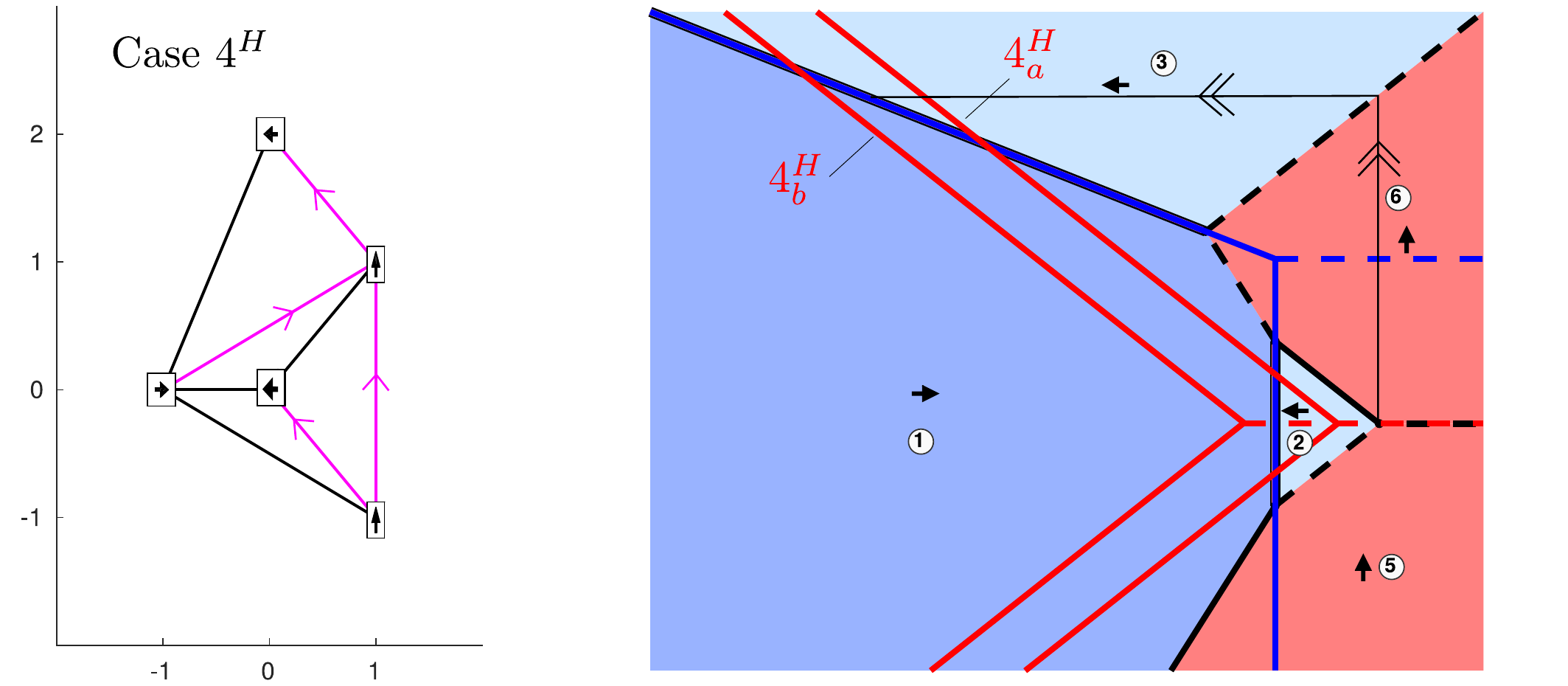}
        \caption{The subdivision associated with case $4^H$ and the associated tropical curves.  }
\figlab{tropauto4H}
\end{figure}

\begin{figure}
\centering
   \begin{subfigure}{0.5\textwidth}
    \centering
        \includegraphics[width=0.96\linewidth]{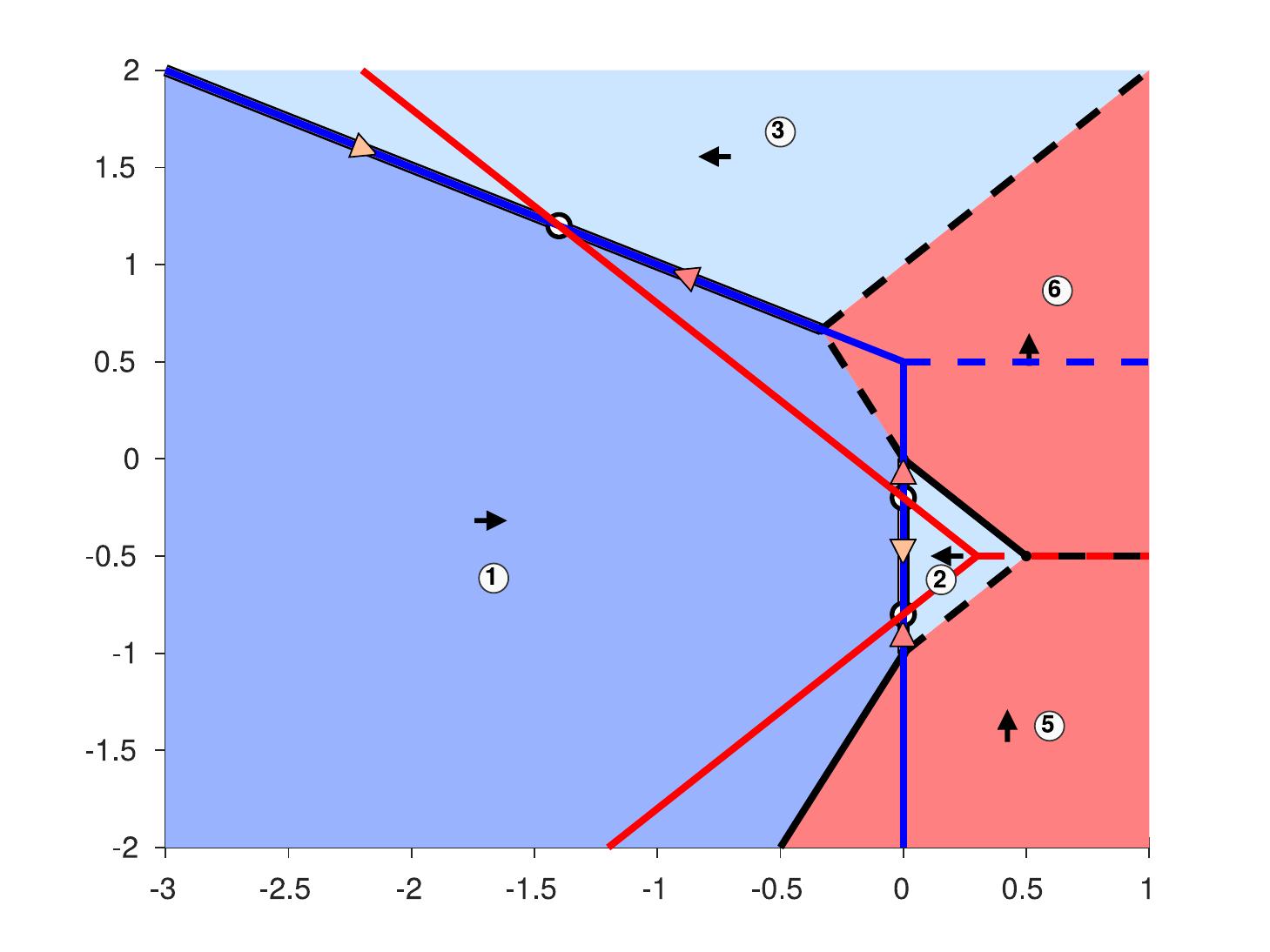}
        \caption{$4^H_a$}
    \end{subfigure}%
    \begin{subfigure}{0.5\textwidth}
    \centering
        \includegraphics[width=0.96\linewidth]{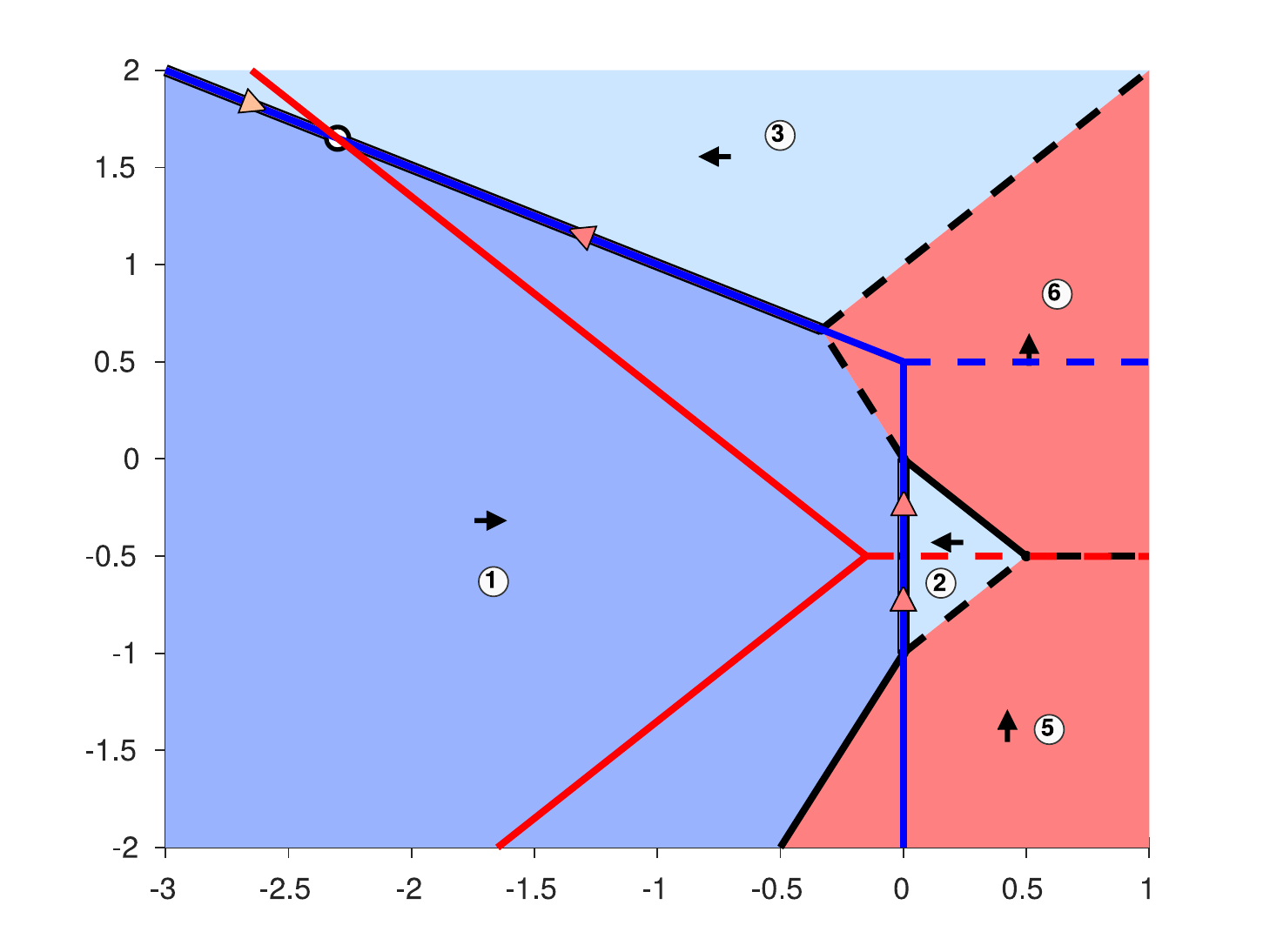}
        \caption{$4^H_b$}
    \end{subfigure}
        \caption{Phase portraits in the case $4^H$, see \figref{tropauto4H}. The values of the tropical coefficients are: $\alpha_1 = 0$, $\alpha_2=0$, $\alpha_3=-1$, $\alpha_5=-1$ and $\alpha_6 = 0$ and $\alpha_4=-0.2$ in (a) and $\alpha_4=-0.65$ in (b). }
\figlab{genauto14H}
\end{figure}

\subsection{Case $4^V$}
We illustrate the labelled subdivision $\mathcal S^{\mathcal I}$ and the associate tropical curves $\mathcal T^{\mathcal U}$ (blue), $\mathcal T^{\mathcal V}$ (red), and $\mathcal T^{\mathcal I}$ (black) in \figref{tropauto4V}. There are two cases of $\rspp{(\mathcal T,\mathcal T^{\mathcal U},\mathcal T^{\mathcal V})}$ being in general position. We denote these by $4^V_a$ and $4^V_b$. There is a bifurcation in between, where two singularities along $\mathcal E_{12}^{\mathcal U}$ coalesce. We illustrate  examples of the two  structurally stable phase portraits (up to equivalence) associated to case $4^V$ in \figref{genauto14V} (see parameters in the figure caption). 
\begin{figure}
    \centering
        \includegraphics[width=0.75\linewidth]{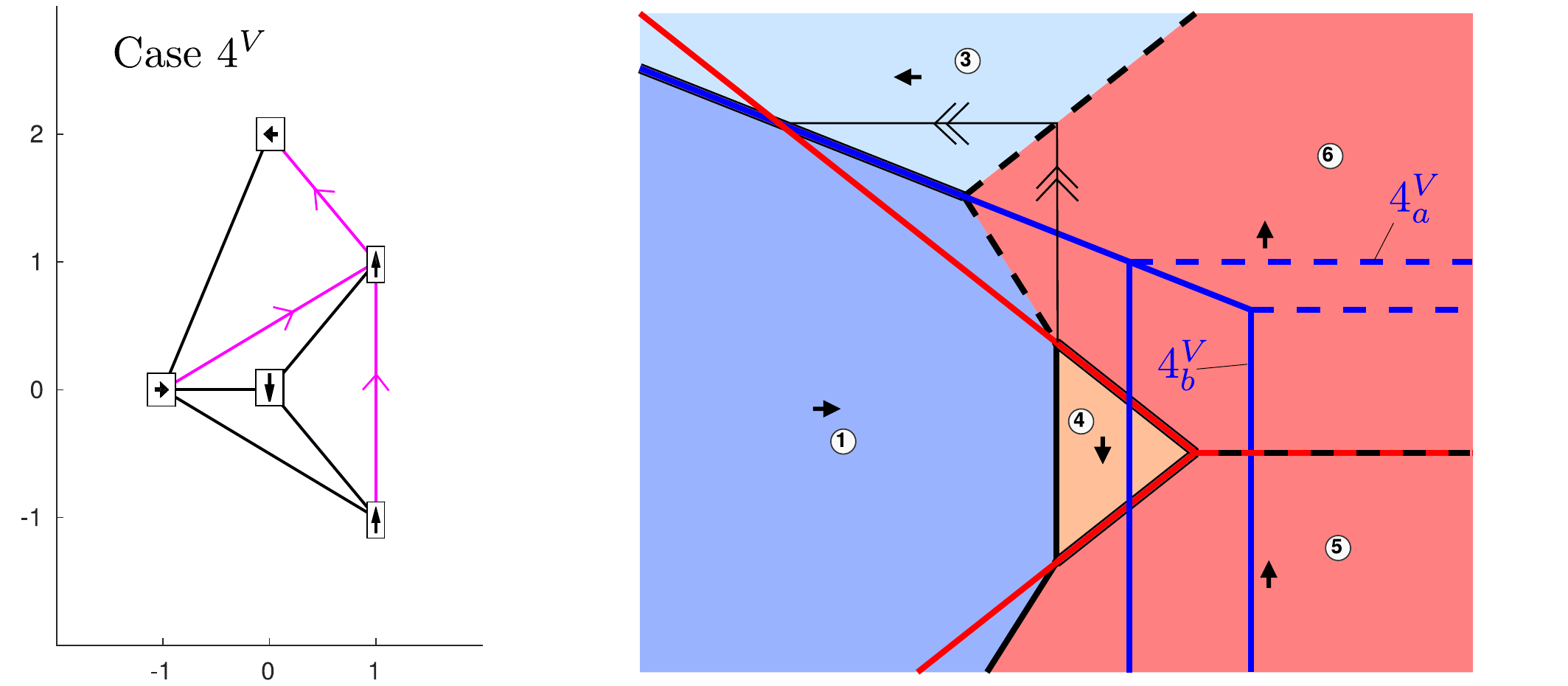}
        \caption{The subdivision associated with case $4^V$ and the associated tropical curves.  }
\figlab{tropauto4V}
\end{figure}

\begin{figure}
\centering
   \begin{subfigure}{0.5\textwidth}
    \centering
        \includegraphics[width=0.96\linewidth]{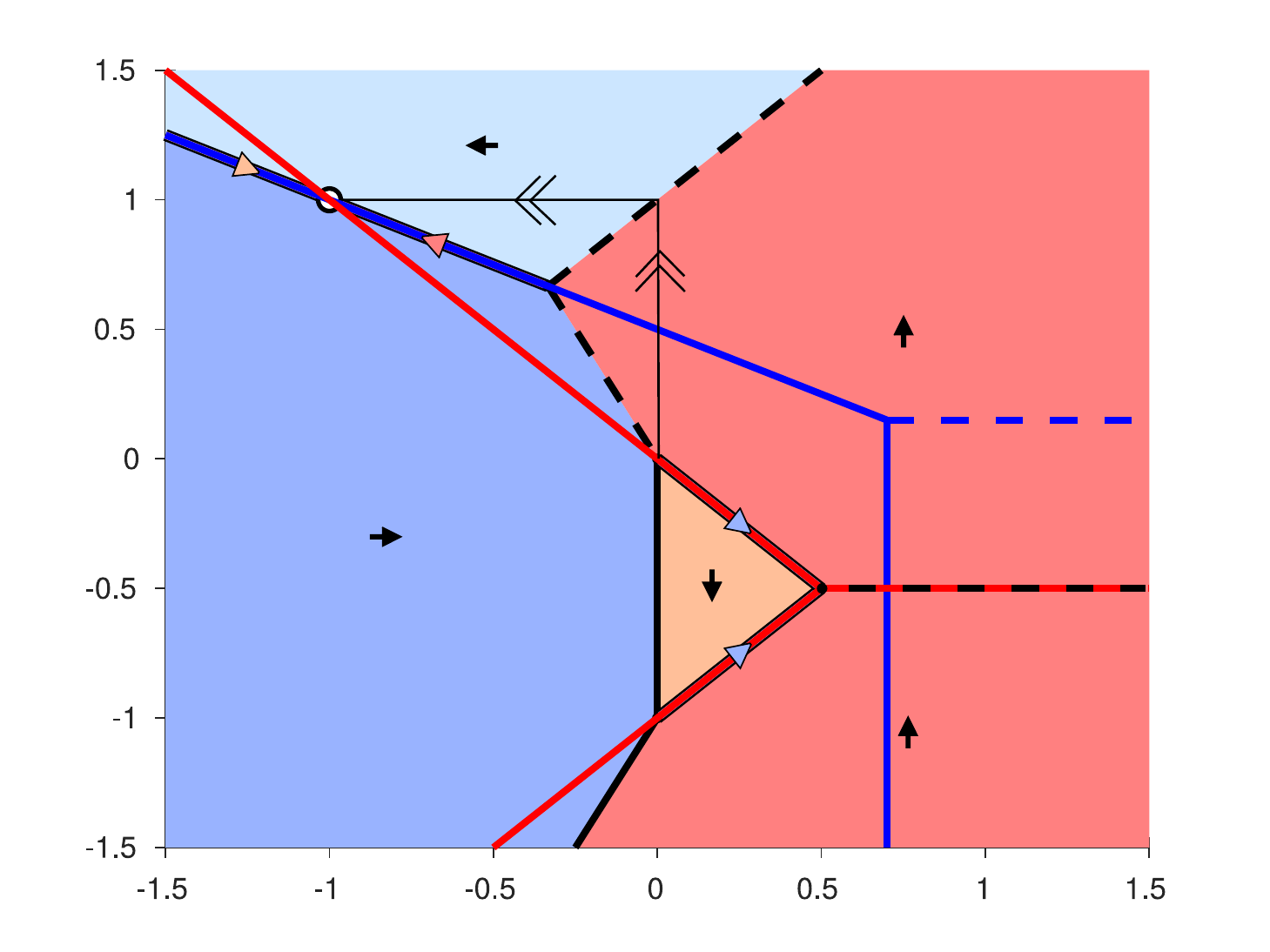}
        \caption{$4^V_a$}
    \end{subfigure}%
    \begin{subfigure}{0.5\textwidth}
    \centering
        \includegraphics[width=0.96\linewidth]{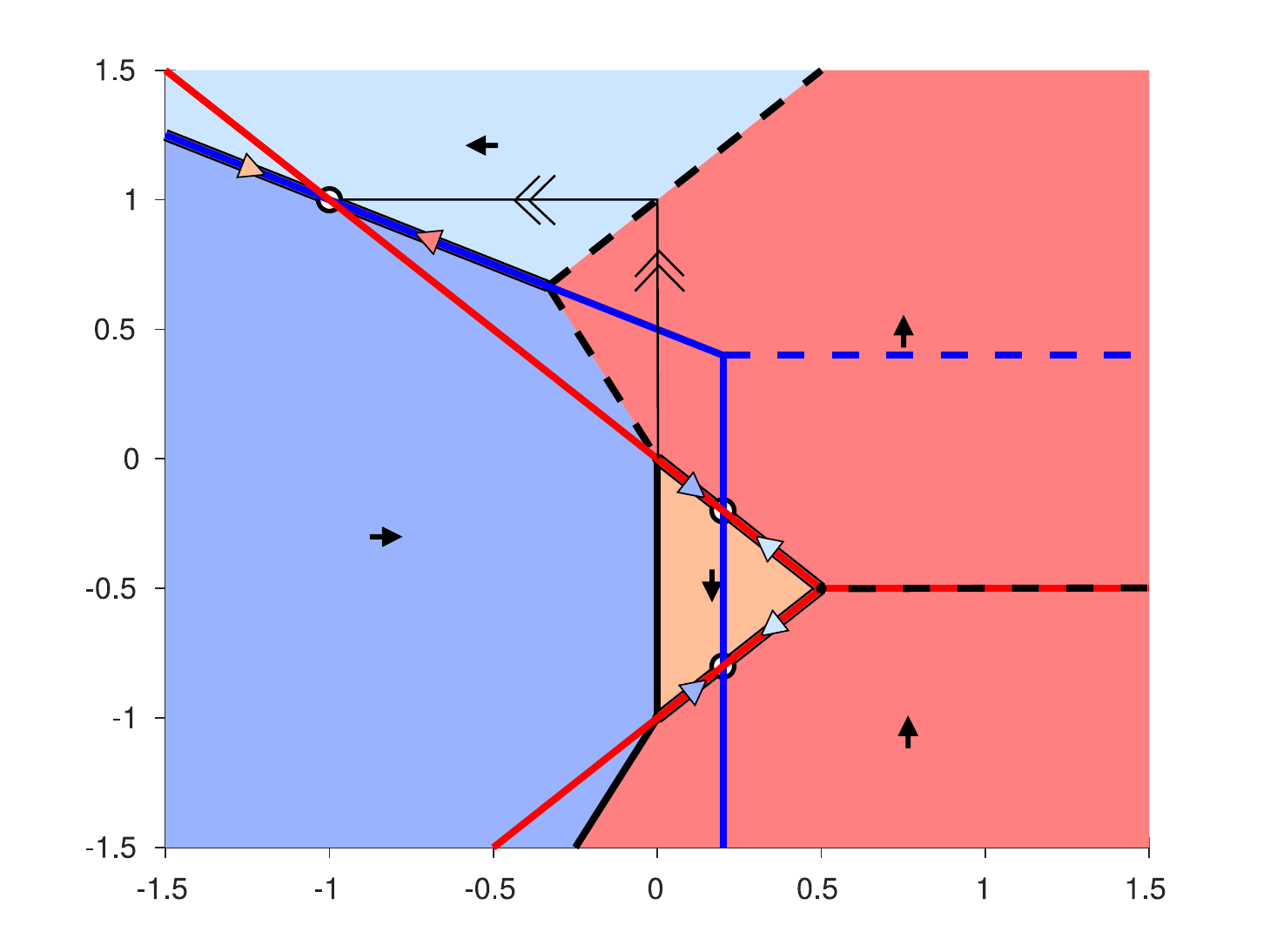}
        \caption{$4^V_b$}
    \end{subfigure}
        \caption{Phase portraits in the case $4^V$, see \figref{tropauto4V}. The values of the tropical coefficients are: $\alpha_1 = 0$, $\alpha_3=-1$, $\alpha_4=0$, $\alpha_5=-1$ and $\alpha_6 = 0$ and $\alpha_2=-0.7$ in (a) and $\alpha_4=-0.2$ in (b). }
\figlab{genauto14V}
\end{figure}
\subsection{Case $5^H$}

We illustrate the labelled subdivision $\mathcal S^{\mathcal I}$ and the associate tropical curves $\mathcal T^{\mathcal U}$ (blue), $\mathcal T^{\mathcal V}$ (red), and $\mathcal T^{\mathcal I}$ (black) in \figref{tropauto5H}. There are three cases of $\rspp{(\mathcal T,\mathcal T^{\mathcal U},\mathcal T^{\mathcal V})}$ being in general position. We denote these by $5^H_a$, $5^H_b$ and $5^H_c$. There are bifurcations in between, where two singularities along $\mathcal E_{12}^{\mathcal U}$ coalesce. The structurally unstable case between $5^H_b$ and $5^H_c$ is (again) also associated with a global bifurcation due to a separatrix connection (indicated in black). We illustrate examples of the three structurally stable phase portraits (up to equivalence) associated to case $5^H$ in \figref{genauto15H} (see parameters in the figure caption). 
\begin{figure}
    \centering
        \includegraphics[width=0.75\linewidth]{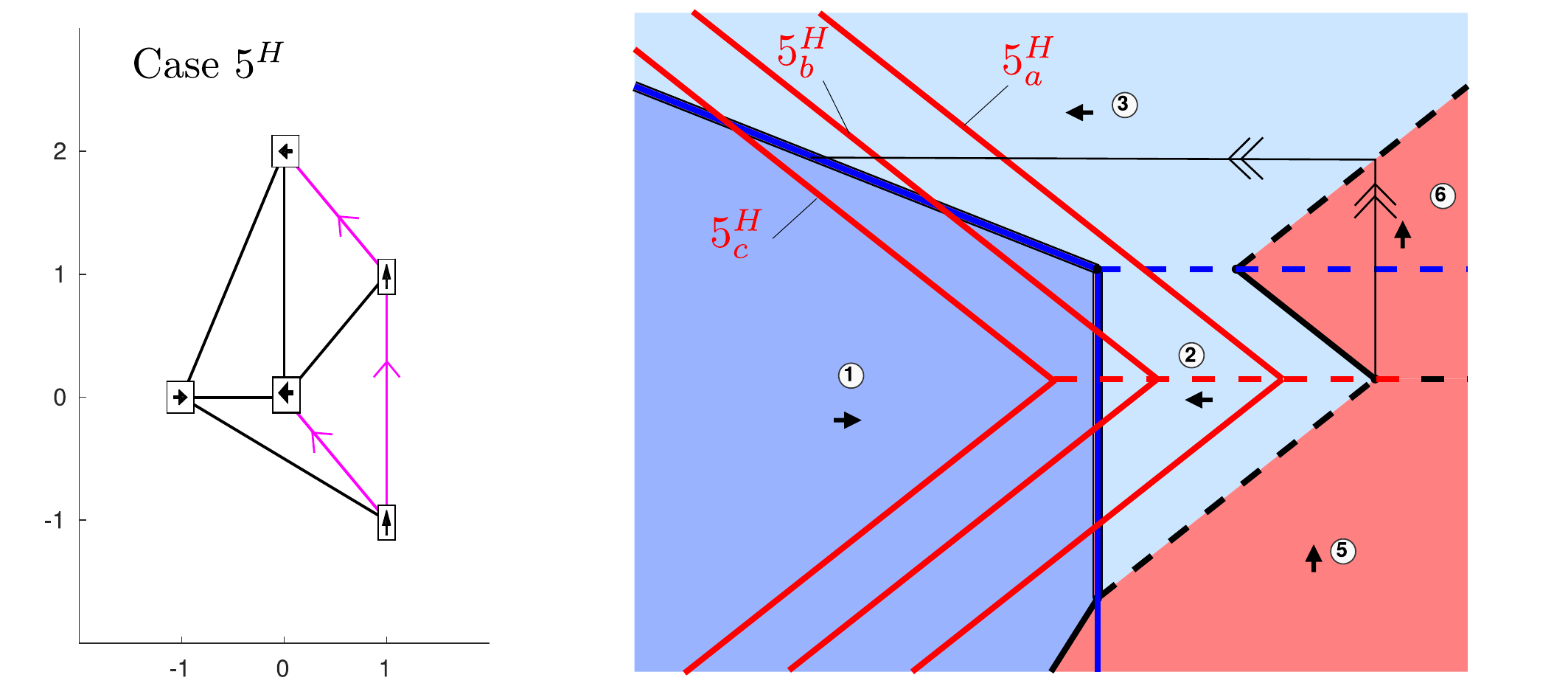}
        \caption{The subdivision associated with case $5^H$ and the associated tropical curves.  }
\figlab{tropauto5H}
\end{figure}

\begin{figure}
\centering
   \begin{subfigure}{0.495\textwidth}
    \centering
        \includegraphics[width=0.96\linewidth]{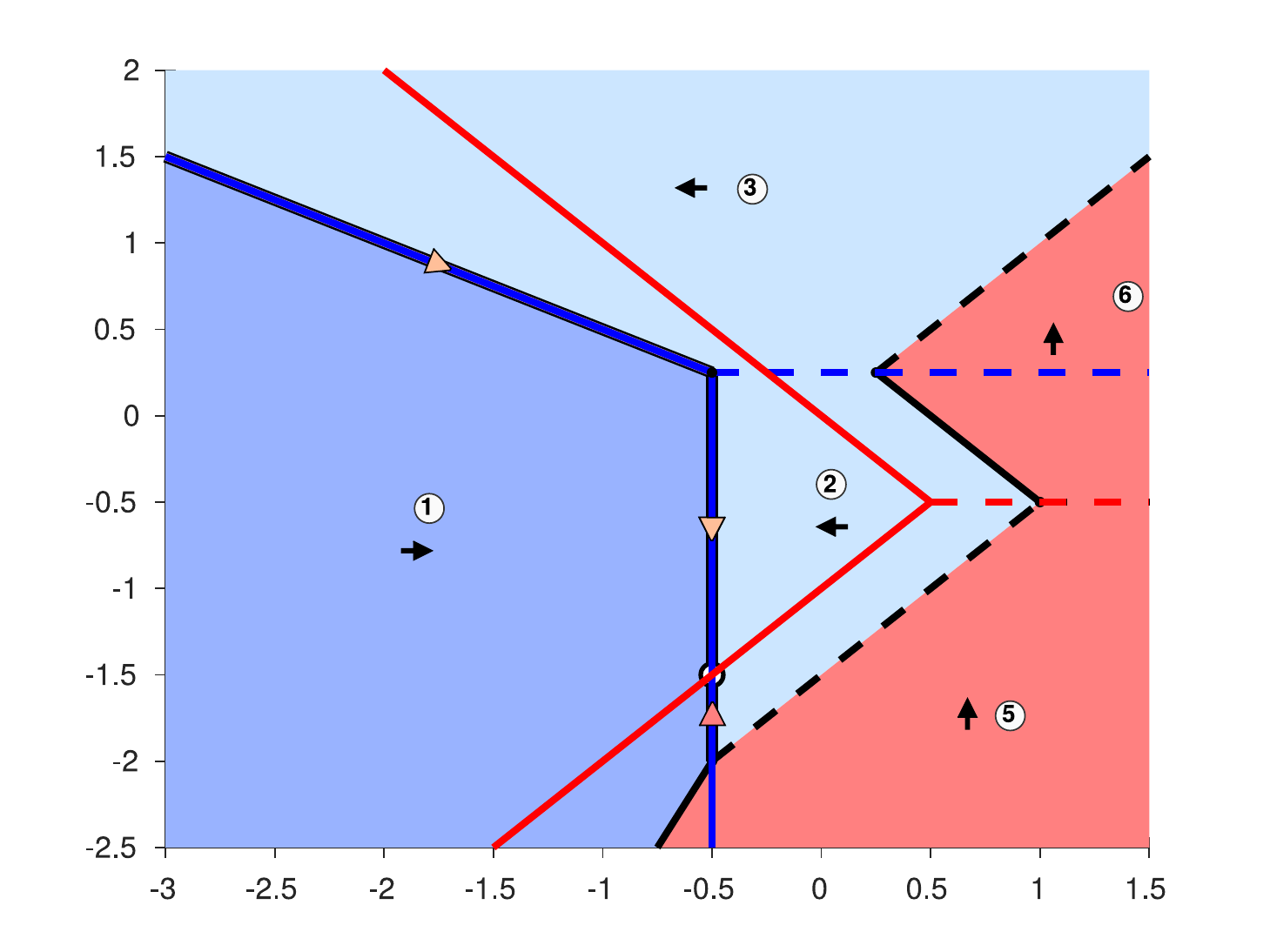}
        \caption{$5^H_a$}
    \end{subfigure}%
    \begin{subfigure}{0.495\textwidth}
    \centering
        \includegraphics[width=0.96\linewidth]{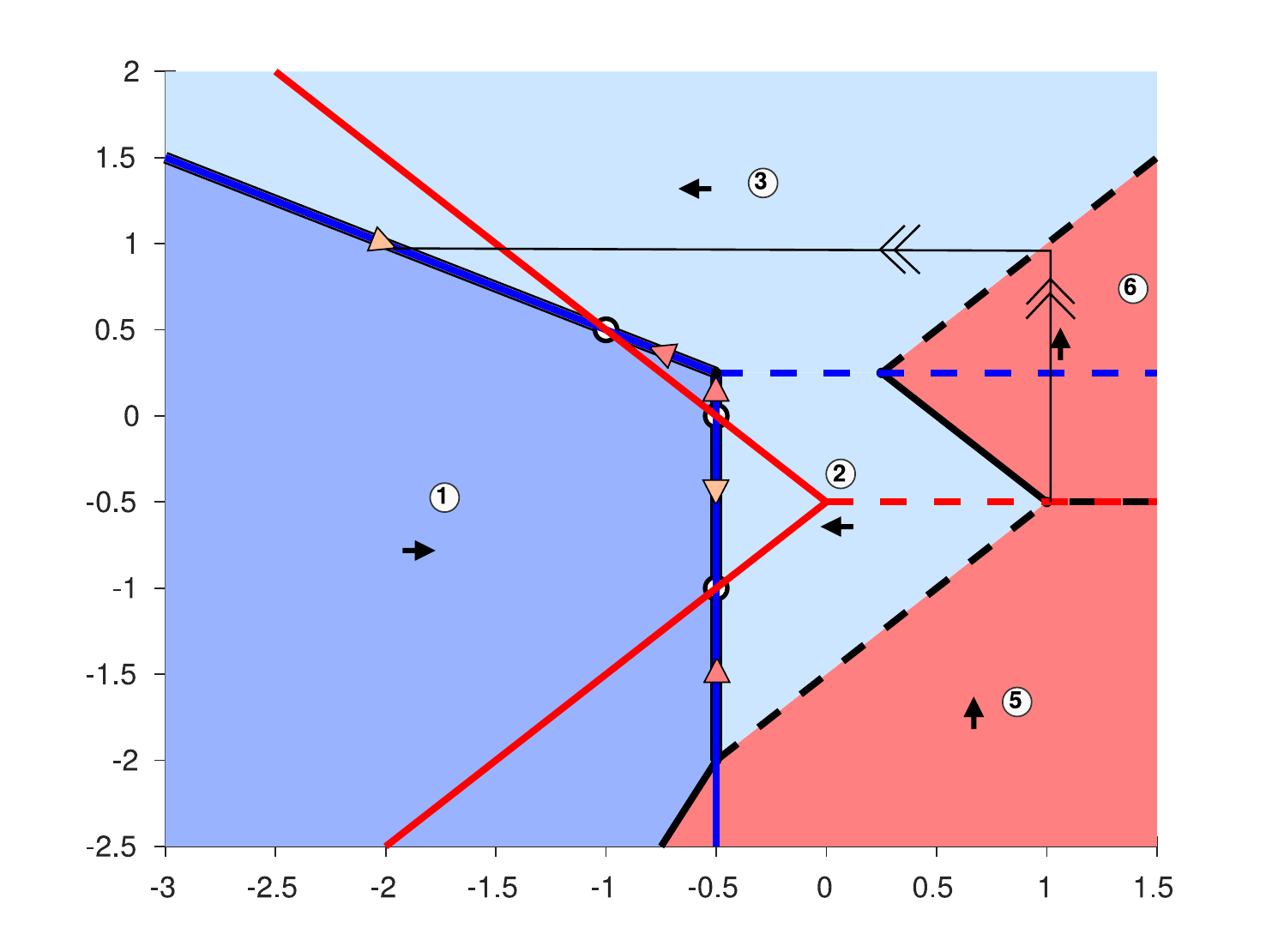}
        \caption{$5^H_b$}
    \end{subfigure}
    \begin{subfigure}{0.495\textwidth}
    \centering
        \includegraphics[width=0.96\linewidth]{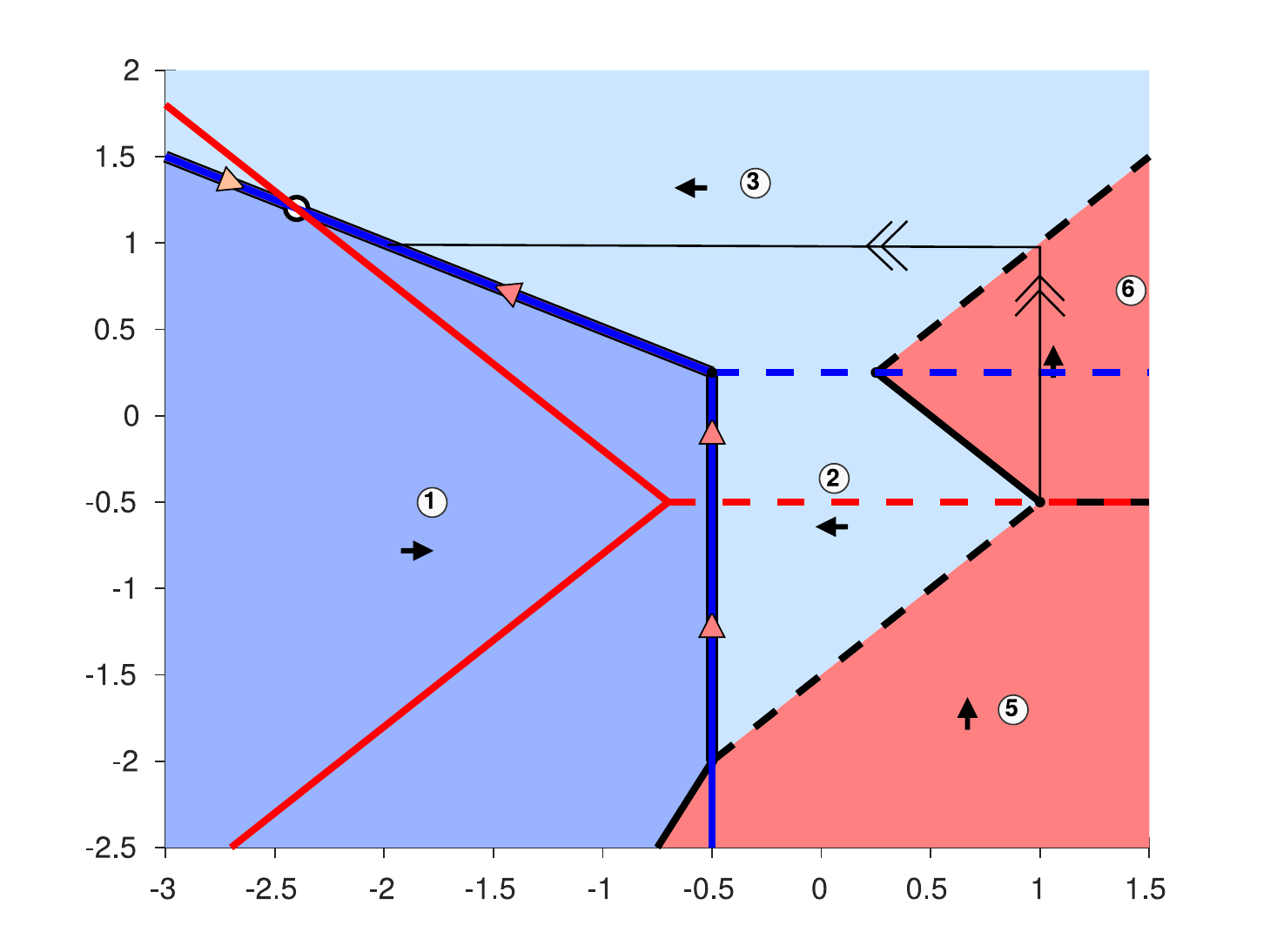}
        \caption{$5^H_c$}
    \end{subfigure}
        \caption{Phase portraits in the case $5^H$, see \figref{tropauto5H}. The values of the tropical coefficients are: $\alpha_1 = 0$, $\alpha_2=0.5$, $\alpha_3=0$, $\alpha_5=-1$ and $\alpha_6 = 0$ and $\alpha_4=0$ in (a), $\alpha_4=-0.5$ in (b)  and $\alpha_4=-1.2$ in (c). }
\figlab{genauto15H}
\end{figure}

\subsection{Case $5^V$}
We illustrate the labelled subdivision $\mathcal S^{\mathcal I}$ and the associate tropical curves $\mathcal T^{\mathcal U}$ (blue), $\mathcal T^{\mathcal V}$ (red), and $\mathcal T^{\mathcal I}$ (black) in \figref{tropauto5V}. There are three cases of $\rspp{(\mathcal T,\mathcal T^{\mathcal U},\mathcal T^{\mathcal V})}$ being in general position. We denote these by $5^V_a$, $5^V_b$ and $5^V_c$. There are bifurcations in between, where two singularities coalesce.  We illustrate examples of the three structurally stable phase portraits (up to equivalence) associated to case $5^V$ in \figref{genauto15V} (see parameters in the figure caption).
There is a limit cycle in case $5^V_c$, see \figref{genauto15V}(c).
\begin{figure}
    \centering
        \includegraphics[width=0.75\linewidth]{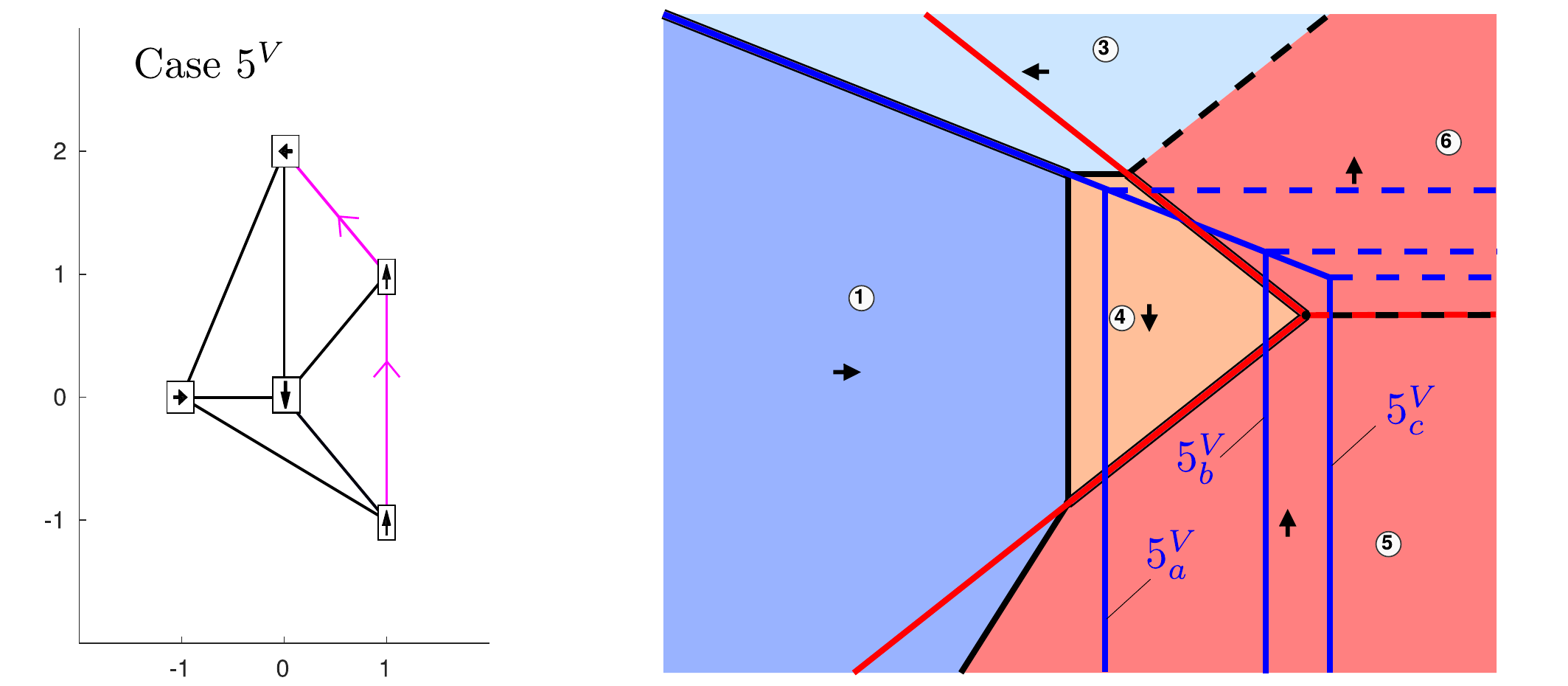}
        \caption{The subdivision associated with case $5^V$ and the associated tropical curves.  }
\figlab{tropauto5V}
\end{figure}

\begin{figure}
\centering
   \begin{subfigure}{0.495\textwidth}
    \centering
        \includegraphics[width=0.96\linewidth]{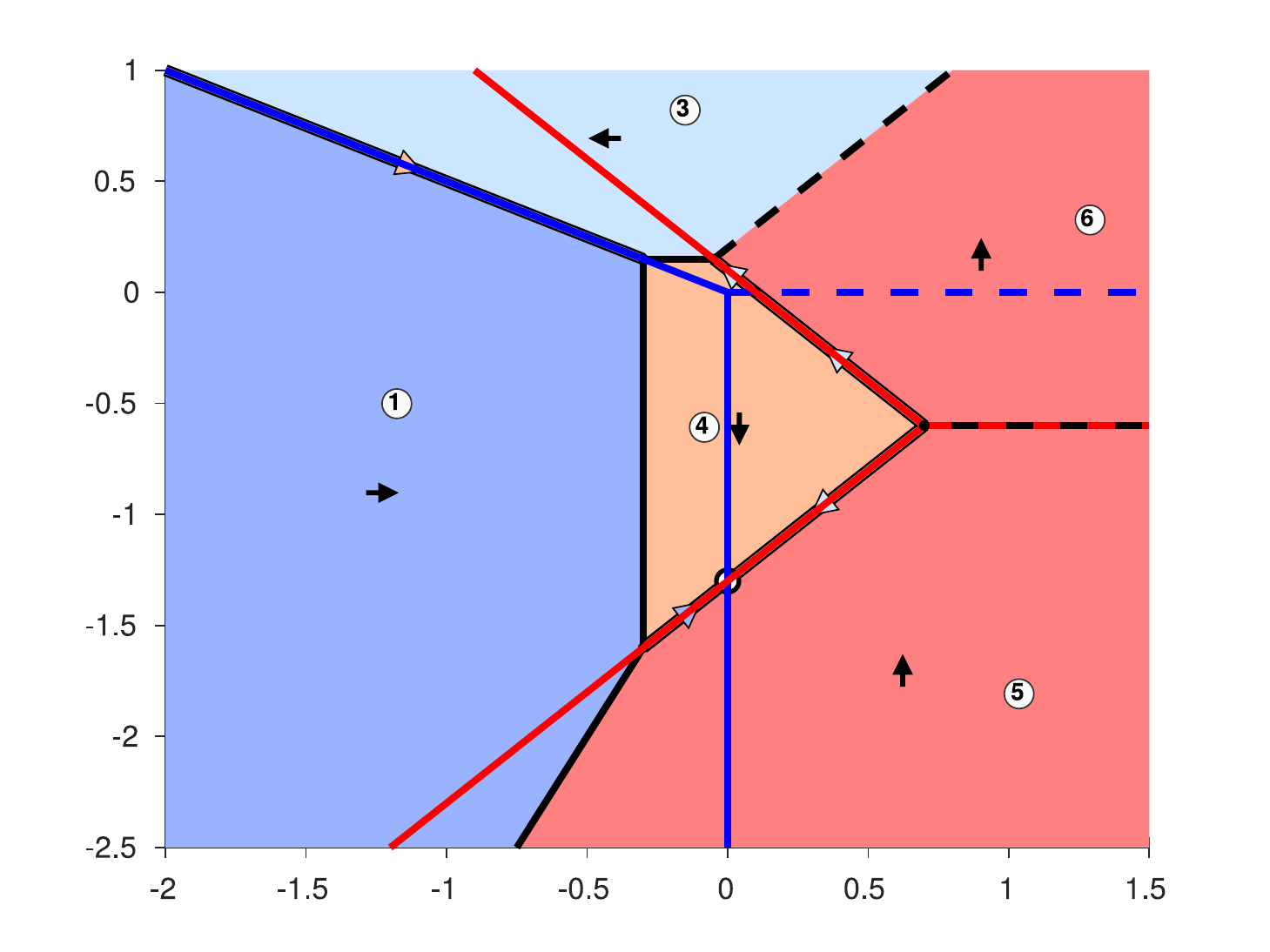}
        \caption{$5^V_a$}
    \end{subfigure}%
    \begin{subfigure}{0.495\textwidth}
    \centering
        \includegraphics[width=0.96\linewidth]{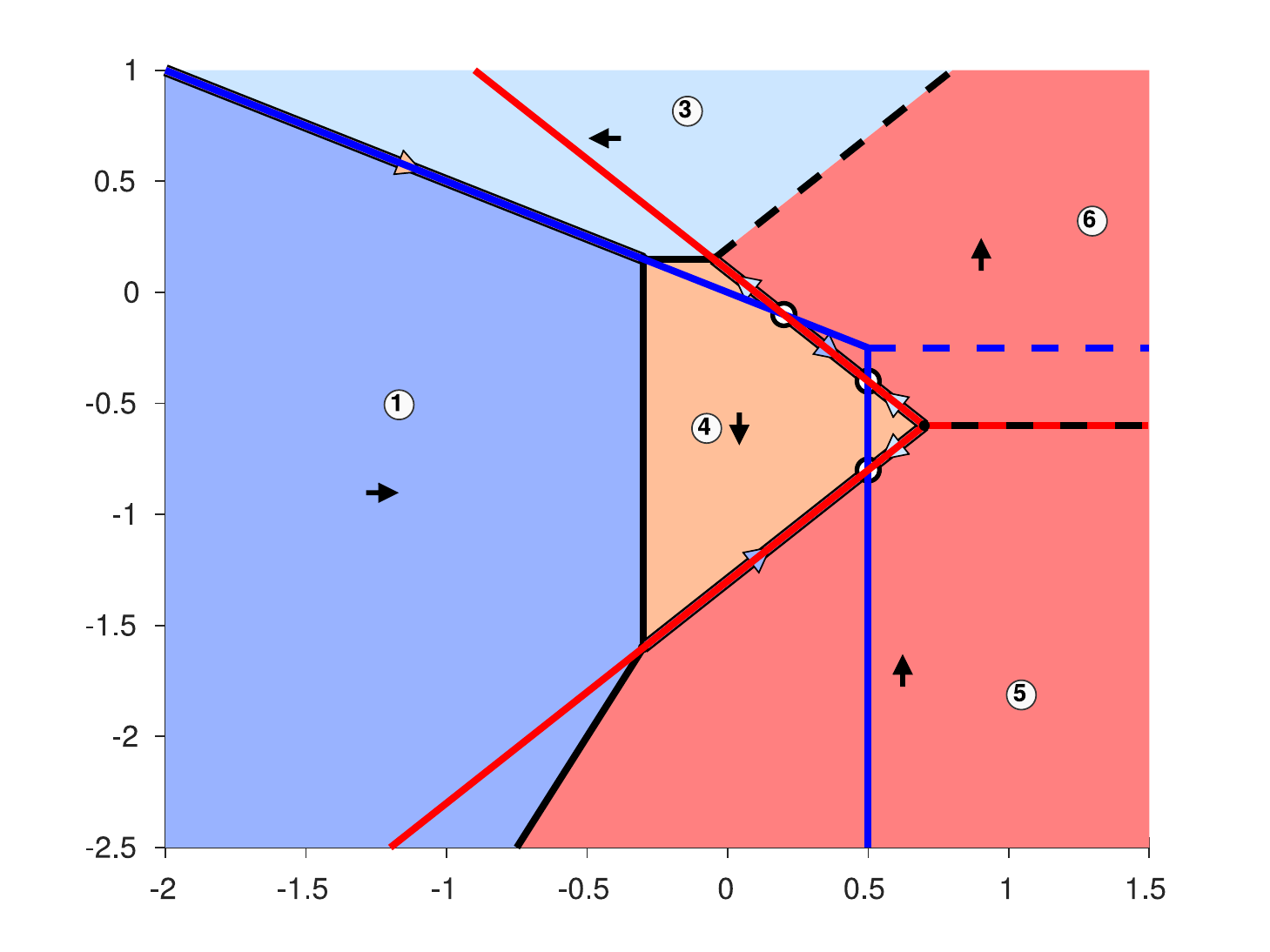}
        \caption{$5^V_b$}
    \end{subfigure}
    \begin{subfigure}{0.495\textwidth}
    \centering
        \includegraphics[width=0.96\linewidth]{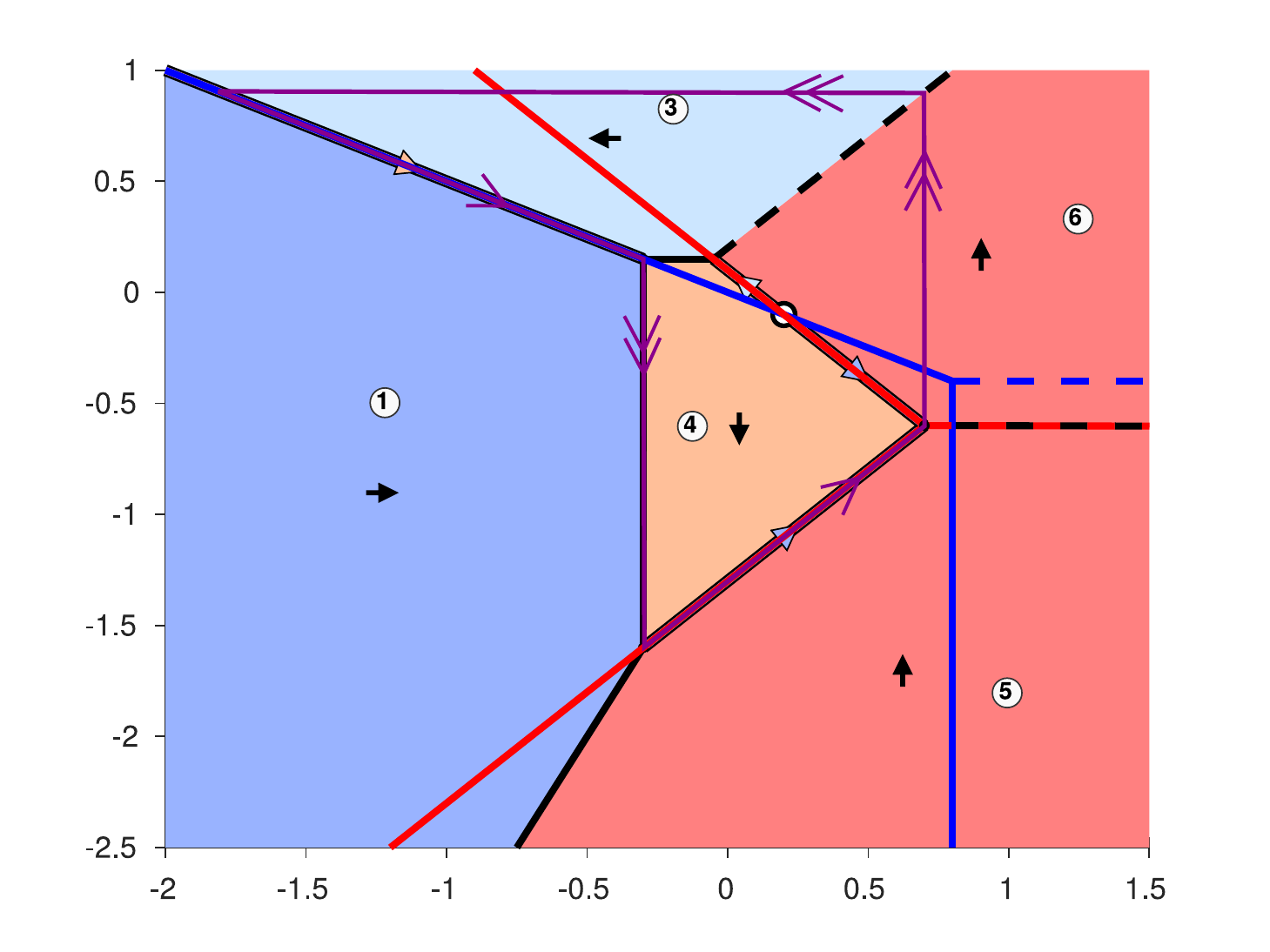}
        \caption{$5^V_c$}
    \end{subfigure}
        \caption{Phase portraits in the case $5^V$, see \figref{tropauto5V}. The values of the tropical coefficients are: $\alpha_1 = 0$, $\alpha_3=0$, $\alpha_4=0.3$, $\alpha_5=-1$ and $\alpha_6 = 0.2$ and $\alpha_2=0$ in (a), $\alpha_2=-0.5$ in (b)  and $\alpha_4=-0.8$ in (c). There is a limit cycle (purple) in (c).}
\figlab{genauto15V}
\end{figure}
\subsection{Completing the proof of \propref{genauto}}
In conclusion, we have found precisely two different phase portraits of $1^H$, one single phase portrait for each of $1^V$, $2$, and $3$, two different phase portraits for each of $4^H$ and $4^V$, and finally, three different phase portraits for each of $5^H$ and $5^V$. This gives $15$ in total, as claimed by \propref{genauto}. Notice that there is only a limit cycle in the case $5^V_c$ and it is unique.
\begin{remark}
 The $15$ different structurally stable phase portraits of the generalized autocatalator model are only determined by the tropical curves $\mathcal T^{\mathcal I}$, $\mathcal T^{\mathcal U}$ and $\mathcal T^{\mathcal V}$ and the different intersections of the two latter ones. In particular, there are no crossing cycles  and no structurally unstable separatrix connections that divide the $\alpha$-space into further equivalence classes. 
 \end{remark}
 \begin{remark}
 For the ``original'' tropical autocatalator, defined by \eqref{tropauto}, having only the single tropical coefficient  $\alpha-1$ of $F_1$ as a free parameter, there are four distinct structurally stable phase portraits (up to equivalence). These correspond to cases $3$ ($\alpha>1$), $4^V_a$ ($\alpha\in (\frac12,1)$), $5^V_a$ ($\alpha<0$) and $5^V_c$ ($\alpha\in (0,\frac12)$), compare with \secref{auto} and \figref{tropauto} and \figref{tropauto2}. (Notice that \eqref{tropauto} is special in the sense that $\mathcal E_{2,3}^{\mathcal U}\cap \mathcal E_{5,6}^\mathcal U\subset \{v=0\}$ is non-empty for all $\alpha$. Although this does not play a role for the classification of structurally stable systems, this property is not the case in \figref{genauto15V}. )
\end{remark}

\section{Discussion}\seclab{discussion}
In this paper, we have introduced the notion of a tropical dynamical system by building upon \rspp{the approach initiated in} \cite{portisch2021novel} and using a differential inclusion framework. This led to a (new) notion of equivalence of tropical dynamical systems and in our main result, \thmref{thm2}, we showed that there are finitely many structurally stable phase portraits. In contrast, it is known for \textit{quadratic} polynomial vector-fields that there are $44$ different phase potraits on the Poincar\'e sphere, but only modulo limit cycles \cite{art1998a}. A similar result for general polynomial systems, would depend upon a solution to Hilbert's 16th problem \cite{art2021a}.


\rspp{In future work, we aim to relate our notion of equivalence to the usual notion of equivalence. For example, we would like to study the other direction in \lemmaref{usual}. We are also interested in enumerating the structurally stable phase portraits for each $N$ (which -- following the proof of \thmref{thm2}  -- is essentially ``computable''). In fact, we believe that there are many directions for future research. We elaborate on a few additional ones below. }
%

\subsection{Limit cycles}
Firstly, since separatrix connections can be structurally stable (for zero splitting constants), we have not bounded the number of limit cycles on the finitely many structurally stable phase portraits,  see \figref{foldfold} \rspp{(and the figure caption for details) as well as} \remref{hybrid}. 
It requires a more thorough description of the case of zero splitting constant for the description of canard-like phenomena as in \figref{foldfold}; the larger cycle in \figref{foldfold} is reminiscent of a canard cycle of jump type, see \cite{de2021a}. (Notice that there are more than just two equivalence classes of limit cycles in \figref{foldfold}. Indeed, as solutions of \eqref{uvtrop}, we can (without further assumptions) have (pathological) periodic solutions of arbitrary period, e.g. going $n$ times around the inner red limit cycles and then subsequently $m$ times around the outer red limit cycles before repeating itself.)

Having said that, the number of \textit{crossing limit cycles} is obviously bounded on $\TDSN$ since the number of cycles in the planar crossing graphs $\mathcal G$ provides an upper bound, see \lemmaref{graph}. It would be interesting to determine the maximal number of crossing cycles in future work. This is essentially a graph theoretical problem.  (The upper bound on general directed planar graphs in \cite{aldred2008a} is not appropriate in the present context; for one thing, since orbits of the crossing flow only consist of horizontal and vertical line segments, the path length of cycles of $\mathcal G$ are greater than or equal to four, see \lemmaref{graph}). In fact, we believe that the graph theoretical viewpoint can be expanded in such a way that $\mathcal G$ also encodes sliding cycles. 

On the other hand, the system \eqref{uvEqn} is a singular perturbation problem for $\epsilon\to 0$. The reference \cite{MR4053594} realizes the asymptotic bound \eqref{lowerboundHN} on the number of limit cycles through canard cycles in slow-fast systems. From this perspective, we do not believe that it is reasonable to expect that a bound on the number of limit cycles on the set of structurally stable tropical dynamical systems will lead to an improvement beyond \eqref{lowerboundHN} (see discussion about the connection to $\epsilon>0$ below). In fact, disregarding the possibility of canards as in \figref{foldfold} on the set of structurally stable systems, so that at most one tropical limit cycle passes through each \rsp{tropical vertex}, then \textit{the number of sliding limit cycles is bounded by the number of triangles  in  $\mathcal S^{\mathcal I}$}, which is given by
\begin{align*}
 2\left[\frac12 (N+1)(N+4)\right]-(3N+2)-2 = N(N+2),
\end{align*}
see \cite[Lemma 3.1.3]{de2010a} and \cite{andreas}. Here the square bracket is the number points in the point configuration, see \eqref{NIpoints}, whereas the other bracket ($3N+2$) is the number of points on the boundary of the Newton polygon. In his thesis \cite{andreas}, the second author constructed a family of tropical dynamical systems with 
\begin{align*}
 \frac12 (N-2)(N-3),
\end{align*}
many sliding limit cycles for each degree $N\ge 4$. The family is constructed by duplicating (using the labelled subdivision) a certain template system for $N=4$. For further details, we refer to the script
\begin{center}\verb#lower_bound_script.m# 
\end{center}
available through the link provided in \secref{TPP},
for running the program \verb#Tropical# \verb#Phase# \verb#Portrait#. This script constructs the tropical phase portrait of the family for each $N\ge 4$. \rsp{We expect that these limit cycles} can be perturbed into limit cycles of \eqref{uvEqn} (and therefore \eqref{xyPQ}) of relaxation type for all $0<\epsilon\ll 1$, see also \conjref{per1} and the discussion below. 

Although these bounds are not competitive with \eqref{lowerboundHN}, we do believe that pursuing such bounds in the context of tropical dynamical systems is meaningful nonetheless. Indeed, ``optimal'' lower bounds for the number of limit cycles in polynomial slow-fast systems occur in a very narrow parameter regime (due to canards). In fact, this is a general observation: When we pick a polynomial system at ``random'', then we ``do not tend'' to see limit cycles at the order of \eqref{lowerboundHN}. (Here ``do not tend'' is very imprecise but refers to something strictly less than probability one, since the limit cycles of \cite{MR4053594} are hyperbolic.) The well-known example by Songling of a quadratic system with four limit cycles, see \cite{songling}, is also very sensitive to variations of the parameters. This is of course just an observation, difficult (if not outright impossible) to make precise as a general mathematical result on \eqref{xyPQ}. However, we speculate that tropical dynamical systems on the set of structurally stable systems offer a route for addressing this observation.

Moreover, as the phase portrait of a tropical dynamical system is determined by linear inequalities on the tropical coefficients, tropical dynamical systems essentially offers a framework where Coppel's problem of characterizing
the phase portraits of quadratic systems by means
of algebraic inequalities on the coefficients, see \cite{coppel1965a}, is solvable on any degree. This is in contrast to the classical case, see \cite[Section 4.14]{perko1991a}, where such a classification \cite{dumortier1991a} (if possible at all) would require nonanalytic inequalities. 


\begin{figure}[H]
    \centering
    \begin{subfigure}{0.5\textwidth}
    \centering
        \includegraphics[width=0.95\linewidth]{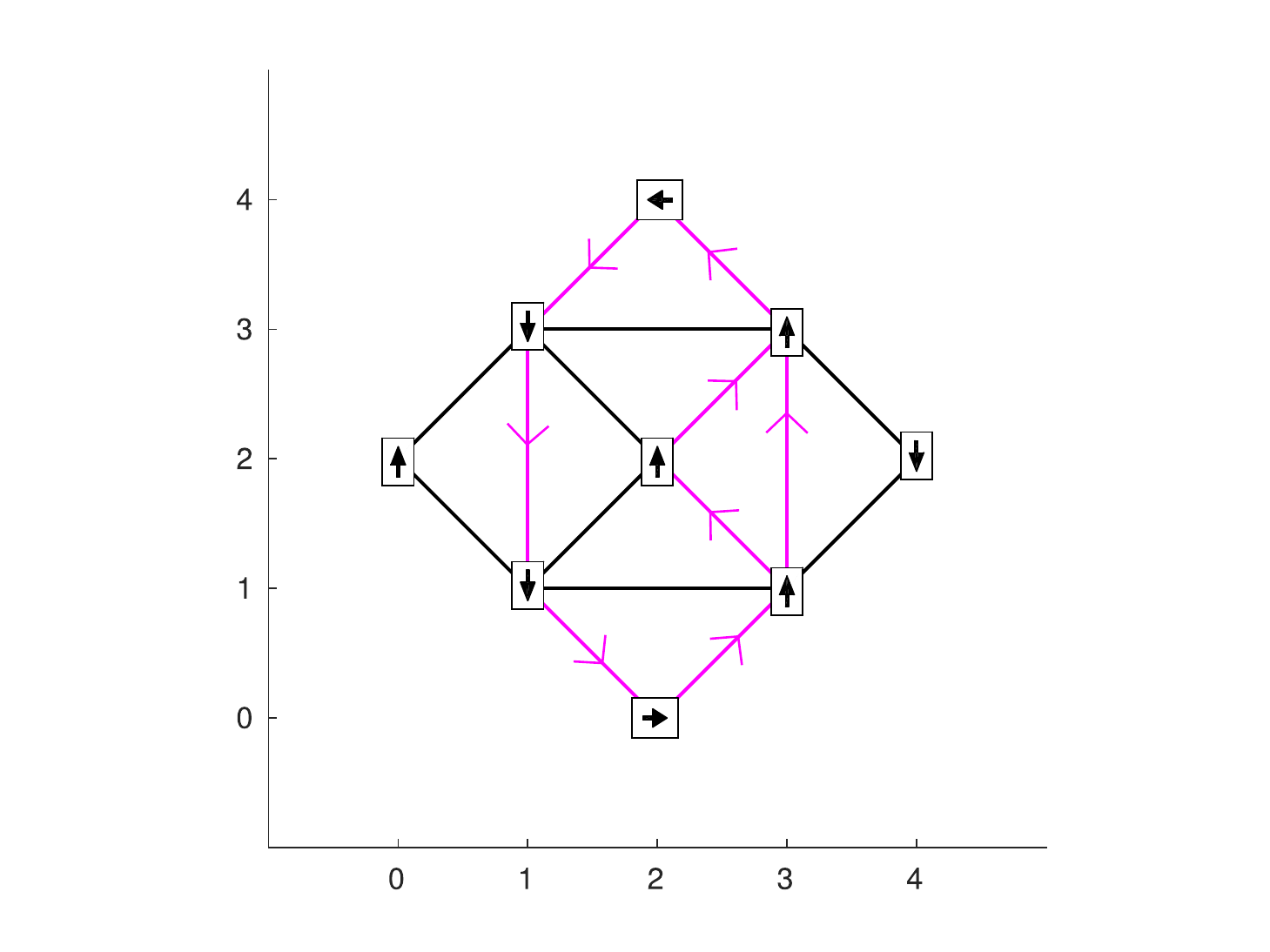}
        \caption{}
    \end{subfigure}%
    \begin{subfigure}{0.5\textwidth}
    \centering
        \includegraphics[width=0.94\linewidth]{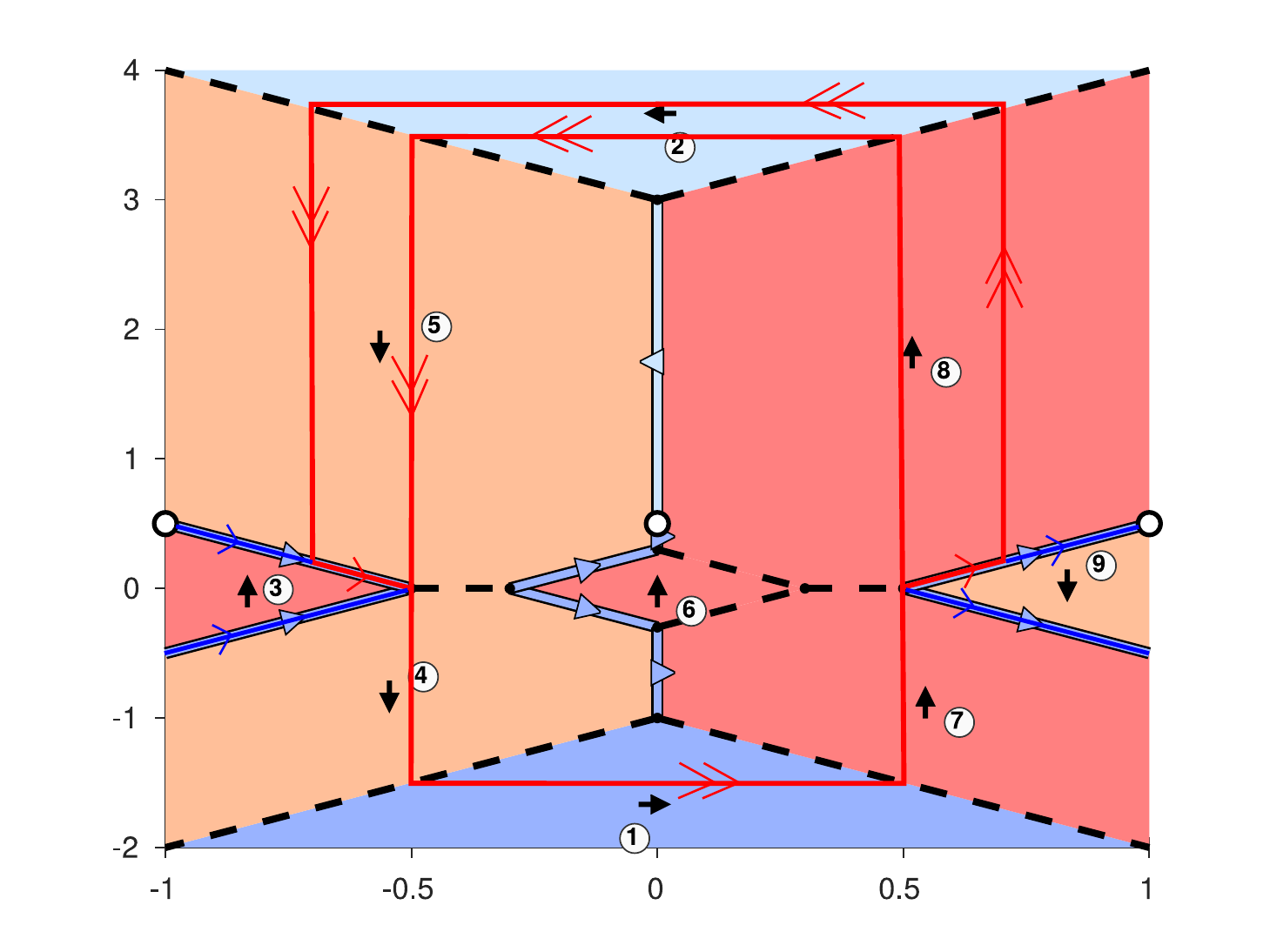}
        \caption{}
    \end{subfigure}
    \caption{An example with infinitely many limit cycles. In (a): the subdivision and the associated crossing graph $\mathcal G$. In (b): The associated phase portrait. The red orbits are examples of limit cycles, whereas the blue orbits are examples of orbits having the limit cycles as their $\omega$- or $\alpha$-limit sets. The situation in (b) is structurally unstable, since the \rsp{tropical vertices} $\mathcal P_{345}$ and $\mathcal P_{789}$ can move independently. Consequently, the heteroclinic connection between $\mathcal P_{345}$ and $\mathcal P_{789}$ can be broken (its splitting constant $b$ is nonzero). To bound the number of limit cycles on the set of structurally stable systems, we need to show that all cases of this type (and others) are structurally unstable.   }
    \figlab{foldfold}
    \end{figure}

%
\subsection{Tropicalization of the full plane}
Secondly, there is the issue of tropicalization of \eqref{xyPQ} in the full plane $\mathbb R^2$. Our methods and results apply to each quadrant. Specifically, we can tropicalize \eqref{xyPQ} in each quadrant by setting
\begin{equation}\eqlab{uvpm}
\begin{aligned}
 u_\text{sign(x)}&=\epsilon \log (\sign(x) x),\\
 v_\text{sign(y)}&=\epsilon \log (\sign(y) y),
\end{aligned}
\end{equation}
for $(x,y)\ne (0,0)$. However, the present paper does not deal with the connection problem that arises when matching across the quadrants. By \eqref{uvpm}, we have coordinates $(u_+,v_+)$ in the first quadrant, $(u_-,v_+)$ in the second quadrant and so on. It is clear that the matching between the first quadrant and the second has to be done along $u_\pm \to -\infty$. This naturally leads to a compactification of each of the four copies of $\mathbb R^2\ni (u_\pm,v_\pm)$ as $(\mathbb R\cup \{-\infty\})^2$. This is also very natural from the perspective that the tropical monomials themselves take values in $\mathbb R\cup \{-\infty\}$.  We will explore this further in future work. 

\subsection{Extensions to $\mathbb R^n$}
Thirdly, there is the problem of extending this theory to higher dimensions. Much of the initial development translates directly, including the definition of crossing and sliding. This is even the basis for the computational methods \cite{kruff2021a} based on tropical geometry. In this regard, it is important to highlight that there is a recursive structure to these systems, in the sense that along a higher dimensional switching manifold of transversal nullcline-type, there is a tropical system of dimension one less. In particular, in three dimensions we can (locally) reduce to a planar tropical dynamical systems, as studied in the present paper, along two-dimensional sliding switching manifolds of transversal nullcline-type. 

\subsection{\rsp{The singular perturbation problem defined by $0<\epsilon\ll 1$}}\seclab{perturb}
Finally, there is the problem of addressing the connection between the tropical dynamical system and the associated smooth system \eqref{uvEqn} for all $0<\epsilon\ll 1$.  We leave this to future work, but do state the following:
\begin{conj}\conjlab{per1}
Suppose that a tropical dynamical system has a \rsp{tropical edge} $\mathcal E$ that is either of transversal Filippov-, or transversal nullcline sliding type and consider a compact line segment $\mathcal S_0\subset \mathcal E$ upon which $\md_{\Fil}$ or $\md_{\trop}$ is constant. Then for all $0<\epsilon\ll 1$, \eqref{uvEqn} has a locally invariant manifold $\mathcal S_\epsilon$, diffeomorphic to $\mathcal S_0$, of the same stability type, and with $\mathcal S_\epsilon\to \mathcal S_0$ for $\epsilon\to 0$.  The one-dimensional flow on $\mathcal S_\epsilon$ is topologically equivalent to the flow on $\mathcal S_0$ given by $\md_{\Fil}$ or $\md_{\trop}$ in the two cases.
\end{conj}
\begin{conj}
 Suppose that a tropical dynamical system has a hyperbolic crossing cycle $\gamma_0$. Then for all $0<\epsilon\ll 1$, \eqref{uvEqn} has a limit cycle $\gamma_\epsilon$ of the same stability type with $\gamma_\epsilon\to \gamma_0$ for $\epsilon\to 0$.
\end{conj}
\begin{conj}\conjlab{tropeq}
 Suppose that a tropical dynamical system has a \rsp{tropical singularity}  at $(u_0,v_0)$ of either sink-, source- or saddle-type, recall \defnref{sinks}. Then there is a neighborhood $X$ of $(u_0,v_0)$ in which the associated smooth system \eqref{uvEqn} for all $0<\epsilon\ll 1$  has a unique hyperbolic \rspp{singularity} at $(u_\epsilon,v_\epsilon)$, with $\lim_{\epsilon\to 0}(u_\epsilon,v_\epsilon)=(u_0,v_0)$, of the same type (stable node, unstable node or saddle).
\end{conj}

\rsp{We feel certain that} these results can be obtained using blowup and Fenichel's theory, see  \cite{fen1,fen2,fen3,jones_1995} and \cite{jelbart2021a,kristiansen2015a,kristiansen2020a} for applications of these methods in smooth systems approaching nonsmooth ones. We also refer to \cite{portisch2021novel}, where a version of \conjref{per1} is proven in the context of the Michaelis-Menten model. 

It is also possible to perturb certain sliding limit cycles, like those in \figref{tropauto}.
Nonhyperbolic crossing cycles and hybrid points, on the other hand, do not persist (in general) as similar objects for \eqref{uvEqn} for $0<\epsilon\ll 1$. It is straightforward to construct examples, see also \remref{hybridpoint}. A consequence of this last fact is that structural stability of the tropical system does not imply structural stability of  \eqref{uvEqn} (in the usual sense) for all $0<\epsilon\ll 1$, at least not without imposing additional assumptions. 

\textbf{Acknowledgement}. The first author would like to offer genuine thanks to Peter Szmolyan for introducing the subject of tropical geometry and for providing valuable insights into its potential applications in chemical reaction networks. The first author also thanks Sam Jelbart for helpful discussions. Finally, both  authors thank the students Jonas Dammann, Thomas Haastrup, Victor Hansen (all project students at the bachelor level) and Daniel Munch Nielsen (master thesis, see \cite{daniel}) for their contribution, which helped plant the seed for the development of the theory.

\bibliography{refs}

\begin{thebibliography}{10}

\bibitem{aldred2008a}
R.~E.~L. Aldred and C.~Thomassen.
\newblock On the maximum number of cycles in a planar graph.
\newblock {\em Journal of Graph Theory}, 57(3):255--264, 2008.

\bibitem{MR4053594}
M.~J. \'{A}lvarez, B.~Coll, P.~De~Maesschalck, and R.~Prohens.
\newblock Asymptotic lower bounds on {H}ilbert numbers using canard cycles.
\newblock {\em J. Differential Equations}, 268(7):3370--3391, 2020.

\bibitem{art1998a}
J.~C. Artés, R.~E. Kooij, and Jaume Llibre.
\newblock Structurally stable quadratic vector fields.
\newblock {\em Memoirs of the American Mathematical Society}, 134(639):VIII--+,
  1998.

\bibitem{art2021a}
J.~C. Artés, M.~C. Mota, and A.~C. Rezende.
\newblock Structurally unstable quadratic vector fields of codimension two:
  Families possessing a finite saddle-node and an infinite saddle-node.
\newblock {\em Electronic Journal of Qualitative Theory of Differential
  Equations}, 2021(35):35, 2021.

\bibitem{aubin1984a}
J.~P. Aubin and A.~Cellina.
\newblock {\em Differential Inclusions : Set-Valued Maps and Viability Theory}.
\newblock Springer, 1984.

\bibitem{baker2016a}
M.~Baker, S.~Payne, and J.~Rabinoff.
\newblock Nonarchimedean geometry, tropicalization, and metrics on curves.
\newblock {\em Algebraic Geometry}, 3(1):63--105, 2016.

\bibitem{braaksma1992a}
B.~L.~J. Braaksma.
\newblock Multisummability of formal power-series solutions of nonlinear
  meromorphic differential-equations.
\newblock {\em Annales De L Institut Fourier}, 42(3):517--540, 1992.

\bibitem{brauer2012a}
F.~Brauer and C.~Castillo-Chavez.
\newblock {\em Mathematical Models in Population Biology and Epidemiology}.
\newblock Springer Science+Business Media, LLC, 2012.

\bibitem{broucke2001a}
M.~E. Broucke, C.~C. Pugh, and S.~N. Simic.
\newblock Structural stability of piecewise smooth systems.
\newblock {\em Computational and Applied Mathematics}, 20(1-2):51--89, 2001.

\bibitem{carmona2021a}
V.~Carmona and F.~Fernández-Sánchez.
\newblock Integral characterization for poincaré half-maps in planar linear
  systems.
\newblock {\em Journal of Differential Equations}, 305:319--346, 2021.

\bibitem{carmona2023a}
V.~Carmona, F.~Fernández-Sánchez, and D.~D. Novaes.
\newblock Uniform upper bound for the number of limit cycles of planar
  piecewise linear differential systems with two zones separated by a straight
  line.
\newblock {\em Applied Mathematics Letters}, 137:108501, 2023.

\bibitem{coppel1965a}
W.~A. Coppel.
\newblock A survey of quadratic systems.
\newblock {\em Journal of Differential Equations}, 2(2):293--304, 1965.

\bibitem{de2010a}
J.~A. De~Loera, J.~Rambau, and F.~Santos.
\newblock {\em Triangulations : Structures for Algorithms and Applications}.
\newblock Springer-Verlag Berlin Heidelberg, 2010.

\bibitem{de2021a}
P.~De~Maesschalck, F.~Dumortier, and R.~Roussarie.
\newblock {\em Canard Cycles : From Birth to Transition}.
\newblock Springer International Publishing, 2021.

\bibitem{desoeuvres2022a}
A.~Desoeuvres, P.~Szmolyan, and O.~Radulescu.
\newblock Qualitative dynamics of chemical reaction networks: An investigation
  using partial tropical equilibrations.
\newblock {\em Lecture Notes in Computer Science (including Subseries Lecture
  Notes in Artificial Intelligence and Lecture Notes in Bioinformatics)},
  13447:61--85, 2022.

\bibitem{Bernardo08}
M.~di~Bernardo, C.~J. Budd, A.~R. Champneys, and P.~Kowalczyk.
\newblock {\em Piecewise-smooth Dynamical Systems: Theory and Applications}.
\newblock Springer Verlag, 2008.

\bibitem{pwsbook}
M.~di~Bernardo, A.R. Chapneys, C.~J. Budd, and P.~Kowalczyk.
\newblock {\em Piecewise-smooth Dynamical Systems}.
\newblock Springer London, 2008.

\bibitem{dumortier1991a}
F.~Dumortier and P.~Fiddelaers.
\newblock Quadratic models for generic local 3-parameter bifurcations on the
  plane.
\newblock {\em Transactions of the American Mathematical Society},
  326(1):101--126, 1991.

\bibitem{dumortier2006a}
F.~Dumortier, J.~Llibre, and J.~C. Art\'es.
\newblock {\em Qualitative theory of planar differential systems}.
\newblock Springer Berlin Heidelberg, 2006.

\bibitem{esteban2021a}
M.~Esteban, J.~Llibre, and C.~Valls.
\newblock {The 16th Hilbert problem for discontinuous piecewise isochronous
  centers of degree one or two separated by a straight line}.
\newblock {\em Chaos}, 31(4):043112, 2021.

\bibitem{feinberg2019a}
M.~Feinberg.
\newblock {\em {Foundations of chemical reaction network theory}}.
\newblock Springer, 2019.

\bibitem{fen1}
N.~Fenichel.
\newblock Persistence and smoothness of invariant manifolds for flows.
\newblock {\em Indiana University Mathematics Journal}, 21:193--226, 1971.

\bibitem{fen2}
N.~Fenichel.
\newblock Asymptotic stability with rate conditions.
\newblock {\em Indiana University Mathematics Journal}, 23:1109--1137, 1974.

\bibitem{fen3}
N.~Fenichel.
\newblock Geometric singular perturbation theory for ordinary differential
  equations.
\newblock {\em J. Diff. Eq.}, 31:53--98, 1979.

\bibitem{filippov1988differential}
A.F. Filippov.
\newblock {\em Differential Equations with Discontinuous Righthand Sides}.
\newblock Mathematics and its Applications. Kluwer Academic Publishers, 1988.

\bibitem{guardia2011a}
M.~Guardia, T.~M. Seara, and M.~A. Teixeira.
\newblock Generic bifurcations of low codimension of planar filippov systems.
\newblock {\em Journal of Differential Equations}, 250(4):1967--2023, 2011.

\bibitem{HUZAK202334}
R.~Huzak and K.~U. Kristiansen.
\newblock The number of limit cycles for regularized piecewise polynomial
  systems is unbounded.
\newblock {\em Journal of Differential Equations}, 342:34--62, 2023.

\bibitem{ilyashenko2002a}
Y.~Ilyashenko.
\newblock {Centennial history of Hilbert's 16th problem}.
\newblock {\em Bulletin of the American Mathematical Society}, 39(3):301--354,
  2002.

\bibitem{MR2508011}
I.~Itenberg, G.~Mikhalkin, and E.~Shustin.
\newblock {\em Tropical algebraic geometry}, volume~35 of {\em Oberwolfach
  Seminars}.
\newblock Birkh\"{a}user Verlag, Basel, second edition, 2009.

\bibitem{itenberg2000a}
I.~Itenberg and E.~Shustin.
\newblock Singular points and limit cycles of planar polynomial vector fields.
\newblock {\em Duke Mathematical Journal}, 102(1):1--37, 2000.

\bibitem{jeffrey2018a}
M.~R. Jeffrey.
\newblock {\em Hidden dynamics: The mathematics of switches, decisions and
  other discontinuous behaviour}.
\newblock Springer International Publishing, 2018.

\bibitem{jeffrey_geometry_2011}
M.~R. Jeffrey and S.~J. Hogan.
\newblock The geometry of generic sliding bifurcations.
\newblock {\em {SIAM} Review}, 53(3):505--525, January 2011.

\bibitem{jelbart2021a}
S.~Jelbart, K.~U. Kristiansen, P.~Szmolyan, and M.~Wechselberger.
\newblock Singularly perturbed oscillators with exponential nonlinearities.
\newblock {\em Journal of Dynamics and Differential Equations}, pages 1--53,
  2021.

\bibitem{jelbart2021c}
S.~Jelbart, K.~U. Kristiansen, and M.~Wechselberger.
\newblock Singularly perturbed boundary-equilibrium bifurcations.
\newblock {\em Nonlinearity}, 34(11):7371--7314, 2021.

\bibitem{jelbart2021b}
S.~Jelbart, K.~U. Kristiansen, and M.~Wechselberger.
\newblock Singularly perturbed boundary-focus bifurcations.
\newblock {\em Journal of Differential Equations}, 296:412--492, 2021.

\bibitem{jones_1995}
C.~K. R.~T. Jones.
\newblock {\em Geometric Singular Perturbation Theory, Lecture Notes in
  Mathematics, Dynamical Systems (Montecatini Terme)}.
\newblock Springer, Berlin, 1995.

\bibitem{Gucwa2009783}
I.~Kosiuk and P.~Szmolyan.
\newblock Geometric singular perturbation analysis of an autocatalator model.
\newblock {\em Discrete and Continuous Dynamical Systems - Series S},
  2(4):783--806, 2009.

\bibitem{kristiansen2017a}
K.~U. Kristiansen.
\newblock Blowup for flat slow manifolds.
\newblock {\em Nonlinearity}, 30(5):2138--2184, 2017.

\bibitem{kristiansen2020a}
K.~U. Kristiansen.
\newblock The regularized visible fold revisited.
\newblock {\em Journal of Nonlinear Science}, 30(6):2463--2511, 2020.

\bibitem{kristiansen2015a}
K.~U. Kristiansen and S.~J. Hogan.
\newblock On the use of blowup to study regularizations of singularities of
  piecewise smooth dynamical systems in $\mathbb{R}^3$.
\newblock {\em {SIAM Journal on Applied Dynamical Systems}}, 14(1):382--422,
  2015.

\bibitem{kristiansen2018a}
K~U. Kristiansen and S.~J. Hogan.
\newblock Resolution of the piecewise smooth visible-invisible two-fold
  singularity in $\mathbb{R}^3$ using regularization and blowup.
\newblock {\em Journal of Nonlinear Science}, 29(2):723--787, 2018.

\bibitem{uldall2021a}
K.~U. Kristiansen and P.~Szmolyan.
\newblock Relaxation oscillations in substrate-depletion oscillators close to
  the nonsmooth limit.
\newblock {\em Nonlinearity}, 34(2):1030--1083, 2021.

\bibitem{kruff2021a}
N.~Kruff, C.~L\"uders, O.~Radulescu, T.~Sturm, and S.~Walcher.
\newblock Algorithmic reduction of biological networks with multiple time
  scales.
\newblock {\em Mathematics in Computer Science}, 15(3):499--534, 2021.

\bibitem{krupa2001a}
M.~Krupa and P.~Szmolyan.
\newblock Relaxation oscillation and canard explosion.
\newblock {\em Journal of Differential Equations}, 174(2):312--368, 2001.

\bibitem{li2003a}
J.~Li.
\newblock Hilbert's 16th problem and bifurcations of planar polynomial vector
  fields.
\newblock {\em International Journal of Bifurcation and Chaos in Applied
  Sciences and Engineering}, 13(1):47--106, 2003.

\bibitem{li2021a}
T.~Li and J.~Llibre.
\newblock {On the 16th Hilbert Problem for Discontinuous Piecewise Polynomial
  Hamiltonian Systems}.
\newblock {\em Journal of Dynamics and Differential Equations}, pages 1--16,
  2021.

\bibitem{litvinov2007a}
G.~L. Litvinov.
\newblock Maslov dequantization, idempotent and tropical mathematics: A brief
  introduction.
\newblock {\em Journal of Mathematical Sciences}, 140(3):426--444, 2007.

\bibitem{litvinov1996a}
G.~L. Litvinov and V.~P. Maslov.
\newblock Idempotent mathematics: A correspondence principle and its
  applications to computing.
\newblock {\em Russian Mathematical Surveys}, 51(6):1210--1211, 1996.

\bibitem{llibre2013a}
J.~Llibre, M.~A. Teixeira, and J.~Torregrosa.
\newblock Lower bounds for the maximum number of limit cycles of discontinuous
  piecewise linear differential systems with a straight line of separation.
\newblock {\em International Journal of Bifurcation and Chaos}, 23(4):1350066,
  2013.

\bibitem{maclagan2015introduction}
D.~Maclagan and B.~Sturmfels.
\newblock {\em Introduction to Tropical Geometry}.
\newblock Graduate Studies in Mathematics, 161. American Mathematical Society,
  2015.

\bibitem{morrison2020a}
R.~Morrison.
\newblock Tropical geometry.
\newblock {\em A Project-based Guide To Undergraduate Research in Mathematics},
  pages 63--105, 2020.

\bibitem{daniel}
D.~M. Nielsen.
\newblock Tropical geometry and geometric singular perturbation theory with
  applications in biology, Technical University of Denmark, Master thesis
  (supervised by Kristiansen, K.~U.), 2021.

\bibitem{noel2012a}
V.~Noel, D.~Grigoriev, S.~Vakulenko, and O.~Radulescu.
\newblock Tropical geometries and dynamics of biochemical networks application
  to hybrid cell cycle models.
\newblock {\em Electronic Notes in Theoretical Computer Science}, 284:75--91,
  2012.

\bibitem{perko1991a}
L.~Perko.
\newblock {\em Differential equations and dynamical systems}, volume~7.
\newblock Springer-Verlag, Berlin, 1991.

\bibitem{petrov1992a}
V.~Petrov, S.K. Scott, and K.~Showalter.
\newblock Mixed-mode oscillations in chemical-systems.
\newblock {\em Journal of Chemical Physics}, 97(9):6191--6198, 1992.

\bibitem{portisch2021novel}
S.~Portisch.
\newblock A novel approach to dimension reduction in enzyme kinetics, Vienna
  University of Technology, Diploma thesis (supervised by Szmolyan, P.), 2017.

\bibitem{DBLP:conf/casc/SamalGFR15}
S.~S. Samal, D.~Grigoriev, H.~Fr{\"{o}}hlich, and O.~Radulescu.
\newblock Analysis of reaction network systems using tropical geometry.
\newblock In {\em Computer Algebra in Scientific Computing - 17th International
  Workshop, {CASC} 2015, Aachen, Germany, September 14-18, 2015, Proceedings},
  pages 424--439, 2015.

\bibitem{andreas}
A.~H. Sarantaris.
\newblock Tropical dynamics: Dimension reduction and graph representation of
  multi-scale polynomial systems, Technical University of Denmark, Bachelor
  thesis (supervised by Kristiansen, K.~U.), 2021.

\bibitem{soliman2014a}
S.~Soliman, F.~Fages, and O.~Radulescu.
\newblock A constraint solving approach to model reduction by tropical
  equilibration.
\newblock {\em Algorithms for Molecular Biology}, 9(1):24, 2014.

\bibitem{songling}
S.~Songling.
\newblock A conrete example of the existence of four limit cycles for plane
  quadratic systems.
\newblock {\em Sci. Scientia}, 2:153--158, 1980.

\bibitem{Sotomayor96}
J.~Sotomayor and M.~A. Teixeira.
\newblock Regularization of discontinuous vector fields.
\newblock In {\em Proceedings of the International Conference on Differential
  Equations, Lisboa}, pages 207--223, 1996.

\bibitem{MR1925796}
B.~Sturmfels.
\newblock {\em Solving systems of polynomial equations}, volume~97 of {\em CBMS
  Regional Conference Series in Mathematics}.
\newblock Published for the Conference Board of the Mathematical Sciences,
  Washington, DC; by the American Mathematical Society, Providence, RI, 2002.

\bibitem{viro2001a}
O.~Viro.
\newblock Dequantization of real algebraic geometry on logarithmic paper.
\newblock {\em European Congress of Mathematics, Vol I}, 201:135--146, 2001.

\bibitem{MR1036837}
O.~Y. Viro.
\newblock Real plane algebraic curves: constructions with controlled topology.
\newblock {\em Algebra i Analiz}, 1(5):1--73, 1989.

\end{thebibliography}
\bibliographystyle{plain}
\newpage
\appendix 
\end{document}